\documentclass{memo-l}

\usepackage{amsxtra,amssymb,amsmath,url,amsrefs, pdfsync, mathrsfs}
\usepackage{array}
\usepackage[utf8]{inputenc}
\usepackage{eucal}
\usepackage{scrtime}
\usepackage[colorlinks]{hyperref}
\usepackage{scalerel,stackengine}
\usepackage{pict2e}
\usepackage{longtable}
\stackMath

\newcommand\widecheck[1]{%
\savestack{\tmpbox}{\stretchto{%
  \scaleto{%
    \scalerel*[\widthof{\ensuremath{#1}}]{\kern-.4pt\bigwedge\kern-.4pt}%
    {\rule[-\textheight/2]{1ex}{\textheight}}
  }{\textheight}%
}{0.5ex}}%
\stackon[2pt]{#1}{\scalebox{-1}{\tmpbox}}%
}

\newcommand{\bessel}[1]{\widecheck{{#1}}}

\newcommand{\fourier}[1]{\widehat{{#1}}}

\renewcommand{\leq}{\leqslant}
\renewcommand{\geq}{\geqslant}

\let\polishl\l
\newcommand{\maks}{Radziwi\polishl\polishl}
\renewcommand{\l}{\ell}

\numberwithin{section}{chapter}
\numberwithin{equation}{chapter}

\newcommand{\uple}[1]{\text{\boldmath${#1}$}}

\def\stacksum#1#2{{\stackrel{{\scriptstyle #1}}
{{\scriptstyle #2}}}}

\newcommand{\bfalpha}{\uple{\alpha}}
\newcommand{\bfbeta}{\uple{\beta}}

\newcommand{\bfx}{\uple{x}}

\newcommand{\ftchi}{f\otimes\chi}

\newcommand{\fotg}{f\otimes g}
\newcommand{\gtchi}{g\otimes\chi}
\newcommand{\symf}{\mathrm{Sym}^2f}
\newcommand{\sym}{\mathrm{Sym}}

\newcommand{\Cc}{\mathbf{C}}

\newcommand{\Aa}{\mathbf{A}}
\newcommand{\AQ}{\mathbf{A}_{\Qq}}

\newcommand{\Zz}{\mathbf{Z}}

\newcommand{\Rr}{\mathbf{R}}
\newcommand{\Gg}{\mathbf{G}}

\newcommand{\Qq}{\mathbf{Q}}

\newcommand{\Fq}{{\mathbf{F}_q}}
\newcommand{\Fqt}{{\mathbf{F}^\times_q}}
\newcommand{\Ff}{\mathbf{F}}

\newcommand{\bFq}{\bar{\Ff}_q}

\newcommand{\mcQ}{\mathcal{Q}}
\newcommand{\mcR}{\mathcal{R}}

\newcommand{\HYPK}{\mathcal{K}\ell}

\newcommand{\proba}{\mathbf{P}}
\newcommand{\expect}{\mathbf{E}}

\newcommand{\spr}{\ov s}

\newcommand{\mods}[1]{\,(\mathrm{mod}\,{#1})}

\newcommand{\what}{\widehat}
\newcommand{\wtilde}{\widetilde}

\DeclareMathOperator{\hypk}{Kl}

\newcommand{\ra}{\rightarrow}

\DeclareMathOperator{\rk}{rk}
\DeclareMathOperator*{\res}{res}

\DeclareMathOperator*{\ord}{ord}

\DeclareMathOperator{\Reel}{\Re}
\DeclareMathOperator{\syms}{Sym^{2}}

\DeclareMathOperator{\Kl}{\mathrm{Kl}}
\DeclareMathOperator{\Kld}{\mathrm{Kl}_{2}}

\DeclareMathOperator{\supp}{supp}

\DeclareMathOperator{\Tr}{tr}

\DeclareMathOperator{\cond}{\mathbf{c}}

\let\Re\relax
\DeclareMathOperator{\Re}{\mathfrak{Re}}

\newcommand{\eps}{\varepsilon}
\renewcommand{\rho}{\varrho}

\DeclareMathOperator{\SL}{SL}
\DeclareMathOperator{\GL}{GL}
\DeclareMathOperator{\PGL}{PGL}

\DeclareMathOperator{\Ort}{O}
\DeclareMathOperator{\SU}{SU}
\DeclareMathOperator{\Un}{U}

\newcommand{\demi}{{\textstyle{\frac{1}{2}}}}

\newcommand{\sheaf}[1]{\mathcal{{#1}}}

\DeclareMathSymbol{\gena}{\mathord}{letters}{"3C}
\DeclareMathSymbol{\genb}{\mathord}{letters}{"3E}

\def\sump{\mathop{\sum \Bigl.^{+}}\limits}
\def\summ{\mathop{\sum \Bigl.^{-}}\limits}
\def\sumpm{\mathop{\sum \Bigl.^{\pm}}\limits}

\def\sums{\mathop{\sum \Bigl.^{*}}\limits}

\def\sump{\mathop{\sum \Bigl.^{+}}\limits}
\def\intc{\frac{1}{2i\pi}\mathop{\int}\limits}

\def\sumstar{\mathop{{\sum\nolimits^{\star}}}}

\theoremstyle{plain}
\newtheorem{theorem}{Theorem}[chapter]
\newtheorem*{theorem*}{Theorem}
\newtheorem{lemma}[theorem]{Lemma}
\newtheorem{corollary}[theorem]{Corollary}

\newtheorem{proposition}[theorem]{Proposition}

\theoremstyle{remark}

\newtheorem*{rems}{Remarks}

\theoremstyle{definition}
\newtheorem{convention}[theorem]{Convention}
\newtheorem*{claim}{Claim}
\newtheorem{definition}[theorem]{Definition}

\newtheorem{example}[theorem]{Example}
\newtheorem{remark}[theorem]{Remark}

\newcommand{\mcM}{\mathcal{M}}
\newcommand{\mcL}{\mathcal{L}}

\newcommand{\mcF}{\mathcal{F}}

\newcommand{\mcG}{\mathcal{G}}

\newcommand{\mfa}{\mathfrak{a}}

\newcommand{\lf}{\lambda_f}
\newcommand{\lamg}{\lambda_g}
\newcommand{\vphis}{\varphi^*}
\newcommand{\MT}{\mathrm{MT}}
\newcommand{\ET}{\mathrm{ET}}
\newcommand{\vphi}{\varphi}

\renewcommand{\geq}{\geqslant}
\renewcommand{\leq}{\leqslant}

\newcommand{\refs}{\eqref}

\newcommand{\ov}[1]{\overline{#1}}

\newcommand{\msym}[1]{\Bigl\langle{#1}\Bigr\rangle}
\newcommand\sumsum{\mathop{\sum\sum}\limits}

\newcommand{\expoL}{{3/2}}
\newcommand{\expoLt}{{5/2}}
\newcommand{\expoq}{{1/144}}
\newcommand{\expoLmax}{{1/360}}

\newcommand{\expopropdenom}{{1443}}

\def\msa{\mathcal{A}}
\def\msl{\mathcal{L}}
\def\mcl{\mathcal{L}}

\def\sump{\mathop{{\sum\nolimits^{+}}}}
\def\summ{\mathop{{\sum\nolimits^{-}}}}
\def\sumpm{\mathop{{\sum\nolimits^{\pm}}}}
\def\Kl{\mathrm{Kl}}
\def\supp{\mathop{\mathrm{supp}}}
\def\sgn{\mathop{\mathrm{sgn}}}
\newcommand\bfone{\mathbf{1}}

\begin{document}

\title[The second moment theory of families of $L$-functions]{The
  second moment
  theory \\
  of families of $L$-functions \\[0.5cm] {\Large The case of twisted
    Hecke $L$-functions}}

\author{Valentin Blomer}
\address{Mathematisches Institut, Universit\"at Bonn,
  Endenicher Allee 60, 53115 Bonn, Germany} \email{blomer@math.uni-bonn.de}

\author{\'Etienne Fouvry}
\address{Universit\'e Paris--Saclay, CNRS \\
Laboratoire de math\' ematiques d'Orsay\\ 91405 Orsay\\France}
\email{etienne.fouvry@u-psud.fr}

\author{Emmanuel Kowalski}
\address{ETH Z\"urich -- D-MATH\\
  R\"amistrasse 101\\
  CH-8092 Z\"urich\\
  Switzerland} \email{kowalski@math.ethz.ch}

\author{Philippe Michel} \address{EPFL/SB/TAN, Station 8, CH-1015
  Lausanne, Switzerland } \email{philippe.michel@epfl.ch}

 \author{Djordje Mili\'cevi\'c}
 \address{Department of Mathematics,
Bryn Mawr College,
101 North Mer\-ion Avenue,
Bryn Mawr, PA 19010-2899, U.S.A.}
 \email{dmilicevic@brynmawr.edu}
\date{\today,\ \thistime} 

\author{Will Sawin}
\address{Columbia University, 2990 Broadway, New York, NY, USA 10027}
\email{sawin@math.columbia.edu}





\begin{abstract}
  For a fairly general family of $L$-functions, we survey the known
  consequences of the existence of asymptotic formulas with
  power-saving error term for the (twisted) first and second moments
  of the central values in the family. 
\par
We then consider in detail the important special case of the family of
twists of a fixed cusp form by primitive Dirichlet characters modulo a
prime $q$, and prove that it satisfies such formulas. We  derive
arithmetic consequences: 
\begin{itemize}
\item  a positive proportion of central values
$L(f\otimes\chi,1/2)$ are non-zero, and indeed bounded from below; 
\item there exist many characters $\chi$ for which the central
  $L$-value is very large; 
\item the probability of a large analytic rank decays exponentially
  fast.
\end{itemize}
We finally show how the second moment estimate establishes a special
case of a conjecture of Mazur and Rubin concerning the distribution of
modular symbols.
\end{abstract}

\maketitle

\setcounter{tocdepth}{1}

\tableofcontents

\chapter{The second moment theory of families of \texorpdfstring{$L$-functions}{L-functions}}

\section{General introduction}
\label{sec-general}

\subsection{Families and moments}

In the analytic theory of automorphic forms, many problems are out of
reach, or make little sense, when specialized to single $L$-functions
or modular forms. It has therefore been a very common theme of
research to study \emph{families} of $L$-functions, and to search for
\emph{statistical results} on average over families. This point of
view has led to numerous insights. In fact, it also sometimes provides
a viable approach to questions for individual objects, as in most
works concerning the subconvexity problem for $L$-functions. An
excellent survey of this point of view is that of Iwaniec and
Sarnak~\cite{IwSa2}.

When studying $L$-functions on average, it has emerged from a series
of works in the last ten to fifteen years that a remarkable array of
results can be obtained as soon as one has sufficiently strong
information concerning the first and especially the \emph{second}
moment of the values of the $L$-functions on the critical line. More
precisely, what is often the crucial input needed is ``a bit more''
than the second moment, which is most easily captured in practice by
an asymptotic formula \emph{with power saving} for the second moment,
together with some basic information for individual $L$-functions
(such as versions of the Prime Number Theorem, sometimes for auxiliary
$L$-functions, or bounds for averages or mean-square averages of
coefficients). This phenomenon is of course consistent with
probabilistic intuition: recall for instance that the Law of Large
Numbers only requires the first moment to exist, and that the Central
Limit Theorem only depends on the second moment.

In this section, we will explain the basic principle and describe some
of its applications in a fairly general and informal setting. In the
next sections, we will introduce the particular family that will be
the focus of the remainder of this book, and we will state the precise
new results that we have obtained in that case.

Let $d\geq 1$ be an integer. We interpret here a \emph{family} of cusp
forms of rank $d\geq 1$ as the data, for any integer $N$, of a finite
set $\mathcal{F}_N$ of cusp forms (cuspidal automorphic
representations) on $\GL_d$ (over $\Qq$ for simplicity). Given such a
family $\mathcal{F}$, we obtain probability and average operators
$$
\proba_N(f\in
A)=\frac{1}{|\mathcal{F}_N|}\sum_{\substack{f\in\mathcal{F}_N\\f\in
    A}} 1 \quad\quad \expect_N(T(f))=
\frac{1}{|\mathcal{F}_N|}\sum_{f\in \mathcal{F}_N} T(f)
$$
for any $N$ such that $\mathcal{F}_N$ is not empty. Here $A$ is any
subset of cusp forms on $\GL_d$, and $T$ is any complex-valued
function defined on the set of cusp forms on $\GL_d$. We will
sometimes informally write $f\in\mathcal{F}$ to say that
$f\in\mathcal{F}_N$ for some $N$ (which might not be unique).
\par
We also require that the size of $\mathcal{F}_N$ and the analytic
conductors $q(f)$ of the cusp forms $f\in\mathcal{F}_N$ grow with $N$
in a nice way, say $|\mathcal{F}_N|\asymp N^{\alpha}$ and
$q(f)\asymp N^{\beta}$ for some $\alpha>0$ and $\beta>0$.

The basic invariants are the standard (Godement-Jacquet) $L$-functions
$$
L(f,s)=\sum_{n\geq 1}\lambda_f(n)n^{-s}
$$
associated to a cusp form $f\in \mathcal{F}$, that we always normalize
in this book so that the center of the critical strip is
$s=\demi$. These are indeed often so important that one speaks of
\emph{families of $L$-functions} instead of families of cusp forms.

For any \emph{reasonable} family of $L$-functions, one can make
precise conjectures for the asymptotic behavior as $N\to +\infty$ of
the complex moments
$$
\expect_N\Bigl(L(f,\sigma+it)^k\overline{L(f,\sigma+it)}^l\Bigr),
$$
at any point $\sigma+it\in\Cc$, where $k$ and $l$ are non-negative
integers. Here, ``reasonable'' has no precise generally accepted
formal definition. A minimal requirement is that the family should
satisfy some form of ``local spectral equidistribution''
(see~\cite{KoFam}), which means that for any fixed prime $p$, the
local component at $p$ of the cusp forms $f\in\mathcal{F}_N$ should
become equidistributed with respect to some measure $\mu_p$ as
$N\to+\infty$ in the unitary spectrum $\widehat{\GL}_d(\Qq_p)$ of
$\GL_d(\Qq_p)$. Such a statement is more or less required to express
for instance the arithmetic component of the leading term of the
asymptotic of the moments at $s=\demi$.
\par
Indeed, following the work of Keating and Snaith~\cite{KeSn} and the
ideas of Katz and Sarnak~\cite{KaSa}, for any integer $k\geq 0$, one
expects an asymptotic formula of the type
\begin{equation}\label{eq-moment-conj}
\expect_N\Bigl(|L(f,\demi)|^{2k}\Bigr)\sim a_kg_k (\log N)^{c_k}
\end{equation}
as $N\to +\infty$, where $a_k$ is an arithmetic factor, whereas $g_k$
and $c_k$ are real numbers that depend only on the so-called
``symmetry type'' of the family, and have an interpretation in terms
of Random Matrix Theory. In general, these invariants can be
predicted, based on the local spectral equidistribution properties of
the family. More precisely, one can often deduce the ``symmetry
type'', in the sense of Katz-Sarnak, from the limiting behavior of the
measures $\mu_p$ as $p\to +\infty$ (see~\cite{KoFam}*{\S 9, \S 10},
but note that this line of reasoning wouldn't always work in the case
of ``algebraic'' families~\cite{SST}). From this symmetry type, which
is either unitary, symplectic, or orthogonal (with some variants in
the orthogonal case related to root numbers), one can predict the
values of $g_k$ and $c_k$. These are related to the asymptotic
beahvior of the moments of the value of the characteristic polynomials
of random matrices in families of compact Lie groups of unitary, or
symplectic or orthogonal matrices, in the limit when the size of the
matrices increases. For instance, in a unitary family, we have
$c_k=k^2$ and
$$
g_k=\frac{G(1+k)}{G(1+2k)}=\prod_{j=0}^k\frac{j!}{(j+k)!}
$$
where $G$ is the Barnes function.
\par
Using this information, $a_k$ is an Euler product given by
$$
a_k=\prod_p (1-p^{-1})^{c_k}\int |L_p(\pi,\demi)|^{2k}\ d\mu_p(\pi)
$$
where the integral is over the unitary spectrum of $\GL_d(\Qq_p)$ and
$L_p$ is the local $L$-factor at $p$. (The value of $c_k$ is exactly
such that the Euler product converges).


\begin{remark}
  Families can also be defined with non-uniform weights over finite
  sets, instead of the uniform measure, or can also be continuous
  families with a finite probability measure (typically to consider
  $L$-functions in the $t$-aspect on the critical line). The sets
  $\mathcal{F}_N$ might also be defined only for a subset of the
  integers $N$. This does not affect the general discussion. We will
  see certain variants of this type in the examples below.
\end{remark}

For orientation, here are some examples of families that have been
studied extensively, and that we will refer to in the list of
applications below. One is given for each of the three basic symmetry
types.

\begin{example}
  (1) For $N\geq 1$, let $\mathcal{D}_N$ be the finite set of
  primitive Dirichlet characters modulo $q\leq N$. This is a unitary
  family.

  (2) For $q\geq 17$ prime, let $\mathcal{C}_q$ be the finite set of
  primitive weight $2$ cusp forms of level $q$. This family has
  particularly nice arithmetic applications, because the
  Eichler--Shimura formula implies that
$$
\prod_{f\in\mathcal{F}_q} L(f,s)
$$
is the (normalized) Hasse-Weil $L$-function of the jacobian $J_0(q)$
of the modular curve $X_0(q)$.  This illustrates one way in which
average studies of families of $L$-functions may have consequences for
a single arithmetic object of natural interest. The family
$\mathcal{C}_q$ is of orthogonal type. It is often of interest to
restrict to cusp forms where $L(f,s)$ has a given root number $1$ or
$-1$, which would split into even and odd orthogonal types.

(3) For $N\geq 2$, let $\mathcal{Q}_N$ be the finite set of primitive
real Dirichlet characters modulo $q\leq N$. This is a family of
symplectic type.

(4) Finally, there have been a number of important works recently that
show that many natural families of cusp forms on $\GL_d$, or other
groups, satisfy the basic local spectral equidistribution properties
(see for instance the work of Shin and Templier~\cite{shin-templier}
and surveys by Sarnak--Shin--Templier~\cite{SST} and
Matz~\cite{Matz}).
\end{example}

The assumption of the \emph{second moment theory} of a family of
$L$-functions is that the expected asymptotic formula holds for the
first and second moments on the critical line, with a power-saving in
the error terms with respect to $N$, and polynomial dependency with
respect to the imaginary part. More precisely, we assume that there
exists $\delta>0$ and $A\geq 0$ such that
\begin{equation}\label{eq-bm1}
\expect_N\Bigl(L(f,\demi + it)\Bigr)=
\MT_1(N;t)+O((1+|t|)^AN^{-\delta}),
\end{equation}
and
\begin{equation}\label{eq-bm2}
\expect_N\Bigl(|L(f,\demi +
it)|^2\Bigr)=\MT_2(N;t)+O((1+|t|)^AN^{-\delta}),
\end{equation}
for $N\geq 1$ and $t\in\Rr$. The main terms are polynomials of some
fixed degree in $\log N$, as also predicted by the precise forms of
the moment conjectures (due to Conrey, Farmer, Keating, Rubinstein and
Snaith~\cite{CFKRS}).  In fact, it is not required in practice to know
that the main terms exactly fit the moment conjectures, provided they
are given in sufficiently manageable form for the computations that
will follow (the degree of the polynomial in $\log N$ is of crucial
importance).

As we hinted at above when saying that one needs ``a bit more'', these
estimates are in fact intermediate steps. The really crucial point is
that, if they can be proved with almost any of the currently known
techniques, then it is also possible to improve them to derive
asymptotic formulas for the first and second moments \emph{twisted} by
the coefficients $\lambda_f(\ell)$ of the $L$-functions, namely
\begin{equation}\label{eq-tw1}
\expect_N\Bigl(\lambda_f(\ell)L(f,\demi +
it)\Bigr)=\MT_1(N;t,\ell)+O((1+|t|)^AL^BN^{-\delta}),
\end{equation}
and
\begin{equation}\label{eq-tw2}
  \expect_N\Bigl(\lambda_f(\ell) |L(f,\demi +
  it)|^2\Bigr)=\MT_2(N;t,\ell)+O((1+|t|)^AL^BN^{-\delta}),
\end{equation}
where $1\leq \ell\leq L$ is an integer (maybe with   some restrictions)
and $B\geq 0$.

The consequences that follow from such asymptotic formulas are
remarkably varied. We will now discuss some of them, with references
to cases where the corresponding results have been established. The
discussion is still informal. The ordering goes (roughly and not
systematically) in increasing order of the amount of information
required of the moments. We will make no attempt to be exhaustive.

\subsection{Universality outside the critical line}

One can generalize Bagchi's version of Voronin's Universality Theorem
to establish a functional limit theorem for the distribution of the
holomorphic functions $L(f,s)$ restricted to a fixed suitable compact
subset $D$ of the strip $0<\Reel(s)<1$ (see~\cite{voronin} for
Voronin's original paper and~\cite{bagchi-orig} for Bagchi's
probabilistic interpretation).  This result is much softer than those
that follow. It first requires an upper-bound of the right order of
magnitude (with respect to $N$) of the untwisted second moment, which
is used to get an upper bound for
$$
\expect_N\Bigl(|L(f,\demi+it)|\Bigr)
$$
using the Cauchy-Schwarz inequality. Using this (and local spectral
equidistribution), one proves a form of equidistribution of $L(f,s)$
restricted to $D$ in a space of holomorphic functions on $D$. Then
some form of the Prime Number Theorem (for an auxiliary $L$-function)
is required to compute the support of the random holomorphic function
that appeared in the first step, in order to deduce the universality
statement.  
\par
For instance, in the case of the family $\mathcal{C}$ above, it is
proved in~\cite{bagchi} that the $L$-functions become distributed like
the random Euler products
$$
\prod_p\det(1-X_pp^{-s})^{-1}
$$
where $(X_p)$ is a sequence of independent random variables that have
the Sato-Tate distribution. The support of this random Euler product
is (for $D$ a small disc centered on the real axis and contained in
the interior of the strip $\demi<\Reel(s)<1$) the set of
non-vanishing holomorphic functions $\varphi$ on $D$, continuous on
the boundary, that satisfy the real condition
$\varphi(\bar{s})=\overline{\varphi(s)}$.

\subsection{Upper and lower bounds for integral moments}

For a family $\mathcal{F}$ with a given symmetry type (in the
Katz-Sarnak sense described above), the asymptotic formula from the
moment conjectures~(\ref{eq-moment-conj}) imply that the order of
magnitude of $\expect_N(|L(f,\demi)|^{2k})$ should be $(\log N)^{c_k}$
for some constant $c_k$ depending only on the symmetry type, with
$c_k=k^2$ if the family is unitary, for instance. Although the
asymptotic remains very mysterious, the order of magnitude is much
better understood.

First, there exists a robust method due to Rudnick and
Soundararajan~\cite{RuSo} to derive \emph{lower bounds} of the right
form. We illustrate it in the case of a unitary family.  The method
involves evaluating the two averages
$$
S_1=\expect_N\Bigl(L(f,\demi)A(f)^{k-1}\overline{A(f)^{k}}\Bigr),
\quad\quad S_2=\expect_N\Bigl(|A(f)|^{2k}\Bigr),
$$
where
$$
A(f)=\sum_{n\leq L}\frac{\lambda_f(n)}{\sqrt{n}}
$$
for some parameter $L$. H\"older's inequality gives the lower bound
$$
\expect_N\Bigl(|L(f,\demi)|^{2k}\Bigr)\geq
\frac{|S_1|^{2k}}{S_2^{2k-1}},
$$
and hence we obtain the desired lower bounds if we can prove that
$$
S_2\ll (\log N)^{k^2}\ll S_1.
$$
After expanding the value of
$A(f)$, and using multiplicativity, we see that
$S_1$ is a combination of twisted first moments involving integers
$\ell\leq
L^{2k-1}$. It is therefore to be expected that we can evaluate
$S_1$, provided we have an asymptotic formula for the twisted first
moments~(\ref{eq-tw1}) valid for the corresponding values of
$\ell$. We can expect to evaluate the first moment in such a range
only when the ``pure'' first moment has an asymptotic formula with
power saving. The evaluation of
$S_2$ is, in principle, simpler. It can be expected (and turns out to
be true when the method is applicable) that one requires
$L$ to be comparable to the conductor
$N^{\alpha}$ in logarithmic scale for the bounds above to hold.

There is no corresponding unconditional upper bound. However,
Soundararajan~\cite{SoMom} devised a method to obtain almost sharp
upper bounds when one assumes that the $L$-functions in the family
$\mathcal{F}$ satisfy the Riemann Hypothesis (i.e., all zeros of
$L(f,s)$ with positive real part   have real part $1/2$). Precisely, he
obtained results like
$$
\expect_N\Bigl(|L(f,\demi)|^{2k}\Bigr)\ll (\log N)^{c_k+\eps},
$$
for any $\eps>0$ for some important families (or the analogue for the
$k$-th moment in symplectic and orthogonal families). His approach was
refined by Harper~\cite{Harper}, who obtained the upper-bound
$(\log N)^{c_k}$ (still under the Riemann Hypothesis for the
$L$-functions). We refer to the introductions to both papers for a
description of the ideas involved.

\subsection{Proportion of non-vanishing}\label{ssec-nonva}

Because of the Riemann Hypothesis, the problem of the location of
zeros of $L$-functions is especially important. In particular, much
interest has been concentrated on the special point $s=\demi$. This is
obviously natural in families where the order of vanishing at this
point has some arithmetic interpretation. This is the case, for
instance, in the family $\mathcal{C}$ of cusp forms of weight $2$:
indeed, for any $f\in\mathcal{C}$, Shimura has constructed an abelian
variety $A_f$ over $\Qq$, of dimension equal to the degree of the
field generated by the coefficients $\sqrt{p}\lambda_f(p)$ for $p$
prime, such that $L(f,s)$ is the Hasse-Weil $L$-function of $A_f$;
then the Birch and Swinnerton-Dyer conjecture predicts that the order
of vanishing of $f$ at $\demi$ should be equal to the rank of the
group $A_f(\Qq)$. However, there are also other applications to an
understanding of the behavior of central values (see the highly
influential study of Landau-Siegel zeros by Iwaniec and
Sarnak~\cites{IwSa,IwSa2}).

If the $L$-function $L(f,s)$ is self-dual, and the sign of its
functional equation of $L(f,s)$ is $-1$, then we get $L(f,\demi)=0$
trivially. One may expect conversely that few $L$-functions satisfy
$L(f,\demi)=0$ otherwise (some do have this property, but they are not
easy to come by; see, e.g.,~\cite{IwKo}*{Ch. 22--23} for an account of
the construction of a single such $L$-function by Gross and Zagier,
and how it completed Goldfeld's effective lower-bound for class
numbers of imaginary quadratic fields).  

Using ideas reminiscent of Markov's inequality in probability theory,
one can obtain rather good information on the proportion of non-vanishing of
central values. The basic observation is that, assuming asymptotic
formulas~(\ref{eq-bm1}) and~(\ref{eq-bm2}), a simple application of the
Cauchy-Schwarz inequality (or of Markov's inequality), leads to the
lower bound
$$
\proba_N\Bigl(L(f,\demi)\not=0\Bigr)\geq
\frac{\expect_N\bigl(L(f,\demi)\bigr)^2}
{\expect_N\bigl(|L(f,\demi)|^2\bigr)} \geq
\frac{\MT_1(N;0)^2+o(1)}{\MT_2(N;0)+o(1)}
$$
as $N\to +\infty$. Since the main terms are polynomials in $\log N$,
the lower-bound is of the form $(\log N)^{-k}$ for some integer
$k\geq 0$. This suffices to obtain a large number of non-vanishing
central critical values, but in practice, one finds that $k \geq 1$
(which can be guessed from the degrees of the polynomials, predicted
by the moment conjectures), so we do not obtain an asymptotic positive
proportion of non-vanishing.

The \emph{mollification method}, pioneered by Selberg~\cite{selberg},
exploits the twisted first and second moments to overcome this loss in
the case where $k=1$, which is the most common. This method introduces
a \emph{mollifier}\label{pg-mollifier}
$$
M(f)=\sum_{1\leq \ell\leq L} \alpha(\ell)\lambda_f(\ell)
$$
where the coefficients $\alpha(\ell)$ are chosen so that $M(f)$
approximates (in some sense) the inverse of $L(f,\demi)$. Using the
asymptotic formulas for the twisted first and second moment, one
obtains asymptotic formulas for the mollified moments
$$
\expect_N\Bigl(M(f)L(f,\demi)\Bigr)=\widetilde{\MT}_1(N,L)+
O(L^BN^{-\delta}),
$$
and
$$
\expect_N\Bigl(|M(f)L(f,\demi)|^2\Bigr)=
\widetilde{\MT}_2(N,L)+O(L^BN^{-\delta}).
$$
The effect of the power-saving with respect to $N$ is that we can
select $L=N^{\gamma}$ to be a small enough (fixed) power of $N$ so
that the Cauchy-Schwarz inequality now leads to the lower bound
$$
\proba_N\Bigl(L(f,\demi)\not=0\Bigr)\geq
\frac{\widetilde{\MT}_1(N,L)^2}{\widetilde{\MT}_2(N,L)}
(1+o(1))
$$
as $N\to +\infty$.  It turns out that the leading term is now a
positive constant (of size depending on $\gamma$), so we get a
positive lower bound for the proportion of non-vanishing special
values.

A version of this method was used by Selberg to prove his celebrated
result on a positive proportion of critical zeros of the Riemann zeta
function. It applies also, for instance, in proving that there is a
positive proportion of non-vanishing critical values in the families
$\mathcal{D}$ and $\mathcal{C}$ above (due to
Iwaniec--Sarnak~\cite{IwSa} and Kowalski--Michel~\cite{KMDMJ},
respectively), among other important families.

\subsection{Existence of large values}
\label{LargeValuesIntroSection}

The problem of the possible extreme sizes of values of $L$-functions
is one of the most difficult and mysterious. This is due, in part, to
the fact that their ``typical'' average behavior seems to be quite
accurately predicted using various probabilistic models, but there is
no particular reason to expect that such models can be reliable at the
level of ``large deviations''. And even if one is convinced (rightly
or wrongly!) that such a model is accurate, rigorous results are very
difficult to come by. Soundararajan~\cite{Soundararajan2008}
introduced a tool called the ``resonator method'' to produce
remarkably large values of $L(f,\demi)$ for some $f\in\mathcal{F}_N$
(this method is also related to the ideas introduced by
Goldston-Pintz-Y\i ld\i r\i m~\cite{GPY} to study gaps between
primes). The idea is again to select coefficients
$(\alpha(\ell))_{\ell\leq L}$ and form the corresponding sums
$$
R(f)=\sum_{\ell\leq L}\alpha(\ell)\lambda_f(\ell),
$$
\label{pg-resonator}
(or some variations thereof, cf.\ Section
\ref{SubsectionProductsResonator}) which are now called
``resonators''. Indeed, they are constructed so that the sizes of the
two quantities
$$
Q_1=\expect_N(|R(f)|^2)
$$
and
$$
Q_2=\expect_N\Bigl(|R(f)|^2L(f,\demi)\Bigr)
$$
are such that $|Q_2|/Q_1$ is as large as possible. These sums can be
evaluated asymptotically as quadratic forms (with variables
$\alpha(\ell)$) if $L=N^{\gamma}$ with $\gamma>0$ small enough,
because of the asymptotic formula for twisted first
moments~(\ref{eq-tw1}). We then have
$$
\max_{f\in\mathcal{F}_N} |L(f,\demi)|\geq \sqrt{\frac{|Q_2|}{Q_1}}.
$$
It remains a delicate issue to optimize the choice of the coefficients
$\alpha(\ell)$, but as in the previous application, we see that we can
certainly expect to find asymptotic formulas for $Q_1$ and $Q_2$,
provided $L$ is not too large, if we have access to an asymptotic
formula for twisted moments for $\ell\leq L^2$. Once more, the
dependency on $L$ is such that we obtain really good results only if
we can take $L$ of size comparable to $N$ in logarithmic scale, which
is what power-savings in the first moment leads to.

\subsection{Decay of probability of large order of vanishing}
\label{ssec-decay}

As we already indicated, the distribution of the order of vanishing
$\rk_{an}(f)$ of an $L$-function $L(f,s)$ at the critical point
$\demi$ (which is also called the \emph{analytic rank}) has been
extensively studied, often because of its links to arithmetic geometry
in special cases. A method due to Heath-Brown and Michel~\cite{HBMDMJ}
exploits a variant of the mollification method to study how often the
analytic rank might be a very large integer. The starting point are
the moments
$$
\expect_N\Bigl(\Bigl| \sum_p \frac{\lambda_f(p)(\log p)}{\sqrt{p}}
\phi(\log p) \Bigr|^{2k}\Bigr), \quad\quad \expect_N\Bigl(\Bigl|
\sum_{\rho} \widehat{\phi}(\rho-\demi) \Bigr|^{2k}\Bigr)
$$
for integers $k\geq 0$, where $\rho$ runs over zeros of $L(f,s)$ ``far
away'' from $\demi$ in some sense, and $\phi$ are suitable test
functions. The first of these two can be studied relatively
elementarily, if $\phi$ has sufficiently small support. The second is
estimated by a delicate computation using the explicit formula
(relating zeros of $L$-functions and their coefficients), and the
asymptotic formulas for twisted second moments~(\ref{eq-tw2}). From
these bounds, one can deduce that the analytic rank cannot be large
very often. In fact, one obtains exponentially-decaying tail-bounds:
there exists a constant $c>0$ (depending on the family) such that
$$
\expect_N(e^{c\rk_{an}(f)})\ll 1
$$
for $N\geq 1$, from which it follows that
$$
\limsup_{N\to +\infty}\proba_N(\rk_{an}(f)\geq r)\ll e^{-cr}
$$
for $r\geq 0$.

\subsection{Subgaussian bounds for critical values}

Given a family $\mathcal{F}$ of $L$-functions, the most general moment
conjectures of Keating-Snaith (including suitable complex exponents;
see~\cite{KoNi} for some general discussion of these) lead to the
expectation that $\log |L(f,\demi)|$ should have an approximately
normal distribution as $N\to +\infty$, after a suitable
normalization. There are currently very few results of this type. The
first one is due to Selberg~\cite{selberg-clt} (see also the short
proof by \maks\ and Soundararajan in~\cite{RadSo1}), and applies to
the Riemann zeta function. It states that for $T\geq 3$, and
$t\in [T,2T]$, the distribution of
$$
\frac{\log\zeta(\demi+it)}{\sqrt{\demi\log\log T}}
$$
converges as $T\to +\infty$ to a standard complex gaussian. 

\maks\ and Soundararajan~\cite{RadSo2} have developed a robust method
to prove \emph{subgaussian upper bounds} in many families. For a given
family, they show that such bounds hold whenever one has suitable
asymptotic formulas for the twisted moments~(\ref{eq-tw1}) when $\ell$
is as large as a small power of the conductor (which, in turn, usually
follows once an asymptotic formula for the first moment is known with
power-saving error term).
\par
The results have a different form depending on the symmetry type of
the family. In the orthogonal case, for instance, the method leads to
$$
\proba_N\Bigl(\frac{\log L(f,\demi)+\demi \log\log
  N}{\sqrt{\log\log N}}\geq V\Bigr) \leq
\frac{1}{\sqrt{2\pi}}\int_V^{+\infty}e^{-x^2/2}dx+o(1),
$$
for any fixed $V\in\Rr$; in this case, the gaussian conjecture would
be that the left-hand side is \emph{equal} to the right-hand side.

The method is quite intricate. Roughly speaking, it starts with the
proof that the sums over primes
$$
P_N(f)=\sum_{p\leq P}\frac{\lambda_f(p)}{\sqrt{p}}
$$
have a gaussian distribution if $P$ is well-chosen, typically
$P=N^{1/(\log\log N)^2}$, which in turn is an effect of quantitative
local spectral equidistribution (with independence of the local
components at distinct primes).

One expects that $P_N(f)$ is a good approximation to
$\log L(f,\demi)+\demi\log\log N$ in some statistical sense (the
additional term is the contribution of squares of primes, and the plus
sign reflects the orthogonal symmetry). Fixing $V$, one distinguishes
(again, roughly speaking; see~\cite[p. 1046]{RadSo2} for a precise
discussion) between three possibilities to compute the probability
that
$$
\frac{\log L(f,\demi)+\demi \log\log
  N}{\sqrt{\log\log N}}\geq V,
$$
namely:
\begin{itemize}
\item It may be that $P_N(f)\geq (V-\eps)\sqrt{\log\log N}$ for some
  small $\eps>0$, and the gaussian distribution of $P_N(f)$ gives a
  suitable gaussian bound for that event;
\item It may be that $|P_N(f)|\geq \log\log N$, but the gaussian
  behavior shows that this is very unlikely;
\item In the remaining case, we have 
\begin{gather*}
  |P_N(f)|\leq \log\log N,\\
  L(f,\demi)(\log N)^{1/2}\exp(-P_N(f))\geq \exp(\eps\sqrt{\log\log
    N}).
\end{gather*}
\end{itemize}
To control this last critical case, one shows that it implies
$$
L(f,\demi)(\log N)^{1/2} \Bigl(\sum_{j=0}^k
\frac{(-P_N(f))^j}{j!}\Bigr) \gg \exp(\eps\sqrt{\log\log N})
$$
for some suitable integer $k\geq 1$. But one can obtain an upper bound
for
$$
\expect_N\Bigl(L(f,\demi)\Bigl(\sum_{j=0}^k
\frac{(-P_N(f))^j}{j!}\Bigr)\Bigr)
$$
using the twisted first moments (where the power-saving gives as
before the crucial control of a suitable value of the length $P$ of
$P_N(f)$). Then, by the Markov inequality, we get
$$
\proba_N(\text{third case})\leq \exp(-\eps\sqrt{\log\log N})
\expect_N\Bigl(L(f,\demi)\Bigl(\sum_{j=0}^k
\frac{(-P_N(f))^j}{j!}\Bigr)\Bigr)
$$
which shows that the third event is also unlikely.

\maks\ and Soundararajan~\cite{RadSo3} have recently announced another
method that leads to gaussian lower bounds for \emph{conditional}
probabilities that normalized values of $\log L(f,\demi)$ belong to
some interval, \emph{knowing that they are non-zero}. These rely on
(and in some sense incorporate) the proof of existence of a positive
proportion of non-vanishing (discussed in
Section~\ref{ssec-nonva}). For an orthogonal family, the statements
are of the type
\begin{multline*}
\proba_N\Bigl(\alpha\leq \frac{\log L(f,\demi)+\demi\log\log
  N}{\sqrt{\log\log N}}\leq \beta\Bigr) \\
\geq
\proba_N\Bigl(L(f,\demi)\not=0\Bigr)\times 
\frac{1}{\sqrt{2\pi}}\int_{\alpha}^{\beta}e^{-x^2/2}dx+o(1),
\end{multline*}
as $N\to +\infty$, when it is known that
$$
\liminf_{N\to+\infty} \proba_N(L(f,\demi)\not=0)>0.
$$

\subsection{Paucity of real zeros}

The problem of possible existence of real zeros of $L$-functions on
the right of the critical line is fascinating and difficult,
especially in the case of self-dual $L$-functions, with the famous
problem of Landau-Siegel zeros (concerning real zeros close to $1$ of
$L$-functions of real Dirichlet characters) remaining one of the key
open problems of analytic number theory.

In this respect, Conrey and Soundararajan~\cite{CoSo} discovered a
very subtle variant of the mollification method (related to some of
the techniques of Section~\ref{ssec-decay}, in particular a critical
lemma of Selberg, see~\cite{CoSo}*{Lemma 2.1}) that allowed them to
prove that the specific family of real Dirichlet characters
$\mathcal{Q}_N$ (which is of special interest in this respect)
satisfies
$$
\liminf_{N\to +\infty} \proba_N\Bigl(L(\chi,s)\text{ has no real
  zero } s > 0\Bigr)>0.
$$
\par
It is unclear how general this method is, because it ultimately
depends on the numerical evaluation of a certain quantity.  Conrey and
Soundararajan~\cite{CoSo}*{end of \S 2} explain that the success can be
motivated by computations from Random Matrix Theory for symplectic
families, but these assume (at least) the Generalized Riemann
Hypothesis, and therefore are no guarantee of success in practice. In
fact, we may note that Ricotta~\cite{Ricotta} obtained a similar
result for families of Rankin-Selberg $L$-functions, but obtaining a
positive proportion with at most three real zeros. This type of result
is probably more robust.

\section{The family of twists of a fixed modular form}\label{intro}

We present in this section the (quite classical) family of
$L$-functions that we will study in the remainder of the book.

We fix \emph{throughout the book} a primitive cusp form (newform) $f$
with respect to some congruence subgroup $\Gamma_0(r)$, with trivial
central character, i.e., trivial nebentypus, which we will denote
$\chi_r$. The modular form $f$ may be either a holomorphic cusp form
of some weight $k_f$ or a Maa{\ss} cusp form with Laplace eigenvalue
$1/4+t_f^2$.

To simplify some computations, the following convention will be
useful: 
\begin{convention}\label{def-conv}
 For a modular form $f$ as above we define a  quantity, the \emph{signed level}, denoted by $r$, which is equal to the level of $f$ (and thus positive) if $f$ is holomorphic and equal to minus the level (and thus negative) if $f$ is a Maa{\ss} form. 
\end{convention}

\begin{remark}
In general, the level is an ideal in the ring of integers of the underlying number field (in our case the number field is simply $\mathbf{Q}$), and the signed level should be thought of a suitable idele generating the ideal. For
  simplicity of notation, we continue to write $\chi_r$ for the
  trivial character modulo $|r|$.
\end{remark}

We denote by $\lf(n)$, for $n\geq 1$, the Hecke eigenvalues of $f$,
normalized so that the mean square is $1$ by Rankin-Selberg theory, or
equivalently so that the standard $L$-function of $f$,
$$
\sum_{n\geq 1}\lf(n)n^{-s},
$$
is absolutely convergent in $\Re s > 1$.

From the point of view of cusp forms, we now consider the family
parameterized by primes $q$, not dividing $r$, which is given by
$$
\mathcal{F}_q=\{f\otimes \chi\,\mid\, \chi\pmod{q},\ \chi\not=\chi_q\},
$$\label{pg-fq}
(where $\chi$ runs over the set of primitive Dirichlet characters
modulo $q$). Since $(q,r)=1$, this is a subset of the set of primitive
cusp forms of level $rq^2$. The associated $L$-functions are the
twisted $L$-functions
$$
L(f\otimes\chi,s)=\sum_{n\geq 1}\frac{\lf(n)\chi(n)}{n^s}= \prod_{p}
\left(1-\frac{\chi(p)\lf(p)}{p^s}+\frac{\chi_r\chi^2(p)}{p^{2s}}\right)^{-1}.
$$

\begin{remark}
  \textbf{ We emphasize that \emph{throughout the remainder of this
      memoir, the modulus~$q$ will be assumed to be prime}, unless
    explicitly stated otherwise.}
\end{remark}

We will usually think of the $L$-functions as simply parameterized by
$q$, and write the probability and expectation explicitly as
$$
\frac{1}{q-2}|\{\chi\bmod{q}\,\mid\, \chi\not=\chi_q,\ \chi\in A\}|,
\quad\quad \frac{1}{q-2}\sums_{\chi\bmod{q}}T(\chi)
$$
for any set $A$ of Dirichlet characters, or any function $T$ defined
for Dirichlet characters; the notation $\sum^{\ast}$ restricts the sum to
primitive characters. We will write $\vphis(q)=q-2$ to clarify the
notation.

This family has been studied in a number a papers (for instance by
Duke, Friedlander, Iwaniec \cite{DFI}, Stefanicki~\cite{St},
Chinta~\cite{Chinta}, Gao, Khan and Ricotta~\cite{GKR}, Hoffstein and
Lee~\cite{HL} and in our own papers~\cites{BloMil,BFKMM,KMS}). It is a
very challenging family from the analytic point of view, and also has
some very interesting algebraic aspects, at least when $f$ is a
holomorphic cusp form of weight $2$. Indeed, if $A_f$ is the abelian
variety over $\Qq$ constructed by Shimura with Hasse-Weil $L$-function
equal to $L(f,s)$, then the product
$$
\prod_{\chi\bmod q}L(f\otimes \chi,s)=L(f,s)\ 
\prod_{\substack{\chi\bmod q\\\chi\not=\chi_q}}L(f\otimes \chi,s)
$$
is the Hasse-Weil $L$-function of the base change of $A_f$ to the
cyclotomic field $K_q$ generated by $q$-th roots of unity. According
to the Birch and Swinnerton-Dyer conjecture, the vanishing (or not) of
critical values of $L(f\otimes\chi,s)$ for $\chi\not=\chi_q$ is
therefore related to the increase of rank of the Mordell-Weil group of
$A_f$ over $K_q$ compared with that over $\Qq$. (We will come back to
this relation, as related to recent conjectures and questions of Mazur
and Rubin).
\par
\medskip
\par
The starting point of this book is that our recent
papers~\cites{BloMil,BFKMM,KMS} give access to the second moment
theory of this family, in the sense sketched in the previous
section. Precisely, the combination of these works provides a formula
with power-saving error term for the second moment of the central
value at $s=1/2$, namely
$$
\frac{1}{\vphis(q)}\sums_{\chi\mods q}|L(\ftchi,\demi)|^2.
$$

From this, it is a relatively simple matter to derive asymptotic
formulas including twists, and to include values at other points of
the critical line with polynomial dependency on the imaginary part,
giving formulas of the type~(\ref{eq-tw1}) and~(\ref{eq-tw2}). We can
then attempt to implement the various applications of the previous
section, and we will now list those which are found in this book.
\par

\begin{remark}
  (1) The family of twists is very simple from the point of view of local
  spectral equidistribution, which in that case amounts merely to an
  application of the orthogonality relations for Dirichlet characters
  modulo $q$. Precisely, for any prime $p$, the local components at
  $p$ (in the sense of automorphic representations) of $f\otimes\chi$
  become equidistributed in the unitary spectrum of $\GL_2(\Qq_p)$ as
  $q\to +\infty$, with limit measure the uniform probability measure
  on the set of unramified twists $\pi_p(f)\otimes |\cdot|_p^{it}$,
  where $\pi_p(f)$ is the $p$-component of $f$. This fact will not
  actually play a role in our arguments so we skip the easy proof. It
  implies however (and this can easily be checked by means of the
  distribution of low-lying zeros) that the family is of unitary type.
\par
(2) Because the orthogonality relations for Dirichlet characters are
easier to manipulate when summing over all characters modulo $q$, we
will sometimes use sums over all Dirichlet characters. This amounts to
adding the $L$-function of $f$ into our family (with the Euler factor
at $q$ removed) and  changing the normalizing factor from
$1/\varphi^{\ast}(q)$ to $1/\varphi(q)$, and has no consequence in the
asymptotic picture.
\end{remark}

In the next sections, we state precise forms of the results we will
prove concerning this family. These concern non-vanishing properties
and extremal values. We leave as an exercise to the interested reader
the proof of the universality theorem (following~\cite{bagchi}), which
takes the following form:

\begin{theorem}
  Let $0<R<1/4$ be a real number and let $D$ be the disc of radius $R$
  centered at $3/4$, which has closure $\bar{D}$ contained in the
  critical strip $1/2<\Reel(s)<1$. Let $\varphi\colon \bar{D}\to\Cc$
  be a function that is continuous and holomorphic in $D$, and does
  not vanish in $D$. Then, for all $\eps>0$, we have
$$
\liminf_{q\to+\infty} \frac{1}{\vphis(q)} \Bigl|\{\chi\mods{q} \,\,
\text{{\rm non-trivial}} \mid\, \sup_{s\in\bar{D}}
|L(f\otimes\chi,s)-\varphi(s)|<\eps\} \Bigr|>0.
$$
\end{theorem}

We do not consider lower-bounds for integral moments, but refer to the
earlier paper of Blomer and Mili\'cevi\'c~\cite{BloMil}*{Th.\ 4} where a
special case is treated.



\section{Positive proportion of non-vanishing}

We will first use the mollification method to show that the central
value $L(\ftchi,\demi)$ is not zero for a positive proportion of
$\chi\mods q$. In fact, as is classical, we will obtain a quantitative
lower-bound.  Moreover, inspired by the work of B. Hough~\cite{Hough},
we will establish a result of this type with an additional constraint
on the argument of the $L$-value.

For $\chi$ such that $L(\ftchi,\demi)\not=0$, we let
\begin{equation}\label{deftheta}
\theta(\ftchi)=\arg(L(\ftchi,\demi))\in\Rr/2\pi\Zz
\end{equation}
be the argument of $L(\ftchi,\demi)$. We will also use the same
notation for the reduction of this argument in $\Rr/\pi\Zz$ (we
emphasize that this is $\Rr/\pi\Zz$, and not $\Rr/2\pi\Zz$; as we will
see, our method is not sensitive enough to detect angles modulo
$2\pi$).

We say that a subset $I\subset \Rr/\pi\Zz$ (or $I\subset \Rr/2\pi\Zz$)
is an \emph{interval} if it is the image of an interval of $\Rr$ under
the canonical projection.

\begin{theorem}\label{thmnonvanishing+angle} 
  Let $I\subset \Rr/\pi\Zz$ be an interval of positive measure. There
  exists a constant $\eta > 0$, depending only on $I$, such that
$$
\frac{1}{\vphis(q)} |\{\chi\mods q \, \, \text{{\rm
    non-trivial}}\mid\, |L(f\otimes\chi,\demi)|\geq (\log q)^{-1},\
\theta(\ftchi)\in I\}| \geq \eta+o_{f,I}(1)
$$
as $q\ra\infty$ among the primes.
\end{theorem}

\begin{remark} \label{remark19}
  (1) Our proof will show that one can take
  $\eta =\frac{\mu(I)^2}{\expopropdenom \zeta(2)}$, where $\mu(I)$
  denotes the (Haar probability) measure of $I$.
  It also shows that the lower bound $(\log q)^{-1}$ can be replaced
  with $(\log q)^{-1/2 -\varepsilon}$ for any $\varepsilon > 0$.  For more details see \S \ref{improvedthm} below.
\par
(2) When $f$ is a holomorphic form with rational coefficients (i.e.,
it is the cusp form associated to an elliptic curve over $\Qq$),
Chinta~\cite{Chinta} has proved the following very strong
non-vanishing result: for any $\eps>0$, we have
\begin{equation}\label{chintabound} 
  \frac{1}{\vphi(q)}
  |\{\chi\mods q \,\mid\, L(f\otimes\chi,\demi)\not=0\}|=1+
  O_{f,\eps}(q^{-1/8+\eps}).	
\end{equation}
His argument uses ideas of Rohrlich, and in particular the fact that
in this case, the vanishing or non-vanishing of
$L(f\otimes\chi,\demi)$ depends only on the orbit of $\chi$ under the
action of the absolute Galois group of $\Qq$.  The Galois invariance of the
non-vanishing of $L(f\otimes\chi,\demi)$ is not known if $f$ is a
Maa{\ss} form, and it is not known either whether a lower bound such
as $|L(f\otimes\chi,\demi)|\geq (\log q)^{-1}$ is Galois-invariant,
even when $f$ is holomorphic.
\end{remark}

Theorem~\ref{thmnonvanishing+angle} can be seen as a special case of a
more general class of new non-vanishing results for $L(\ftchi,1/2)$
under additional constraints on $\chi$. In Section~\ref{sec-mellin},
we combine the mollification method with Katz's work on the
equidistribution of Mellin transforms of trace functions over finite
fields (see~\cite{Ka-CE}) to prove a very general theorem of this type
(see Theorem~\ref{nonvanishing+anglegeneral}). We state a
representative special case here.

For any $\chi\mods q$, the \emph{Evans sum} is defined as
$$
\widetilde{t}_e(\chi)=\frac{1}{\sqrt{q}}\sum_{x\in\Fqt}
\chi(x)e\Bigl(\frac{x-\bar{x}}{q}\Bigr).
$$
By Weil's bound for exponential sums in one variable, the Evans sums
are real numbers in the interval $[-2,2]$. A result of
Katz~\cite[Th. 14.2]{Ka-CE} implies that they become equidistributed,
as $q\to+\infty$, with respect to the Sato--Tate measure on $[-2,2]$.
We then have:

\begin{theorem}\label{thmnonvanishing+angleevans}
  Let $I\subset [-2,2]$ be a set of positive measure with non-empty
  interior. There exists a constant $\eta > 0$, depending only on $I$,
  such that
$$
\frac{1}{\vphi(q)} |\{\chi\mods q \, \, \mid\,
|L(f\otimes\chi,\demi)|\geq (\log q)^{-1},\ \widetilde{t}_e(\chi)\in
I\}| \geq \eta+o_{f,I}(1)
$$
as $q\ra\infty$ among the primes.
\end{theorem}

\section{Large central values}

Our next result exhibits large central values of twisted $L$-functions
in our family, using Soundararajan's resonator method.  More
precisely, we first prove a result that includes an angular
constraint, similar to that in the previous section.

\begin{theorem}\label{thm-extremal} Let $I\subset \Rr/\pi\Zz$ be an
  interval of positive measure. There exists a constant $c>0$ such for
  all primes $q$ large enough, depending on $I$ and $f$, there exists
  a non-trivial character $\chi\mods q$ such that
$$
L(\ftchi,\demi)\geq \exp\left(\Big(\frac{c\log q}{\log\log
    q}\Big)^{1/2}\right)\quad \text{and}\quad \theta(\ftchi)\in I.
$$
\end{theorem}

We will also prove a second version which involves a product of
twisted $L$-functions (and thus a slightly different family of
$L$-functions), without angular restriction.

\begin{theorem}\label{thm-extremal2}
  Let $g$ be a fixed primitive cusp form of level $r'$ and trivial
  central character. There exists a constant $c>0$, depending only on
  $f$ and $g$, such that for all primes $q$ large enough in terms of
  $f$ and $g$, there exists a non-trivial character $\chi\bmod{q}$
  such that
$$
|L(\ftchi,\demi)L(\gtchi,\demi)|\geq \exp\left(\Big(\frac{c\log
    q}{\log\log q}\Big)^{1/2}\right).
$$
\end{theorem}

Note that because we have a product of two special values, the
resonator method is now not a ``first moment'' method, but will
involve the average of these products, which is of the level of
difficulty of the second moment for a single cusp form $f$, and once
more, a power-saving in the error term is crucial for success.


\section{Bounds on the analytic rank}\label{pg-an}

Our third result concerns the order of vanishing (the analytic rank)
$$
\rk_{an}(f\otimes\chi)=\ord_{s=1/2}L(f\otimes\chi,s)
$$
of the twisted $L$-functions at the central point.  Using the methods
of \cites{KMDMJ,KMVcrelle,HBMDMJ} (as in Section~\ref{ssec-decay}) we
prove the exponential decay of the probability that the analytic rank
exceeds a certain value:

\begin{theorem}\label{thm-exponentialdecay} There exist   
  constants $R\geq 0$, $c>0$, depending only on $f$, such that
$$
  \frac{1}{\vphis(q)}\sums_{\chi\mods
    q}\exp(c\rk_{an}(f\otimes\chi))\leq \exp(cR)
$$
for all primes $q$.  In particular, by the inequality of arithmetic
and geometric means, we have
$$
\frac{1}{\vphis(q)}\sums_{\chi\mods q}\rk_{an}(f\otimes\chi)\leq R,
$$
for all primes $q$, and for any $t\geq 0$ we have
$$
\limsup_{q\to +\infty}\frac{1}{\vphis(q)} |\{\chi\mods q\,\,
\text{{\rm non-trivial}}\mid\, \rk_{an}(f\otimes\chi)\geq t\}|\ll_f
\exp(-ct).
$$
\end{theorem}

\begin{remark} 
  If
  $f$ is holomorphic with rational coefficients, an immediate
  consequence of Chinta's bound \eqref{chintabound} (using the bound
  $\rk_{an}(f\otimes\chi)\ll_f \log
  q$, for which see, e.g.,~\cite{IwKo}*{Th.\ 5.7}) is that
  $$
  \frac{1}{\vphis(q)}\sums_{\chi\mods
    q}\rk_{an}(f\otimes\chi)\ll_{f,\eps}q^{-1/8+\eps}.
  $$
\end{remark}

\section{A conjecture of Mazur-Rubin concerning modular symbols}

Suppose that $f$ is a holomorphic form   of weight $2$. For any
integers $q\geq 1$ coprime with $r$ and $a$ coprime to $q$, the
corresponding \emph{modular symbol} (associated to $f$) is defined by
$$
\msym{\frac aq}_f=2\pi i\int_{i\infty}^{a/q}f(z)dz =2\pi
\int_{0}^{\infty}f\Bigl(\,\frac{a}q+iy\,\Bigr)\,dy,
$$
\label{pg-mr}
where the path of integration can be taken as the vertical line
joining $i\infty$ to $a/q$ in the upper half-plane.  This quantity, as
a function of $a$, depends only on $a\mods q$.

It turns out that modular symbols are closely related to the special
values $L(f\otimes \chi,\demi)$ for Dirichlet characters
$\chi\mods q$, by means of a formula due to Birch and Stevens
(cf. \cite{MTT}*{(8.6)}).
\par
Recently, Mazur and Rubin~\cite{MaRu} have investigated the variation
of the rank of a fixed elliptic curve $E/\Qq$ in abelian extensions of
$\Qq$ (including infinite extensions). This has led them (via the
Birch--Swinnerton-Dyer conjecture and the Birch--Stevens formula) to a
number of questions and conjectures concerning the modular symbols of
the cusp form $f$ attached to $E$ (i.e., the cusp form whose
$L$-function coincides with the Hasse-Weil $L$-function of $E$ by the
modularity theorem). In particular, they raised a number of problems
concerning the distribution of these modular symbols.

Many of these questions have now been solved by Petridis and
Risager~\cite{PeRi} on average over $q$. In Chapter~\ref{ch-modular},
we will study the distribution of modular symbols associated to an
individual prime modulus $q$ (see also the recent work~\cite{KiSu} by
Kim and Sun for a more arithmetic/algebraic perspective on modular
symbols). Among other things, we will solve a conjecture of Mazur and
Rubin (see~\cite{PeRi}*{Conj. 1.2}) concerning their variance. Let
$$
M_f(q) = \frac{1}{\varphi(q)} \sum_{\substack{a\bmod{q}\\(a,q)=1}}
\msym{\frac aq}_f 
$$ 
be the mean value, which will be computed in Theorem \ref{th-MaRu2}.

\begin{theorem}\label{th-MaRu} 
  For $q$ a prime, the variance of modular symbols
$$
V_f(q)=\frac{1}{\vphi(q)}
\sum_{\substack{a\bmod{q}\\(a,q)=1}}\Bigl\vert\,\msym{\frac
  aq}_f-M_f(q)\Bigr\vert^2
$$
satisfies
$$
V_f(q)= 2 \prod_{p\mid r}(1+p^{-1})^{-1}
\frac{L^{\ast}(\symf,1)}{\zeta(2)}\log q+\beta_f+O(q^{-1/145})
$$
for $q$ prime, where $\beta_f\in\Cc$ is a constant and
$L^{\ast}(\symf,s)$ denotes the \emph{imprimitive} symmetric square
$L$-function of $f$ \emph{(cf.\ Section \ref{symsquare})}.
\end{theorem}

\section{Twisted moment estimates}

As we have explained in Section~\ref{sec-general}, the proofs of most
of these results rely on the {\em amplification method} and the {\em
  resonator method}, and involve various asymptotic formulas for
moments and twisted moments of the $L$-functions in the family.

In our case, since the Fourier coefficients of $f\otimes\chi$ are
$\lambda_f(n)\chi(n)$, and the first factor is fixed, it is most
natural to consider moments twisted simply by character values
$\chi(\ell)$ for some integers $\ell$. Moreover, in order to
incorporate angular restrictions on the central values, as in
Theorems~\ref{thmnonvanishing+angle} and \ref{thm-extremal}, it is
useful to also consider twists by powers of the Gau\ss\  sums of the
characters, at least in the first moment.

Hence, our basic sums of interests are
\begin{equation}\label{moments}
\begin{split}
  \mcL(f,s;\l,k) &:=\frac{1}{\vphis(q)}\sums_{\chi\mods q} L(f\otimes\chi,s)\eps_\chi^{k}\chi(\l),\\
  \mcQ(f,s;\l,\l') &:=\frac{1}{\vphis(q)}\sums_{\chi\mods q}
  \bigl|L(f\otimes\chi,s)\bigr|^2\chi(\l)\overline{\chi(\l')}
\end{split}
\end{equation}
where $s$ is a complex parameter (with real part close to $\demi$ in
practice), $\ell$ and $\ell'$ are coprime integers, $k\in\Zz$ and
\begin{equation}\label{gauss}
\eps_\chi = \frac{1}{\sqrt{q}}
\sum_{h\mods{q}}\chi(h)e\Bigl(\frac{h}{q}\Bigr)
\end{equation}
is the normalized Gau{\ss } sum of $\chi$. If $s=1/2$, we will drop it
from the notation and write $\mcL(f;\l,k) = \mcL(f,1/2;\l,k)$,
$\mcQ(f;\l,\l') = \mcQ(f,1/2;\l,\l')$.

Using these, we can build the mollified moments (or resonating
moments, depending on the application), namely
$$
\frac{1}{\vphis(q)}\sums_{\chi\mods q} L(f\otimes\chi,
s)e(2k\theta(f\otimes\chi))M(\ftchi,s;\bfx_L)
$$ 
and
$$
\frac{1}{\vphis(q)}\sums_{\chi\mods q}
\bigl|L(f\otimes\chi,s)\bigr|^2\bigl|M(\ftchi,s;\bfx_L)\bigr|^2
$$
where $M(\ftchi,s;\bfx_L)$ is a finite sum
$$M(\ftchi,s;\bfx_L)=\sum_{\l\leq L}x_\l\frac{\chi(\l)}{\l^{s}}$$
involving complex parameters $\bfx_L=(x_\l)_{\l\leq L}$ that we select
carefully depending on each application.


We need to evaluate the first moment only for
$s=\frac12$, and by the functional equation of
$L(\ftchi,s)$ is it sufficient to do so when $k\geq -1$.

\begin{theorem}\label{thm-moment1}  
  For $k\geq -1$, $\l\in(\Zz/q\Zz)^\times$ and any $\eps>0$, we have
$$
\mcL(f;\l,k)=\delta_{k=0}\frac{\lf(\ov \l_q)}{\ov
  \l_q^{1/2}}+O_{f,\eps, k}(q^{-1/8+\eps}),
$$
for $q$ prime, where $\ov\l_q$ denotes the unique integer in the
interval $[1,q]$ satisfying the congruence $\l\ov\l_q\equiv 1\mods q$.
\end{theorem}

The proof of this theorem is rather elementary when $k=0$, but it
requires the results of Fouvry, Kowalski and Michel~\cite{FKM1} on
twists of Fourier coefficients by trace functions otherwise. 

The evaluation of the second moment is significantly more
challenging. The combination of our three papers
\cites{BloMil,BFKMM,KMS} successfully handles the case
$\l=\l'=|r|=1$. Precisely, by~\cite{KMS}*{Th.\ 1.5} (which relies on the
previous papers), we have:

\begin{theorem}\label{thm-moment2}
  Assume that the level of $f$ is $r=1$. For any $\delta< 1/144$, we
  have
\begin{equation*}
  \frac{1}{\vphis(q)}\sums_{\chi\mods q}|
  L(f\otimes\chi,\demi)|^2= P_f(\log q)+O_{f,\delta}(q^{-\delta}),
\end{equation*}
for $q$ prime, where $P_f(X)$ is a polynomial of degree $1$ depending
on $f$ only with leading coefficient $ 2 L( \symf,1)/\zeta(2)$. %
\end{theorem}

As we discussed above, this is in a certain sense the main case, and
from there it is possible to evaluate the more general second moments
$\mcQ(f,s;\l,\l')$, which we do here for $f$ of general level. In
fact, for the proof of Theorem \ref{thm-extremal2}, we will require an
estimate involving two cusp forms (which of course may be equal!).

\begin{theorem}\label{thsecondmoment} 
  Let $f, g$ be primitive cusp forms of signed levels $r$ and $r'$ coprime to
  $q$, with trivial central character.  Define
$$
\mcQ(f,g,s;\l,\l')=\frac{1}{\vphis(q)}\sums_{\chi\mods q}
L(f\otimes\chi,s)\ov{L(g\otimes\chi,s)}\chi(\l)\overline{\chi(\l')}
$$
for integers $1\leq \l,\l'\leq L$, with $(\l\l',qrr')=(\l,\l')=1$, and
$s\in\Cc$.
\par
Then, for $s=\frac{1}2+\beta+it$, $\beta,t\in\Rr$ with
$|\beta|\leq 1/\log q$, we have the asymptotic formula
$$
\mcQ(f,g,s;\l,\l')=\MT(f,g,s;\l,\l')+O_{f, g,
  \varepsilon}(|s|^{O(1)}L^{\expoL}q^{-\expoq+\eps})
$$
for $q$ prime, where
$$\MT(f,g,s;\l,\l')=\frac12 \MT^+(f,g,s;\l,\l')+\frac12
\MT^-(f,g,s;\l,\l')$$ is a ``main term'' whose even and odd parts
$\MT^\pm(f,g,s;\l,\l')$ are given in \eqref{MTsecondmoment}. 
\end{theorem}

The main terms as we express them here are well-suited to further
transformations for our main applications. If one is interested in the
second moment $\mcQ(f,g,\demi;1,1)$ only (as in
Theorem~\ref{th-MaRu}), then one can express the main term more
concretely, but there are a number of cases to consider.
\par
If $f=g$ is of squarefree level $r$, then
$$
  \MT(f,f,\demi;1,1)=2
  \prod_{p\mid r}(1+p^{-1})^{-1}
  \frac{L( \symf,1)}{\zeta(2)}(\log q )+\beta_f + O(q^{-2/5})
$$
for some constant $\beta_f$, where $\symf$ is the symmetric square of
$f$ (cf.\ Section \ref{symsquare}). 
\par
If $f\not=g$, it may be that $\mcQ(f,g,\demi;1,1)$ is \emph{exactly}
zero for ``trivial'' reasons. This happens if $f$ and $g$ have the
same signed level $r=r'$ (recall that, with the convention~\ref{def-conv},
this implies that either both are holomorphic, or that both are
non-holomorphic) and their root numbers $\eps(f)$ and $\eps(g)$
satisfy $\eps(f)\eps(g)=-1$. In that case, computations with root
numbers show that
$$
L(f\otimes\chi,\demi)\overline{L(g\otimes \chi,\demi)}=-
L(f\otimes\ov{\chi},\demi)\overline{L(g\otimes \ov{\chi},\demi)}
$$
so the second moment cancels by pairing each character with its
conjugate (see Remark~\ref{rm-cancel}).

If, on the other hand, we have $\eps(f)\eps(g)=1$ and $f$, $g$ are of
the same type, then we have
$$
  \MT(f,g,\demi;1,1)=2\gamma_{f,g}\frac{ L( f\otimes
    g,1)}{\zeta(2)}  + O(q^{-2/5})
$$
where $L(f\otimes g,s)$ is the Rankin-Selberg convolution of $f$ and
$g$ (cf.\ Section \ref{RankinSelberg}) and $\gamma_{f,g}$ is some
non-zero constant depending on $f$ and $g$.  We defer a more detailed
discussion to Proposition~\ref{pr-mt}.

\begin{remark} (1) Let $d$ be the usual divisor function. Then $d(n)$
  is the $n$-th Hecke eigenvalue of a non-holomorphic Eisenstein
  series $E(s)$, and the identity
$$
L(\chi,s)^2=L(E\otimes \chi,s)
$$
shows that the problem of estimating the second moment of twists of
$E$ is equivalent to the problem of estimating the fourth moment of
the values of the Dirichlet $L$-functions $L(\chi,s)$.  This remark
shows that several parts of this memoir have obvious links with the
beautiful work of Young \cite{MY} (later improved in \cite{BFKMM})
where he proves the existence of a polynomial $P_4$ of degree $4$ and
of a constant $\delta >0$ such that, for all $q \geq 2$, we have
$$
\frac{1}{\varphi^* (q)} \sum_{\chi \mod q} \vert L(\chi, 1/2)\vert^4=
P_4 (\log q) + O(q^{-\delta}).
$$
\par
(2) Recently Zacharias \cite{Zac2} used the evaluation of the
mollified second moment $\mcQ(f,s;\l,\l')$ of this memoir together
with his own evaluation of the mixed twisted moment
$$\frac{1}{\vphis(q)}\sums_{\chi\mods q}L(\chi,\demi)L(\ftchi,\demi)\chi(\ell)$$
to establish the existence of a positive proportion of primitive
$\chi\mods q$ such that
$$L(\chi,\demi)L(\ftchi,\demi)\not=0.$$
He also obtained similar results when $f$ is an Eisenstein series (in
which case the $L$-function $L(f\otimes\chi,s)$ is a product of
Dirichlet $L$-functions): using his evaluation of the fourth mollified
moment of Dirichlet $L$-functions (\cite{Zac1}) he shows that for any
pair of characters $\chi_1,\chi_2\mods q$, there exists a positive
proportion of primitive characters $\chi\mods q$ for which
$L(\chi,\demi)L(\chi\cdot \chi_1,\demi)L(\chi\cdot
\chi_2,\demi)\not=0$.
\par
(3) The approach of Hoffstein and Lee~\cite{HL} towards the second
moment, based on multiple Dirichlet series, reduces a proof of
Theorem~\ref{thm-moment2} (with some power-saving exponent) to a
non-trivial estimate for a certain special value of a double Dirichlet
series, which is denoted $\tilde{Z}_q(1-k/2,1/2;f,f)$ in loc.\ cit.
When $q$ is prime, our theorem therefore indirectly provides such an
estimate
$$
\tilde{Z}_q(1-k/2,1/2;f,f)\ll q^{-1/144}.
$$
\end{remark}

\section*{Outline of the book}

This book is organized as follows:
\begin{enumerate}
\item Chapter~\ref{chapLfunctions} is preliminary to the main results;
  we set up the notation, and recall a number of important facts
  concerning Hecke $L$-functions (such as those of our family), as
  well as auxiliary $L$-functions that arise during the proofs of the
  main results (such as Rankin-Selberg $L$-functions). We require, in
  particular, some forms of the Prime Number Theorem and zero-free
  regions for these $L$-functions, and since the literature is not
  fully clear in this matter, we discuss some of these in some
  detail. We also discuss briefly a shifted convolution bound that is
  a slight adaptation of one of Blomer and
  Mili\'cevi\'c~\cite{BloMil}.
\item Chapter~\ref{ch-sums} gives an account of the algebraic
  exponential and character sums that occur in the book; on the one
  hand, these are the elementary orthogonality properties of character
  sums, and the averages of Gau\ss\ sums that give rise to
  hyper-Kloosterman sums, and on the other hand, we state a number of
  deep bounds for various sums of Kloosterman sums. Although we do not
  need to develop new bounds of this type, we give a quick sketch of
  the arguments that lead to them, with references to the original
  proofs. It is worth mentioning that these proofs rely in an
  absolutely essential way on the most general form of the Riemann
  Hypothesis over finite fields, due to Deligne, as well as on works
  of Katz. In Sections~\ref{sec-trace} and~\ref{sec-equi-mellin}, we
  present some background on trace functions and discuss the results
  of Katz on discrete Mellin transforms over finite fields that are
  involved in the proof of (the general form of)
  Theorem~\ref{thmnonvanishing+angleevans}.
\item In Chapter~\ref{ch-first}, we prove the necessary asymptotic
  estimates for the first twisted moment of our family. The proof is
  very short, which illustrates the principle that the complexity of
  moment computations in families of $L$-functions increases steeply
  as the order of moment increases.
\item In turn Chapter~\ref{ch-second} gives the proof of the required
  twisted second moment estimates. Although this is much more involved than
  the first moment, most of the necessary ingredients are found in our
  previous works, and the chapter is relatively short.
\item Finally, Chapters~\ref{ch-central},
  \ref{ch-extreme},~\ref{analyticrank} and~\ref{ch-modular} are
  devoted to the proofs of our main results: positive proportion of
  non-vanishing (including Theorem~\ref{thmnonvanishing+angle}),
  existence of large values, bounds for the analytic rank and the
  variance of modular symbols, respectively.  These chapters are
  essentially independent of each other (the last one is extremely
  short, as the proof of Theorem~\ref{th-MaRu} is mostly a direct
  translation of the second moment estimate), and many readers will
  find it preferable to start reading one of them, and to refer to the
  required results of the previous chapters only as needed.
\par
Since the theory of trace functions and its required background
involve prerequisites that may be unfamiliar to some readers, the
corresponding statements and results are isolated in independent
sections (besides the background sections~\ref{sec-trace} and
\ref{sec-equi-mellin}, they are in Sections~\ref{sec-m1-trace} and
\ref{sec-mellin}).
\end{enumerate}

\section*{Acknowledgments} 

\'E.\ F.\ thanks ETH Z\"urich and EPF Lausanne for financial
support. Ph.\ M.\ was partially supported by the SNF (grant
200021-137488) and by the NSF Grant 1440140, while in residence at
MSRI during the winter 2017. V.\ B., Ph.\ M.\ and E.\ K.\ were
partially supported by the DFG-SNF lead agency program grants BL
915/2, 200021L\_153647, 200020L\_175755. D.\ M.\ was supported by the
NSF (Grant DMS-1503629) and ARC (through Grant DP130100674). W.S. was
partially supported by Dr. Max R\"ossler, the Walter Haefner
Foundation and the ETH Zurich Foundation.
\par
\medskip
\par
We warmly thank the referee for his or her detailed and in-depth
report.
\par
We thank F. Brumley for useful discussions, especially concerning
Prime Number Theorems for automorphic forms, and G. Henniart for
useful information and remarks concerning the computations of root
numbers and local factors of various $L$-functions. We also thank
A. Saha for some references.
\par
The applications to modular symbols were elaborated after the talk of
M. Risager during an Oberwolfach meeting organized by V.B., E.K. and
Ph.M. We thank M. Risager and K. Rubin for enlightening discussions
about these problems, and we acknowledge the excellent conditions
provided to organizers by the Mathematisches Forschungsinstitut
Oberwolfach. 

Parts of the introduction were sketched during the conference
``Aspects of Automorphic Forms and Applications'' at the Institute of
Mathematical Research of Hong Kong University; E.K. and Ph.M. thank
the organizers, and especially Y-K. Lau, for inviting them to this
conference and giving the occasion to present some of the results of
this work.


\chapter{Preliminaries}
\label{chapLfunctions}

We collect in this chapter some preliminary material. Most of it is
well-known, however some cases of the Prime Number Theorem
(Proposition~\ref{AutomorphicPNTLemma}) are difficult to locate in the
literature, and the computation of the ramified factors of the
symmetric square $L$-function in Section~\ref{symsquare} are even more
problematic.

\section{Notation and conventions}\label{21}

\begin{itemize}
\item[--] We use the notation $\delta_{x,y}$ or $\delta(x,y)$ or $\delta_{x=y}$
for the Kronecker delta symbol.
\item[--] The notation $A\asymp B$ means
$$A\ll B\ll A,$$
where $\ll$ denotes the Vinogradov symbol.
\item[-] In this book, we will denote generically by $W$, sometimes with
subscripts, some smooth complex-valued functions, compactly supported
on $[1/2,2]$ and possibly depending on a finite set $\mathscr{S}$ of
complex numbers, whose derivatives satisfy
\begin{equation}\label{Wbound}
 W^{(j)}(x)\ll_{j}  \prod_{s \in \mathscr{S}} (1+|s|)^{cj}
\end{equation}
for some fixed constant $c > 0$ and any $j\geq 0$ (as usual, an empty
product is defined to be equal to $1$). In practice, $\mathscr{S}$ may
be empty, or may contain the (signed) levels $r, r'$ of two cusp forms, their
weight/spectral parameter, and/or a complex number $s$ on or close to
the $\demi$-line. Of course, the set $\mathscr{S}$ must not contain our
basic parameter $q$, but no harm is done if some $s\in \mathscr{S}$
grows like $(\log q)^2$, say, since all our estimates contain a 
$q^{\varepsilon}$-valve.  To lighten the notation, we will not the
display the dependence on parameters $s \in \mathscr{S}$ in implied
constants and just keep in mind that it is polynomial.
\item[--] Throughout this book, we will use the $\eps$-convention, according to
which a statement involving $\eps$ holds for all sufficiently small
$\varepsilon > 0$ (with implied constants depending on $\varepsilon$)
and the value of $\varepsilon$ may change from line to line. A typical
example is \eqref{eps}, where the various $\varepsilon$'s in
\eqref{Vassymp} and \eqref{KSbound} combine to a new $\varepsilon$.
\item[--] For $z\in\Cc$, we denote $e(z)=e^{2\pi iz}$. We recall that if
$q\geq 1$ is an integer, then $x\mapsto e(x/q)$ is a well-defined
additive character modulo $q$.
\item[--] For an integer $c\geq 1$ and $a\in\Zz$ coprime to $c$, we often write
$\bar{a}$ for the inverse of $a$ modulo $c$ in
$(\Zz/c\Zz)^{\times}$. The value of $c$ will always be clear from the
context.
\item[--] For any (polynomially bounded) multiplicative function $a(n)$, we
define a Dir\-ich\-let series
$$
A(s)=\sum_{n\geq 1}a(n)n^{-s},
$$
we denote by $A_p(s)$ the $p$-factor of the corresponding Euler
product, so that
$$
A(s)=\prod_p A_p(s)
$$
in the region of absolute convergence. For any integer $r$, we also
write $A^{(r)}(s)$ for the Euler product restricted to primes
$p\nmid r$.
\item[--] Let $c\geq 1$ and let $a$, $b$ be integers. We denote
$$
S(a, b; c) = \sum_{\substack{d\mods {c}\\(d,c)=1}} e\Bigl(\frac{a d +
  b\bar{d}}{c}\Bigr)
$$ 
the Kloosterman sum modulo $c$. We also denote
\begin{equation*}
\Kl(a;c)=\frac{1}{\sqrt{c}}S(a,1;c),
\end{equation*}
the normalized Kloosterman sum.
\item[--]
As we already mentioned, \textbf{\textit{unless otherwise specified,
    $q$ will be a prime number}}.
\end{itemize}

\section{Hecke \texorpdfstring{$L$-functions}{L-functions}}\label{hecke}

Let $f$ be a primitive cusp form (holomorphic or Maa\ss) of signed level $r$
(i.e.\ for the group $\Gamma_{0}(|r|)$) with trivial central character
$\chi_r$. The Hecke $L$-function of $f$ is a degree $2$ Euler product
absolutely convergent for $\Re s>1$:
\begin{align*}
L(f, s)&:=\prod_pL_p(f,s)=\prod_{p}\prod_{i=1}^2
\Bigl(\,1-\frac{\alpha_{f,i}(p)}{p^s}\,\Bigr)^{-1}\\
&= \prod_{p}\Bigl(1-\frac{\lambda_{f}(p)}{p^s} +\frac{\chi_r(p)}{p^{2s}}\Bigr)^{-1}=\sum_{n\geq 1}\frac{\lambda_{f}(n)}{n^s}, \quad \Re s >1.
\end{align*}
The factor $L_p(f,s)$ is the local $L$-factor at the prime $p$ and the
coefficients $\alpha_{f,i}(p)$ for $i=1,2$ are called the local
parameters of $f$ at $p$. The coefficients of this Dirichlet series
$(\lambda_{f}(n))_{n\geq 1}$ have a simple expression in terms of
these parameters: for any prime $p$, we have
$$\lf(p)=\alpha_{f,1}(p)+\alpha_{f,2}(p),\ \alpha_{f,1}(p)\alpha_{f,2}(p)=\chi_r(p),$$ 
and  we have the multiplicativity relations
\begin{gather*}
  \lf(m)\lf(n)=\sum_{d|(m,n)}\chi_r(d)\lf\left(\frac{mn}{d^2}\right),\\
  \lf(mn)=\sum_{d|(m,n)}\chi_r(d)\mu(d)\lf\left(\frac
    md\right)\lf\left(\frac nd\right).
\end{gather*}
For a primitive form, the Dirichlet coefficient $\lf(n)$ is the eigenvalue of
$f$ for the $n$-th Hecke operator.

The local $L$-factors $(L_p(f,s))_{p}$ are completed by an archimedean
local factor which is a product of  shifted Gamma functions
\begin{equation}
	 \label{defLinfty}L_\infty(f,s)=\Gamma_\Rr(s-\mu_{f,1})\Gamma_\Rr(s-\mu_{f,2}),\ \Gamma_\Rr(s)=\pi^{-s/2}\Gamma(s/2).
\end{equation}
The coefficients $\mu_{f,i},\ i=1,2$ are called the local archimedean
parameters of $f$ and are related to the classical invariants of $f$
as follows:
$$
\mu_{f,1}=-\frac{k-1}2,\ \mu_{f,2}=-\frac{k}2
$$
if $f$ is holomorphic of weight $k\geq 2$ and
$$
\mu_{f,1}=\frac{1- \kappa_f}{2}+ it_f,\ \mu_{f,2}=\frac{1-
  \kappa_f}{2}-it_f
$$ 
if $f$ is a Maa{\ss} form with Laplace eigenvalue
$\lambda_f(\infty)=(\frac12+it_f)(\frac12-it_f)$ and
$\kappa_f \in \{\pm 1\}$ is the eigenvalue of $f$ under the involution
$f \mapsto f (-\bar{z})$. 
The completed product
$$\Lambda(f,s)=|r|^{s/2}L_\infty(f,s)L(f,s)$$
admits a holomorphic continuation to the whole complex plane and satisfies a functional equation of the shape
$$\Lambda(f,s)=\eps(f)\ov{\Lambda(f,1-\ov s)}$$
where $\eps(f)$ (the root number) is a complex number satisfying $|\eps(f)|=1$.

\subsection{Character twists}
Let $\chi$ be a non-trivial Dirichlet character of prime modulus
$q$ also coprime with $r$. The twisted $L$-function
\begin{align*}
  L(\ftchi,s)&=\prod_pL_p(\ftchi,s)=\prod_{p}\prod_{i=1,2}
               \Big(1-\frac{\alpha_{f,i}(p)\chi(p)}{p^s}\Big)^{-1}\\
             &= \prod_{p}\Bigl(1-\frac{\lambda_{f}(p)\chi(p)}{p^s}
               +\frac{\chi_r(p)\chi^2(p)}{p^{2s}}\Bigr)^{-1}=\sum_{n\geq
               1}\frac{\lambda_{f}(n)\chi(n)}{n^s}, 
               \quad \Re s >1.
\end{align*}
is in fact the Hecke $L$-function of a primitive cusp form
$\ftchi$ for the group
$\Gamma_{0}(q^2|r|)$ with central character
$\chi^2\chi_r$ (see \cite{IwKo}*{Propositions 14.19 \& 14.20}, for
instance, in the holomorphic case which carries over to the general
case).

\begin{lemma} \label{lm-fneq}
  Let $f$ be a primitive (holomorphic or Maa{\ss}) cusp form of signed level
  $r$ and trivial central character, and let $\chi$ be a primitive
  character modulo $q$, not necessarily prime. Then the twisted
  $L$-function satisfies the functional equation
$$
\Lambda(\ftchi,s)=\eps(\ftchi)\ov{\Lambda(\ftchi,1-\ov
  s)}=\eps(\ftchi)\Lambda(f\otimes\ov\chi,1-s)
$$
with
$$
\Lambda(\ftchi,s)=(q^2|r|)^{s/2}L_\infty(f\otimes\chi,s)L(f\otimes\chi,s).
$$
Setting
$$
\mfa=\frac{1-\kappa_f\chi(-1)}2=
\begin{cases}0,&\hbox{ if $\chi$ and $f$ have the same parity} ,\\
  1,&\hbox{ if $\chi$ and $f$ have different parity} ,
\end{cases}
$$
we have
$$
L_\infty(f\otimes\chi,s)=
\begin{cases} L_\infty(f,s) \hbox{ if $f$ is holomorphic of weight $k$},\\
  L_\infty(f,s+\mfa)\hbox{ if $f$ is an even Maa{\ss} form},\\
   L_\infty(f,s-1+\mfa)\hbox{ if $f$ is an odd Maa{\ss} form},
\end{cases}
$$
and 
\begin{equation}\label{twistedroot}
\eps(\ftchi)=  \eps(f) \chi(r)\eps_\chi^2
\end{equation}
where $\eps(f)$ is the root number of $L(f,s)$ and $\eps_\chi$ is the
normalized Gau{\ss} sum, cf.\ \eqref{gauss}.
\end{lemma}

Recall Convention \ref{def-conv} that $r$ can be positive or negative depending on whether $f$ is holomorphic or not. Observe that $L_\infty(f\otimes\chi,s)$ depends at most on the parity
of $\chi$, and is independent of $\chi$ if $f$ is holomorphic. The
following notation will be useful: for $\chi(-1)=\pm 1$ we write
\begin{equation*}
L_\infty(f,\pm,s):=L_\infty(f\otimes\chi,s).	
\end{equation*}

\begin{proof} 
  This is standard (see, e.g.,~\cite{IwKo}*{Th.\ 14.17 and Prop.\ 14.20}
  in the holomorphic case). We did not find a reference for the
  explicit root number computation \eqref{twistedroot} in the Maa{\ss}
  case, so for the reader's convenience we include the details. We
  start with some general ``converse type'' computations. Let
$$F(z) = \sqrt{y} \sum_{n \not= 0} a(n)   K_{it}(2 \pi  |n|y) e(nx), \quad G(z) = \sqrt{y} \sum_{n \not= 0} b(n)   K_{it}(2 \pi  |n|y) e(nx)$$
be two Maa{\ss} form that are both even or both odd and satisfy  
$$F(-1/Nz) = \bar{\eta} G(z)$$
for some integer $N\geq 1$ and some complex number $\eta$ of modulus
$1$. Differentiating both sides of the functional equation with
respect to $x$, we obtain
$$\bar{\eta} G_x(z) = \frac{\partial}{\partial x} F\left( \frac{-x+iy}{N(x^2 + y^2)} \right)  =   F_x\left(- \frac{1}{Nz}\right) 
 \frac{1}{Nz^2}.$$
If both $F$ and $G$ are even, 
we compute
 \begin{multline*}
   2\int_0^{\infty} F(iy) y^{s-1/2} \frac{dy}{y} = 4 \sum_{n > 0} a(n)
   \int_0^{\infty} K_{it}(2\pi n y) y^{s } \frac{dy}{y}
   \\
   = 4\sum_{n > 0} \frac{a(n)}{n^{s}} \frac{\pi^{ - s} \Gamma(1/2(s -
     it)) \Gamma(1/2(s + it))}{4} = L(F, s)L_{\infty}(F, s). 
\end{multline*}
 On the other hand, by the functional equation, this equals
$$2\bar{\eta} \int_0^{\infty} G(i/Ny) y^{s-1/2} \frac{dy}{y} = \bar{\eta} N^{1/2-s} 2 \int_0^{\infty} G(iy) y^{-s+1/2} \frac{dy}{y} $$
so that by the above computation we have
$$   L(F, s)L_{\infty}(F, s) = \bar{\eta} N^{1/2-s} L(G,1- s)L_{\infty}(G, 1-s).$$

If both $F$ and $G$ are odd, we compute 
 \begin{multline*}
\frac{1}{ i} \int_0^{\infty} F_x(iy) y^{s+1/2} \frac{dy}{y}  =  4 \pi   \sum_{n > 0} a(n) n \int_0^{\infty}   K_{it}(2\pi  n y) y^{s+1 } \frac{dy}{y} \\
  =   4 \pi \sum_{n > 0} \frac{a(n)}{n^{s}} \frac{\pi^{ - s-1}
    \Gamma(1/2(s+1 - it)) \Gamma(1/2(s+1 + it))}{4} =  L(F, s)L_{\infty}(F, s).   
\end{multline*}
On the other hand, by the functional equation for the derivative, this
equals
\begin{multline*}
  \frac{1}{i} \int_0^{\infty} \bar{\eta} G_x(i/Ny) \frac{1}{N (iy)^2}
  y^{s+1/2} \frac{dy}{y} = -\frac{\bar{\eta}}{iN}\int_0^{\infty}
  G_x(i/Ny) y^{s-3/2} \frac{dy}{y}
  \\
  = -\frac{\bar{\eta}}{i} N^{1/2 - s} \int_0^{\infty} G_x(iy) y^{3/2
    -s} \frac{dy}{y},
\end{multline*}
so that by the above computation we have
$$   L(F, s)L_{\infty}(F, s)  = -\bar{\eta} N^{1/2-s} L(G,1- s)L_{\infty}(G, 1-s).$$

After these general considerations, we return to the functional
equation of twisted $L$-functions. Let $f$ be a Maa{\ss} form of
parity $\kappa_f \in \{\pm 1\}$ and signed level $r < 0$ and trivial central
character. Then $f \otimes \chi$ has parity
$\kappa_f \chi(-1) \in \{\pm 1\}$. Write
$$f(-1/(|r|z)) = \bar{\eta} f(z)$$
so that by the above computation we have
$\varepsilon(f) = \bar{\eta} \kappa_f$ for the root number. By a
formal matrix computation \cite{Iw-Topics}*{Theorem 7.5} we see that
$$(f \otimes \chi)\left(-\frac{1}{q^2|r| z}\right) = \varepsilon_{\chi}^2\chi(|r|) \bar{\eta}\cdot (f \otimes \bar{\chi})(z)$$
So by the above computation, the root number of $f\otimes \chi$ is
indeed
$$\varepsilon_{\chi}^2 \chi(|r|) \bar{\eta} \cdot \kappa_f \chi(-1) =  \varepsilon_{\chi}^2 \chi(r) \varepsilon(f)  $$
as claimed.
\end{proof}

\begin{remark}
  One could also recover the root number from general principles of
  automorphic representation theory, since it is given by a product
  over all places, and the behavior of local root numbers under
  twisting is relatively straightforward. The above ``classical''
  treatment doesn't require knowledge of, say, the classification of
  local representations at infinity.
\end{remark}

In particular, taking $s=1/2$ and recalling that 
$\theta(\ftchi)$ 
 is the argument of $L(\ftchi,1/2)$ (if
the latter is non-zero), cf.\  \eqref{deftheta}, we obtain:
\begin{equation}\label{phaseformula}
  \text{ If $L(\ftchi,1/2)\not=0$, one has }
  \exp(2  i\theta(\ftchi))=\frac{L(\ftchi,1/2)}
  {L(f\otimes\ov\chi,1/2)}=\eps(f) \chi( r)\eps_\chi^2.
\end{equation}

From the above discussion, and from the formula
$\eps_{\ov{\chi}}^2=\ov{\eps_\chi}^2$, we can derive an explicit form
of the functional equation for a product of twisted $L$-functions.

\begin{lemma}\label{propfcteqn}
  Let $f,g$ be primitive cusp forms with trivial central character, of
  signed levels $r$ and $r'$ respectively, both coprime to $q$.  We have
\begin{equation*}
  \Lambda(f\otimes\chi,s)
  \Lambda(g\otimes\ov\chi,\ov s)=\eps(f)\eps(g)\chi(r \overline{r'})
  \Lambda(f\otimes\ov\chi,1-s)\Lambda(g\otimes\chi,1-\ov s).
\end{equation*}
\end{lemma}

We recall our convention~\ref{def-conv}: if $f$ is a Maa\ss\ form, then
its level is defined to be the opposite of the arithmetic conductor.

\begin{remark}\label{rm-cancel}
If we assume that $r=r'$ (so $f$ and $g$ are of the same type) and
$s=1/2$, then using the fact that $\lambda_f(n)$ and $\lambda_g(n)$
are real-valued, we obtain from the functional equation the relation
$$
L(f\otimes\chi,\demi)\overline{L(g\otimes \chi,\demi)}=\eps(f)\eps(g)
L(f\otimes\ov{\chi},\demi)\overline{L(g\otimes \ov{\chi},\demi)}.
$$
In particular, if furthermore $\eps(f)\eps(g)=-1$, it follows that
\begin{equation}\label{eq-self-dual}
  L(f\otimes\chi,\demi)\overline{L(g\otimes \chi,\demi)}=-
  L(f\otimes\ov{\chi},\demi)\overline{L(g\otimes \ov{\chi},\demi)}.
\end{equation}
\end{remark}

\subsection{The explicit formula}
\label{pg-explicit}
In Chapter \ref{analyticrank} we will obtain upper bounds for the
analytic rank of $L(\ftchi,s)$ (i.e.\ the order of vanishing at $s=1/2$)
on average over $\chi$. For this we will need the explicit formula in
this specific situation. Define $\Lambda_f$ and $\Lambda_{\ftchi}$ by
the formulas
$$-\frac{L'}{L}(f,s)=\sum_{n\geq 1}\frac{\Lambda_f(n)}{n^s},$$
$$-\frac{L'}{L}(\ftchi,s)=\sum_{n\geq 1}\frac{\Lambda_{\ftchi}(n)}{n^s}=\sum_{n\geq 1}\frac{\Lambda_{f}(n)\chi(n)}{n^s}.$$
The explicit formula for $L(\ftchi,s)$ is:

\begin{proposition}\label{Weilexplicit} 
  Let $\vphi:]0,+\infty[\ra\Cc$ be smooth and compactly supported, and
  let $$\wtilde\vphi(s)=\int_{0}^\infty\vphi(x)x^s\frac{dx}x$$ be its
  Mellin transform and $\psi(x)=x^{-1}\vphi(x^{-1})$ so that
  $\wtilde\psi(s)=\wtilde\vphi(1-s)$. One has
\begin{multline}
  \sum_{n\geq 1}\Bigl(\,\Lambda_f(n)\chi(n)\vphi(n)+\Lambda_f(n)\ov\chi(n)\psi(n)\, \Bigr)=\\
  \vphi(1)\log(q^2|r|)+\intc_{(1/2)}\Bigl(\,\frac{L'_\infty(f,\pm,s)}{L_\infty(f,\pm,s)}+\frac{L'_\infty(f,\pm,1-s)}{L_\infty(f,\pm,1-s)}\,\Bigr)\wtilde\vphi(s)ds-\sum_{\rho}\wtilde\vphi(\rho)\label{eq-Weil}
\end{multline}
where $\rho$ ranges over the multiset of zeros of $\Lambda(\ftchi,s)$
in the strip $0< \Re s< 1$.
\end{proposition}

See \cite{IwKo}*{\S 5.5} for the proof.

\section{Auxiliary \texorpdfstring{$L$-functions}{L-functions}}\label{sectionSymRS}

In addition to Hecke $L$-functions and their twists by characters, 
several auxiliary $L$-functions will play an important role in this
memoir. They will arise as \emph{individual} $L$-functions (not in a
family), typically in expressions for leading terms of various
asymptotic formulas. As a consequence, it is   their behavior
close to $\Reel(s)=1$ that is of most interest.
\par 
We review in this section the definitions of these $L$-functions, and
summarize their analytic properties. We then list some useful
consequences.

\subsection{Ranking-Selberg \texorpdfstring{$L$-functions}{L-functions} on
  \texorpdfstring{$\GL_2$}{GL(2)}}\label{RankinSelberg}

We recall the basic theory of Rankin-Selberg convolution for $\GL_2$.
Given two primitive modular forms $f$ and $g$ of level $r$ and $r'$
respectively with trivial central character, the Rankin-Selberg
$L$-function of $f$ and $g$ is a degree $4$ Euler product
$$
L(\fotg,s)=\prod_p L_p(\fotg,s)=\sum_{n\geq
  1}\frac{\lambda_{\fotg}(n)}{n^s},\ \Re s>1
$$
such that, for $p\nmid rr'$, we have
$$
L_p(\fotg,s)=\prod_{i,j=1}^2\Bigl(\,1-\frac{\alpha_{f,i}(p)\alpha_{g,j}(p)}{p^s}\,
\Bigr)^{-1}
$$
and in general
$$
L_p(\fotg,s)=\prod_{i}^4 \Bigl(\,1-\frac{\alpha_{f\otimes
    g,i}(p)}{p^s}\, \Bigr)^{-1}.
$$ 
In particular, $\lambda_{f\otimes g}(n)=\lf(n)\lamg(n)$ for any $n$
squarefree coprime with $rr'$. An exact description of all Dirichlet
coefficients is given by Winnie Li in~\cite{WinnieLi}, but this is
rather complicated.
\par
By Rankin-Selberg theory, $L(\fotg,s)$ admits analytic continuation to
$\Cc$ with at most one simple pole at $s=1$, which occurs if and only
if $f=g$. This $L$-function satisfies a functional equation of the
shape
$$\Lambda(\fotg,s)=\eps(\fotg)\Lambda(\fotg,1-s)$$
with
$$\Lambda(\fotg,s)=r(\fotg)^{s/2}L_\infty(\fotg,s)L(\fotg,s)$$
where $r(\fotg)$ is a positive integer, $L_\infty(\fotg,s)$ is a
product of Gamma factors and $\eps(\fotg)=\pm 1$.
\label{pg-fg}
Moreover, as a consequence of the descriptions above and of the
approximation to the Ramanujan-Petersson conjecture (cf.\ Section \ref{235}), for any prime $p$
the local factor $L_p(f\otimes g,s)$ has no poles for $\Re s\geq 1/2$. 

In some of our applications, we will also encounter the Dirichlet
series
\begin{equation}\label{lstarfg}
  L^{\ast}(f \otimes g, s) = 
  \sum_{n\geq 1} \frac{\lambda_f(n) \lambda_g(n)}{n^s},
\end{equation}
initially defined in $\Re s > 1$. By the above discussion, it has
holomorphic continuation to $\Re s > 1/2$, except for a pole at $s=1$
which exists if and only if $f=g$.  If $f\not=g$, then
$$
L^{\ast}(f\otimes g,s)=\frac{L(f\otimes
  g,s)}{\zeta^{(rr')}(2s)}\prod_{p\mid rr'}A_p(f,g;s),
$$
for some correction factors $A_p(f,g;s)$ which have been computed
explicitly by Winnie Li~\cite{WinnieLi}*{\S 2, Th.\ 2.2} when $f$ and $g$
are both holomorphic. Here and throughout this book a superscript ${(r)}$ denotes the removal of the Euler factors at primes dividing $r$.

\begin{lemma}\label{lm-ast}
  For two newforms $f \not= g$ we have $L^{\ast}(f\otimes g,1)\not=0$.
\end{lemma}

\begin{proof}
For $p \mid rr'$ we have
$$
A_p(f,g;s)=L^{\ast}_p(f\otimes g,s)L_p(f\otimes g,s)^{-1}.
$$
The factor $L^{\ast}_p(f\otimes g,s)$ is the inverse of a polynomial
at $p^{-s}$ (by multiplicativity), so doesn't vanish. On the other
hand, it follows e.g. from results of Gelbart and
Jacquet~\cite{GJ}*{Prop.\ 1.2, 1.4} that $L_p(f\otimes g,s)$ has no
poles in $\Re(s)\geq 1$, so that $A_p(f,g;1)\not=0$. (More precisely,
if one of the local representations of $f$ or $g$ at $p$ is
supercuspidal, then Prop.\ 1.2 in~\cite{GJ}, and the fact that the
central character is unitary, imply that all poles of
$L_p(f\otimes g,s)$ satisfy $\Reel(s)=0$; on the other hand, if none
of the local representation is supercuspidal, then Prop.\ 1.4
of~\cite{GJ} implies that $L_p(f\otimes g,s)$ is a product of
$\GL_2$-local factors, which have no poles for~$\Reel(s)=1$,
e.g. because they are products of at most two factors
$1/(1-\alpha p^{-s})$ where~$|\alpha|<p^{1/2}$ by elementary results
towards the Ramanujan-Petersson conjecture, as recalled in
Section~\ref{235}). The lemma now follows from the fact that
$L(f\otimes g,1)\not=0$ (see Proposition~\ref{AutomorphicPNTLemma}).
\end{proof}

If $f=g$, then we define $L^{\ast}(\symf ,s)$ through the relation
\begin{equation}\label{eq-symast}
  \zeta^{(r)}(s)L^{\ast}(\symf,s)=\zeta^{(r)}(2s)L^{\ast}(f\otimes f,s).
\end{equation}
In particular, we have
\begin{equation}\label{lstarres}
  L^{\ast}(\symf , 1)=\zeta(2)\prod_{p\mid r}(1+p^{-1})
  \res_{s=1} L^{\ast}(f \otimes f, s).
\end{equation}
Using the formulas in~\cite{WinnieLi}*{p.\ 145, Example 1}, it follows
that if $r$ is squarefree, then we have
$$
L^{\ast}(\symf , 1) = L (\symf,1)
$$ 
where $\symf$ is the symmetric square; we will explain how to recover
this fact (and describe the corresponding formulas if $r$ is not
squarefree) in Section~\ref{symsquare}, using the local Langlands
correspondence.

\subsection{Rankin-Selberg convolutions on \texorpdfstring{$\GL_d$}{GL(d)}}

The previous examples are special cases of Rankin-Selberg
$L$-functions attached to two general automorphic representations of
$\GL_d(\AQ)$. The general theory is due to
Jacquet--Piatetskii-Shapiro--Shalika~\cite{JPSS}, and we recall it
briefly here.

Let $d,e\geq 1$ be integers, and let $\pi,\pi'$ be automorphic
cuspidal representations of $\GL_{d}(\Aa_\Qq)$ and $\GL_{e}(\Aa_\Qq)$,
respectively, whose central characters $\omega,\omega'$ are trivial on
$\Rr_{>0}$. We denote by $\wtilde{\pi}$ and $\wtilde{\pi}'$ their
contragredient representations.

The Rankin-Selberg $L$-function associated to $\pi$ and $\pi'$ is an Euler
product, absolutely convergent for $\Re s>1$, of the form
\begin{align*}
L(\pi\otimes\pi',s)&=\prod_p L_p(\pi\otimes\pi',s)=\prod_p\prod_{i=1}^{d}\prod_{j=1}^{e}\Bigl(\,1-\frac{\alpha_{\pi\otimes\pi',(i,j)}(p)}{p^s}\,\Bigr)^{-1}\\
&=\sum_{n\geq 1}\frac{\lambda_{\pi\otimes\pi'}(n)}{n^s},\ \Re s>1
\end{align*}
such that, for $p$ not dividing the product of the conductors
$q(\pi)q(\pi')$, we have
$$
 \alpha_{\pi\otimes\pi',(i,j)}(p)=\alpha_{\pi,i}(p)\alpha_{\pi',j}(p)
$$
where $\alpha_{\pi,i}(p)$ and $\alpha_{\pi',j}(p)$ are the local
parameters of $\pi,\pi'$ at the place $p$, i.e.
$$
L_p(\pi,s)=\prod_{i=1}^{d}\Bigl(\,1-\frac{\alpha_{\pi,i}(p)}{p^s}\,
\Bigr)^{-1},\
L_p(\pi',s)=\prod_{j=1}^{e}\Bigl(\,1-\frac{\alpha_{\pi',j}(p)}{p^s}\,\Bigr)^{-1}
$$
are the local factors of the standard $L$-functions of $\pi,\pi'$.

When $e=1$ and $\pi'=1$ is the trivial representation, this
Rankin-Selberg $L$-function is the standard $L$-function: we have then
$L(\pi\otimes\pi',s)=L(\pi,s)$.
 
The Rankin-Selberg $L$-functions admit meromorphic continuation to $\Cc$,
and satisfy a functional equations of the shape
$$\Lambda(\pi\otimes \pi',s)=\eps(\pi\otimes\pi')\Lambda(\wtilde\pi\otimes \wtilde\pi',1-s)$$
with $|\eps(\pi\otimes\pi')|=1$ and
$$\Lambda(\pi\otimes \pi',s)=q(\pi\otimes\pi')^{s/2}L_\infty(\pi\otimes\pi',s)L(\pi\otimes\pi',s)$$
where $q(\pi\otimes\pi')\geq 1$ is an integer
and
$$L_\infty(\pi\otimes\pi',s)=\prod_{i=1}^d\prod_{j=1}^{e}\Gamma_\Rr(s+\mu_{\pi\otimes\pi',(i,j)})$$ is a product of Gamma factors. 
The completed $L$-function $\Lambda(\pi\otimes \pi',s)$ is holomorphic
on $\Cc$, unless $\pi'\simeq \wtilde\pi$, in which case it has simple
poles at $s=0,1$.

\makeatletter
\newcommand*{\bigboxplus}{\DOTSB\mathop{\mathpalette\big@boxplus\relax}\slimits@}

\newcommand{\big@boxplus}[2]{%
  \vcenter{%
    \m@th\bigbox@thickness{#1}%
    \sbox\z@{$#1\bigoplus$}%
    \dimen@=\ht\z@ \advance\dimen@\dp\z@
    \hbox{%
      \setlength{\unitlength}{\dimen@}%
      \begin{picture}(1,1)
      \polyline(0.1,0.1)(0.9,0.1)(0.9,0.9)(0.1,0.9)(0.1,0.1)(0.5,0.1)
      \polyline(0.5,0.1)(0.5,0.9)
      \polyline(0.1,0.5)(0.9,0.5)
      \end{picture}%
    }%
  }%
}

\newcommand{\bigbox@thickness}[1]{%
  \ifx#1\displaystyle
    \linethickness{0.2ex}%
  \else
    \ifx#1\textstyle
      \linethickness{0.16ex}%
    \else
      \ifx#1\scriptstyle
        \linethickness{0.12ex}%
      \else
        \linethickness{0.1ex}%
      \fi
    \fi
  \fi
}
\makeatother

If $\pi$ and $\pi'$ are not necessarily cuspidal, but are isobaric
sums of cuspidal representations $\pi_i$ and $\pi'_j$ (whose central
characters are trivial on $\Rr_{>0}$), say
$$
\pi=\bigboxplus_i\mu_i\pi_i,\ \pi'=\bigboxplus_j\nu_j\pi'_j,\
\mu_i,\nu_j\geq 1,
$$
then the Rankin-Selberg $L$-function exists and is given as the
product
$$
L(\pi\otimes\pi',s)=\prod_{i,j}L(\pi_i\otimes\pi'_j,s)^{\mu_i\nu_j}
$$
so that analytic properties in the isobaric case are deduced
immediately from the purely cuspidal case.

The local parameters of the Rankin-Selberg $L$-function enjoy the
following additional properties:
\begin{enumerate}
\item For a prime $p\nmid q(\pi)q(\pi')$, we have
$$\lambda_{\pi\otimes\pi'}(p)=\lambda_\pi(p)\lambda_{\pi'}(p).$$
\item If $d=e$ and $\pi'=\wtilde \pi$ is the contragredient of $\pi$,
  then we have
$$
\lambda_{\pi\otimes\wtilde\pi}(n)\geq 0
$$
for all $n\geq 1$ (cf.\ \cite{RudnickSarnak1996}*{p.\ 318}). 
\item The archimedean local factor $L_\infty(\pi\otimes\pi',s)$ has no
  poles in the half-plane $\Re s> 1$, and likewise for any of the
  local factors $L_p(\pi\otimes\pi',s)$ for $p$ prime, because of the
  absolute convergence of the series $L(\pi\otimes\pi',s)$ in this
  region.
\end{enumerate}

To measure the complexity of an $L$-function, we use the {\em analytic
  conductor}, which is defined as the function\label{pg-conductor}
$$
Q(\pi\otimes\pi',s)=q(\pi\otimes\pi')\prod_{i,j}(1+|\mu_{\pi\otimes\pi',(i,j)}+s|).
$$
The conductor of a Rankin-Selberg $L$-function is controlled by that
of the factors, more precisely we have
\begin{equation}\label{eq-bh}
q(\pi\otimes\pi')\leq q(\pi)^{n'}q(\pi')^n,\quad
Q(\pi\otimes\pi',0)\leq Q(\pi,0)^{n'}Q(\pi',0)^n
\end{equation}
for some $n, n'$ (due to Bushnell and
Henniart~\cite{bushnell-henniart} for the non-archimedean part).
Analogously, we will use the notation $Q(\pi)$ for the analytic
conductor of $\pi$.

\subsection{The symmetric square \texorpdfstring{$L$-function}{L-function}}\label{symsquare}

We return to the case $d=2$. When $f=g$ (of level $r$ and with trivial
central character), it is possible to factor the Rankin-Selberg
$L$-function
$$
  L(f\otimes f,s)=\zeta(s)L(\symf,s)	
$$
where $L(\symf,s)$ is the symmetric square $L$-function of $f$. This
is an Euler product of degree three given by
$$
  L(\symf,s)=\prod_p L_p(\symf,s)=
  \prod_p \prod_{i=1}^3\Bigl(1-\frac{\alpha_{\symf,i}(p)}{p^s}
  \Bigr)^{-1} =\sum_{n\geq 1}\frac{\lambda_{\symf}(n)}{n^s},
$$
for $\Re s>1$. For all $p\nmid r$, we have
\begin{multline*}
L_p(\symf,s)=\frac{L_p(f\otimes f,s)}{(1-p^{-s})^{-1}}\\
=\Bigl(\,1-\frac{\alpha_{f,1}(p)^2}{p^s}\,
\Bigr)^{-1}\Bigl(\,1-\frac{\alpha_{f,2}(p)^2}{p^s}\,
\Bigr)^{-1}\Bigl(\,1-\frac{\alpha_{f,1}\alpha_{f,2}(p)}{p^s}\,
\Bigr)^{-1}.
\end{multline*}
This $L$-function admits analytic continuation to $\Cc$ and satisfies
a functional equation of the shape
$$\Lambda(\symf,s)=\eps(\symf)\Lambda(\symf,1-s)$$
with $\eps(\symf)=+1$ and
$$\Lambda(\symf,s)=q(\symf)^{s/2}L_\infty(\symf,s)L(\symf,s)$$
where $L_\infty(\symf,s)$ is a product of Gamma factors.  In fact, it
was proved by Gelbart-Jacquet \cite{GJ} that $L(\symf,s)$ is the
$L$-function of an automorphic representation on $\GL_{3}$ over $\Qq$,
which we denote $\symf$, and that $L(\symf,s)$ is entire.
This result also implies that the Rankin-Selberg $L$-function
$L(f\otimes f,s)$ is the $L$-function of a (non-cuspidal)
$\GL_{4}(\AQ)$-automorphic representation.


In some applications, as in our Chapter~\ref{ch-modular}, it is of
some importance to understand the precise relation between the
automorphic symmetric square $L$-function of Gelbart-Jacquet and the
``imprimitive'' version $L^{\ast}(\symf, s)$ defined
by~(\ref{eq-symast}), i.e.
$$
\zeta^{(r)}(2s)L^{\ast}(f\otimes f,s)=
\zeta^{(r)}(s)L^{\ast}(\symf,s).
$$

Since it is quite complicated to track the literature concerning this
point (especially when the level of $f$ is not squarefree), we record
the result in our case of interest, and sketch the proof using the
local Langlands correspondance.

Let $f$ be a primitive cusp form with trivial central character and
level $r$. For any prime $p$, let $\pi_p$ be the local representation
of the automorphic representation corresponding to $f$.  The following
list enumerates the possibilities for $\pi_p$, the corresponding
inverse $L$-factors at $p$, namely $L_p(\pi_p,s)^{-1}$ for the
standard $L$-function, and $L_p(\symf,s)^{-1}$ for the automorphic
symmetric square $L$-function, and finally the ``correction factor''
$$
C_p=\frac{L^{\ast}_p(\symf ,s)}{L_p(\symf,s)}.
$$
\par
\medskip
\par
\begin{enumerate}
\item Unramified:
\par
\setlength\extrarowheight{5pt}
\begin{tabular}{c|l}
  $L_p(\pi_p,s)^{-1}$ &   $(1-\alpha_p p^{-s})(1-\beta_p
                        p^{-s})$,\quad\quad$\alpha_p\beta_p=1$\\
  $L_p(\symf,s)^{-1}$ & $(1-\alpha_p^2
                        p^{-s})(1-p^{-s})(1-\beta_p^2p^{-s})$\\
  $C_p$ & $1$
\end{tabular}
\item Unramified up to quadratic twist ($\pi_p=\pi'_p\otimes \eta$ for
  some ramified quadratic character $\eta$ and some unramified
  representation $\pi'_p$):
\par
\begin{tabular}{c|l}
  $L_p(\pi_p,s)^{-1}$ &   $1$\\
  $L_p(\symf,s)^{-1}$ &  $ (1-(\alpha'_p)^2p^{-s})(1-p^{-s})(1-(\beta'_p)^2p^{-s})$\\
  $C_p$ & $(1-(\alpha'_p)^2p^{-s})(1-p^{-s})(1-(\beta'_p)^2p^{-s})$
\end{tabular}
\item Steinberg:
\par
\begin{tabular}{c|l}
  $L_p(\pi_p,s)^{-1}$ &   $1-\alpha_p p^{-s}$,\quad\quad $\alpha_p^2=p^{-1}$\\
  $L_p(\symf,s)^{-1}$ & $1-p^{-1-s}$\\
  $C_p$ & $1$
\end{tabular}
\item Steinberg up to a quadratic twist ($\pi_p=\sigma\otimes\eta$ for
  $\sigma$ the Steinberg representation and some ramified quadratic
  character $\eta$):
\par
\begin{tabular}{c|l}
  $L_p(\pi_p,s)^{-1}$ &   $1$\\
   $L_p(\symf,s)^{-1}$ & $1-p^{-1-s}$\\
  $C_p$ & $1-p^{-1-s}$
\end{tabular}
\item Ramified principal series and not of Type (2):
\par
\begin{tabular}{c|l}
  $L_p(\pi_p,s)^{-1}$ &   $1$\\
  $L_p(\symf,s)^{-1}$ & $1-p^{-s}$\\
  $C_p$ & $1-p^{-s}$
\end{tabular}
\item Supercuspidal equal to its twist by the unramified quadratic
  character:
\par
\begin{tabular}{c|l}
  $L_p(\pi_p,s)^{-1}$ &   $1$\\
   $L_p(\symf,s)^{-1}$ & $1+p^{-s}$\\
  $C_p$ & $1+p^{-s}$
\end{tabular}
\item Supercuspidal not equal to its twist by the unramified quadratic
  character:
\par
\begin{tabular}{c|l}
  $L_p(\pi_p,s)^{-1}$ &  $1$\\
   $L_p(\symf,s)^{-1}$ & $1$\\
  $C_p$ & $1$
\end{tabular}
\end{enumerate}
\par
\medskip
\par



\begin{remark}
  (1) In some references, only ``twist minimal'' representations are
  considered, i.e, those $f$ which have minimal conductor among all
  their twists $f\otimes\chi$ by (all) Dirichlet characters. Cases
  (2), (4) and (5) cannot happen for such representations.
\par
(2) All cases may happen for elliptic curves. Case (1) comes from good
reduction, case (2) from good reduction up to a quadratic twist, (3)
from semistable reduction, (4) from semistable reduction up to a
quadratic twist, while (5), (6) and (7) can all come from potentially
good reduction, with (7) only occurring at primes $2$ and $3$.
\end{remark}

\begin{proposition}
  The above list is correct and complete.
\end{proposition}

\begin{proof}
  We use the local Langlands correspondance (due to Harris and
  Taylor~\cite{taylor-harris}), and its compatibility with the
  symmetric square (due to Henniart~\cite{henniart}).
  The Langlands parameter corresponding to $\pi_p$ is a
  two-dimensional Weil-Deligne representation representation $V$ of
  $\Qq_p$ with trivial determinant. The local $L$-factor is then
$$
\det\bigl(1-p^{-s}F\,\mid\, (\syms V)^{I_p, N}\bigr)^{-1},
$$
where $F$ is the Frobenius automorphism of $\Qq_p$, $I_p$ is the
inertia subgroup of $W_{\Qq_p}$, and $N$ is the monodromy operator
(whose invariant subspace is defined to be its kernel).
The restriction of $V$ to $I_p$ is semisimple, and thus can be of
three possible types:
\begin{enumerate}
\item[(a)] A sum of two copies of the same character.
\item[(b)] A sum of two different characters.
\item[(c)] A single irreducible character.
\end{enumerate}

Only in case (a) can $N$ act non-trivially, as $N$ is nilpotent and commutes with $I_p$, and we will handle that separately. 


We now compute the symmetric square and the Frobenius action on the
inertia invariants in each case. It will be convenient to recall in
some of the cases that, because $V$ has trivial determinant, we can
canonically identify $\syms V$ with the space of endomorphisms of $V$
with trace zero.

\textbf{Case (a) -- $N$ trivial.} Because the determinant is trivial,
the inertia character that appears must be either quadratic or
trivial. In this case the representation is unramified, potentially
after a quadratic twist.  This gives cases (1) and (2); the $L$-factor
calculation is well-known (see, e.g.,~\cite[Section 3.5]{GJ},
or~\cite[p. 107, case 1]{coates-schmidt} on the Galois side).

\textbf{Case (a) -- $N$ nontrivial.} By the same logic, the character
is quadratic or trivial. Then the inertia representation is trivial,
potentially after a quadratic twist, and the associated smooth
representation is Steinberg, potentially after a quadratic twist. This
gives cases (3) and (4), and the $L$-factor calculation is also
well-known (e.g., the $GL_2$ case can be found
in~\cite[Th. 6.15]{gelbart}, and the symmetric square factor is
computed, on the Galois side, in~\cite[p. 107, case
2]{coates-schmidt}, recalling that the symmetric square is unchanged
by a quadratic twist).

\textbf{Case (b).} The inertia invariants of the adjoint
representation form a one-dimensional space. The Frobenius action on
this space defines a one-dimensional unramified character $\eta$,
which is either trivial or nontrivial. It is trivial if and only if
there is a non-scalar endomorphism of the whole representation, i.e.,
if it fails to be irreducible, or in other words if the corresponding
smooth representation is a principal series. In this case, the
Frobenius action on inertia invariants is trivial, so the factor is
$1/(1-p^{-s})$. This is case (5). If $\eta$ is non-trivial, then we
have an isomorphism $V\otimes\eta\to V$, which taking determinants
implies that $\eta$ is quadratic. Hence we have the corresponding
isomorphism on the automorphic side, and the $L$-factor is
$1/(1-p^{-s})$. This is case (6).

\textbf{Case (c).} The space of inertia invariants of the adjoint
representation vanishes. Then the local $L$-factor is $1$, and so the
representation has no nontrivial endomorphisms and thus is
irreducible, hence the corresponding automorphic representation is
supercuspidal, and has $L$-factor $1$. This is case (7).
\end{proof}


\subsection{Symmetric power \texorpdfstring{$L$-functions}{L-functions}}
\label{ssec-symk}

More generally, for any integer $k\geq 1$, one can form the symmetric
$k$-th power $L$-function $L(\sym^kf,s)$, which is an Euler product of
degree $k+1$, namely
$$
  L(\sym^kf,s)=\prod_pL_p(\sym^k,f,s)
  =\prod_p\prod_{i=0}^k (1-\alpha_{\sym^kf,i}p^{-s})^{-1} =\sum_{n\geq
    1}\frac{\lambda_{\sym^k f}(n)}{n^s}
$$
and for $p\nmid r$,
$$
L_p(\sym^k,f,s)=\prod_{i=0}^{k}\Bigl(\,1-\frac{\alpha_{f,1}^i(p)\alpha_{f,2}^{k-i}(p)}{p^s}\,
\Bigr)^{-1}.
$$
The analytic continuation of these Euler products is not known in
general.  For $k=3$ and $k=4$, Kim and Shahidi \cites{KiSh1,KiSh2}
have proven that $L(\sym^kf,s)$ is the $L$-function of a self-dual
automorphic (not-necessarily cuspidal) representation of $\GL_{k+1}$
and in particular it admits analytic continuation to $\Cc$ and
satisfies a functional equation of the usual shape:
$$\Lambda(\sym^kf,s)=\eps(\sym^k f)\Lambda(\sym^k f,1-s)$$
where 
$\eps(\sym^kf)
=\pm 1$ and
$$\Lambda(\symf,s)=q(\sym^kf)^{s/2}L_\infty(\sym^kf,s)L(\sym^kf,s),$$
and again $L_\infty(\sym^kf,s)$ is a product of Gamma factors.

We summarize the results of Kim and Shahidi, as well as those of
Gelbart and Jacquet that were already mentioned, as follows.
\par
For $k \leq 4$ the $L$-function $L(\sym^kf,s)$ is the $L$-function of an automorphic
representation $\sym^kf$ of $\GL_{k+1}(\AQ)$. The representation
$\sym^kf$ decomposes into an isobaric sum
\begin{equation}\label{eq-isobaric}
 \sym^kf\simeq\bigboxplus_{j=1}^{n_{f,k}}\mu_j\pi_j 
\end{equation}
where $n_{f,k}\geq 1$ and $(\pi_j)$ are cuspidal automorphic
representations on $\GL_{d_j}(\AQ)$.  This implies
$$
L(\sym^kf,s)=\prod_jL(\pi_j,s)^{\mu_j}.
$$
The decomposition~(\ref{eq-isobaric}) satisfies $\sum_j\mu_jd_j=k+1$.
The automorphic representation $\sym^k\pi$ is self-dual, hence its
decomposition into isotypical components is invariant by taking
contragredient, i.e., the multiset
$\{(\mu_j,\pi_j)\,\mid\, j\leq n_{f,k}\}$ is invariant under
contragredient. Moreover, for every $1\leqslant j\leqslant n_{f,k}$,
we have
$$
\pi_j\simeq\tilde{\pi}_j\quad\text{or}\quad d_j\leqslant 2.
$$ 
\par
We now list the precise possibilities for the decomposition. Let $\pi$
be the automorphic representation associated to $f$. It is self-dual
with trivial central character, and $\sym^kf=\sym^k\pi$.  If $\pi$ is
of CM-type, then $\pi\otimes\eta\simeq\pi$ for a nontrivial quadratic
Dirichlet character $\eta$, which determines a quadratic extension
$E/\Qq$, and there exists a Gr\"o\ss en\-char\-acter $\chi$ of $E$
such that $\pi=\pi(\chi)$ (the automorphic induction of $\chi$). Write
$\chi'$ for the conjugate of $\chi$ by the nontrivial element of
$\mathrm{Gal}(E/\Qq)$. Then we have:
\begin{itemize}
\item $\sym^2\pi=\pi(\chi^2)\boxplus\eta$,
\item $\sym^3\pi=\pi(\chi^3)\boxplus\pi(\chi^2 \chi')$,
\item $\sym^4\pi=\pi(\chi^4)\boxplus\pi(\chi^3 \chi')\boxplus\bfone$.
\end{itemize}
The individual terms $\pi(\chi^a \chi'^b)$ either remain cuspidal, and have unitary central character, or split into the Eisenstein series of two unitary characters.

If $\pi$ is not of CM-type,  
then
the automorphic representations $\sym^2\pi$, $\sym^3\pi$ and
$\sym^4\pi$ are all cuspidal and self-dual, with central character
trivial on $\Rr_{>0}$. 

In either case, we conclude

\begin{corollary}\label{cor-iso}
  The automorphic representations $\pi_j$ on $\GL_{d_j}(\AQ)$ of the
  isobaric decomposition of $\sym^k\pi$ have unitary central
  characters trivial on $\Rr_{>0}$. They satisfy either
  $\pi_j\simeq\tilde{\pi}_j$, or $d_j\leqslant 2$.
\end{corollary}

\begin{remark}
  One can check that this corollary remains true for any cuspidal
  automorphic representation $\pi$, even if the central character of
  $\pi$ is non-trivial. However, checking this requires the
  consideration of more cases, since $\pi$ could be of polyhedral
  type.
\end{remark}


\subsection{The Ramanujan-Petersson conjecture and its approximation}\label{235}

The \emph{Ramanujan-Petersson conjecture} at unramified places
predicts optimal bounds for the local parameters of $f$ (equivalently
a pole free region for the local $L$-factors), namely it predicts that
\begin{align}\label{RPconj}
  |\alpha_{f,i}(p)|&\leq 1,\ i=1,2,\\
  \Re \mu_{f,i}&\leq 0,\ i=1,2.  
\nonumber
\end{align}
This would imply that for any $(n, q(f)) = 1$ one has
\begin{equation}\label{RPconjbound}
|\lf(n)|\leq d(n)
\end{equation}
where $d(n)$ is the divisor function.

If $f$ is holomorphic, the Ramanujan-Petersson conjecture is known by
the work of Deligne \cite{WeilI}.  Moreover, it is known that
\eqref{RPconj} holds for any prime $p\mid q(f)$, and so
\eqref{RPconjbound} holds for all integers $n$.

The results on the functoriality of the symmetric power $L$-functions
$L(\sym^k f,s)$ mentioned above together with Rankin-Selberg theory
imply that the Ramanujan-Petersson conjecture is true on average in a
strong form: for any $x\geq 1$ and $\eps>0$, we have
\begin{equation}\label{RP4}
  \sum_{n\leq x}\Bigl(\,|\lf(n^4)|^2+|\lf(n^2)|^4+|\lf(n)|^8\, \Bigr)
  \ll_f x^{1+\eps},
\end{equation}
(see e.g.\ \cite[Theorem 1.2]{GLu} for the last bound, the other cases being very similar) where the implied constant depends on~$\eps$, and also
\begin{equation}\label{RSbis}
\sum_{n\leq x} \vert \lambda_f (n)\vert^2 \ll_f x
\end{equation}
This  implies that
$$|\alpha_{f,i}(p)|\leq p^{1/8},\ i=1,2.$$
\par
With additional more sophisticated arguments, Kim and Sarnak
\cite{KiSa} have obtained the currently best approximation to the
Ramanujan-Petersson conjecture. For $\theta=7/64$, we have
\begin{align*}
  |\alpha_{f,i}(p)|&\leq p^\theta\\
  \Re \mu_{f,i}&\leq \theta,\ i=1,2,
\end{align*}
and therefore, for any $n\geq 1$, we have
\begin{equation}\label{KSbound}
|\lf(n)|\leq d(n)n^\theta.
\end{equation}
On the other hand, for $p \mid r$, we have (cf.\ e.g. \cite[Theorem 4.6.17]{miyake} or more generally \cite{gelbart})
\begin{equation}\label{ramified}
  |\lf(p)|= p^{-1/2}\text{ or } \lambda_f(p)=0.
\end{equation}

\emph{For the rest of the book the letter $\theta \leq 7/64$ is
  reserved for an admissible exponent towards the Ramanujan-Petersson
  conjecture.}

\section{Prime Number Theorems}
\label{MomentsHeckeEigenvaluesSection}

By ``Prime Number Theorems'' we mean the problem of evaluating
asymptotically certain sums over the primes of arithmetic functions
associated to Hecke eigenvalues of $f$ and $g$. The main tool for this
is the determination of zero-free regions of the relevant
$L$-functions. We first state a general result concerning the
zero-free domain for Rankin-Selberg $L$-functions. 

\begin{proposition}
\label{AutomorphicPNTLemma}
Let $\pi$ and $\pi'$ be irreducible cuspidal automorphic
representations of $\GL_d(\Aa_\Qq)$ and
$\GL_{e}(\Aa_\Qq)$. Assume that the central characters
$\omega_\pi$ and $\omega_{\pi'}$ are unitary and trivial on $\Rr_{>0}$
and either:
\begin{enumerate}
\item At least one of $\pi$ or $\pi'$ is a $\GL_1$-twist of a
  self-dual representation, possibly the trivial one, or
\item $d\leq 3$ and $e\leq 2$, or vice-versa.
\end{enumerate} 
Then, there is  an explicitly computable constant $c=c(d,e)>0$ such that the Rankin--Selberg $L$-function $L(\pi\times\pi',\sigma+it)$ has no zeros in the region
\begin{equation}\label{HdvPregion}
\sigma>1-\frac{c}{\log(Q({\pi})Q({\pi'})(|t|+2))}	
\end{equation}
except for at most one exceptional simple Landau-Siegel real zero
$<1$. Such a zero may only occur if $\pi\otimes\pi'$ is self-dual,
i.e., if $\tilde\pi\otimes\tilde\pi'\simeq\pi\otimes\pi'$ as
admissible representations.
\end{proposition}

\begin{proof}
  If $\pi$ and $\pi'$ are both self-dual, then this is a result of
  Moreno~\cite{Moreno1985}*{Theorem~3.3}. If only one of the two is
  self-dual it was observed by Sarnak that Moreno's method
  extends~\cite{Sarnak2004}. However we could not find a proof of this
  in the literature and we take this opportunity to report a proof
  kindly provided by F. Brumley. We assume that $\pi'$ is self-dual
  and that $\pi$ is not (in particular
  $\pi'\not\simeq\pi,\tilde\pi$). Given a non-zero real number $t$,
  consider the isobaric representation
$$
\Pi=
(\pi\otimes|\cdot|^{-it})\boxplus(\tilde\pi\otimes|\cdot|^{it})\boxplus
\pi'
$$
and its Rankin-Selberg $L$-function
$$
L(s)=L(\Pi\otimes\tilde\Pi,s).
$$ 
This $L$-function factors as a product of the following nine
$L$-functions:
\begin{gather*}
  L(\pi\otimes\tilde\pi,s),\quad 
  L(\pi\otimes\tilde\pi,s),\quad
  L(\pi'\otimes\pi',s)	,\\
  L(\pi\otimes\pi,s+2it),\quad
  L(\pi\otimes\pi',s+it),\\
  L(\tilde\pi\otimes\tilde\pi,s-2it),\quad
  L(\tilde\pi\otimes\pi',s-it),\\
  L(\pi\otimes\pi',s+it),\quad
  L(\tilde\pi\otimes\pi',s-it).
\end{gather*}
Also by construction the coefficients of $-L'/L(s)$ are non-negative
so that we can use the Goldfeld-Hoffstein-Lieman Lemma~\cite{IwKo}*{Lemma
5.9}.

The $L$-function $L(s)$ has a pole of order $3$ at $s=1$. On the other
hand, suppose that $L(\pi\otimes\pi',\sigma+it)=0$ for $\sigma<1$
satisfying \eqref{HdvPregion}. Then $L(s)$ vanishes to order at least
$4$ at $\sigma$ (the two factors $L(\pi\otimes\pi',s+it)$ and the two
factors $L(\tilde\pi\otimes\pi',s-it)$), thus contradicting the
Goldfeld-Hoffstein-Lieman Lemma if $c$ is small enough, depending on
$d,e$.


Suppose now that neither $\pi$ nor $\pi'$ are self-dual up to
$GL_1$-twists. If $d=e=2$, then the result follows from the functorial
lift $\GL_2\times\GL_2\to\GL_4$ of Ramakrishnan~\cite{dinakar}*{Theorem
M}, according to which there exists an isobaric automorphic
representation $\pi\boxtimes\pi'$ of $\GL_4(\Aa_\Qq)$ (with unitary
central character trivial on $\Rr^{+}$) such
that $$L(s,\pi\boxtimes\pi')=L(s,\pi\times\pi').$$ If $d=2$ and $e=3$
this follows from the functorial lift $$\GL_2\times\GL_3\to\GL_6$$
established by Kim and Shahidi \cite{KiSh1}.
\end{proof}

\begin{remark}
  (1) This result covers the case when at least one of $\pi$ or $\pi'$
  is a $GL_1$-twist of the self-dual representation by passing the
  twist to the other factor.

  In particular, this contains the case where (say) $\pi'=1$ is the
  (self-dual) trivial representation, that is the standard zero-free
  region 
$$
\sigma>1-\frac{c}{\log(Q({\pi})(|t|+2))}
$$ 
for the standard $L$-function $L(\pi,s)$ of any cuspidal
representation, except for the possible Landau--Siegel zero if
$\pi\simeq\tilde{\pi}$.

(2) In our actual applications in this book, we will apply the result
only to a finite set of auxiliary $L$-functions (depending on the
given cusp forms $f$ and $g$, which are fixed), hence the issue of
Landau-Siegel zeros is not an important one, as long as we have a
standard zero-free region in $t$-aspect.
\end{remark}

From this, we deduce the next result.

\begin{proposition} Let $f,g$ be primitive cusp forms of levels
  $r, r'$ with trivial central character. There exists an absolute
  constant $c>0$ such that for $k,k'\leq 4$ the Rankin-Selberg
  $L$-function $L(\sym^kf\otimes\sym^{k'}g,s)$ has no zeros in the
  domain
$$\Re s\ge 1-\frac{c}{\log(Q(f)Q(g)(2+|s|))}$$
except for possible real zeros $<1$.
\end{proposition}

\begin{proof} 
  In terms of the isobaric decompositions~(\ref{eq-isobaric}) of $f$
  and $g$ given in Section~\ref{ssec-symk}, we have
$$L(\sym^kf\otimes\sym^{k'}g,s)=\prod_{i,j}L(\pi_i\otimes\pi'_j,s)^{\mu_i\nu_j}.$$
It will then be sufficient to prove the result for each factor
$L(\pi_i\otimes\pi'_j,s)$, since
$$
Q(\pi_i\otimes\pi'_j)\leq (Q( \pi_i)Q(\pi'_j))^{O(1)}\leq (Q(
\sym^kf)Q(\sym^{k'}f))^{O(1)}\leq (Q(f)Q(g))^{O(1)}
$$
by~(\ref{eq-bh}) (and~\cites{KiSh1,KiSh2}). By Corollary~\ref{cor-iso} 
we see that at least one of the two
sufficient conditions of Proposition~\ref{AutomorphicPNTLemma} is
always satisfied. 
\end{proof}

We now spell out several corollaries which are deduced from these
zero-free domains by standard techniques.
The first one concerns upper and lower bounds for values of this
$L$-function in the zero-free region:

\begin{corollary}\label{cor29}
  Let $f,g$ be primitive cusp forms of levels $r, r'$ with trivial
  central character. For $0\leq k,k'\leq 4$, there exist two constants
  $0<c=c_{f,g}<1/10$ and $A=A_{f,g}\geq 0$ such that for $s$
  satisfying
$$\Re s\geq 1-\frac{c}{\log(2+|s|)}$$
the following bounds hold:
\begin{equation*}
  \log^{-A}(2+|s|)\ll \Bigl\vert\,\frac{s-1}{s}\, 
  \Bigr\vert^{\rho}L(\sym^k f\otimes \sym^{k'}g,s)\ll \log^{A}(2+|s|)
\end{equation*}
where
$$\rho=\rho_{f,g,k,k'}=\ord_{s=1}L(\sym^k f\otimes \sym^{k'}g,s)\geq
0$$ is the order of the pole of $L(\sym^k f\otimes \sym^{k'}g,s)$ at
$s=1$ and the implicit constants depends on $f$ and $g$ only. Here we
also make the convention that for $k=k'=0$, we have
$L(\sym^k f\otimes \sym^{k'}g,s)=\zeta(s)$.
\end{corollary}

The second corollary concerns the versions of the Prime Number
Theorem that can be deduced from these zero-free regions:

\begin{corollary}\label{PNTsymandpower} 
  Let $f,g$ be primitive cusp forms of levels $r, r'$ with trivial
  central character. Let $0\leq k,k'\leq 4$. There exists a constant
  $C>0$ such that:
\par
\emph{(1)} There exist $\gamma_{k,k'}\in\Rr$ and an integer
$m_{k,k'}\geq 0$ (possibly also depending on $f, g$) such that for any $x \geq 2$, we have
\begin{align*}
  \sum_{p\leq x}\lambda_{\sym^kf}(p)\lambda_{\sym^{k'}g}(p)\log
  p&=m_{k,k'} x+O(x\exp(-C\sqrt{\log x}))\\
  \sum_{p\leq x}\lambda_{\sym^kf}(p)\lambda_{\sym^{k'}g}(p)\frac{\log
  p}{p}&=m_{k,k'}\log x+\gamma_{k,k'}+O\Bigl(\frac{1}{\log x}\Bigr).
\end{align*}
\par
\emph{(2)} There exists $\gamma'_{k,k'}\in\Rr$ and an integer
$n_{k, k'} \geq 0$ (possibly depending on $f, g$) such that
\begin{align}\label{powerPNT}
  \sum_{p\leq x}\lf(p)^k\lamg(p)^{k'}\log p
  &=n_{k,k'}x+O(x\exp(-C\sqrt{\log x})),	
  \\
  \label{powermertens}
  \sum_{p\leq x}\lf(p)^k\lamg(p)^{k'}\frac{\log p}{p}
  &=n_{k,k'}\log   x+\gamma'_{k,k'}+O\Bigl(\frac{1}{\log x}\Bigr),
\end{align}
and for $2\leq x\leq y/2$ we have
\begin{equation}\label{powerresonator}
  \sum_{x\leq p\leq y}\frac{\lf(p)^k\lamg(p)^{k'}}{p\log p}
  =\left( n_{k,k'} +  O\Bigl(\frac{1}{\log x}\Bigr)\right) \Bigl(\frac{1}{\log x}-\frac{1}{\log y}\Bigr).
\end{equation}
\par
In these estimates, the implied constants depend on $f$ and $g$ only.
\end{corollary}

\begin{proof} 
  The first two equalities are deduced from the zero free region for
  $L(\sym^kf\otimes\sym^{k'}g,s)$ (see for instance
  Liu--Ye~\cite{LiuYe2007}).

  The remaining ones follow by partial summation, using the
  decompositions
\begin{gather*}
  \lambda_{f}(p)=\lambda_{\sym^1f}(p),\\
  \lambda_{f}(p)^2=\lambda_{f}(p^2)+1=\lambda_{\sym^2f}(p)+1,\\
  \lambda_{f}(p)^3=\lambda_{f}(p^3)+2\lambda_{f}(p)=\lambda_{\sym^3f}(p)+2\lambda_{\sym^1f}(p),\\
  \lambda_{f}(p)^4=\lambda_{f}(p^4)+3\lambda_{f}(p^2)+2=\lambda_{\sym^4f}(p)+3\lambda_{\sym^2f}(p)+2
\end{gather*}
for $p\nmid rr'$, which reflect the decomposition of tensor powers of
the standard representation of $\SL_2$ in terms of irreducible
representation (in particular, all coefficients are non-negative
integers).
\end{proof}

\begin{remark}
From
\begin{multline*}
  \lf(p)^2\lamg(p)^2=(\lambda_{\sym^2f}(p)+1)(\lambda_{\sym^2g}(p)+1)\\
  =\lambda_{\sym^2f}(p)\lambda_{\sym^2g}(p)
  +\lambda_{\sym^2f}(p)+\lambda_{\sym^2g}(p)+1
\end{multline*}
for $p\nmid rr'$, we see that
$$
n_{2,2}=m_{2,2}+m_{2,0}+m_{0,2}+1\geq 1,
$$
and similarly
$$
n_{4,4}=m_{4,4}+3m_{4,2}+3m_{2,4}+
2m_{4,0}+2m_{0,4}+9m_{2,2}+3m_{2,0}+3m_{0,2}+4\geq 4.
$$
\end{remark}

We will also need a variant. We denote by $\lambda_f^{\ast}$ and
$\lambda_{g}^{\ast}$ any multiplicative functions such that
\begin{equation}
\label{LambdaStarAdjustment}
\lambda_f^{\ast}(p)=\lambda_f(p)+O(p^{\theta-1}),\quad 
\lambda_{g}^{\ast}(p)=\lambda_{g}(p)+O(p^{\theta-1}),
\end{equation}
where the implied constants depend on $f$ and $g$.  (Note that these
functions may depend on both $f$ and $g$).

\begin{corollary}\label{corlambdaast} 
  The estimates \eqref{powerPNT}, \eqref{powermertens} and
  \eqref{powerresonator} are valid with $\lf,\lamg$ replaced by
  $\lambda_f^{\ast},\ \lambda_g^{\ast}$, with the same integers
  $n_{k,k'}$, but with possibly different values for $\gamma_{k,k'}$.
\end{corollary}

\begin{proof} 
  It suffices to verify \eqref{powerPNT}. Since
  $|\lf(p)|,|\lamg(p)|\leq 2p^\theta$, we have
$$\lambda_f^{\ast}(p)^k\lambda_g^{\ast}(p)^{k'}
=\lambda_f(p)^k\lambda_g(p)^{k'}+O(p^{(k+k')\theta-1})$$ and, since
$(k+k')\theta\leq 8\theta <1$, the difference between
$$\sum_{p\leq x}\lambda_f^{\ast}(p)^k\lambda_g^{\ast}(p)^{k'}\log p\ \hbox{  and  }\ 
\sum_{p\leq x}\lf(p)^k\lamg(p)^{k'}\log p$$ is
$\ll x\exp(-C\sqrt{\log x})$.
\end{proof}

\section{Consequences  of the functional equations}

The functional equation satisfied by an $L$-function makes it possible
to obtain (by inverse Mellin transform) either a representation of its
values by rapidly converging smooth sums (this is called, somewhat
improperly, the ``approximate functional equation''), or identities
between rapidly converging smooth sums of these coefficients (an
example is the Voronoi summation formula). We discuss the versions of
these identities that we need in this section.

\subsection{Approximate functional equations}

The following proposition is obtained by specializing~\cite{IwKo}*{Thm.\ 5.3,
Prop.\ 5.4} to twisted $L$-functions and to the product of two
twisted $L$-functions, using the functional equations of
Lemmas~\ref{lm-fneq} and~\ref{propfcteqn}. Again, we recall that we
use the convention~\ref{def-conv} about the signed level of a Maa\ss\ form.

\begin{proposition}\label{pr-approx}
  Let $f,g$ be two primitive cusp forms of signed levels $r$ and $r'$ coprime
  to $q$, where $f=g$ is possible. Given any $A>2$, let $G=G_A$ be the
  holomorphic function defined in the strip $|\Re u| < 2A$ by
\begin{equation}\label{Gdef}
G(u)=\Bigl(\cos\frac{\pi u}{4A}\Bigr)^{-16 A}.	
\end{equation}
Let $s\in\Cc$ be such that $1/4 < \Re s < 3/4$, and let $\chi$ be a
primitive Dirichlet character modulo $q$, with parity
$\chi(-1)=\pm 1$.
\par
\emph{(1)} We have
\begin{multline}\label{fcteqn1}
  L(f \otimes \chi, s)= \sum_{m\geq 1} \frac{ \lf(m)}{m^{s}}\chi(m)
  V_{f,\pm,s}\Big(\frac{m}{q\sqrt{|r|}}\Big) \\
  +\eps(f,\pm,s) \sum_{m\geq
    1}\frac{\lf(m)}{m^{1-s}}\ov\chi(m)
  V_{f,\pm,1-s}\Big(\frac{m}{q\sqrt{|r|}}\Big),
 \end{multline}
where 
$$
\eps(f,\pm,s)=\eps(\ftchi)(q^2|r|)^{\frac12-s}\frac{L_\infty(f,\pm,1-s)}{L_\infty(f,\pm,s)}
$$
and 
\begin{equation*}
  V_{f,\pm,s}(y) = \frac{1}{2\pi i} \int_{(2)} \frac{L_\infty(f,\pm,s+u)}{L_\infty(f,\pm,s)}G(u) y^{-u} \frac{du}{u}.
\end{equation*}
\par
\emph{(2)} We have
\begin{multline}\label{fcteqn2}
  L(f \otimes \chi, s) \ov{L(g \otimes \chi,s)} =
  \sum_{m,n\geqslant 1} \frac{\lambda_f(m)
    \lamg(n)}{m^sn^{\spr}}\chi(m) \ov{\chi}(n)
  W_{f,g,\pm,s}\left(\frac{mn}{q^2|rr'|}\right) \\+ 
  \eps(f,g,\pm,s)\chi(r\bar{r}')\sum_{m,n\geq 1}\frac{\lf(m)\lamg(n)}{m^{1-s}n^{1-\spr}}\ov{\chi}(m)
   {\chi}(n) W_{f,g,\pm,1-s}\left(\frac{mn}{q^2|rr'|}\right), 
\end{multline}
where
\begin{equation}\label{eq-epspm}
  \eps(f,g,\pm,s)=\eps(f)\eps(g)(q^2r)^{\demi - s} (q^2r)^{\demi - \bar{s}} \frac{L_\infty(f,\pm,1-s)}{L_\infty(f,\pm,s)}\frac{L_\infty(g,\pm,1-\spr)}{L_\infty(g,\pm,\spr)}
\end{equation}
(again these expressions depend only on the parity of $\chi$) and
\begin{equation}\label{Wfgdef}
  W_{f,g,\pm,s}(y) = \frac{1}{2\pi i} \int_{(2)} \frac{L_\infty(f,\pm,s+u)}{L_\infty(f,\pm,s)}\frac{L_\infty(g,\pm,\spr+u)}{L_\infty(g,\pm,\spr)}G(u) y^{-u} \frac{du}{u}.
\end{equation}
\end{proposition}

Note that the Hecke eigenvalues $\lambda_f(n), \lambda_g(n)$ are real. Also note   the special cases
\begin{equation}\label{eq-epspm-demi}
  \eps(f,\pm,\demi)=\eps(f\otimes\chi)=\eps(f)\chi(r)\eps_{\chi}^2,
  \quad\quad
  \eps(f,g,\pm,\demi)=\eps(f)\eps(g).
\end{equation}

We need to record some decay properties for $V_{f,\pm,s}, W_{f,g,\pm,s}$ and their derivatives.


 Shifting the contour to $\Re u= A$ or $\Re u=-(\sigma-\theta)+\eps$ for $\eps>0$ and using Stirling's formula, we have
\begin{lemma} Assume that $\sigma=\Re s\in]1/4,3/4[$. For any integer $j\geq 0$ any $y>0$, we have 
\begin{equation}\label{Vassymp}
V_{f,\pm,s}(y)-1\ll (y/|s|)^{\sigma-\theta-\eps},\ W_{f,g,\pm,s}(y)-1\ll (y/|s|^2)^{\sigma-\theta-\eps}
\end{equation}
and
\begin{equation*}
y^{j}V_{f,\pm,s}^{(j)}(y)\ll (1+y/|s|)^{-A},\ y^{j}W_{f,g,\pm,s}^{(j)}(y)\ll (1+y/|s|^2)^{-A}
\end{equation*}
where the constant implied depends on $f,g$, $\eps$ and $j$ (where applicable).
\end{lemma}

\begin{convention}\label{conveven} 
  In most of this book, we will only treat in detail averages over the
  {\em even} characters, since the odd case is entirely similar.  To
  simplify notation, we may then write $\eps(f,s)$, $V_{f,s}$ and
  $W_{f,g,s}$ in place of $\eps(f,+,s)$, $V_{f,+,s}$ and
  $W_{f,g,+,s}$. Moreover, for $s=1/2$, we may simplify further, and
  write $V_{f}$ and $W_{f,g}$ in place of $V_{f,1/2}$ and
  $W_{f,g,1/2}$.
\end{convention}

\subsection{The Voronoi summation formula}

The next  lemma is a version of the Voronoi formula. 

\begin{lemma}\label{221}
  Let $q$ be a positive integer and $a$ an integer coprime to $q$, and
  let $W$ be a smooth function compactly supported in $]0,\infty[$.
  Let $f$ a primitive cusp form of signed level $r$ coprime with $q$ and
  trivial central character. For any real number $N>0$, we have
$$
\sum_{n\geq 1} \lf(n)W\Bigl(\frac{n}N\Bigr)e\Bigl(\frac{an}{q}\Bigr) =
\eps(f)\sum_{\pm}\frac{N}{q|r|^{1/2}} \sum_{n\geq
  1}\lf(n)e\Bigl(\mp\frac{\overline{a|r|}n}{q}\Bigr) \widetilde
W_{\pm}\Bigl(\frac{Nn}{q^2|r|}\Bigr)
$$
with
$$\widetilde W_{\pm}(y)=\int_{0}^\infty W(u)\mathcal{J}_{\pm
}(4\pi\sqrt{ uy})du,$$ where \emph{(1)} for $f$ holomorphic of weight
$k_f$ we write
\begin{equation*}
  \mathcal{J}_{+}(u)  =
  2\pi i^{k_f}J_{k_f-1}(u)
  , \quad  \mathcal{J}_{-}(u)=0;
\end{equation*}
 
\emph{(2)} for $f$ a Maa{\ss} form with Laplace eigenvalue
$(\frac12+it_f)(\frac12-it_f)$ and reflection eigenvalue $\eps_f=\pm1$
we write
\begin{equation*}
  \mathcal{J}_{+}(u)  =\frac{-\pi}{\sin(\pi t_f)}
  (J_{2it_f}(u)-J_{-2it_f}(u)), 
  \quad \mathcal{J}_{-}(u)= 4\eps_f\cosh(\pi t_f)K_{2it_f}(u). 
 \end{equation*} 
\end{lemma}

See \cite{KMVDMJ}*{Theorem A.4} for the proof. Note that $\widetilde{W}_{\pm}$ depends on the archimedean parameters of $f$, which we suppress from the notation.   In particular, the passage from a smooth weight function $W$ to $\widetilde{W}_{\pm}$ may increase the set of parameters $\mathscr{S}$, cf.\ Section \ref{21}. 
\par
\medskip
\par
Let $K\colon \Zz\to \Cc$ be a $q$-periodic function. Its
\emph{normalized Fourier transform} is the $q$-periodic function
defined by
\begin{equation*} \label{pg-what}
  \fourier{K}(h) = \frac{1}{\sqrt{q}}
  \sum_{n\mods q} K(n) e\Bigl(\frac{nh}{q}\Bigr)
\end{equation*}
for $h\in\Zz$. The \emph{Voronoi transform} of $K$ is the $q$-periodic
function defined by
$$
\bessel{K}(n) = \frac{1}{\sqrt{q}}\sum_{\substack{h\mods q\\(h,q) =1}}
\fourier{K}(h) e\Bigl(\frac{\bar{h}n}{q}\Bigr)
$$
for $n\in\Zz$ (see~\cite{FKM3}*{\S 2.2}). Combining the Voronoi formula
above with the discrete Fourier inversion formula
$$
K(n)=\frac{1}{\sqrt{q}}\sum_{a\mods q}\what
K(a)e\Bigl(-\frac{an}{q}\Bigr),
$$
we deduce:

\begin{corollary}\label{corvoronoi} Let $q$ be a prime number. Let $W$
  be a smooth function compactly supported in $]0,\infty[$.  Let $f$
  be a primitive cusp form of signed level $r$ coprime with $q$. For any real
  number $N>0$, we have
\begin{align*}
  \sum_{n}\lf(n)K(n)W\Bigl(\frac nN\Bigr)
  &= \frac{\what K(0)}{q^{1/2}}\sum_{n\geq 1} \lf(n)W\Bigl(\frac nN\Bigr)+\\ 
  &\ \eps(f)\sum_{\pm}\frac{N}{q|r|^{1/2}} \sum_{n\geq
    1}\lf(n)\widecheck K(\pm \overline{|r|} n)\widetilde W_{\pm}\Bigl(\frac{nN}{q^2|r|}\Bigr) 
\end{align*}
In particular, for any integer $a$ coprime to $q$, if we take
$$
K(n)=q^{1/2}\delta_{n\equiv a\mods q}=\begin{cases}
q^{1/2}&\text{ if } n\equiv a\mods{q}\\
0&\text{ otherwise,}
\end{cases}
$$
then we have
\begin{align*}
  q^{1/2}\sum_{n\equiv a\mods q} \lf(n)W\Bigl(\frac nN\Bigr)
  &= \frac{1}{q^{1/2}}\sum_{n\geq 1} \lf(n)W\Bigl(\frac nN\Bigr)+\\ 
  &\ \eps(f)\sum_{\pm}\frac{N}{q|r|^{1/2}} \sum_{n\geq
    1}\lf(n)\widetilde W_{\pm}\Bigl(\frac{nN}{q^2|r|}\Bigr) 
    \Kl(\pm a\overline{|r|}n;q). 	
\end{align*}
\end{corollary}

Finally, we recall the decay properties of the Bessel transforms
$\widetilde W_{\pm}$ which follow from repeated integration by parts
and the decay properties of of Bessel functions and their
derivatives. These are proved in~\cite{BFKMM}*{Lemma 2.4}.

\begin{lemma}\label{besseldecay} Let $W$ be a smooth function
  compactly supported in $[1/2,2]$ and satisfying \eqref{Wbound}. In the Maa{\ss} case set $\vartheta = |\Re it|$, otherwise set $\vartheta = 0$.   
    For
  $M\geq 1$ let $W_M(x)=W(x/M)$. For any $\varepsilon$, for any $i,j\geq 0$ and for all
  $y>0$, we have
\begin{displaymath}
\begin{split}
  y^j\widetilde{(W_M)}^{(j)}_{\pm}(y) &\ll_{i,j,\varepsilon} M(1+My)^{j/2}
 \big (1+ (My)^{-2\vartheta -\varepsilon}\big)\big(1 +   (My)^{1/2} \big)^{-i}.
\end{split}
\end{displaymath}
In particular, the functions 
$\widetilde{(W_M)}_{\pm}(y)$ decay rapidly when $y\gg 1/M$.
\end{lemma}

We recall that all implied constants may depend polynomially on the
parameters $s\in \mathscr{S}$ that $W$ and $\tilde{W}$ depend on.

\section{A factorization lemma}
\label{pg-factor}
To shorten notations, let us  write 
\begin{equation}\label{defT(s)}
T(s)=L(f\otimes f,s).
\end{equation}  We denote by
$(\mu_f(n))_{n\geq 1}$ the convolution inverse of $(\lf(n))_{n\geq1}$,
which is given by
\begin{equation}\label{defmurevision}
L(f,s)^{-1}=\prod_{p}\Bigl(1-\frac{\lf(p)}{p^s}+\frac{\chi_r(p)}{p^{2s}}\Bigr)=\sum_{n\geq
  1}\frac{\mu_f(n)}{n^s},\quad \Re s\geq 1.
\end{equation}

We then define an auxiliary function of six complex variables by
\begin{equation}\label{bigLdef}
L(s,z,z',u,v,w)= \sumsum_{\stacksum{d,\l_1,\l_2,n}{(\l_1,\l_2)= (d\l_1\l_2, r) = 1}}
\frac{\mu_f(d\l_1)\lf(\l_1 n)\mu_f(d\l_2)\lf(\l_2n)}
{\l_1^{s+z+u+v}{\l_2}^{s+z'+u+w} {d}^{z+z'+v+w}n^{2s+2u}}.	
\end{equation}

In Chapters~\ref{ch-central} and~\ref{analyticrank}, we will use the
following lemma.

\begin{lemma}\label{lm-factor}
For $\eta\in\Rr$, let $\mcR({\eta})$ be the open subset of
$(s,z,z',u,v,w)\in\Cc^6$ defined by the inequalities
\begin{align*}
  \mathcal R (\eta):=
  \begin{cases}
   \Re s  > \frac{1}{2}-\eta, \  \Re z > \frac{1}{2}-\eta, & \ \Re z' > \frac{1}{2}-\eta,\\
   \Re u >-\eta, \ \Re v >-\eta, &\ \Re w >-\eta.
 \end{cases}
 \end{align*}
 There exists $\eta>0$ and a holomorphic function $D(s,z,z',u,v,w)$
 defined on $\mcR(\eta)$ such that $D$ is absolutely bounded on
 $\mcR(\eta)$, and such that the holomorphic function
 $L(s,z,z',u,v,w)$ admits meromorphic continuation to
 $\mathcal R (\eta)$ and satisfies the equality
$$
L(s,z,z',u,v,w)=\frac{T(2s+2u)T(z+z'+v+w)}{T(s+z+u+v)T(s+z'+u+w)}
D(s,z,z',u,v,w).
$$
\end{lemma}

As a special case:

\begin{corollary}\label{cor-factor}
  The function $(u,v,w)\mapsto L(\demi,\demi,\demi,u,v,w)$ initially
  defined as a convergent holomorphic series over a domain of the
  shape
  $$\Re u,\Re v,\Re w\gg 1$$ extends meromorphically to the domain
 $$\Re u,\Re v,\Re w> -\eta$$ for some absolute constant $\eta > 0$ and satisfies
 \begin{displaymath}
 \begin{split}
   \nonumber L(\demi,\demi,\demi,u,v,w) &=\frac{T(1+2u)T(1+v+w)}
   {T(1+u+v)T(1+u+w)}D(u,v,w)\\
  & =\eta_3(f,u,v,w)\frac{(u+v)(u+w)}{u(v+w)}.
  \end{split}
 \end{displaymath}
 where
  \begin{itemize}
  \item $D$ is an Euler product absolutely convergent for
    $\Re u,\Re v,\Re w\geq -\eta$,
  \item $\eta_3$ is holomorphic and non-vanishing in a neighborhood of
    $(u, v, w) = (0,0,0)$.
  \end{itemize}
\end{corollary}

\begin{proof}[Proof of Lemma~\ref{lm-factor}]
  The function $\mu_f$ is multiplicative and satisfies
 $$
 \mu_f (p) =-\lambda_f (p),\ \mu_f (p^2)=\chi_r (p), \ \mu_f
 (p^k)=0\text{ for } k\geq 3.
$$

By \eqref{RP4}, the series \eqref {bigLdef} is absolutely convergent
in the intersection 
\begin{align*}
  \mathcal C :=
 \begin{cases}
 \Re (s+z+u+v) >1,&  \  \Re (s+z'+u+w) >1,\\
  \Re (z+z'+v+w) >1,& \   \Re (2s+2u) >1
  \end{cases}
 \end{align*}
 of four half spaces of $\Cc^6$.  In particular, the region
 $\mathcal C$ contains the region $\mathcal R (0)$ of $\Cc^6$.
 
 In this region we have the factorization
 $$L(s,z,z',u,v,w)=\prod_p L_p(s,z,z',u,v,w)$$
where
\begin{multline*}
L_p(s,z,z',u,v,w)=\\ \sumsum_{\stacksum{\delta,\lambda_1,\lambda_2,\nu\geq 0}{\lambda_1\lambda_2=0}}
\frac{\mu_f(p^{\delta+\lambda_1})\lf(p^{\lambda_1+\nu})\mu_f(p^{\delta+\lambda_2})\lf(p^{\lambda_2+\nu})}
{p^{\lambda_1(s+z+u+v)+\lambda_2(s+z'+u+w)+\delta(z+z'+v+w)+\nu(2s+2u)}}
\end{multline*}
for $p \nmid r$. For $\alpha\geq 0$ and $\theta=7/64$,  we have
\begin{equation}\label{generalboundpalpha}
|\lambda_f(p^\alpha)|\leq (\alpha+1)p^{\alpha\theta},
\end{equation}
hence the factor $L_p(s,z,z',u,v,w)$ is absolutely convergent for
$(s,z,z,u,v,w)$ such that
\begin{gather*}
  \Re(z+z'+v+w)>2\theta,\quad\quad
  \Re(s+z+u+v)>2\theta, \\
  \Re(s+z'+u+w)>2\theta,\quad\quad \Re(s+u)>\theta.
\end{gather*}
This includes the region $\mcR(\eta)$ with $\eta:=\theta/4$. Now
splitting the summation over the set of
$(\delta,\lambda_1,\lambda_2,\nu) $ with
$0 \leq \delta +\lambda_1+\lambda_2 +\nu \leq 1$ and the complementary
set, we see that for $(s,z,z',u,v,w)\in\mcR(\theta/4)$ we have the
equality
 \begin{multline}\label{Lpapprox}
 L_p(s,z,z',u,v,w)\\=1+\frac{\lf(p)^2}{p^{z+z'+v+w}}-\frac{\lf(p)^2}{p^{s+z+u+v}}-\frac{\lf(p)^2}{p^{s+z'+u+w}}+\frac{\lf(p)^2}{p^{2s+2u}}+\mathrm{EL}_p(s,z,z',u,v,w)
 \end{multline}
with $\mathrm{EL}_p(s,z,z',u,v,w)$ holomorphic in that region and satisfying
\begin{equation}\label{ELpbound}\mathrm{EL}_p(s,z,z',u,v,w)=O\Bigl(\, \frac{p^{2\theta}}{p^{2(1-\theta)}}\, \Bigr)=O\Bigl(\,\frac{1}{p^{2-4\theta}}\, \Bigr)	
\end{equation}
 where the implied constant is absolute.
 
 We consider now the multivariable Dirichlet series
 $$M(s,z,z',u,v,w):=\frac{T(2s+2u)T(z+z'+v+w)}{T(s+z+u+v)T(s+z'+u+w)}.$$
 In the region $\mathcal C$, it is absolutely convergent and factors as
 $$M(s,z,z',u,v,w)=\prod_p M_p(s,z,z',u,v,w)$$
 where
 $$M_p(s,z,z',u,v,w)=\frac{T_p(2s+2u)T_p(z+z'+v+w)}{T_p(s+z+u+v)T_p(s+z'+u+w)}.$$
Let us recall that for any $p$ we have
$$T_p(s)=\zeta_p(s)\prod_{i=1}^3\Bigl(\,1-\frac{\alpha_{\symf,i}(p)}{p^s}\Bigr)^{-1}.$$
with
\begin{equation}\label{symlocalbound}
  |\alpha_{\symf,i}(p)|\leq p^{2\theta};
\end{equation}
in particular $T_p(s)$ is holomorphic and non-vanishing for
$\Re s>2\theta$. Moreover, for $p\nmid r$, we have
 $$
T_p(s)=(1-\frac{1}{p^{2s}})^{-1}\sum_{\alpha\geq 0}\frac{\lambda_{f}(p^{\alpha})^2}{p^{\alpha
    s}}=1+\frac{\lambda_{f}(p)^2}{p^{
    s}}+\sum_{\alpha\geq 2}\frac{
    \xi_f(p^\alpha)}{p^{\alpha s}}
$$
where the coefficients $\xi_f(p^\alpha)$ satisfy
$$|\xi_f(p^\alpha)|= \vert \lambda_f(p^\alpha)^2 +\lambda_f (p^{\alpha -2})^2 +\cdots \vert\leq (\alpha+1)^3 p^{2\alpha\theta},$$
by \eqref{generalboundpalpha}.
Hence, by the same reasoning as before, we have for $p\nmid r$ the
equality
\begin{multline}\label{Mpapprox}
  M_p(s,z,z',u,v,w)\\=1+\frac{\lf(p)^2}{p^{z+z'+v+w}}-\frac{\lf(p)^2}{p^{s+z+u+v}}-\frac{\lf(p)^2}{p^{s+z'+u+w}}+\frac{\lf(p)^2}{p^{2s+2u}}\\
  +\mathrm{EM}_p(s,z,z',u,v,w)
\end{multline}
with $\mathrm{EM}_p(s,z,z',u,v,w)$ holomorphic in $\mcR(\theta/4)$ and satisfying
\begin{equation}\label{EMpbound}
\mathrm{EM}_p(s,z,z',u,v,w)=O\Bigl(\frac{p^{2\theta}}{p^{2(1-\theta)}}\Bigr)=O\Bigl(\frac{1}{p^{2-4\theta}}\Bigr).	
\end{equation}

Let $P\geq 1$ be a parameter to be chosen sufficiently large; given some converging Euler product
$$L=\prod_p L_p$$ we set
$$L_{\leq P}=\prod_{p\leq P}L_p,\ L_{>P }=\prod_{p> P}L_p$$
so that
$$L=L_{\leq P}L_{> P}.$$

We apply this decomposition to $L(s,z,z',u,v,w)$ for $P>r$.  In the
region of absolute convergence, we have
 $$L(s,z,z',u,v,w)=L_{\leq P}(s,z,z',u,v,w) L_{>P}(s,z,z',u,v,w).$$
 We write
 $$L_{>P}(s,z,z',u,v,w)=M_{>P}(s,z,z',u,v,w)D_{>P}(s,z,z',u,v,w)$$
where
$$D_{>P}(s,z,z',u,v,w)=\prod_{p>P}\frac{L_{p}(s,z,z',u,v,w)}{M_{p}(s,z,z',u,v,w)}$$
By \eqref{Mpapprox} and \eqref{EMpbound} we can choose $P>|r|$
sufficiently large so that for $p>P$, $M_p(s,z,z',u,v,w)^{-1}$ is
holomorphic in the region $\mcR(\theta/4)$, then by \eqref{Lpapprox},
\eqref{ELpbound} \eqref{Mpapprox} and \eqref{EMpbound} we have, in
that same region the equality
 $$\frac{L_{p}(s,z,z',u,v,w)}{M_{p}(s,z,z',u,v,w)}=1+O\Bigl(\,\frac{1}{p^{2(1-3\theta)}}+\frac{1}{p^{2-4\theta}}\,\Bigr).$$
 Since $2-6\theta>1$ the product   
 $$D_{>P}(s,z,z',u,v,w)=\prod_{p>P}\Bigl(\,1+O\Bigl(\,\frac{1}{p^{2(1-3\theta)}}\,\Bigr)\,\Bigr)$$
 is absolutely convergent and uniformly bounded in the region
 $\mcR(\theta/4)$. We now write the finite product
 $$L_{\leq P}(s,z,z',u,v,w)=M_{\leq P}(s,z,z',u,v,w)D_{\leq P}(s,z,z',u,v,w).$$
 By \eqref{symlocalbound} the finite product
 $$D_{\leq P}(s,z,z',u,v,w)=\prod_{p\leq P}\frac{L_{p}(s,z,z',u,v,w)}{M_{p}(s,z,z',u,v,w)}$$
 is holomorphic and uniformly bounded in the region $\mcR(\theta/4)$ and 
 $$D(s,z,z',u,v,w)=D_{\leq P}(s,z,z',u,v,w)D_{> P}(s,z,z',u,v,w)$$
 has the required properties.
\end{proof}


\section{A shifted convolution problem}

The objective of this section is to adapt the work of Blomer and
Mili\'cevi\'c~\cite{BloMil} to prove a variant of the shifted
convolution problem that is required in this book. The following
result is proved in loc.\ cit.\ in the case of cusp forms of level
one. Since, the generalization to arbitrary fixed (signed) level $r$ is
straightforward, we will only briefly indicate the changes that are
required.

Most of the notation in this section is borrowed from \cite{BloMil},
except that the modulus which is denoted $q$ in this book is denoted
$d$ in loc.\ cit.

\begin{proposition}\label{thmSCP} 
  Let $\ell_1, \ell_2\geq 1$ two integers, $q \geq 1$ and
  $N \geq M \geq 1$.  Let $f_1, f_2$ be two primitive cusp forms of
  signed levels $r_1$ and $r_2$ and Hecke eigenvalues
  $(\lambda_1(n))_{n\geq 1}$ and $(\lambda_2(m))_{m\geq 1}$,
  respectively. Assume that
  $(\ell_1\ell_2, r_1r_2) = 1$. Let $V_1, V_2$ be
  fixed smooth weight functions satisfying \eqref{Wbound}. Then for
  $\theta=7/64$, we have
\begin{multline*}
  \sum_{\substack{\ell_2 m - \l_1 n\equiv 0
      \mods{q}\\\ell_2m-\ell_1n\not=0}}
  \lambda_{1}(m) \lambda_{2}(n) V_1\left(\frac{\l_2m}{M}\right) V_2\left(\frac{\l_1 n}{N}\right)\\
  \ll (qN)^{\varepsilon}\left( \left(\frac{N}{q^{1/2}} + \frac{N^{3/4}
        M^{1/4}}{q^{1/4}} \right) \left( 1 +
      \frac{(NM)^{1/4}}{q^{1/2}}\right) + \frac{M^{3/2 + \theta}}{q}
  \right)
\end{multline*}
uniformly in $\l_1 , \l_2$, with an implied constant depending
on $f_1, f_2$ and the parameters $\mathscr{S}$ that $V_1, V_2$ depend on. The same bound holds if the congruence condition
$\ell_2 m - \l_1 n\equiv 0 \, (\text{{\rm mod }} q)$ is replaced by
$$\ell_2 m + \l_1 n\equiv 0 \, (\text{{\rm mod }} q).$$
\end{proposition}

\begin{proof} 
  If $N \asymp M$, we write $ \ell_2 m \pm \l_1 n = hq$ with
  $0 \not= h \ll M/q$. For each value of $h$, we use \cite{Bl} to
  bound the corresponding shifted convolution sum by
  $M^{1/2 + \theta +\varepsilon}$, so that we get a total contribution
  of
$$\ll  q^{-1}  M^{3/2 + \theta+\varepsilon}.$$
 
If $N \geq 20 M$, say, then the bound is a straightforward adaptation
of \cite{BloMil}*{Proposition 8} to cusp forms with general level. The key
observation is that Jutila's circle method allows us to impose extra
conditions on the moduli $c$. It is easiest to work with moduli $c$
such that $r_1r_2 \mid c$ (the condition $(c, r_1r_2) = 1$ would also
do the job). With this in mind, we follow the argument and the
notation of \cite{BloMil}*{Sections 7 and 8}. We replace the definition
\cite{BloMil}*{(7.1)} with $Q = (N|r_1r_2|)^{1000}$ (note that this has
no influence on the dependency of the implied constant on the levels,
since an important feature of Jutila's method is the fact that $Q$
enters the final bound only as $Q^{\varepsilon}$).  The definition of
the weight function $w$ in \cite{BloMil}*{(7.5)} is non-trivial only
for $\l_1 \l_2 r_1r_2 \mid c$, so that
$$\Lambda \gg C^2 \frac{\vphi(\l_1 \l_2 |r_1r_2|)}{(\l_1 \l_2 |r_1r_2|)^2 }$$ in
\cite{BloMil}*{(7.6)}. From there, the argument proceeds identically
with the Voronoi summation formula and the Kuznetsov formula for level
$|r_1r_2| \l_1 \l_2$. In \cite{BloMil}*{(8.1)}, we put
$\beta = {\rm lcm}(\l_1 , \l_2, d, r_1,r_2)$. Again the argument
proceeds verbatim as before. Wilton's bound in \cite{BloMil}*{Section
8.2} is polynomial in the level, see \cite{HaMi}*{Proposition
5}. The rest of the argument remains unchanged, except that the
level of the relevant subgroup for the spectral decomposition in
\cite{BloMil}*{}(7.14) and below is $\Gamma_0(\l_1 \l_2 |r_1r_2|)$
instead of $\Gamma_0(\l_1 \l_2)$; as a consequence, the sum over
$\delta$ before and after \cite{BloMil}*{(8.8)} must be over
$\delta \mid \l_1 \l_2r_1r_2$.

The changes that are required to handle the congruence
$\ell_2m + \ell_1 n \equiv 0\bmod{q}$ are explained in Section 11 of
\cite{BloMil}.
\end{proof}

\section{Partition of unity}

We will use partitions of unity repeatedly in order to decompose a
long sum over integers into smooth localized sums (see e.g.\
\cite{FoCrelle}*{Lemme 2}).

\begin{lemma}
There exists a smooth non-negative function $W(x)$ supported on
$[1/2,2]$ and satisfying \eqref{Wbound} such that
$$\sum_{k\geq 0}W\Bigl(\frac{x}{2^k}\Bigr)=1$$ for any $x\geq 1$.
\end{lemma}


\chapter{Algebraic exponential sums}
\label{ch-sums}

In this chapter, we will first summarize elementary orthogonality
properties of Dirichlet characters, then state and sketch some ideas
of the proofs of bilinear estimates with Kloosterman sums.  These are
the core results that we use in all main results of this book. In
Sections~\ref{sec-trace} and~\ref{sec-equi-mellin}, which are only
used later in Sections~\ref{sec-m1-trace} and~\ref{sec-mellin}, we
discuss briefly trace functions over finite fields, and the
equidistribution properties of their discrete Mellin transforms
(following Katz~\cite{Ka-CE}).

\section{Averages over Dirichlet characters}
\label{sec-chars}

Let $q$ be an odd prime.  Given a function $\tau$ defined on Dirichlet
characters modulo $q$, we will write
\begin{align*}
  \sump_{\chi\mods q} \tau(\chi)
  &=\sum_{\chi\mods
    q}\frac{1+\chi(-1)}2\tau(\chi),\\
  \summ_{\chi\mods q} \tau(\chi)
  &=\sum_{\chi\mods q}\frac{1-\chi(-1)}2 \tau(\chi),\\
  \sums_{\chi\mods q} \tau(\chi)
  &=\sum_{\substack{\chi\mods q\\ \chi
  \text{ primitive}}} \tau(\chi)
\end{align*}
for the sum of $\tau$ over even (resp. odd, primitive) Dirichlet
characters modulo $q$.
\par
We recall the basic orthogonality relations
\begin{equation}\label{average}
\begin{aligned}
  \frac{1}{\vphi(q)}\sum_\chi \chi(m)\ov{\chi(n)}&=\delta_{(mn,q)=1}\delta_{m\equiv n\mods q}, 
  \\
  \frac{2}{\vphi(q)}\sump_\chi
  \chi(m)\ov{\chi(n)}&=\delta_{(mn,q)=1}\delta_{m\equiv \pm n\mods
    q}. 
\end{aligned}
\end{equation}
\par
As in \eqref{gauss}, we denote
$$
\eps_\chi =\frac{1}{q^{1/2}}\sum_{h\mods
  q}\chi(h)e\Bigl(\frac{h}{q}\Bigr)
$$
the normalized Gau{\ss} sum of a character $\chi$ modulo $q$. If
$\chi=\chi_q$ is the trivial character, then we have
$\eps_{\chi_q}=- q^{-1/2}$. 
\par
Since we are interested in the distribution of root numbers, we will
need to handle moments of the Gau{\ss} sums. These are well-known
(see, e.g.,~\cite[Proof of Th. 21.6]{IwKo}): for any integer $k\geq 1$
and $(m,q)=1$, we have
\begin{equation*}
\frac{1}{\vphi(q)}\sum_\chi \chi(m)\eps_\chi^k=q^{-1/2}\Kl_k(\ov m;q),
\end{equation*}
where
\begin{equation}\label{eq-hypk}
\Kl_k(m;q)= \frac{1}{q^{\frac{k-1}2}}\sum_{x_1\cdots x_k=m\mods
  q}e\Bigl(\frac{x_1+\cdots+x_k}{q}\Bigr)
\end{equation}
is the normalized hyper-Kloosterman sum modulo $q$.  Consequently, we
have
\begin{equation}\label{averagesignprim}
  \frac{1}{\vphis(q)}\sums_{\chi\mods{q}} \chi(m)\eps_\chi^k=
  q^{-1/2}\Kl_k(\ov m;q) + O(q^{-1-|k|/2}).
\end{equation}
This formula remains true for $k=0$ if we define
$$
\Kl_0(m;q)=q^{1/2}\delta_{m=1\mods q}.
$$
Moreover, since
\begin{equation*}
\eps_{\chi}=\ov{\eps_{\chi}}^{-1}=\chi(-1)\eps^{-1}_{\ov\chi}
\end{equation*}
for primitive characters, the formula \eqref{averagesignprim} extends
to negative $k$ when we define
$$
\Kl_k(m;q):=\Kl_{|k|}((-1)^k \ov{m};q), \quad k \leq -1.
$$

Similarly, we obtain
\begin{equation}\label{averagesigneven}
  \frac{2}{\vphis(q)}\sump_{\chi \text{ primitive}} 
  \chi(m)\eps_\chi^k=\frac{1}{q^{1/2}}\sum_\pm\Kl_k(\pm\ov m;q)  
  + O(q^{-1-|k|/2}).
\end{equation}
for the sum restricted to even characters only. 

The following deep bound of Deligne is essential at many points, in
particular it implies the equidistribution of angles of Gau\ss\ sums.


\begin{proposition}[Deligne]\label{lem-kloos}
  Let $k$ be a non-zero integer. For any prime $q$ and any integer $m$
  coprime to $q$, we have
$$
|\Kl_k(m;q)|\leq |k|.	
$$
\end{proposition}

\begin{remark}
  To avoid confusion, we will never use the notation~$k$ to refer to a
  finite field.
\end{remark}

This was proved by Deligne~\cites{WeilI,WeilII} as a consequence of his
general form of the Riemann Hypothesis over finite fields. For $k=2$,
it is simply the Weil bound for classical Kloosterman sums.

\section{Bounds for Kloosterman sums}

In this section we recall various bounds for sums of Kloosterman sums
which will be required in some of our applications. For a prime $q$
and an integer $a$ coprime with $q$, we define
\begin{equation}\label{eq-bilform}
B(\Kld,\bfalpha,\bfbeta)=\sumsum_{m\leq M,\ n\leq
  N}\alpha_m\beta_n\Kld(amn;q),
\end{equation}
where $\bfalpha=(\alpha_m)_{1\leq m\leq M}$,
$\bfbeta=(\beta_n)_{1\leq n\leq N}$ are sequences of complex
numbers. We write
$$
\| \bfalpha \|^2_2=\sum_{m\leq M}|\alpha_m|^2, \quad  \| \bfbeta
\|^2_2=\sum_{n\leq N}|\beta_n|^2.
$$

The following bound is a special case of a result of Fouvry, Kowalski
and Michel~\cite{FKM2}*{Thm.\ 1.17}.

\begin{proposition}\label{CSPV} 
  For any $\eps>0$, we have
\begin{equation*}
B(\Kld,\bfalpha,\bfbeta)\ll_{\eps} q^{\eps}\| \bfalpha \|_2 \| \bfbeta \|_2(MN)^{1/2}  \Bigl(\frac{1}{q^{1/4}}+\frac{1}{M^{1/2}}+\frac{q^{1/4}}{N^{1/2}}\Bigr),
\end{equation*} 
uniformly for $(a,q)=1$.
\end{proposition}

We will also need the following bound which is a special case of
another result of Fouvry, Kowalski and Michel~\cite{FKM1}*{Thm.\ 1.2 and
(1.3)}: 

\begin{proposition}\label{propFKM1} 
  For any
  primitive cusp form $f$ with trivial central character and level
  $r$, and any smooth function $W$ satisfying \eqref{Wbound}, we have
$$
\sum_{n\geq 1}\lf(n)\Kl_k(an;q)W\Bigl(\frac{n}N\Bigr) \ll_{\eps} q^{\eps+1/2-1/8}{N}^{1/2}\Bigl(1+\frac{N}q\Bigr)^{1/2}
$$
for any $(a,q)=1$, any $k\in\Zz-\{0\}$, any integer $N\geq 1$ and any
$\eps>0$, where the implied constant depends polynomially on $f$ and  $k$ (and the parameters $\mathscr{S}$ that $W$ depends on). 
\end{proposition}

The last estimate we require was conjectured by Blomer, Fouvry,
Kowalski, Michel and Mili\'cevi\'c in \cite{BFKMM}, and was proved by
Kowalski, Michel and Sawin~\cite{KMS}*{Thm.\ 1.1}:

\begin{proposition}\label{thmtypeII} 
  Suppose that $M,N\geq 1$ satisfy
$$
1\leq M\leq Nq^{1/4},\quad 
MN\leq q^{5/4}.
$$
For any $\eps>0$, we have
\begin{equation}\label{typeIIgen}
  B(\hypk_2,\uple{\alpha},\uple{\beta})
  \ll_\eps q^\eps\|\bfalpha\|_2\|\bfbeta\|_2(MN)^{\frac12}
  \left(\frac{1}{M^{1/2}}+q^{-\frac1{64}}\Bigl(\frac{q}{MN}\Bigr)^{\frac{3}{16}}  \right),
\end{equation}
uniformly for $(a,q)=1$.
\end{proposition}

\begin{remark} 
  (1) The point of this result, in comparison with
  Proposition~\ref{CSPV} (which applies in much greater generality
  than~\cite{KMS}) is that it is non-trivial even in ranges where
  $MN<q$. More precisely,  Proposition~\ref{thmtypeII} gives a
  non-trivial estimate as soon as $MN\geq q^{7/8+\delta}$ for some
  $\delta>0$.
\par
(2) Notice that we may always assume in addition that $MN>q^{1/4}$ (as
in \cite{KMS}), since otherwise the bound \eqref{typeIIgen} is implied
by the trivial bound
  $$
  B(\hypk_2,\uple{\alpha},\uple{\beta})\ll
  \|\bfalpha\|_2\|\bfbeta\|_2(MN)^{1/2}.
$$
\end{remark}

\section{Sketch of the arguments}
	
We summarize here the key ideas of the proofs of the estimates of the
previous section. We hope that this informal discussion will be
helpful to readers yet unfamiliar with the tools involved in the use
of trace functions and of Deligne's form of the Riemann Hypothesis
over finite fields.

\emph{Trace functions} modulo a prime $q$ are functions on $\Fq$ attached to
$\ell$-adic sheaves on the affine line or on the multiplicative group
over $\Fq$ (here $\ell$ is a prime number different from $q$).  We
will not recall precise definitions of trace functions
(see~\cite{Pisa} for an accessible survey), but we note that $\hypk_2$ is a fundamental example of a trace function, and will give further examples later.

\par
\medskip
\par
Proposition \ref{CSPV} is proved by the use of the Cauchy-Schwarz
inequality to eliminate the arbitrary coefficients $\beta_n$, and then
by completing the sum in the $n$ variable to obtain sums over
$\Fq$. This reduces the proof to the estimation of correlation sums
$$
\sum_{n \in \Fq} \Kld(a m_1 n;q) \Kld(a m_2 n;q) e\left(
  \frac{hn}{q} \right).
$$ 
In the generalization of this problem considered in~\cite{FKM2}, where $\hypk_2$ is replaced with a general trace function, such sums are estimated
using Deligne's most general form of the Riemann hypothesis over
finite fields~\cite{WeilII}. This argument gives square-root
cancellation for sums of trace functions over algebraic curves, as
long as an associated cohomology group vanishes, and this vanishing
reduces to an elementary problem of representation theory for the
geometric monodromy group of the sheaves associated to these trace
functions (in fact, one that does not require knowing precisely what
the monodromy group is). 
\par
However, a number of special cases, including this one, have a more
elementary proof. In this case, the bound follows directly from Weil's
bound for Kloosterman sums, since one can check elementarily that
$$
\sum_{n \in \Fq} \Kld(a m_1 n;q) \Kld(a m_2 n;q) e\left( \frac{hn}{q}
\right)= \sum_{y\in\Fqt}
e\Bigl(\frac{am_2\bar{y}-am_1\overline{(h+y)}}{q}\Bigr)
$$
which, for $h\in \Fqt$, is equal to
$$
q^{1/2} e\left( \frac{ -a\bar{h} (m_1 +m_2)}{q} \right) \Kld\left(
  \frac{a^2 m_1 m_2}{h^2};q \right)
$$
(such an identity is to be expected since $\Kld(a;q)$ is a discrete
Fourier transform of $x\mapsto e(\bar{x}/q)$, so that the correlation
sum can be evaluated by the discrete Plancherel formula).

\par
\medskip
\par 
Proposition \ref{propFKM1} is proven using the amplification
method. This involves amplifying over modular forms, which means that
the sum is enlarged dramatically to a sum of similar expressions over
a basis $\mathcal{B}$ of Hecke eigenforms $g$ of the space of modular
forms of the same weight as $f$ and of level $pr$ (both holomorphic
and non-holomorphic), so that the Petersson-Kuznetsov formula may be
applied. Here we view $f$, a form of level $r$, as being of level
$pr$, and we can assume that $f\in\mathcal{B}$.
\par
To get a nontrivial bound using this approach, it is necessary to
insert an amplifier $A(g)$, which is here a weighted sum of Hecke
eigenvalues, of the form
$$
A(g)=\sum_{\ell\leq L}\alpha_{\ell}\lambda_g(\ell),
$$
for $g\in\mathcal{B}$, chosen so that $A(g)$ is ``large''. We hope to
get an upper bound for
$$
\Sigma=\sum_{g\in\mathcal{B}} |A(g)|^2 \Bigl| \sum_{n\geq
  1}\lambda_g(n)\Kl_k(an;q)W\Bigl(\frac{n}N\Bigr) \Bigr|^2,
$$
in order to claim by positivity that
$$
\Bigl| \sum_{n\geq 1}\lambda_f(n)\Kl_k(an;q)W\Bigl(\frac{n}N\Bigr)
\Bigr|^2\leq \frac{\Sigma}{|A(f)|^2}.
$$
\par
The application of the Kuznetsov formula produces a complicated sum on
the arithmetic side. In the off-diagonal terms of the
amplified sums, we see that correlation sums of the following shapes
$$ 
\mathcal C(\Kld;\gamma) = \sum_{z \in \Fq} \widehat{\Kld}(
\gamma \cdot z;q ) \overline{\widehat{\Kld}( z;q)},
$$ 
appear, for certain quite specific $\gamma \in \PGL_2(\Fq)$. In the
generalized version, one again uses Deligne's Theorem to estimate such sums;
this involves separating the possible $\gamma$ for which there is no
square-root cancellation, and exploiting the fact that they are very
rare, except for very special input sheaves, and cannot coincide too
often with the ``special'' $\gamma$ that occur in the application of
the Kuznetsov formula.  
\par
Here also, in the special case of the Kloosterman sums that we are
dealing with, whose Fourier transform is $x\mapsto e(\bar{x}/q)$, the
estimate for correlation sums reduce to Weil's bound for Kloosterman
sums (see \cite{FKM1}*{1.5(3)}).  In the simplest case of Dirichlet
characters, this method was pioneered by Bykovski\u i \cite{By}, cf.\
also \cite{BH}. However, for general hyper-Kloosterman sums $\Kl_k$,
it seems very unlikely that a similarly elementary argument exists to
prove this bound.
\par
\medskip
\par
Finally, the proof of Proposition \ref{thmtypeII} is by far the most
difficult and involves highly non-trivial algebraic geometry. In
particular, it uses heavily some special properties of Kloosterman
sums, and does not apply to an arbitrary trace function (although one
can certainly expect that a similar result should be true for any
trace function that is not an additive character times a
multiplicative character, in which case it is trivially false).

The completion step in this case should be thought of as primarily
an analogue of the Burgess bound~\cite{Bur} for short sums of
Dirichlet characters. Similarly to the standard proof of the Burgess
bound, we use the multiplicative structure of the function $\Kld(amn)$
to bound sums over intervals of length smaller than $\sqrt{q}$ by
reducing them to high moments of sums over even shorter intervals,
which themselves can be controlled by more complicated complete
sums. 
\par
More precisely, we begin in the same way as in the proof of
Proposition~\ref{CSPV} by applying the Cauchy-Schwarz inequality to
eliminate the coefficients $\beta_n$. However, the resulting sum over
$n$ is now too short to be usefully completed directly. Instead, we
apply the Burgess argument in the form of the ``shift by $ab$'' trick
of Karatsuba and Vinogradov. This ends up reducing the problem to the
estimation of certain complete exponential sums in three
variables. The key result that we need to prove (a special case of
\cite{KMS}*{Theorem 2.6}) is the following:

\begin{theorem}\label{KMSCES} 
For a prime $q$, for $r\in \Fq$, $\lambda\in\Fq$ and
$\uple{b}=(b_1,\ldots,b_4)\in\Ff_q^4$, let
$$
R(r, \lambda,\uple{b}) = \sum_{s \in \Fqt} e\left( \frac{\lambda a}{q}
\right) \prod_{i=1}^2 \Kld (s (r+b_i)) \Kld(s(r+b_{i+2})).
$$
Then we have
\begin{equation}\label{kms-correlation-sum} 
  \sum_{r \in \Fq} R(r, \lambda_1,\uple{b})
  \overline{R(r,\lambda_2,\uple{b})} = q^2
  \delta(\lambda_1,\lambda_2) + O(q^{3/2})
\end{equation} 
for all $\uple{b}$, except those that satisfy a certain non-trivial
polynomial equation $Q(\uple{b})=0$ of degree bounded independently of
$q$.
\end{theorem}

A key difference with the Burgess bound is that, whereas the Weil
bound for multiplicative character sums over curves which is used
there gives square-root cancellation outside of an explicit and very
small set of diagonal parameters, the exceptional set of parameter
values $\uple{b}$ in Theorem~\ref{KMSCES}
is not explicit, and is also relatively large (it has codimension
one). This is the main difficulty in generalizing the
bound~(\ref{typeIIgen}) to shorter ranges, since in order to do so, we
must take higher moments of short sums, leading to complete sums of
more variables, for which even best-possible estimates are not
helpful \emph{unless} one can show that the codimension of the
diagonal locus diminishes proportionally to the exponent. 
\par
However, this difficulty is not significant for the applications in
this book, since we need the bound~(\ref{typeIIgen}) only in the case
where $M$ and $N$ are very close to $\sqrt{q}$. In this case, using
higher moments would not give better results, even if the analogue of
Theorem~\ref{KMSCES} was obtained with an exceptional locus of the
highest possible codimension (as in the Burgess case).

We now give a longer but still informal
summary of the techniques behind Theorem~\ref{KMSCES}, which involve simpler exponential sum
estimates, topology, elementary representation theory, and simple
arguments with Galois representations, as well as more technical steps
based on vanishing cycles.

The proof of the theorem begins by constructing (in \cite[\S4.1]{KMS}) a sheaf $\mcR$ on
$\Aa^6_{\Fq}$ whose trace function exponential sum
$R$, which depends on $6$ variables $r,\lambda, \uple{b}$. This is proven using the $\ell$-adic machinery in a relatively
formal way, exploiting known sheaf-theoretic analogues of the
algebraic operations involved in the definition of $R$. One begins
with a fundamental result of Deligne (related to
Proposition~\ref{lem-kloos}) which implies that there is a sheaf (of
conductor bounded in terms of $q$ only) with trace function equal to
hyper-Kloosterman sums; then taking tensor products of two sheaves
multiplies their trace functions, and the sum over $s$ is obtained by
computing sheaf cohomology (precisely, computing a higher direct image
with compact support of a $7$-variable sheaf). Both the result of Deligne and the step where we sum over $s$ involve key
results in étale cohomology, such as the Grothendieck--Lefschetz trace
formula.

At this point, we apply Katz's Diophantine Criterion for
Irreducibility and Deligne's Riemann Hypothesis (in \cite[Theorem 4.11 and (3.4)]{KMS}) . These imply that the
bound~(\ref{kms-correlation-sum}) holds for a given $\uple{b}$
\emph{if and only if} the sheaves $\sheaf{R}_{\uple{b},\lambda}$ in
one variable $r$ obtained by specializing the parameters
$(\uple{b},\lambda)$ of the $\sheaf{R}$ are geometrically irreducible, and that $\sheaf{R}_{\uple{b},\lambda_1}$ and $\sheaf{R}_{\uple{b},\lambda_2}$ are geometrically non-isomorphic for $\lambda_1\neq \lambda_2$. We will use this equivalence in both directions. We note
that proving the second part is easier, because there are in general
many ways to prove that two sheaves are non-isomorphic, and we can in
fact handle most cases using byproducts of the arguments involved in
the proof of irreducibility.

For this irreducibility statement, we begin by computing directly the
diagonal average over $\lambda$ of the sum $R$ in the case
$\lambda_1=\lambda_2=\lambda$, and the average of $R$ over all
$\uple{b}$ in the special case $\lambda_1=\lambda_2=0$ (in
\cite[Proposition 4.3 and (3.4)]{KMS}). These computations reveal that
the restrictions of the sheaf $\mcR$ to certain higher dimensional
spaces are irreducible.
\par
In general, estimating the average value of a sum such as $R$ will
give very little concrete information on any of it specific
values. The \emph{geometric} analogue of this operation here is to
show that the restriction of an irreducible sheaf on some variety to a
proper subvariety remains irreducible, and this turns out to be often
tractable.

The proof of irreducibility requires different methods in the
$\lambda=0$ and $\lambda \neq 0$ cases.

For $\lambda=0$, an elementary computation shows that the values of
the $R$ sum are independent of the choice of additive character used
to define Kloosterman sums. The geometric analogue of this fact is
that the sheaf $\mcR$ (specialized to $\lambda=0$) may be defined
without the use of additive character sheaves \cite[Lemma 4.27]{KMS}. Since it turns out that
this is the only part of the construction that requires working in
positive characteristic $q$, we deduce that the sheaf $\mcR_{\lambda=0}$ can
actually be constructed over the integers and over the complex
numbers. Over $\Cc$, we may apply topological arguments to study the
irreducibility of the sheaf, and it is then possible to derive the
same conclusion for sufficiently large prime characteristic $q$. (It
is actually ultimately more convenient to apply the argument in
characteristic $q$, using only the intuitions from topology; the
integrality property of the sheaf is used in \cite[\S4.4]{KMS} to show that it is tamely
ramified, and the topological properties and arguments carry over to
the tamely ramified case).

To be a bit more precise we may view the complex version of the sheaf
$\mcR_{\lambda=0}$ as a representation of the fundamental group of the
open subset $X$ of $\Cc^5$ (with coordinates $(r,\uple{b})$) where it
is lisse. We can think of this space as a family of punctured Riemann
surfaces parametrized by $\uple{b}$. Over the open subset of this
parameter space $X$ where the punctures do not collide, we can
``follow'' a loop in one Riemann surface into a loop in any other, so
their fundamental groups are equal (as subgroups of the fundamental
group $\pi_1(X)$ of the total space) and thus have the same action on
the sheaf. An immediate consequence of this is that, if one fiber of
the sheaf over some $\uple{b}$ is irreducible, then \emph{all} fibers
are irreducible. However, we can do better, because the (common)
fundamental group of our Riemann surfaces is a normal subgroup of
$\pi_1(X)$, with quotient isomorphic to the fundamental group of the
base, minus the set $Y$ of points whose fibers are empty. We can show
that the variety $Y$ of points whose fibers are empty has codimension
$2$, so the quotient is in fact the fundamental group of $\Cc^4$,
which is trivial, hence our subgroup is in fact equal to the whole
group $\pi_1(X)$. Since the representation associated to the sheaf
$\mcR_{\lambda=0}$ is an irreducible representation of $\pi_1(X)$, it
is therefore irreducible on each fiber, away from the points where the
punctures collide. (Note that in practice, the argument is phrased
using Galois theory, instead of loops, but the conclusion is the
same.)

\begin{remark}
  In the special case of $\Kld$, one can compute that the rank of
  $\mcR$ over points where $\lambda=0$ is $2$. As a rank $2$ sheaf
  whose trace function takes values in $\Qq$, it looks very much like
  the sheaf of Tate modules of a family of elliptic curves, and it is
  possible that there exists an argument reducing the $R$-sum for
  $\lambda=0$ to the number of points on a family of elliptic
  curves. If this is so, then checking the irreducibility property
  would be the same as checking that the $j$-invariant of this family
  is nonconstant. However, such an argument is unlikely to apply for
  $k\geq 3$.
\end{remark}

For $\lambda \neq 0$, the sum $R$ depends on the choice of additive
character, and is in general an element of $\Qq(\mu_p)$ and not
$\Qq$. Geometrically, the associated sheaf has wild ramification. This
causes difficulties if one tries a direct analogue of the previous
argument. Indeed, the argument that, if one fiber of the sheaf is irreducible, then all are irreducible, is not valid in the wildly ramified setting without additional work. What's more, our previous argument that one fiber of the sheaf is irreducible also does not generalize to the $\lambda\neq 0$ setting.
\par
Instead, we use (in \cite[\S4.5]{KMS}) arguments from the theory of \emph{vanishing
  cycles}. After interpreting the irreducibility at a given
$\lambda\not=0$ in terms of the rank of the stalks of a suitable
auxiliary sheaf $\sheaf{E}$ (namely, the sheaf $\mcR$ tensored by its
dual), Deligne's semicontinuity theorem gives a tool to check that the
irreducibility is independent of $\lambda\not=0$. The key input that
is needed is the proof that the Swan conductors of the local monodromy
representations associated to $\sheaf{E}$, which are numerical
invariants of wild ramification, are themselves independent of
$\lambda$. (In the tame case, the Swan conductors are always zero,
which explains partly why it is easier to handle).

In order to check this constancy property, we must compute the local
monodromy representations at every singular point. These are known for
Kloosterman sheaves (by work of Katz) and for additive character
sheaves (by elementary means) and it is easy to combine this
information when taking tensor products. The main difficulty is to
understand the local monodromy representations after taking cohomology
(which amounts to computing the sum over $s$ that defines the
$R$-sum). This is precisely what the theory of vanishing cycles
achieves in situations where the local geometry is sufficiently
``nice''. 
\par
In our case of interest when $\lambda \neq 0$, the singularities of
the one-variable specialized sheaf $\mcR_{\uple{b},\lambda}$ are those
$r$ where we can ``see'' that the sum degenerates in an obvious way,
namely those $r$ such that $r+b_i=0$, and $r=\infty$. One can then
compute that the local monodromy representation where $r+b_i=0$ is
tame (so has Swan conductor $0$), and the local monodromy at $\infty$
is wild, with large but constant Swan conductor \cite[Lemma 4.32, Corollary 4.37]{KMS}.

The computation of the local monodromy representation at $\infty$ also
allows us to prove irreducibility for generic $\lambda$, because the
problem still involves restricting an irreducible representation to a
normal subgroup, making it isotypic (up to conjugacy). Because of
this, if it were not irreducible, then the unique isomorphism class of
its irreducible components would be repeated with multiplicity at
least two. Then, when we restrict further to the local monodromy group
at $\infty$, each irreducible component must have multiplicity at
least two. But the explicit computation (using vanishing cycles)
allows us to detect an irreducible component of multiplicity one,
which is not conjugate to any other.

Finally, combining these arguments, we prove irreducibility for every
value of $\lambda$. We require some fairly elementary arguments to
conclude the proof by excluding that some specialized sheaves are
isomorphic for different values of $\lambda$. The most difficult case
is when $\lambda_2=-\lambda_1$ and we are dealing with the
generalization of~(\ref{typeIIgen}) to hyper-Kloosterman $\hypk_k$
with $k$ odd, in which case some extra steps are needed.

\section{Trace functions and their Mellin transforms}
\label{sec-trace}

Let $\ell$ be a prime distinct from $q$.  Let $\mcF$ be a
geometrically irreducible $\ell$-adic sheaf on $\Aa^1_{\Fq}$, which we
assume to be a middle-extension of weight $0$. The complexity of
$\mcF$ is measured by its conductor $\cond(\mcF)$, in the sense
of~\cite{FKM1}. Among its properties, we mention that $|t(x)|\leq c(\mathcal F)$
for all $x\in \Fq$. 
\par
An important property is that if we denote
$$
\what{t}(x)=\frac{1}{\sqrt{q}}\sum_{a\in\Fq}t(a)e\Bigl(\frac{ax}{q}\Bigr)
$$
the discrete Fourier transform of a function $t\colon \Fq\to\Cc$, then
unless the trace function $t$ is proportional to $e(ax/p)$ for some
$a$, then we have
$$
|\what{t}(x)|\ll 1
$$
where the implied constant depends only on $\cond(\mcF)$, as a
consequence of Deligne's general form of the Riemann Hypothesis over
finite fields; see the statement and references
in~\cite{Pisa}*{Th.\ 4.1}. More precisely, if $\mcF$ is not
geometrically isomorphic to an Artin-Schreier sheaf, then $\what{t}$
is itself the trace function of a geometrically irreducible
middle-extension $\ell$-adic sheaf of weight $0$, whose conductor is
bounded (polynomially) in terms of $\cond(\mcF)$ only (see the survey
previously mentioned and~\cite{FKM1}*{Prop.\ 8.2} for the bound on the
conductor), so that $\what{t}\ll 1$ is a special case of the assertion
that a trace function is bounded by its conductor.
\par
Similarly, if we define the Mellin transform of $t$ by
  $$
  \wtilde{t}(\chi)=\frac{1}{\sqrt{q}}\sum_{a\in\Fqt}t(a)\chi(a)
$$
\label{pg-mellin}for any Dirichlet character $\chi$ modulo $q$, then we have
$$
|\wtilde{t}(\chi)|\ll 1
$$
where the implied constant depends only on $\cond(\mcF)$, unless $t$
is itself proportional to a Dirichlet character (loc.\
cit.).

\begin{example}\label{ex-trace}
  (1) Let $k\geq 2$ be an integer. The function $x\mapsto \Kl_k(x;q)$
  defined by~(\ref{eq-hypk}) is a trace function (for any
  $\ell\not=q$) of a sheaf $\HYPK_k$ with conductor bounded by a
  constant depending only on $k$. These sheaves, constructed by
  Deligne and extensively studied by Katz, are called
  \emph{Kloosterman sheaves}; they are fundamental in the proof of
  Theorem~\ref{KMSCES}.
\par
(2) Let $f\in\Zz[X]$ be a polynomial and $\chi\mods{q}$ a non-trivial
Dirichlet character. Define
$$
t_1(x)=e\Bigl(\frac{f(x)}{q}\Bigr),\quad\quad
t_2(x)=\chi(f(x)).
$$
Then $t_1$ and $t_2$ are trace functions, with conductor depending
only on $\deg(f)$. If $f$ has degree $1$, then $t_1$ is associated to
an \emph{Artin-Schreier sheaf}, and if $f=aX$ for some $a\not=0$, then
$f$ is associated to a \emph{Kummer sheaf}.
\end{example}

Below we will use the following definition:

\begin{definition}\label{def-sheaf}
  A \emph{Mellin sheaf} over $\Fq$ is a geometrically irreducible,
  geometrically non-constant, middle-extension sheaf of weight $0$ on
  $\Gg_{m,\Fq}$ that is not geometrically isomorphic to a Kummer sheaf.
\end{definition}



By orthogonality of characters, we have the discrete Mellin inversion
formula
$$
\sum_{\chi \mods q}\widetilde{t}(\chi)\chi(x)=\frac{q-1}{\sqrt{q}}
t(x^{-1})
$$
for $x\in \Fqt$. Similarly, we get
$$
  \sum_{\chi\mods q}\widetilde{t}(\chi)\chi(x)\eps_{\chi}^2=
  \frac{q-1}{\sqrt{q}}(t\star \Kl_2)(x)
$$
by opening the Gau\ss\ sums (this is also a case of the discrete
Plancherel formula), where\label{pg-convol}
$$
(t_1\star t_2)(x)=\frac{1}{\sqrt{q}}\sum_{ab=x}t_1(a)t_2(b)
$$
is the multiplicative convolution of two functions on $\Fqt$.

We will need:

\begin{lemma}\label{lm-convol}
  Let $\mcF$ be a Mellin sheaf with trace function $t$. Then one of
  the following two conditions holds:
\par
\emph{(1)} There exists a Mellin sheaf $\mcG$ with conductor $\ll
\cond(\mcF)^4$ with trace function $\tau$ such that
$$
(t\star \Kl_2)(x)=\tau(x)+O(q^{-1/2})
$$
for $x\in\Fqt$, where the implied constant depends only on
$\cond(\mcF)$.
\par
\emph{(2)} The sheaf $\mcF$ is geometrically isomorphic to a pullback
$[x\mapsto a/x]^*\HYPK_2$ of a Kloosterman sheaf for some $a\in\Fqt$,
in which case there exists $\alpha\in\Cc$ with modulus $1$ such that
$$
t(x)=\alpha \Kl_2(a\bar{x};q)
$$
for all $x\in\Fqt$. We then have
$$
\widetilde{t}(\chi)=\alpha\chi(a)\eps_{\chi}^{-2}
$$
for all $\chi$.
\end{lemma}

\begin{proof}
  If $\mcF$ is not geometrically isomorphic to a pullback
  $[x\mapsto a/x]^*\HYPK_2$ of a Kloosterman sheaf, then the
  ``shriek'' convolution $\mcF\star_!\HYPK_2$ has trace function
  $t\star \Kl_2$, and the middle-convolution
  $\mcG=\mcF\star_{mid}\HYPK_2$ of $\mcF$ and $\HYPK_2$ is a sheaf
  with trace function $t\star \Kl_2+O(q^{-1/2})$, as a consequence of
  the properties of middle-convolution~\cite{Ka-CE}*{Ch. 2}.

  The middle-convolution is a Mellin sheaf in this case: indeed, it is
  geometrically irreducible because $\HYPK_2$ is of ``dimension'' one
  in the Tannakian sense, so $(\mcF\star_{mid}\HYPK_2)[1]$ is an
  irreducible object in the Tannakian sense, which implies the result
  by~\cite{Ka-CE}*{p. 20}). In that case, we obtain (1), where the
  conductor bound is a special case of the results from the Appendix
  by Fouvry, Kowalski and Michel to P. Xi's paper~\cite{xi}.
\par
\end{proof}

\begin{remark}
(1)  The ``error term'' in Case (1) of this lemma is linked to the
  possible existence of Frobenius eigenvalues of weight $\leq -1$ in
  the ``naive'' convolution. One can think of the middle-convolution
  here as the ``weight $0$'' part of this naive convolution.
\par
(2) Using a more intrinsic definition of the conductor than the one
in~\cite{FKM1}, one could obtain a better exponent that
$\cond(\mcG)\ll \cond(\mcF)^4$ (see~\cite{sawin-bounds}).
\end{remark}

In Chapter~\ref{ch-modular}, we will use the following variant of
Lemma~\ref{lm-convol}.

\begin{lemma}\label{lm-except-symbols}
  Let $\mcF$ be a geometrically irreducible $\ell$-adic that is not
  geometrically isomorphic to a Kummer sheaf, an Artin-Schreier sheaf
  or the pull-back of an Artin-Schreier sheaf by the map
  $x\mapsto x^{-1}$.  There exists a Mellin sheaf $\mcG$ with
  conductor bounded polynomially in terms of $\cond(\mcF)$, not
  geometrically isomorphic to $[x\mapsto a/x]^*\HYPK_2$ for any
  $a\in\Fqt$, such that the trace function $\tau$ of $\mcG$ satisfies
\begin{equation}\label{eq-tr-2}
\tau(x)=\frac{1}{\sqrt{q}} \sum_{y\in\Fqt}\overline{t(y)}
e\Bigl(-\frac{\bar{x}y}{p}\Bigr)+O(q^{-1/2})
\end{equation}
where the implied constant depends only on $\cond(\mcF)$.
\end{lemma}

\begin{proof}
  The principle is the same as in Lemma~\ref{lm-convol}. We denote by
  $\mcL$ the Artin-Schreier sheaf with trace function
  $x\mapsto e(x/p)$ and its Tannakian dual $\mcL^{\vee}$ with trace
  function $x\mapsto e(-\bar{x}/p)$. We consider the
  middle-convolution object $\mcF\star_{mid} \mcL^{\vee}$. Because
  $\mcF$ is not an Artin-Schreier sheaf, this object is associated to
  a middle-extension sheaf $\mcG$ of weight $0$, which is
  geometrically irreducible because $\mcL^{\vee}$ is of Tannakian
  dimension $1$.
\par
Middle-convolution of $\mcG$ with $\mcL$ gives back the input sheaf
$\mcF$ (again because $\mcL$ is of dimension $1$ with dual
$\mcL^{\vee}$ in the Tannakian sense). Thus, because $\mcF$ is not
geometrically isomorphic to a Kummer sheaf, so is $\mcG$. And because
$\mcF$ is not the pull-back of an Artin-Schreier sheaf by the map
$x\mapsto x^{-1}$, the sheaf $\mcG$ is not of the form
$[x\mapsto a/x]^*\HYPK_2$.
\par
Finally, as in Lemma~\ref{lm-convol}, the sheaf $\mcG$ has trace
function satisfying~(\ref{eq-tr-2}) and has conductor bounded in terms
of the conductor of $\mcF$.
\end{proof}

\section{The equidistribution group of a Mellin transform}
\label{sec-equi-mellin}

In a remarkable recent work, Katz~\cite{Ka-CE} has shown that the
discrete Mellin transforms of quite general trace functions satisfy
equidistribution theorem similar to those known for families of
exponential sums indexed by points of an algebraic variety.
\par
Katz's work relies in an essential way on deep algebraic-geometric
ideas, especially on the so-called Tannakian formalism. We will
minimize what background is needed by presenting this as a black-box,
with examples. We refer, besides Katz's book, to the recent Bourbaki
report of Fresán~\cite{fresan} for an accessible survey.

Let $\mcF$ be a Mellin sheaf over $\Fq$ as in
Definition~\ref{def-sheaf}. Katz~\cite{Ka-CE}*{p. 11} defines two
linear algebraic groups related to $\mcF$, its arithmetic and
geometric Tannakian monodromy groups, the geometric one being a normal
subgroup of the arithmetic one under our
assumptions~\cite{Ka-CE}*{Th.\ 6.1}. In equidistribution statements, it
is often simpler to assume that they are equal, and Katz frequently
does so.

\begin{definition}\label{def-eagm}
  We say that $\mcF$ has Property EAGM (``Equal Arithmetic and
  Geometric Monodromy'') if the two groups defined by Katz
  in~\cite{Ka-CE}*{p. 11} are equal. We then call a maximal compact
  subgroup $K$ of (the base change from $\overline{\mathbb Q}_\ell$ to $\mathbb C$ of) this common group the
  \emph{equidistribution group} of the Mellin transform of $\mcF$. We
  denote by $K^{\sharp}$ the space of conjugacy classes in $K$.
\end{definition}

Assuming that $\mcF$ has EAGM, Katz~\cite{Ka-CE}*{p. 12--13} defines a
subset $X_q$ of the set of characters of $\Fqt$, of cardinality
$\leq 2\rk(\mcF)\leq 2\cond(\mcF)$, and for any $\chi\notin X_q$, he
defines a conjugacy class $\theta_{\chi}\in K^{\sharp}$ such that the\label{pg-thetachi}
Mellin transform $\widetilde{t}$ of the trace function of $\mcF$
satisfies
$$
\widetilde{t}(\chi)=\Tr(\theta_{\chi})
$$
for $\chi\notin X_q$. It will be convenient for us to enlarge $X_q$ to
always include the trivial character.

The key result that we need is the following further consequence of
the work of Katz. It can be considered as a black box in the next
section.

\begin{theorem}\label{th-katz}
  Let $\pi$ be an irreducible representation of the equidistribution
  group $K$ of the Mellin sheaf $\mcF$, which is assumed to have Property EAGM.
Then one of the following
  properties holds:
\par
\emph{(1)} There exists a Mellin sheaf $\pi(\mcF)$, as in
Definition~\emph{\ref{def-sheaf}}, with Property \emph{EAGM}, such
that for any $\chi\notin X_q$, we have
$$
\Tr(\pi(\theta_{\chi}))=\widetilde{t}_{\pi}(\chi),
$$
where $t_{\pi}$ is the trace function of $\pi(\mcF)$, and such that
the conductor of $\pi(\mcF)$ is bounded in terms of $\pi$ and
$\cond(\mcF)$ only.
\par
\emph{(2)} There exists $a\in\Fqt$ such that
$$
\Tr(\pi(\theta_{\chi}))=\chi(a)
$$ 
for all $\chi$. Moreover, if $d$ is the order of the finite group of
characters of finite order of $K$, then $a$ is a $d$-th root of unity
in $\Fqt$.
\end{theorem}

\begin{proof}
  The existence of the sheaf $\pi(\mcF)$ as an object in the Tannakian
  category associated to $\mcF$ by Katz is part of the Tannakian
  formalism~\cite{Ka-CE}*{Ch. 2}. By construction, this object is
  irreducible in the Tannakian sense over $\bar{\mathbf{F}}_q$.  By
  the classification of the geometrically irreducible
  objects~\cite{Ka-CE}, it is either ``punctual'', in which case we
  are in Case (2) (because $\pi$ is then a finite order charcter of
$K$), or there exists a geometrically irreducible
  $\ell$-adic sheaf $\mcG$ on $\Gg_{m,\Fq}$ such that
  $\pi(\mcF)=\mcG$.  By construction of the Tannakian category, this
  sheaf is not geometrically isomorphic to a Kummer sheaf (loc. cit.)
  and it is of weight $0$. Hence it is a Mellin sheaf.
\par
Still in this second case, the Tannakian groups of $\pi(\mcF)$ are the
image by (the algebraic representation corresponding to $\pi$) of the
groups associated to $\mcF$, and are therefore equal, so that
$\pi(\mcF)$ has Property EAGM. The bound for the conductor follow
easily from the computations in~\cite{Ka-CE}*{Ch. 28, Th.\ 28.2}.
\end{proof}
\begin{remark}
  (1) The second case will be called the ``punctual'' case.
\par
(2) The conductor bounds resulting from~\cite{Ka-CE}*{Th.\ 28.2} are
relatively weak because they rely on bounds for tensor products and on
embedding the representation $\pi$ in a tensor product of tensor
powers of the standard representation and its dual. A much stronger
estimate (which would be essential for strong quantitative
applications, such as ``shrinking targets'' problems) has been proved
by Sawin~\cite{sawin-bounds}: we have
$$
\cond(\mcG)\leq 2+\dim(\pi)(1+w(\pi)\rk(\mcF))
$$
where $w(\pi)$ is the minimum of $a+b$ over pairs $(a,b)$ of
non-negative integers such that $\pi$ embeds in
$\rho^{\otimes a}\otimes\rho^{\vee\otimes b}$, where $\rho$ is the
representation of $K$ corresponding to $\mcF$ itself. In turn $w(\pi)$
is bounded by an affine function of the norm of the highest weight
vectors of the restriction of $\pi$ to $K^{0}$.
\end{remark}

\begin{example}\label{ex-ex}
  (1) (The Evans sums, see~\cite{Ka-CE}*{Ch. 14}). Let
$$
t(x)=e\Bigl(\frac{x-\bar{x}}{q}\Bigr)
$$
for $x\in\Fqt$. Then $t$ is the trace function of a Mellin sheaf
$\mcF_e$ of conductor bounded independently of $q$, and
Katz~\cite{Ka-CE}*{Th.\ 14.2} proves that $\mcF_e$ has Property EAGM,
and that its equidistribution group $K$ is $\SU_2(\Cc)$ with
$X_q=\emptyset$. By definition, each value of the Mellin transform
$$
\widetilde{t}_e(\chi)=\frac{1}{\sqrt{q}}\sum_{x\in\Fqt}
\chi(x)e\Bigl(\frac{x-\bar{x}}{q}\Bigr)
$$
is the \emph{Evans sum} associated to $\chi$.
\par
(2) (The Rosenzweig-Rudnick sums, see~\cite{Ka-CE}*{Ch. 14}). Let
$$
t(x)=e\Bigl(\frac{ (x+1)\overline{(x-1)}}{q}\Bigr)
$$
for $x\in\Fqt$ such that $x\not=1$ and $t(1)=0$.  Then $t$ is the
trace function of a Mellin sheaf $\mcF_{rr}$ of conductor bounded
independently of $q$, and Katz~\cite{Ka-CE}*{Th.\ 14.5} proves that
$\mcF_{rr}$ has Property EAGM, and that its equidistribution group $K$
is also $\SU_2(\Cc)$ with $X_q=\emptyset$. Each value of the Mellin transform
$$
\widetilde{t}_{rr}(\chi)=\frac{1}{\sqrt{q}}\sum_{x\in\Fqt}
\chi(x)e\Bigl(\frac{ (x+1)\overline{(x-1)}}{q}\Bigr)
$$
is the Rosenzweig-Rudnick sum associated to $\chi$.\label{pg-rr}
\par
(3) (Unitary examples, see~\cite{Ka-CE}*{Ch.\ 17}) Katz gives many
examples where the equidistribution group is $\Un_N(\Cc)$ for some
integer $N\geq 1$. For instance, fix a non-trivial multiplicative
character $\eta$ modulo $q$, of order $d$. Let $n\geq 2$ be an integer
coprime to $d$, and let $P\in \Fq[X]$ be a monic polynomial of degree
$n$ with distinct roots in $\bFq$, and with $P(0)\not=0$. Write
$$
P=X^n+a_{n-1}X^{n-1}+\cdots+a_1X+a_0,
$$
and assume that $\gcd ( \{i : a_i \neq 0 \} )=1$.
$1$. Then $t(x)=\eta(P(x))$ is the trace function of a Mellin sheaf
with Property EAGM, for which the equidistribution group is
$\Un_n(\Cc)$ by~\cite{Ka-CE}*{Th.\ 17.5}.
\par
If, on the other hand, $P$ has degree coprime to $q$, and if $P'$ has
$n-1$ distinct roots $\alpha\in\bFq$, if the set
$S=\{P(\alpha)\,\mid\, P'(\alpha)=0\}$ has $n-1$ distinct points
(i.e., $P$ is ``weakly super-morse''), and if in addition $S$ is not
invariant by multiplication by any constant $\not=1$, then the
``solution counting'' function
$$
t(x)=\sum_{P(y)=x} 1-1
$$
is the trace function of a Mellin sheaf with Property EAGM and with
equidistribution group $\Un_{n-1}(\Cc)$ (see~\cite{Ka-CE}*{Th.\
17.6}). The discrete Mellin transform in that case is
$$
\widetilde{t}(\chi)=\frac{1}{\sqrt{q}}\sum_{y\in\Fq}\chi(P(y))
$$
(see~\cite{Ka-CE}*{Remark\ 17.7}).
\par
(4) For further examples including groups like $\SU_n(\Cc)$ for some
$n$, $\Ort_{2n}(\Cc)$, $\mathrm{G}_2$ or products, see~\cite{Ka-CE}.
\end{example}


\chapter{Computation of the first twisted moment}\label{ch-first} 

Besides stating and proving the general form of the first moment
formulas twisted by characters that we will need in our main result,
we will also consider in this chapter the first moment twisted by more
general discrete Mellin transforms of trace functions over finite
fields, in the sense of Section~\ref{sec-trace}. We present these last
results in a separate section for greater readability; it may be
safely omitted in a first reading and is only used in
Section~\ref{sec-mellin}.

\section{Introduction}

In this chapter, we will prove Theorem \ref{thm-moment1}, which we
first recall. We fix $f$ as in Section~\ref{intro}, i.e.\ a primitive (holomorphic or Maa{\ss}) cusp form for $\Gamma_0(|r|)$ with trivial nebentypus,  and recall Convention \ref{def-conv} on signed levels. 

Given $\l\in(\Zz/q\Zz)^\times$ and $k\in\Zz$, we consider
\begin{equation}\label{firstmom}
  \mcL(f;\l,k)= \frac{1}{\vphis(q)}\sumstar_{\chi\mods q}\eps_\chi^k\chi(\l)L(\ftchi,1/2), 
\end{equation}
(cf. \eqref{moments} and our convention to drop  the parameter $s$ if it equals $1/2$).  We first observe that
\begin{equation}\label{reflect}
\mcL(f;\l,k)=\eps(f)\mcL\left(f; (-1)^k\ov{\l r},-(k+2)\right).
\end{equation}
  Indeed, for any non-trivial character $\chi\mods{q}$, we have
$\eps_{\chi}=\ov{\eps_{\chi}}^{-1}=\chi(-1)\eps^{-1}_{\ov\chi},$ and moreover
$$
L(\ftchi,1/2)=\eps(f)\chi(  r)\eps^2_\chi L(f\otimes\ov\chi,1/2)=
\eps(f)\chi(  r)\eps^{-2}_{\ov\chi} L(f\otimes\ov\chi,1/2),
$$
by the functional equation (cf.\ \eqref{twistedroot}), which implies the formula.  
\par
It is therefore sufficient to evaluate $\mcL(f;\l,k)$ for $k\geq -1$
in order to handle all values of $k$.  In this case we will prove a slightly more precise version of Theorem \ref{thm-moment1}. 
\begin{theorem}\label{thm-twistedfirstmoment} 
  There exists an absolute constant $B\geq 0$ such that for
  $k\geq -1$, $\l\in(\Zz/q\Zz)^\times$ and any $\eps>0$, we have
$$
\mcL(f;\l,k)=\delta_{k=0}\frac{\lf(\ov \l_q)}{\ov
  \l_q^{1/2}}+O_{f,\eps}((1+|k|)^{B}q^{-1/8+\eps}),
$$
where $\ov\l_q$ denotes the unique integer in the interval $[1,q]$ satisfying the congruence
$$\l\ov\l_q\equiv 1\mods q.$$
\end{theorem}

Combining this theorem with the formula \eqref{reflect}, we obtain

\begin{corollary}\label{cor-k=-2}
  There exists an absolute constant $B\geq 0$ such that, for
  $k\in \Zz$ and $\l\in(\Zz/q\Zz)^\times$, we have
$$
\mcL(f;\l,k)=\delta_{k=0}\frac{\lf(\ov \l_q)}{\ov
  \l_q^{1/2}}+\delta_{k=-2}\eps(f)\frac{\lf((  \l r)_q)}{(  \l
  r)^{1/2}_q}+O_{f,\eps}((1+|k|)^Bq^{-1/8+\eps}).
$$
where we denote by $( \l r)_q$ the unique integer in $[1,q]$
representing the congruence class $ \l r\mods q$.
\end{corollary}

\begin{remark}
  Observe that for $1\leq \ell < q^{1/2}$, we have
  $\ov \l_q \geq q^{1/2}$ unless $\ell = 1$, so that the main term for
  $k=0$ can be absorbed in the error term. Hence for $k \geq -1$, and
  $(\ell, q) = 1$ with $1 \leq \ell < q^{1/2}$, we obtain
$$\mcL(f;\l,k)=\delta_{k=0}\delta_{\ell=1}  +O_{f,\eps}((1+|k|)^{B}q^{-1/8+\eps}). $$	
\end{remark}



\section{Proof}

The first moment  decomposes into the sum over the even and odd characters
$$\mcL(f;\l,k)=\frac12\mcL^+(f;\l,k)+\frac12\mcL^-(f;\l,k)$$ where
$$\mcL^\pm(f;\l,k)=\frac{2}{\vphis(q)}\sumpm_{\substack{\chi\mods q\\ \chi \text{ primitive}}}\eps_\chi^k\chi(\l)L(\ftchi,1/2).$$
We evaluate the ``even'' first moment $\mcL^+(f;\l,k)$ in detail, the
odd part is entirely similar.
 
The approximate functional equation \eqref{fcteqn1},
\eqref{averagesigneven} and \eqref{twistedroot} give (using notations
of Convention \ref{conveven})
\begin{equation}\label{Lfcomp}
\begin{split}
  \mcL^{+}(f;\l,k)
  =&\sum_\pm q^{-1/2}\sum_{(n,q)=1}\Kl_k(\pm\ov{\l n};q)\frac{\lf(n)}{n^{1/2}}V_{f,1/2}\Big(\frac{n}{q\sqrt{|r|}}\Big)\\
  &+{\eps(f)}\sum_\pm q^{-1/2}\sum_{(n,q)=1}\Kl_{k+2}(\pm \ov{|r|\l}
  n;q)\frac{\lf(n)}{n^{1/2}}
  V_{f,1/2}\Big(\frac{n}{q\sqrt{|r|}}\Big)\\
  &+ O\Bigl(\sum_{n} \frac{|\lambda_f(n)|}{n^{1/2}} \Bigl|V_{f,
    1/2}\Big(\frac{n}{q\sqrt{|r|}}\Big)\Bigr| q^{-1-|k|/2} \Bigr).
\end{split}
\end{equation}
The error term is $O_f(q^{-(1 + |k|)/2})$.  Since $k \geq -1$, we have
$k+2 \not= 0$, so that it follows from Proposition~\ref{propFKM1} (after applying a smooth partition of unity into dyadic intervals  to the $n$-sum) 
that
$$
q^{-1/2}\sum_\stacksum{n\geq 1}{(n,q)=1}\Kl_{k+2}(\pm \ov{|r|\l}
n;q)\frac{\lf(n)}{n^{1/2}}
V_{f,1/2}\Big(\frac{n}{q\sqrt{|r|}}\Big)\ll_{\eps,f}(1+|k|)^B
q^{-1/8+\eps}.
$$
If $k\not=0$, the same bound holds for the first term on the right-hand side of \eqref{Lfcomp}, and otherwise 
this term equals
 $$\sum_\pm \sum_{ n\equiv \pm \ov\l\mods q}\frac{\lf(n)}{n^{1/2}}V_{f,1/2}\left(\frac{n}{q\sqrt{|r|}}\right)=
 \sum_\pm \frac{\lf(\ov \l^\pm_q)}{(\ov \l^\pm_q)^{1/2}}V_{f,1/2}\left(\frac{ \ov \l^\pm_q}{q\sqrt{|r|}}\right)+O_f(q^{\eps+\theta-1/2}),$$
 where $\ov \l^\pm_q$ denotes the unique solution $n$ of the equation 
 $$\l n\equiv \pm 1\mods q$$
 contained in the interval $[1,q]$. By \eqref{Vassymp} and \eqref{KSbound} we have 
 \begin{equation}\label{eps}
 \frac{\lf( \ov \l^\pm_q)}{( \ov \l^\pm_q)^{1/2}}V_{f,1/2}\left(\frac{ \ov \l^\pm_q}{q\sqrt{|r|}}\right)=\frac{\lf(\ov \l^\pm_q)}{(\ov \l^\pm_q)^{1/2}}+O_{f,\eps}(q^{\eps+\theta-1/2}).
 \end{equation}
 We have therefore proven that 
 $$\mcL^+(f;\l,k)=\delta_{k=0} \sum_\pm \frac{\lf(\ov \l^\pm_q)}{(\ov \l^\pm_q)^{1/2}} +O_{f,\eps}(q^{-1/8+\eps}).$$
 for $k\geq -1$.  Similarly we have
 $$\mcL^-(f;\l,k)=\delta_{k=0} \sum_\pm (\pm 1) \frac{\lf( \ov \l^\pm_q)}{( \ov \l^\pm_q)^{1/2}} +O_{f,\eps}(q^{-1/8+\eps}).$$
 Combining the two equations above, we obtain Theorem
 \ref{thm-twistedfirstmoment}.

\section{First moment with trace functions}\label{sec-m1-trace}

Let $q$ be a prime, and $\mcF$ a Mellin sheaf over $\Fq$ as in
Definition~\ref{def-sheaf}. Let $t\colon \Fqt\to \Cc$ be its trace
function and $\widetilde{t}$ its discrete Mellin transform. Fix $f$ as
in Section~\ref{intro}.

\begin{theorem}\label{th-m1-mellin}
  Assume that we are not in case \emph{(2)} of
  \emph{Lemma~\ref{lm-convol}}.  Then, for any integer $\ell\geq 1$,
  we have\footnote{Here, the integer $\ell$ is not related to the
    auxiliary prime used in defining the sheaf $\mcF$.}
\begin{equation}\label{eq-m1-mellin}
  \frac{1}{\vphis(q)}
  \sums_{\chi\mods{q}}
  L(f\otimes \chi,\demi)\chi(\ell)\widetilde{t}(\chi)\ll q^{-1/8+\eps}
\end{equation}
for any $\eps>0$, where the implied constant depends on $f$ and
$\eps$, and polynomially on $\cond(\mcF)$.
\end{theorem}

\begin{proof}
  We may assume that $\ell$ is coprime to $q$.  For any non-trivial
  character $\chi$, by the approximate functional
  equation~(\ref{fcteqn1}) and~(\ref{eq-epspm}), we have
$$
L(\ftchi,1/2)=\sum_{n\geq
  1}\frac{\lf(n)}{n^{1/2}}\chi(n)V\Bigl(\frac{n}{q\sqrt{|r|}}\Bigr)+
\eps(f)\chi(r)\eps^2_{\chi} \sum_{n\geq
  1}\frac{\lf(n)}{n^{1/2}}\ov\chi(n)V\Bigl(\frac{n}{q\sqrt{|r|}}\Bigr)
$$
where $V=V_{f,\chi(-1),1/2}$.  Distinguishing according to the
parity of $\chi$, the left-hand side
of~(\ref{eq-m1-mellin}) is
$$
\frac{1}{\vphis(q)} \sump_{\substack{\chi\mods{q}\\\chi\not=1}}
L(f\otimes \chi,\demi)\chi(\ell)\widetilde{t}(\chi)
+\frac{1}{\vphis(q)}\summ_{\chi\mods{q}} L(f\otimes
\chi,\demi)\chi(\ell)\widetilde{t}(\chi).
$$
We consider the sum over non-trivial even characters, since the case
of odd characters is (as usual) entirely similar.  We have
\begin{multline*}
  \frac{1}{\vphis(q)} \sump_{\substack{\chi\mods{q}\\\chi\not=1}}
  L(f\otimes \chi,\demi)\chi(\ell)\widetilde{t}(\chi) \\= \sum_{n\geq
    1}\frac{\lf(n)}{n^{1/2}}V\Bigl(\frac{n}{q\sqrt{r}}\Bigr) T_1(n, \ell)
  +\eps(f)\sum_{n}\frac{\lf(n)}{n^{1/2}}V\Bigl(\frac{n}{q\sqrt{r}}\Bigr)
  T_2(n, \ell, r)
\end{multline*}
where
\begin{align*}
  T_1(n, \ell)
  &=\frac{1}{\vphis(q)}
    \sump_{\substack{\chi\mods{q}\\\chi\not=1}} \widetilde{t}(\chi)  \chi(n\ell)\\
  T_2(n, \ell, r)&=  \frac{1}{\vphis(q)}
             \sump_{\substack{\chi\mods{q}\\\chi\not=1}}
  \eps^2_\chi\chi(r\ell)\bar{\chi}(n)\widetilde{t}(\chi).
\end{align*}
By~(\ref{average}) and discrete Mellin inversion, we compute
\begin{align*}
  T_1(n, \ell)
  &= 
    \frac{\alpha(q)}{2}\Bigl(t(\ov{n\ell})+t(-\ov{n\ell})\Bigr)
    -\frac{1}{\vphis(q)}\widetilde{t}(1) \delta_{(n, q) = 1}
  \\
  T_2(n, \ell, r)
  &=    \frac{\alpha(q)}{2}\Bigl((t\star \Kl_2)(r\ell\ov{n})
    +(t\star \Kl_2)(-r\ell\ov{n})\Bigr)
    -\frac{\eps_1^2}{\vphis(q)}\widetilde{t}(1) \delta_{(n, q) = 1}
\end{align*}
where
$$
\alpha(q)=\frac{\vphi(q)}{\vphis(q)\sqrt{q}}\sim \frac{1}{\sqrt{q}}.
$$
\par
Since $\widetilde{t}(1) \ll 1$, the contribution of the trivial character
to the first moment is $\ll q^{-1}$, where the implied constant
depends only on $\cond(\mcF)$. 
\par
Since we are in Case (1) of Lemma~\ref{lm-convol}, we see that, up to
negligible error, the even part of the first moment is the sum of four
expressions of the type
$$
\gamma\alpha(q)\sum_{\substack{n\geq 1\\(n,q)=1}}
\frac{\lf(n)}{n^{1/2}}V_{f,+,1/2}\Bigl(\frac{n}{q\sqrt{|r|}}\Bigr)\tau(n)
$$
with $\gamma=1$ or $\gamma=\eps(f)$, where $\tau$ is (by
Lemma~\ref{lm-convol} in the cases involving $t\star \Kl_2$) a trace
function of a geometrically irreducible middle-extension sheaf of
weight $0$ with conductor $\ll 1$. By \cite{FKM1}*{Th.\ 1.2}, each of
these sums is $\ll q^{-1/8+\eps}$ (cf.\ Proposition~\ref{propFKM1} for a special case), where the implied constant depends
only on $\eps$, $f$ and polynomially on $\cond(\mcF)$.
\end{proof}

\begin{remark}
Case (2) in Lemma~\ref{lm-convol} leads to a first moment
$$
\frac{1}{\vphis(q)}\sumstar_{\chi\mods
  q}L(\ftchi,1/2)\chi(a\ell)\eps_{\chi}^{-2},
$$
which is evaluated asymptotically in
Theorem~\ref{thm-twistedfirstmoment}.
\end{remark}


\chapter{Computation of the second twisted moment}
\label{ch-second}

\section{Introduction}\label{sec-2ndmoment-intro}

In this chapter, we prove Theorem \ref{thsecondmoment}, which we will
now state with precise main terms. 

We fix $f$ as in Section~\ref{intro}. Let $g$ be a primitive cusp form
of signed level $r'$ coprime to $q$ and trivial central character; we allow
the possibility that $g=f$. We recall that we use
Convention~\ref{def-conv} concerning the signed levels of cusp forms.
We will use the approximate functional equation \eqref{fcteqn2}, and
the corresponding test functions $W_{f,g,\pm,s}$ (see~(\ref{Wfgdef}))
and ``signs'' $\eps(f,g,\pm, s)$ (see~(\ref{eq-epspm})).

Recall that we consider the twisted second moments
\begin{equation}\label{eq-twistedsecond}
\mcQ(f,g,s;\l,\l')=\frac{1}{\vphis(q)}\sums_{\chi\mods q}
L(f\otimes\chi,s)\ov{L(g\otimes\chi,s)}\chi(\l)\overline{\chi(\l')}
\end{equation}
for integers $1\leq \l,\l'\leq L \leq q^{1/2}$ (say), with
$(\l\l',qrr')=(\l,\l')=1$, and $s\in\Cc$ with
$|\Re s -1/2| < (\log q)^{-1}$. We write $\sigma = \Re s$, and we may
assume without loss of generality that $q \geq e^{10}$, say.  

We write $|r| = \rho \delta$, $|r'| = \rho' \delta$ with
$\delta = (r, r')\geq 1$ and $(\rho, \rho') = 1$. In particular, note
that $\rho$ and $\rho'$ are positive.


\begin{theorem}\label{th-second-final}
We have
$$
\mcQ(f,g,s;\l,\l')=\MT(f,g,s;\l,\l')+O(|s|^{O(1)}L^{\expoL}q^{-\expoq+\eps}),
$$
where the main term is given by
$$
\MT(f,g,s;\l,\l')=\frac12 \MT^+(f,g,s;\l,\l')+\frac12
\MT^-(f,g,s;\l,\l')
$$ 
with
\begin{multline}\label{MTsecondmoment}
  \MT^{\pm}(f,g;\l,\l')= \frac12\sum_{n\geq
    1}\frac{\lambda_f(\l'n)\lamg(\l
    n)}{{\l'}^s\l^{\spr}n^{2\sigma}}W_{f,g,\pm,s}\Bigl(\frac{\l\l'n^2}{q^2|rr'|}\Bigr)
  \\+\frac{\eps(f,g,\pm,s)\lambda_f(\rho)\lambda_g(\rho')}{2\rho^{1-s}
    (\rho')^{1-\spr}}\sum_{n\geq 1}\frac{\lambda_f( \l n)\lamg( \l'
    n)}{ \l^{1-s} (\l')^{1-
      \spr}n^{2-2\sigma}}W_{f,g,\pm,1-s}\Bigl(\frac{\l\l'n^2}{q^2\delta^2}\Bigr).
\end{multline}
\par
If $r=r'$ and $\eps(f)\eps(g)=-1$, then $\mcQ(f,g,s;1,1)=0$.
\end{theorem}

In the rest of this chapter, to simplify notation, we will not display
the $s$ dependency and will write $\mcQ(f,g;\l,\l')$ for
$\mcQ(f,g,s;\l,\l').$ Moreover if $f=g$, we will just write
$\mcQ(f;\l,\l')$.

We first justify the last assertion of the theorem concerning the
exact vanishing of the untwisted second moment when $r=r'$ and
$\eps(f)\eps(g)=-1$. Indeed, in that case we have
$$
L(f\otimes\chi,\demi)\overline{L(g\otimes \chi,\demi)}=-
L(f\otimes\ov{\chi},\demi)\overline{L(g\otimes \ov{\chi},\demi)}
$$
for all $\chi$, by~(\ref{eq-self-dual}). If $\chi$ is real, this shows
that $L(f\otimes\chi,\demi)\overline{L(g\otimes \chi,\demi)}=0$, and
otherwise, the sum of the values for $\chi$ and $\ov{\chi}$ is zero.
 
 \section{Isolating the main term}\label{sec-secondmoment}

This moment decomposes as the sum of its even and odd part
$$\mcQ(f,g;\l,\l')=\frac{1}{2}\sum_{\pm}\mcQ^\pm(f,g;\l,\l')$$
where
\begin{equation*}
  \mcQ^\pm(f,g;\l,\l')=\frac{2}{\vphis(q)}\sumpm_{\substack{\chi\mods q\\ \chi \text{ primitive}}} L(f\otimes\chi,s)\ov{L(g\otimes\chi,s)}\chi(\l/\l').
\end{equation*}
We give the details for the even second moment $\mcQ^+(f,g,\l,\l')$;
the treatment of the odd second moment is identical.

We apply the approximate functional equation \eqref{fcteqn2}. A simple
large sieve argument shows that
$\mcQ^+(f,g,\l,\l') = O_{f, g}(|s|^{O(1)})$, so that replacing
$\vphis(q)$ by $\vphi(q)$ introduces an error of
$O_{f, g}(|s|^{O(1)} q^{-1})$. Adding and subtracting the contribution
of the trivial character using the bound
$$  
\sum_{(nm,q)=1}\frac{\lf(m)\lamg(n)}{(mn)^{1/2}}W_{f,g,+
  ,s}\Bigl(\frac{nm}{q^2|rr'|}\Bigr)\ll_{f, g} |s|^{O(1)},
$$
and applying orthogonality \eqref{average}, we obtain
\begin{multline*}
  \mcQ^+(f,g; \l,\l') = \frac12\sum_{\pm}\sumsum_\stacksum{\l m\equiv
    \pm\l'n\mods q}{(mn,q)=1} \frac{\lambda_f(m)
    \lamg(n)}{m^{s}n^{\spr}}
  W_{f,g,+,s}\left(\frac{mn}{q^2|rr'|}\right)\\
  +\eps(f,g,+,s) \frac{1}2\sum_{\pm}\sum_\stacksum{r\l n\equiv \pm r'\l'
    m\mods q}{(mn,q)=1}\frac{\lf(m)\lamg(n)}{m^{1-s}n^{1-\spr}}
  W_{f,g,+,1-s}\left(\frac{mn}{q^2|rr'|}\right)
  \\
  + O\left(\frac{|s|^{O(1)}}{q}\right).
\end{multline*}
The contribution of the terms such that $q\mid m$ (and therefore
$q\mid n$) is bounded trivially by
$\ll |s|^{O(1)} q^{-1+2\theta+o(1)}$ so we can remove the constraint
$(mn,q)=1$.

Let $\eps=\pm 1$ be the sign of $rr'$. The contribution of the terms
satisfying
$$
\l m=\l'n\quad
$$
in the first sum, and 
$$
r\l n=\eps r'\l' m
$$ 
in the second sum, forms the main term, and is denoted
$\MT^+(f,g;\l,\l')$.  (Note that the corresponding equations with
opposite signs have no solutions).
We write as above $|r| = \rho \delta$, $|r'| = \rho' \delta$ with
$\delta = (r, r')\geq 1$ and $(\rho, \rho') = 1$. Since $\l,\l'$ are
coprime and coprime to $\rho\rho'$, the solutions $n\geq 1$,
$m\geq 1$, of the second equation are parameterized in all cases by
$$
n=\rho'\l'k,\quad\quad m=\rho\l k,
$$
where $k\geq 1$.  The main term is then
\begin{multline*}
  \frac12\sum_{n\geq 1}\frac{\lambda_f(\l'n)\lamg(\l
    n)}{{\l'}^s\l^{\spr}n^{2\sigma}}W_{f,g,+,s}\Bigl(\frac{\l\l'n^2}{q^2|rr'|}\Bigr)
  \\+\frac{\eps(f,g,+,s)}{2}\sum_{n\geq 1}\frac{\lambda_f(\rho\l
    n)\lamg(\rho'\l' n)}{(\rho\l)^{1-s}(\rho'\l')^{1-
      \spr}n^{2-2\sigma}}W_{f,g,+,1-s}\Bigl(\frac{\l\l'n^2}{q^2\delta^2}\Bigr).
\end{multline*}
Moreover, since $\rho \mid r$ and $\rho' \mid r'$, we have
$\lambda_f(\rho\ell n)=\lambda_f(\rho)\lambda_f(\ell n)$ and
$\lambda_g(\rho'\ell' n)=\lambda_g(\rho')\lambda_g(\ell' n)$,
hence this main term becomes
\begin{multline}\label{mainplus}
  \frac12\sum_{n\geq 1}\frac{\lambda_f(\l'n)\lamg(\l
    n)}{{\l'}^s\l^{\spr}n^{2\sigma}}W_{f,g,+,s}\Bigl(\frac{\l\l'n^2}{q^2|rr'|}\Bigr)
  \\+\frac{\eps(f,g,+,s)}{2} \frac{\lambda_f(\rho)
    \lambda_g(\rho')}{\rho^{1-s}(\rho')^{1-\spr}} \sum_{n\geq
    1}\frac{\lambda_f( \l n)\lamg( \l' n)}{ \l^{1-s} (\l')^{1-
      \spr}n^{2-2\sigma}}W_{f,g,+,1-s}\Bigl(\frac{\l\l'n^2}{q^2\delta^2}\Bigr).
\end{multline}

The first term of this sum equals
\begin{equation}\label{sun1}
  \frac12\intc_{(2)}\frac{L_\infty(f,+,s+u)}
  {L_\infty(f,+,s)}\frac{L_\infty(g,+,\spr+u)}{L_\infty(g,+,\spr)}
  L(f\times g,2s,u;\l',\l)
  G(u) (q^2|rr'|)^{u} \frac{du}{u}, 
\end{equation}
where
\begin{equation}\label{Lfgdef}
    L(f\times g,2s,u;\l',\l)=\sum_{n\geq 1}\frac{\lambda_f(\l' n)\lamg(\l n)}{{\l'}^{s+u}\l^{\spr+u}n^{2\sigma+2u}}	,\ \Re (2\sigma+2u)>1,
\end{equation}
and the second term equals
\begin{multline}\label{sun2}
  \frac{\eps(f,g,+,s) \lambda_f(\rho)
    \lambda_g(\rho')}{2\rho^{1-s} (\rho')^{1 - \spr}}
  \intc_{(2)}\frac{L_\infty(f,+,1-s+u)}{L_\infty(f,+,1-s)}
  \\
  \times \frac{L_\infty(g,+,1-\spr+u)}{L_\infty(g,+,1-\spr)} L(f\times
  g,2-2s,u; \l, \l') G(u) (q^2\delta^2)^{u} \frac{du}{u}.
\end{multline}

Similarly, the odd part of the second moment $\mcQ^-(f,g, \l,\l')$
yields the second part of the main term, namely
\begin{multline*}
  \MT^-(f,g;\l,\l') = \frac12\sum_{n\geq
    1}\frac{\lambda_f(\l'n)\lamg(\l
    n)}{{\l'}^s\l^{\spr}n^{2\sigma}}W_{f,g,-,s}\Bigl(\frac{\l\l'n^2}{q^2|rr'|}\Bigr)
  \\+\frac{\eps(f,g,-,s)\lambda_f(\rho) \lambda_g(\rho')
  }{2\rho^{1-s} (\rho')^{1 - \spr} }\sum_{n\geq 1}\frac{\lambda_f(
    \l n)\lamg( \l' n)}{ \l^{1-s}( \l')^{1-
      \spr}n^{2-2\sigma}}W_{f,g,-,1-s}\Bigl(\frac{\l\l'n^2}{q^2\delta^2}\Bigr),
\end{multline*}
where the first term equals
\begin{equation}\label{sun3}
  \frac12\intc_{(2)}\frac{L_\infty(f,-,s+u)}{L_\infty(f,-,s)}
  \frac{L_\infty(g,-,\spr+u)}{L_\infty(g,-,\spr)}L(f\times g,2s,u;\l',\l)
  G(u)  (q^2|rr'|)^{u}  \frac{du}{u}
\end{equation}
and the second
\begin{multline}\label{sun4}
  \frac{\eps(f,g,-,s)\lambda_f(\rho) \lambda_g(\rho')}{2\rho^{1-s}
    (\rho')^{1 - \spr}}
  \intc_{(2)}\frac{L_\infty(f,-,1-s+u)}{L_\infty(f,-,1-s)}
  \\
  \times \frac{L_\infty(g,-,1-\spr+u)}{L_\infty(g,-,1-\spr)} L(f\times
  g,2-2s,u; \l, \l') G(u) (q^2\delta^2)^{u} \frac{du}{u}.
\end{multline}
\par
At this stage, we can therefore write
$$
\mcQ(f,g,s;\l,\l')=\MT(f,g,s;\l,\l')+\text{(error term)}
$$
as in Theorem~\ref{th-second-final}, with
$$
\MT(f,g,s;\l,\l')=\frac12 \MT^+(f,g,s;\l,\l')+\frac12
\MT^-(f,g,s;\l,\l'),
$$
and with an error term that we will estimate in the next sections to
conclude the proof of the theorem.
\par
It is not necessary (or, indeed, useful) to evaluate
the mains terms very precisely in general, since in most applications
(as in later chapters) we will perform further averages or
combinations of them. 

However, the special case $\l=\l'=1$ and $s=\demi$ (i.e., the ``pure''
second moment) is important, so we transform the main term in that
case. We recall the notation $L^{\ast}(f\otimes g,1)$ from
\eqref{lstarfg} and \eqref{lstarres}, and recall in particular that
these are non-zero.

\begin{proposition}\label{pr-mt}
If $f=g$, then we have
$$
\MT(f,f;\demi,1,1)= 2 \prod_{p\mid r}(1+p^{-1})^{-1} \frac{L^{\ast}(
  \symf,1)}{\zeta(2)} \log q +\beta_f + O(q^{-2/5})
$$
for some constant $\beta_f$. If $f\not=g$, then
$$
\MT(f,g,\demi;1,1)= \Bigl(1 + \varepsilon(f) \varepsilon(g
)\frac{\lambda_f(\rho)\lambda_g(\rho')}{(\rho\rho')^{1/2}}\Bigr)
L^{\ast}( f\otimes g,1) + O(q^{-2/5}),
$$
where the leading constant has modulus $\leq 2$, and is non-zero
unless $\rho=\rho'=1$ and $\eps(f)\eps(g)=-1$.
\end{proposition}

\begin{proof}
  The formulas follow easily from shifting the contour in
  \eqref{sun1}, \eqref{sun2}, \eqref{sun3}, \eqref{sun4} to
  $\Re u = -1/5$ and applying the residue theorem, involving only a
  pole at $u=0$ occurs, since by definition (see~(\ref{lstarfg})
  and~(\ref{eq-symast})), we have
$$
L(f\times g,1,u;1,1)=L^{\ast}(f\otimes g,1+2u),
$$
and
$$
L(f\times f,1,u;1,1)=\frac{\zeta^{(r)}(1+2u)}{\zeta^{(r)}(2+4u)}
L^{\ast}(\symf ,1+2u),
$$
and moreover $\eps(f,g,\pm,\demi)=\eps(f)\eps(g)$
(see~(\ref{eq-epspm-demi})).
\par 
If $f\not=g$, then since $\rho\mid r$ and $\rho'\mid r'$, we have
$$
\frac{|\lambda_f(\rho)\lambda_g(\rho')|}{\sqrt{\rho\rho'}}\leq
\frac{1}{\sqrt{\rho\rho'}}.
$$
We deduce first that
$$
\Bigl|1 + \varepsilon(f) \varepsilon(g
)\frac{\lambda_f(\rho)\lambda_g(\rho')}{(\rho\rho')^{1/2}}\Bigr| \leq 2,
$$
and next that the leading constant can only be zero if $\rho=\rho'=1$,
and then only if $\eps(f)\eps(g)=-1$ (recall that
$L^{\ast}(f\otimes g,1)\not=0$ by Lemma~\ref{lm-ast}).
\end{proof}

\section{The error term}

The contributions to the error term are of the form
$$\label{pg-et}
\ET(f,g;\l,\pm\l')= \sumsum_\stacksum{\l m\equiv \pm\l'n\mods q}{\l
  m\not= \l'n} \frac{\lambda_f(m)\lamg(n)}{m^sn^{\spr}}
W_{f,g,\pm,s}\left(\frac{mn}{q^2}\right)
$$
or $\ET(f,g;r\l,\pm r'\l')$.  The following bound then implies the
theorem.

\begin{theorem}\label{thmETtwisted} 
  Let $s\in\Cc$ be a complex number such that
  $|\sigma-1/2| \leq (\log q)^{-1}$. Let $L\leq q^{1/2}$.  For any
  coprime integers $\l$ and $\l'$ such that $1\leq \l,\l'\leq L$, we
  have
  $$
  \ET(f,g;\l,\pm\l')\ll (rr'|s|)^{O(1)}L^\expoL q^{-\expoq+\eps}.
$$
\end{theorem}

The proof proceeds as in \cite{BFKMM}. Using a partition of unity on
the $m,n$ variables and a Mellin transform to separate the variables
$m$ and $n$ in the weight function, we reduce to the evaluation of
$O(\log^2 q)$ sums of the shape
$$
\ET(M,N;\l,\pm \l')=\frac{1}{(MN)^{1/2}}\sum_{\substack{\l m\equiv
    \pm\l'n\mods q \\ \l m \not= \l'n}}
{\lf(m)\lambda_g(n)}W_1\Bigl(\frac{m}M\Bigr)W_2\Bigl(\frac{n}N\Bigr)
$$ 
test functions for $W_1,W_2$ satisfying \eqref{Wbound} and for
parameters $1\leq M,N$ such that $MN\leq q^{2+\eps}$ (where we have
removed $f,g$ from the notation for simplicity). As in \cite{BFKMM} the weight functions $W_1, W_2$ depend on a parameter of size $(\log q)^2$. 

We will explain the proof of the estimate for $\ET(M,N;\l,\l')$. The
case of $\ET (M,N;\l,-\l')$ is very similar, and left to the reader.

We first note that we may assume that $|r|+|r'|+|s|\ll q^{\varepsilon}$,
since otherwise Theorem \ref{thmETtwisted} holds trivially.

We will bound the sums $\ET(M,N;\l,\l')$ in different ways, according
to the relative sizes of $M,N$. We may assume without loss of
generality that $M\leq N$. Then we have at our disposal the following
three bounds.

\section{The trivial bound} 

Using \eqref{KSbound} and \eqref{RP4}, and distinguishing the cases
$4LN<q$ (in which case the equation
$\l m\equiv \pm \l'n\mods q,\ \l m \not= \l'n$ has no solutions) and
$2LN\geq q$, we have
\begin{equation}\label{eqtrivial}
  \ET(M,N;\l,\l')\ll q^\eps 
  M^\theta\frac{M}{(MN)^{1/2}}\frac{LN}q\ll q^\eps LN^\theta\frac{(MN)^{1/2}}q.
\end{equation}
	
\section{The shifted convolution bound} 

Next we appeal to the shifted convolution estimate of Proposition
\ref{thmSCP}. Setting
$$
M'=\min(\l M,\l'N)\leq LM, \quad N'=\max(\l M,\l'N)\leq LN,
$$ 
and using the bounds
$$
M'N'\leq L^2 q^{2+\eps},\quad\quad M\leq \min(N, q^{1+\eps}),
$$
we see that $\ET(M,N;\ell,\ell')$ is
\begin{align}
  &\ll \frac{q^{\varepsilon}}{(MN)^{1/2}}\left( \left(\frac{N'}{q^{1/2}} + \frac{{N'}^{3/4} {M'}^{1/4}}{q^{1/4}} \right)\left(1+\frac{(M'N')^{1/4}}{q^{1/2}}\right)   + \frac{{M'}^{3/2 + \theta}}{q} \right)	\nonumber\\
  &\ll 
    \frac{q^{\varepsilon}}{(MN)^{1/2}}\left( \left(\frac{N}{q^{1/2}} + \frac{N^{3/4}M^{1/4}}{q^{1/4}} \right)L^{3/2}   + \frac{{(LM)}^{3/2 + \theta}}{q} \right)\nonumber\\
  &\ll
    q^{\varepsilon}{L^{3/2}}\left(
    \left(\frac{N}{qM}\left(1+\frac{N}{qM}\right) \right)^{1/4}   +
    \frac{L^\theta}{q^{1/2-\theta}} \right).\label{eqSCP}
\end{align}

\section{Bilinear sums of Kloosterman sums}

This is similar to~\cite{BFKMM}*{\S 6.2}.  We apply the Voronoi
summation formula of Corollary \ref{corvoronoi} to the $n$
variable. We will do so only under the assumption
\begin{equation}\label{eqnodiag}
N>4LM, 	
\end{equation}
so that the summation condition $\ell m\not=\l' n$ is automatically satisfied.
This expresses $\ET(M,N;\l,\l')$ as the sum of two terms. The first one
is
$$
\frac1{q(MN)^{1/2}}\sumsum_{m,n}\lf(m)\lamg(n)W_1\left(\frac
  mM\right)W_2\left(\frac{n}N\right)\ll \frac{q^\eps}{q(MN)^{1/2}},
$$
which is very small, and the second is
\begin{multline*}
  \eps(g)\sum_{\pm}\frac{N}{q|r|^{1/2}}\frac{1}{(qMN)^{1/2}}\sum_{m,n\geq
    1}\lf(m)\lamg(n)
  \\
  \times W_1\left(\frac{m}M\right)\widetilde
  W_{2,\pm}\left(\frac{n}{N^*}\right)\Kld(\pm\ov{|r'| \l}\l' mn;q),
\end{multline*}
where $N^*=q^2|r|/N$.
\par
By Lemma \ref{besseldecay}, the function $\widetilde W_{2,\pm}(y)$ has
rapid decay for $y\geq q^{\eps}$. By a further partition of unity, we
reduce to bounding quantities of the shape
$$
E=\frac{(1 + N^{\ast}/N')^{2\theta + \varepsilon}}{(qMN^*)^{1/2}}
B(\Kld,\bfalpha.\bfbeta)
$$
with coefficient sequences
$$
\bfalpha=\bigl(\lf(m)W_1(m/M)\bigr)_{m\leq 2M}\quad \text{and}
\quad\bfbeta=\bigl(\lamg(n)\widetilde{W}_{2,\pm}(n/N*)\bigr)_{n\leq
  N'}
$$
that are supported on $[1,2M]$ and $[1,2q^\eps N^*]$, respectively,
and where
$$
B(\Kld,\bfalpha,\bfbeta)=\sumsum_{m,n}\alpha_m\beta_n\Kld(amn;q)
$$
with $a=\pm\overline{|r'|\l}\l'$ coprime to $q$, as
in~(\ref{eq-bilform}).


 
Bounding the Kloosterman sums trivially and using \eqref{RP4} we
obtain first
\begin{equation}
  E\ll 
  q^\eps \frac{MN^*}{(qMN^*)^{1/2}} =
  q^{\eps}\left(\frac{qM}{N}\right)^{1/2}. 
  \label{eqvoronoiboundtrivial}
\end{equation}

Using instead Proposition~\ref{thmtypeII}, with the sequence $\bfbeta$
viewed as a sequence of length $N^{\ast}$, we obtain 
\begin{align}
  E  &\ll
       q^\eps \frac{MN^*}{(qMN^*)^{1/2}}
       \left(\frac{1}{M}+q^{-\frac1{32}}
       \Bigl(\frac{q}{MN^*}\Bigr)^{{3}/{8}}  \right)^{1/2}\nonumber\\
     &\ll 
       q^{\eps}\left(\frac{qM}{N}\right)^{1/2}\left(\frac{1}{M}
       +q^{-1/32}\Bigl(\frac{N}{qM}\Bigr)^{3/8}\right)^{1/2} ,
\label{eqvoronoibound}
\end{align}
under the assumptions that
$$
M\ll q^{\eps+1/4}N^*,\ MN^*\ll q^{5/4}
$$
or equivalently
\begin{equation}\label{eq-condition}
MN\leq q^{\eps+9/4},\ q^{3/4}\leq N/M.
\end{equation}
Observe that the first inequality is always satisfied.

\section{Optimization}

Set $\eta=\expoq$. We have now derived the four basic bounds
\eqref{eqtrivial}, \eqref{eqSCP}, \eqref{eqvoronoiboundtrivial} and
\eqref{eqvoronoibound}, all of which provide estimates for
$\ET(M, N, \ell, \ell')$.  We define $\beta,\lambda,\mu,\nu$ so that
the identities
 \begin{gather*}
   M=q^\mu,\ N=q^\nu,\ L=q^\lambda,\
   \ET( M, N, \ell, \ell')=q^\beta,\\
   \mu^*=2-\mu,\ \nu^*=2-\nu.
 \end{gather*}
Our objective is to prove that
 \begin{equation}\label{eqobjective}
 \beta\leq \eps+\frac32\lambda-\eta,
\end{equation}
which will conclude the proof of Theorem~\ref{thmETtwisted}.  We use
the same method as in \cite{BFKMM}*{\S 6.2}. 
 
We have
\begin{gather*}
  0\leq\mu\leq\nu,\quad \mu+\nu\leq 2+\eps,\quad \mu\leq 1+\eps,
  \\
  -1-\eps\leq 1+\mu-\nu\leq 1,
\end{gather*}
and assume that 
$
\lambda\leq \frac{1}{2}.
$  
By the trivial bound we have
\begin{equation*}
\beta\leq \eps +\lambda+\theta\nu+\frac12(\mu+\nu-2).	
\end{equation*}
We may therefore assume that
\begin{equation*}
2-2\eta-2\theta\nu \leq \mu+\nu\leq 2+\eps	
\end{equation*}
(otherwise \eqref{eqobjective} holds)
 and therefore
\begin{equation}\label{muplusnubound}
-2\eta-2\theta\nu\leq \mu-\nu^*\leq \eps.
\end{equation}
Applying the shifted convolution estimate~\eqref{eqSCP}, we obtain
$$
\beta\leq
\eps+\frac32\lambda+\sup\Bigl(\frac{1}4(\nu-\mu-1),\frac{1}2(\nu-\mu-1),
(\lambda+1)\theta-\frac{1}2 \Bigr),
$$
so that \eqref{eqobjective} holds unless
\begin{equation}\label{SCPlowerbound}
  1-4\eta\leq \nu-\mu\text{ or equivalently }\mu+\nu^*\leq 1+4\eta.	
\end{equation}
This inequality implies that \eqref{eqnodiag} and \eqref{eq-condition}
both hold. We may therefore apply \eqref{eqvoronoiboundtrivial} and
\eqref{eqvoronoibound}.

Applying \eqref{eqvoronoiboundtrivial} we obtain
$$\beta\leq \eps+\frac12(1+\mu-\nu)=\eps+\frac12(\mu+\nu^*-1)$$
which establishes \eqref{eqobjective} unless 
\begin{equation*}
1-2\eta \leq \mu+\nu^*.	
\end{equation*}
This inequality together with \eqref{muplusnubound} implies that
\begin{equation}\label{eq-new}
\mu\geq \frac12-2\eta-\theta\nu.
\end{equation}
Applying now \eqref{eqvoronoibound} we obtain
\begin{align*}
  \beta&\leq \eps+\frac12(\mu+\nu^*-1)+\max\Bigl(
         -\frac12\mu,-\frac{1}{64}-\frac{3}{16}(\mu+\nu^*-1)\Bigr)
  \\
       &\leq \eps+\max\Bigl(3\eta+\frac12\theta\nu-\frac14,
         \frac5{16}4\eta-\frac1{64}\Bigr)
\end{align*}
by~(\ref{SCPlowerbound}) and~(\ref{eq-new}),
resp.~(\ref{SCPlowerbound})).  This concludes the proof
of~(\ref{eqobjective}), since the first term in the maximum is
$<-\eta$ (recall that $\theta\leq 7/64$ and $\nu\leq 2$) and the
second is equal to $-\eta$.



\chapter{Non-vanishing at the central point}
\label{ch-central}

\section{Introduction}

In this chapter, we will prove Theorem \ref{thmnonvanishing+angle}
using the mollification method. We fix $f$ as in Section~\ref{intro}.
Recall that by ``interval'' in $\Rr/\pi\Zz$, we mean the image of an
interval of $\Rr$ under the canonical projection, and its measure is
the probability Haar measure $\mu$ of this image. The statement to prove is:

\begin{theorem}\label{thm61}
  Let $I\subset \Rr/\pi\Zz$ be an interval of positive measure
  $\mu(I)$. There exists a constant $\eta=\eta_I > 0$ such that, as
  $q\to+\infty$ among the primes, we have
$$
  \frac{1}{\vphis(q)} \Big|\Big\{\chi\mods q\,\, \text{\emph{non-trivial}}\mid\,
  |L(f\otimes\chi,1/2)|\geq \frac{1}{\log q},\ \theta(\ftchi)\in I\Big\}\Big|
  \geq \eta+o_{f,I}(1).
$$
Furthermore this  formula remains true with the following choice of $\eta$ 
$$
\eta =\eta_I = \frac{\mu (I)^2}{1443 \,\zeta (2)}.
$$
\end{theorem}

In Section~\ref{sec-mellin}, which may be omitted in a first reading,
we will also prove a positive proportion of non-vanishing for central
values of the twisted $L$-functions with characters satisfying
conditions on the discrete Mellin transform of a quite general
  trace function.

\section{The Cauchy-Schwarz inequality}

Let $I\subset\Rr/\pi\Zz$ be an interval with positive measure $\mu(I)$
and characteristic function $\chi_I$.  Let
$\chi\mapsto M(\ftchi,\bfx_L)$ be a function defined for Dirichlet
characters modulo $q$ (later, it will be the ``mollifier''), depending
on the fixed cusp form $f$ and on some additional data $\bfx_L$.

By the Cauchy-Schwarz inequality, we have
\begin{multline}\label{CS}
  \Biggl\vert\,\frac1{\vphis(q)}
  \sums_{\chi\mods q}\delta_\stacksum{|L(\ftchi,1/2)|\geq (\log q)^{-1}}{\theta(\ftchi)\in I}L(\ftchi,1/2)M(\ftchi,\bfx_L)\,\Biggr\vert^2\\
  \leq  \mcQ(f;\bfx_L) \times \Bigl(\frac1{\vphis(q)}\sums_{\chi\mods
    q}\delta_\stacksum{|L(\ftchi,1/2)|\geq (\log
    q)^{-1}}{\theta(\ftchi)\in I}\Bigr),
\end{multline}
where 
\begin{equation}\label{defOmegarevision}
\mcQ(f;\bfx_L):=\frac1{\vphis(q)}\sums_{\chi\mods
  q}|L(\ftchi,1/2)M(\ftchi,\bfx_L)|^2.	
\end{equation}

On the left-hand side of \eqref{CS}, we remove the condition $|L(\ftchi,1/2)|\geq (\log q)^{-1}$ in a trivial manner,
\begin{multline}\label{8.2}
  \frac1{\vphis(q)}\sums_{\chi\mods q}\delta_\stacksum{|L(\ftchi,1/2)|\geq (\log q)^{-1}}{\theta(\ftchi)\in I}L(\ftchi,1/2)M(\ftchi,\bfx_L)=\\
  \mathscr{L}(f;\bfx_L,\chi_I)+O\Bigl(\frac1{q \log q} 
  \sums_{\chi\mods q}|M(\ftchi,\bfx_L)|\Bigr)
\end{multline}
where, for any function $\psi:\Rr/\pi\Zz\ra \Cc$, we have defined
\begin{equation}\label{defLrevision}
\mathscr{L}(f;\bfx_L,\psi):=\frac1{\vphis(q)}\sums_{\chi\mods q}\psi(\theta(\ftchi))L(\ftchi,1/2)M(\ftchi,\bfx_L).
\end{equation}

Given $\delta$ with $0< \delta < 1$, let
$$
\psi(\theta)=\psi_{I,\delta}(\theta)=\sum_{|k|\leq
  k_\delta}\what\psi(2k) e^{2ik\theta}
$$ 
be a trigonometric polynomial of period $\pi$ such that
\begin{equation}\label{diffofmeasures}
\int_0^{\pi} |\chi_I-\psi|^2 \,d\theta \leq \delta.
\end{equation}
\par
Since Gau{\ss} sums are equidistributed in $\Rr/2\pi\Zz$ (as follows
by Weyl's criterion from the formula \eqref{averagesignprim} and
Deligne's bound for hyper-Kloosterman sums), we have in particular
$$
\frac{1}{\vphis(q)} \sums_{\chi\mods q}
|\chi_I(\theta(\ftchi))-\psi(\theta(\ftchi))|^2 \ll \delta.
$$
By the definitions  \eqref{defOmegarevision} and  \eqref{defLrevision},  
the Cauchy-Schwarz inequality provides us with the approximation 
\begin{equation*}
  \mathscr{L}(f;\bfx_L,\chi_I)=\mathscr{L}(f;\bfx_L,\psi)+
  O\left(\delta^{1/2}  \mcQ(f;\bfx_L)^{1/2}\right).
\end{equation*}

\section{Choosing the mollifier}\label{sec-mollif}

As customary in the mollification method, we choose the function
$M(\ftchi,\bfx_L)$ to be a suitable Dirichlet polynomial of length
$L=q^\lambda$, for some sufficiently small absolute positive constant
$\lambda\leq 1/2$. Given some complex tuple $\bfx_L=(x_\ell)_{\ell\leq L}$,
we set
\begin{equation}\label{defM}
  M(\ftchi,\bfx_L)=\sum_{\l\leq L}x_\l\frac{\chi(\l)}{\l^{1/2}}.
\end{equation}
We assume throughout that $x_{\ell} = 0$ unless $(\ell, r) = 1$. 
We recall that $(\mu_f(n))$ denotes the convolution inverse of the
Hecke eigenvalues $(\lf(n))$. We consider coefficients
$(x_\l)_{\l\leq L}$ of the shape
\begin{equation}\label{xldef}
  x_\l= \mu_f(\l) 
  P\Bigl(\frac{\log (L/\l)}{\log L}\Bigr)\, \delta_{ \l\leq L } \delta_{ (\ell, r) = 1},	
\end{equation}
where $P:[0,1]\ra\Cc$ is a real-valued polynomial satisfying $P(1)=1$,
$P(0)=0$.

In particular,  as a consequence of \eqref{RSbis} and \eqref{KSbound}, we have the inequalities
\begin{equation*}
|x_\ell|\ll_\eps 
 \l^{\theta+\eps} 
\end{equation*}
for any $\eps>0$, and
\begin{equation}\label{averagexl}
\sum_{\l \leq L} |x_\l|^2 \ll_{f} L.
\end{equation}
 Indeed, to prove \eqref{averagexl} we first notice that the definition of the multiplicative function $n \mapsto \mu_f (n)$ given in \eqref{defmurevision} implies that for $p \nmid r$, we have the equalities
 $$
  \vert \mu_f (p)\vert =\vert \lambda_f (p)\vert, \ \mu_f (p^2) =1, \ \mu_f (p^k)=0 \ (k\geq 3).
 $$
 In the sum studied in \eqref{averagexl}, we factorize each $\l$ as $\l = \l_1 \l_2$ where $\l_1$ and $\l_2$ are coprime, $\l_1$ is squarefull and $\l_2$ is squarefree.
 We can now write the inequalities
 \begin{align*}
 \sum_{\l \leq L} \vert x_{\l}\vert^2& \ll \sum_{\l_1\leq L\atop (\l_1,r) =1} \vert \mu_f (\l_1)\vert^2 \sum_{\l_2 \leq L/\l_1\atop (\l_2, r) =1}\vert  \mu_f (\l_2)\vert^2\\
 & \ll   \sum_{\l_1\leq L\atop (\l_1,r) =1}   \sum_{\l_2 \leq L/\l_1\atop (\l_2, r) =1}\vert  \lambda_f (\l_2)\vert^2\\
 &\ll   \sum_{\l_1\leq L\atop (\l_1,r) =1}  (L/\l_1) \ll L,
 \end{align*}
 by appealing to  \eqref{RSbis}. 
 
Next we have 
\begin{lemma}\label{lemma81} We have
$$\frac1{\vphis(q)}\sums_{\chi\mods q}|M(\ftchi,\bfx_L)|\ll (\log q)^{1/2}.$$
\end{lemma}
\begin{proof} By the Cauchy-Schwarz inequality we have
$$\sums_{\chi\mods q}|M(\ftchi,\bfx_L)|\leq \vphis(q)^{1/2}\Bigl(\sums_{\chi\mods q}|M(\ftchi,\bfx_L)|^2\Bigr)^{1/2}$$
and, as a consequence of \eqref{averagexl},  we have
\begin{align*}
  \frac1{\vphis(q)}\sums_{\chi\mods q}|M(\ftchi,\bfx_L)|^2 &\leq
  \frac{\vphi(q)}{\vphis(q)}\sum_{\l\equiv \l'\mods q}\frac{|x_\l
    x_{\l}'|}{ (\l\l')^{1/2}} \\
&=\frac{\vphi(q)}{\vphis(q)} \sum_{\l\leq
    L}\frac{|x_\l|^2}{ \l}\ll \log L
\end{align*}
since $L< q$.
\end{proof}

We conclude from \eqref{CS}, \eqref{8.2}  and Lemma \ref{lemma81} that
\begin{multline}\label{eq-cs-nonva}
  \frac{1}{\vphis(q)} |\{\chi\mods q\,\, \text{non-trivial}\mid\,
  |L(\ftchi,1/2)|\geq (\log q)^{-1}, \theta(\ftchi)\in I\}| \\
  \geq\frac{ \bigl\vert \mathscr{L}(f;\bfx_L,\psi_{I,\delta})+O_{f}( (\log
      q)^{-1/2} + \delta^{1/2} \mcQ(f;\bfx_L)^{1/2})\bigr\vert^2}{\mcQ(f,
    \bfx_L)}.
\end{multline}

\section{Computation of the first mollified
  moment} \label{sec-m1-moll}

In this section we evaluate $\mathscr{L}(f;\bfx_L,\psi_{I,\delta})$. Since
\begin{equation}\label{eq-square-angle}
\exp (2i\theta(\ftchi))=\eps(f)\chi( r){\eps^2_\chi}
\end{equation}
by \eqref{phaseformula}, we have
\begin{equation}\label{abovesum}
\mathscr{L}(f;\bfx_L,\psi)=\sum_{|k|\leq k_\delta}\what\psi(2k)\eps(f)^k\sum_{\ell\leq L}\frac{x_\ell}{\ell^{1/2}}\mcL(f; r^k\ell,2k)
\end{equation}
where $\mcL(f;r^k\ell,2k)$ is the first moment  defined in \eqref{firstmom}.

\begin{remark}
  It is at this point that our restriction to intervals modulo $\pi$
  instead of modulo $2\pi$ intervenes: because of the factor $2$ on
  the left-hand side of~(\ref{eq-square-angle}), we are not able to
  evaluate a first moment of $L(f\otimes\chi,\demi)$ twisted by
  $e(k\theta(\ftchi))$ for $k$ odd.
\end{remark}

 By \eqref{reflect}, Theorem \ref{thm-twistedfirstmoment} and \eqref{averagexl}, the total contribution, denoted by $S_{k\not= 0, -1}$,  of the terms $k\not=0,-1$ to the  sum \eqref{abovesum}  satisfies
\begin{equation}\label{Sknot=0-1<<}
S_{k\not= 0, -1}\ll_{\delta,I,f,\eps}q^{\eps+\lambda/2-1/8}.
\end{equation}
In particular, this contribution is negligibly small if $\lambda<1/4$,
which we assume from now on. 
 
 The contribution  $S_{k=0} $ of the term $k=0$ to \eqref{abovesum}  is equal to 
 $$ S_{k=0}:=
 \what\psi (0) \sum_{\ell \leq L} \frac{x_\ell}{\ell^{1/2}} \mcL(f;\ell,0)= 
  \what\psi (0)\sum_{\ell \leq L} \frac{x_\ell}{\ell^{1/2} }\Bigl\{ \frac{ \lambda_f (\bar{\ell}_q)}{\bar {\ell}_q^{1/2}} +O_{f, \varepsilon } (q^{\varepsilon -1/8})\Bigr\},
$$
 by Theorem \ref{thm-twistedfirstmoment}. We treat separately the cases $\ell =1$ and $2 \leq \ell \leq L$. The first case contributes by 
\begin{equation}\label{bakershop}
\what\psi(0)x_1+O(q^{-1/8+\eps})=\mu(I)+O_{f, \varepsilon}(\delta^{1/2}+q^{\eps-1/8}),
\end{equation}
by \eqref{diffofmeasures}.

To deal with the cases $2\leq \ell \leq L$ we exploit the inequalities
$$
 \vert \lambda_f (\bar{\ell}_q)\bar {\ell}_q^{-1/2} \vert \leq \bar {\ell}_q^{\theta-1/2}\leq (q/\ell)^{\theta-1/2},
 $$
since $\ell$ satisfies $1<\ell <q^{1/2}$. From this we deduce that the contribution of the $2\leq \ell \leq L$ satisfies
$$
\ll  q^{\theta -1/2} L^{1-\theta} +q^{\varepsilon -1/8} L^{1/2}\ll  q^{\varepsilon -1/8} L^{1/2}.
$$
Gathering with \eqref{bakershop}, we proved the equality
\begin{equation}\label{Sk=0<<}
S_{k=0} =  \mu(I)+O_{f, \varepsilon}( \delta^{1/2}+ q^{\eps +\lambda/2 -1/8})
\end{equation}
\par
Next, by Corollary \ref{cor-k=-2}, if $q$ is large enough (depending
on the level $r$), the contribution $S_{k=-1}$ of the term $k=-1$
equals
\begin{multline} \label{defSk=-1}
S_{k=-1}=   \what\psi(-2)\sum_{\l\leq L}\frac{x_\l\lf( \l_q )}{
    \ell^{1/2}\l_q^{1/2} }
  +O_{\delta,I,f,\eps}(q^{\eps+\lambda/2-1/8}) \\
  =\what\psi(-2)\sum_{\substack{\l\leq
      L\\ (\ell, r) = 1}}\frac{\mu_f(\l)\lf( \l_q )}{ \ell^{1/2} \l_q^{1/2}
  } P\Bigl(\frac{\log (L/\l)}{\log
    L}\Bigr)+O_{\delta,I,f,\eps}(q^{\eps+\lambda/2-1/8})
\end{multline}
where as before $ \ell_q$ denotes the unique integer in $[1, q]$
representing the congruence class $ \ell$ (mod $q$). Since
$1\leq \ell\leq L<q$, we have $\ell_q=\ell$.

We now use the following lemma, which is stated in slightly greater
generality than needed here, for later reference in
Section~\ref{sec-mellin}.

\begin{lemma}\label{lm-decay-mollif}
  Let $a \geq 1$ be an integer.  There is some constant $c>0$,
  depending only on $f$, such that
$$
\sum_{\substack{\l\leq L\\ (\ell, r) =1}}
\frac{x_{\ell}}{\sqrt{\ell}}\frac{\lf(a\ell)}{\sqrt{a\ell}}
= \sum_{\substack{\l\leq L\\
    (\ell, r) =
    1}}\frac{\mu_f(\l)\lf(a\l)}{a^{1/2}\l}P\Bigl(\frac{\log
  (L/\l)}{\log L}\Bigr)\ll_{f} \exp(-c \sqrt{ \lambda\, \log q})
$$
uniformly in $a$. 
\end{lemma}

\begin{proof} 
We can write
$$
\sum_{(\ell,r)=1}\mu_{f}(\ell)\lambda_f(a\ell)\ell^{-s} =
  \sum_{\substack{d\mid a\\ (d, r) = 1}}\mu(d)\lambda_f(a/d)d^{-s}\sum_{(\ell,r)=1}
  \lambda_f(\ell)\mu_f(d\ell)\ell^{-s}.
$$
In turn, since $\mu_f$ is supported on cubefree numbers, for $d$
squarefree and coprime to $r$ we have
$$
\sum_{(\ell,r)=1} \lambda_f(\ell)\mu_f(d\ell)\ell^{-s}= \prod_{p\nmid
  dr}\Bigl( 1-\frac{\lambda_f(p)^2}{p^{s}}+
\frac{\lambda_f(p^2)}{p^{2s}}\Bigr) \prod_{p \mid
  d, \, p \nmid r}(-\lambda_f(p))\Bigl(1-\frac{1}{p^s}\Bigr).
$$
We can therefore write the equality
$$
\sum_{(\ell,r)=1} \lambda_f(\ell)\mu_f(d\ell)\ell^{-s}=
\frac{H_d(s)}{T(s)}
$$
where $T$ is defined in \eqref{defT(s)} and where $H_d(s)$ is an Euler product that converges absolutely for
$\Reel(s)>3/4$ and is bounded by $d^{\theta}$. This gives analytic
continuation of the Dirichlet series
$$
\sum_{(\ell,r)=1} \lambda_f(\ell)\mu_f(d\ell)\ell^{-s}
$$
in the zero-free region of the Rankin-Selberg $L$-function
(Proposition~\ref{AutomorphicPNTLemma}), and moreover the value of
this function at $s=1$ is zero. We then obtain the result  of Lemma \ref{lm-decay-mollif} by a
standard contour integration and by a partial summation. 
\end{proof}
Returning to \eqref{defSk=-1}, we deduce that $S_{k=-1} $ satisfies the bound
$$
S_{k=-1} \ll \exp (-c \sqrt{ \log q}).
$$ 
Gathering this bound with  \eqref{Sknot=0-1<<} and \eqref{Sk=0<<} and supposing that $0< \lambda <1/4$, we deduce the equality
\begin{equation}\label{L=MT+ET}
\mathscr{L}(f;\bfx_L,\psi) = \mu (I) + O \bigl(\delta^{1/2} + \exp( - c \sqrt {\log q})\, \bigr).
\end{equation}

We obtain now the lower bound

\begin{multline}\label{stairway}
  \frac{1}{\vphis(q)} |\{\chi\mods q\,\, \text{non-trivial}\mid\, |L(\ftchi,1/2)|\geq (\log
  q)^{-1},\
  \theta(\ftchi)\in I\}|\\
  \geq \frac{\mu(I)^2+O_{f,\delta,I}\left((\log
      q)^{-1/2}\right)+O_f\left(\delta^{1/2}(1+ \mcQ(f;\bfx_L))\right)
  }{\mcQ(f;\bfx_L)}.
\end{multline}
using~(\ref{eq-cs-nonva}). It remains to evaluate $\mcQ(f;\bfx_L)$.

\section{Computation of the second mollified moment}\label{moll-sec-mom}

In this section we compute
\begin{align*}
  \mcQ(f;\bfx_L) 
  &= 
    \sum_{ \ell, \ell' \leq L  } \frac{x_{\ell}
    \overline{x_{\ell'}}}{(\ell \ell')^{1/2}} \mcQ(f,f,
    1/2, \ell, \ell') \\
  &=  \sum_d \sum_{(\ell, \ell' ) = 1} \frac{x_{d\ell}
    \overline{x_{d\ell'}}}{d(\ell \ell')^{1/2}} \mcQ(f,f,1/2;\ell,
    \ell'),
\end{align*}
which is enough for our purpose, since
$\mcQ(f,f,1/2;d\ell, d\ell')=\mcQ(f,f,1/2;\ell, \ell')$ (recall that
the twisted second moment is defined in~\eqref{eq-twistedsecond}).

  Proposition \ref{pr-mt} is not sufficient for our purpose since we will sum over $\ell$ and $\ell'$. So we use Theorem~\ref{thsecondmoment} which evaluates $\mcQ(f,f,1/2; \ell, \ell')$
for $(\ell, \ell') = 1$ with two main terms given in
\eqref{MTsecondmoment}. Since $\eps(f,f,\pm,1/2)=1$ and
$\rho = \rho' = 1$, $\delta = |r| = |r'|$, we obtain by \eqref{sun1},
\eqref{sun2}, \eqref{sun3}, \eqref{sun4} that
\begin{multline}\label{errorterm}
\mcQ(f;\bfx_L)=\intc_{(2)}\Bigl(\frac{L_\infty(f,1/2+u)^2}{L_\infty(f,1/2)^2}+\frac{L_\infty(f,3/2+u)^2}{L_\infty(f,3/2)^2}\Bigr)G(u)\\
 \times \Bigl( \sum_{d, (\l,\l') = 1}\frac{x_{d\l}\ov{x_{d\l'}}}{d(\l\l')^{1/2}}L(f\times f,1,u;\l,\l')\Bigr)
   (|r|q)^{2u} \frac{du}{u}	
 +O\left(L^{\expoLt}q^{-\expoq+\eps}\right),
\end{multline}
where
$$
L(f\times f,1,u;\l,\l')=\sum_{n\geq 1}\frac{\lambda_f(\l
  n)\lf(\l'n)}{(\l\l'n^2)^{1/2+u}},
$$
(see \eqref{Lfgdef}).  We apply Mellin inversion to the sum over
$\ell$ and $\ell$'. For any polynomial 
$$
Q(X) = \sum_{k} a_k X^k,
$$
and any $L >1$,  we introduce the polynomial $\what{Q_L} (v)$ defined by 
\begin{equation}\label{defhatQM}
\frac{L^v}{v} \what{Q_L}(v)= \int_0^L Q\left(\frac{\log(L/x)}{\log
    L}\right) x^{v-1} dx = \sum_k a_k \frac{k! L^v}{v^{k+1} (\log L)^k
}.
\end{equation}
With this notation, the main term of $\mcQ(f, \bfx_L)$ equals
\begin{multline}\label{tripleintegral}
  \frac{1}{(2\pi i)^3}\int_{(2)}\int_{(2)}\int_{(2)}
  \Bigl(\frac{L_\infty(f,1/2+u)^2}{L_\infty(f,1/2)^2}+\frac{L_\infty(f,3/2+u)^2}{L_\infty(f,3/2)^2}\Bigr)\\\times
  G(u)L(f,u,v,w)\what{P}_L(v)\what{\ov
    P}_L(w)L^{v+w}(|r|q)^{2u}\frac{du}{u}\frac{dv}v\frac{dw}w
\end{multline}
where
\begin{equation} \label{defL(fuvw)}
  L(f, u,v,w)=\sumsum_\stacksum{d,\l,\l',n}{(\l,\l')= (d\ell\ell', r) = 1}\frac{\mu_f(d\l)\lf(\l
    n)\mu_f(d\l')\lf(\l'n)}
  {\l^{1+u+v}{\l'}^{1+u+w}{d}^{1+v+w}n^{1+2u}}=L(\demi,\demi,\demi,u,v,w),
\end{equation}
in terms of the auxiliary function introduced in \eqref{bigLdef}.

Recall that $T(s)$ is defined in \refs{defT(s)} as
$
T(s)=L(f\otimes f,s).
$
From Corollary~\ref{cor-factor}, we obtain the meromorphic
continuation of this function to the domain
$$\Re u,\Re v,\Re w> -\eta$$
for some $\eta > 0$, given by
\begin{align*}
  L(f, u,v,w) &=\frac{T(1+2u)T(1+v+w)}{T(1+u+v)T(1+u+w)}D(u,v,w) \\
              &=\eta_3(u,v,w)\frac{(u+v)(u+w)}{u(v+w)},
\end{align*}
where $D$ is an Euler product absolutely convergent for
$\Re u,\Re v,\Re w\geq -\eta$ and $\eta_3$ is holomorphic and
non-vanishing in a neighborhood of $(u, v, w) = (0,0,0)$.
  
We shift the $v,w$-contours and then $u$-contour to the left of
$u=v=w=0$, using again the standard zero-free regions for
Rankin-Selberg $L$-functions (Proposition~\ref{AutomorphicPNTLemma})
together with the rapid decay of Gamma-quotients.  In this way we see
that the  triple integral in \eqref{tripleintegral} equals
\begin{equation}\label{res}
  2 \eta_3( 0, 0, 0)
  \res_{u=v=w=0}
  \what{P_L}(v)\what{P_L}(w)L^{v+w}(|r|q)^{2u}\frac{(u+v)(u+w)}{u^2(v+w)vw} + O\left(\frac{1}{\log L}\right).
  \end{equation}
  We write
  \begin{equation}\label{fractiondecomp}
  \frac{(u+v)(u+w)}{u^2(v+w)vw}=\frac{1}{(v+w)vw}+\frac{1}{uvw}+\frac{1}{u^2(v+w)}.	
\end{equation}
Our plan is now to compute the residue coming from each of the three
terms on the right--hand side of \eqref{fractiondecomp}.  For this
purpose, we gather in one lemma the contents of
\cite{KMVcrelle}*{Lemma 9.1-Corollary 9.4}.  We have

\begin{lemma}\label{CrelleKMV}
  For $M >1$, $P$ and $Q$ polynomials, let $\what{P_M}$ and
  $\what{Q_M}$ be defined by \eqref{defhatQM}. We then have the
  equalities\footnote{We take this opportunity to mention a misprint
    in the statement of \cite{KMVcrelle}*{Lemma 9.4} (in that paper
    the formula was used in its correct form): the formula should read
  $$
  \res_{s_1,s_2=0}\frac{M^{s_1+s_2}\what{P_M}(s_1)\what{Q_M}(s_2)}{s_1s_2(s_1+s_2)}=\Bigl(\int_0^1
  P(x)Q(x)dx\Bigr)(\log M).
$$}
  $$
  \res_{s=0} \frac{M^s \what{Q_M }(s)} {s} = Q(1),
  $$
  and
  $$
  \res_{s_1=s_2=0} \frac{M^{s_1+s_2}\what{Q_M} (s_1) \what{Q_M}
    (s_2)}{s_1s_2 (s_1+s_2)} =(\log M) \Bigl( \int_0^1 P(x) Q(x) \,
  dx\Bigr).
  $$
\end{lemma}

Up to terms of size $O((\log L)^{-1} )$, the contribution from the
first term of the right side of \eqref{fractiondecomp} is zero while
the contribution from the second term is equal to
$2\eta_3(0,0,0)P(1)^2 = 2\eta_3( 0, 0, 0)$, as a consequence of the
first part of Lemma \ref{CrelleKMV}.

To deal with the residue coming from the third term, we first note the equality (recall that $P(0) =a_0=0$) 
\begin{align*}
vw\what{P_L}(v)\what{P_L}(w)&=\frac{1}{\log^2 L}\Bigl(\sum_{k\geq 1}ka_k\frac{(k-1)!}{(v\log L)^{k-1}}\Bigr)\,\Bigl(\sum_{k\geq 1}ka_k\frac{(k-1)!}{(w\log L)^{k-1}}\Bigr)\\
&=\frac{\what{P'_L}(v)\what{P'_L}(w)}{\log^2 L},
\end{align*}
The second part of Lemma \ref{CrelleKMV}  implies that we have 
\begin{multline*}
  2\eta_3( 0, 0, 0)\res_{u=v=w=0} \frac{\what{P_L}(v)\what{P_L}(w)
    L^{v+w}(|r|q)^{2u}}{u^2(v+w)}\\
  =2\eta_3(0,0,0)\frac{2\log q}{\log L}\Bigl(\int_0^1
  P'(x)^2dx\Bigr)+O\Bigl(\frac{1}{\log L}\Bigr).
\end{multline*}  
Altogether, \eqref{res} equals
$$
2\eta_3(0,0,0)\Bigl(P(1)^2+\frac{2\log q}{\log
  L}\int_0^1P'(x)^2dx\Bigr)+O\Bigl(\frac{1}{\log L}\Bigr).
$$
Taking $L=q^\lambda$ with $0<\lambda<2/5\times\expoq=\expoLmax$ to
deal with the error term in \eqref{errorterm} and $P=X$, it remains to recall \eqref{stairway} to  obtain:

\begin{proposition}\label{prop-moment12} Let $\bfx_L$ be defined as
  above with $P(X)=X$, let $0<\lambda< \expoLmax$ be fixed. For any
  $\delta>0$, we have
$$ 
\mcQ(f;\bfx_L)=2\eta_3(0,0,0)(1+2\lambda^{-1})+O(\log^{-1}q)
$$
and
\begin{multline*}
  \frac{1}{\vphis(q)} |\{\chi\mods q\,\, \text{{\rm
      non-trivial}}\mid\, |L(\ftchi,1/2)|\geq (\log
  q)^{-1}, \theta(\ftchi)\in I\}|\\\
  \geq
  \frac1{2\eta_3(0,0,0)}\cdot \frac{\mu(I)^2}{1+2\lambda^{-1}}+O_{f,\delta,I}\Bigl(\frac{1}{\log^{1/2}
    q}\Bigr)+O_f(\delta^{1/2}).
\end{multline*}
\end{proposition}

To conclude the proof of Theorem \ref{thmnonvanishing+angle}, it
remains to observe that 
\begin{equation}\label{eta3<}
\eta_3(0,0,0)\leq \zeta(2).
\end{equation}  Indeed, this
follows  from the  factorization
$$
\eta_3 (0,0,0) = \prod_p L_p (f,0,0,0)
$$
where  the local factors $L_{p}(f,u,v,w)$ of  $L(f,u,v,w)$  are defined in \eqref{defL(fuvw)}  and satisfy
\begin{equation}\label{Lp<1-}
L_p(f,0, 0, 0)   \leq (1 - p^{-2})^{-1}\text{ if } p\mid r,
\end{equation}
and 
\begin{equation}\label{Lp=1}L_p(f,0,0,0)=1\text{ if } p\nmid r.
\end{equation}  
To prove \eqref{Lp<1-} and \eqref{Lp=1} we will use the following identities where $p$ is arbitrary:
\begin{gather*}
  \lf(p)=\alpha+\beta,\text{ where } \alpha\beta=\chi_r(p)\\
  \lf(p^k)=\alpha^k+\alpha^{k-1}\beta+\cdots+
  \alpha\beta^{k-1}+\beta^k
  \\
  \mu_f(1)=1,\ \mu_f(p)=-\lf(p),\ \mu_f(p^2)=\chi_r (p),\ \mu_f(p^k)=0,\ k\geq
  3.
\end{gather*}
In the case of \eqref{Lp<1-}  we see that the definition \eqref{defL(fuvw)} and the condition $p \mid r$  imply the equality
$$
L_p (f,0, 0, 0) =   \sum_{n \mid p^\infty} \frac{\lf^2(n)}{n} = \sum_{k \geq 0} \frac {\lf^{2k}(p)}{p^k} =( 1 - \lf^2 (p)/p)^{-1}.
$$
It remains to appeal to  \eqref{ramified} to complete the proof of \eqref{Lp<1-}.

The proof of \eqref{Lp=1} requires more attention.  When $p \nmid r$ we write the local factor as  

\begin{align}\label{tediousformula}
  L_{p}(f,0,0,0)&=\sumsum_\stacksum{d,\l,\l',n|p^\infty}{(\l,\l')=1}\frac{\mu_f(d \l)\lf( \l n)\mu_f(d \l')\lf( \l'n)}{\l{\l'}{d} n}\nonumber\\
              &=\sumsum_\stacksum{\delta,\lambda,\lambda',\nu\geq 0}{\lambda.\lambda'=0}\frac{\mu_f(p^{\delta+\lambda})\lf(p^{\lambda+\nu})\mu_f(p^{\delta+\lambda'})\lf(p^{\lambda'+\nu})}{p^{\delta+\lambda+\lambda'+\nu}}\nonumber\\
              &=2\sumsum_{\delta,\lambda,\nu\geq 0}\frac{\mu_f(p^{\delta+\lambda})\lf(p^{\lambda+\nu})\mu_f(p^{\delta})\lf(p^{\nu})}{p^{\delta+\lambda+\nu}}-
                \sumsum_{\delta,\nu\geq
                0}\frac{\mu_f(p^{\delta})^2\lf(p^{\nu})^2}{p^{\delta+\nu}}.
\end{align}
As a consequence of the vanishing of the function $\mu_f$, the first  multiple sum can be restricted to the six subcases
$$
(\delta, \lambda) = (0,0),\, (0,1),\, (0,2), \, (1,0),\, (1,1),\, (2,0),
$$
and the second one to the three subcases
$$
\delta= 0,\, 1, \, 2.
$$
For instance, the contribution of the terms with $(\delta, \lambda) = (0,0)$ to the first multiple sum on the right--hand side of \eqref{tediousformula} can be expressed as
$$
2 \sum_{\nu \geq 0}\frac 1{p^\nu}\,  \Bigl( \frac{\alpha^{\nu +1} -\beta^{\nu+1}}{\alpha -\beta}\Bigr)^2 =\frac 2{(\alpha-\beta)^2} \Bigl( \frac {\alpha^2}{1-\alpha^2/p} +\frac {\beta^2}{1-\beta^2/p} -
\frac 2{1-1/p}
\Bigr).
$$
The other contributions are computed similarly. By straightforward computations (most easily performed by computer--assisted symbolic calculations)   we obtain \eqref{Lp=1}. 
 This completes the proof of Theorem \ref{thm61}.
 
 \section{Improvement of Theorem \ref{thmnonvanishing+angle}}\label{improvedthm}
 We quickly explain Remark \ref{remark19} (1) which asserts that the lower bound $(\log q)^{-1}$ in the statement of Theorem \ref{thmnonvanishing+angle} and Theorem  \ref{thm61} can be improved 
 in $(\log q)^{-\alpha}$ for every   $\alpha >1/2$. Indeed if, in the left--hand side of  \eqref{CS},
 we replace the condition $|L(\ftchi,1/2)|\geq (\log q)^{-1}$ by $|L(\ftchi,1/2)|\geq (\log q)^{-\alpha}$, then
 the coefficient $1/(q\log q)$ in the error term of the right--hand side of the  equality \eqref{8.2}
 has to be replaced by $1/(q \log^\alpha q)$. A direct application of Lemma \ref{lemma81} bounds this error term by $O ((\log q)^{1/2-\alpha})$. In order to make the method work, we only require to this error term  to be negligible when compared with the main term  $ \mathscr{L}(f;\bfx_L,\chi_I)$ of \eqref{8.2}   as $q$ tends to infinity.  The order of magnitude of  $\mathscr{L}(f;\bfx_L,\chi_I)$ is known by  the equality \eqref{L=MT+ET} and we are led to  the sufficient condition $\alpha >1/2$.
\section{Non-vanishing with Mellin constraints}\label{sec-mellin}

The goal of this section is to generalize the positive proportion of
non-vanishing to incorporate certain conditions on $\chi$, which are
roughly of the form
$$
\widetilde{t}(\chi)\in A,
$$
where $A\subset \Cc$ and 
$$
\widetilde{t}(\chi)=\frac{1}{\sqrt{q}}\sum_{x\in\Fqt}\chi(x)t(x)
$$
is the discrete Mellin transform of some suitable function
$t\colon \Fqt\to \Cc$. The functions $t$ that we can handle are some
of the trace functions over $\Fq$ described in
Section~\ref{sec-trace}.

Let $C\geq 1$ be a real number and $K$ be a compact Lie group. For
each prime $q$ (large enough), fix a Mellin sheaf $\mcF_q$ over $\Fq$
as in Definition~\ref{def-sheaf} with conductor $\leq C$, with
Property EAGM and with equidistribution group $K$ (see
Definition~\ref{def-eagm}). 
Examples of such families, with $K=\SU_2(\Cc)$ and $C=5$, are provided
by the sheaves related to Evans or Rosenzweig-Rudnick sums, see
Example~\ref{ex-ex}.
\par
We denote by $X_q$ the set of exceptional characters modulo $q$ as
described in Section~\ref{sec-equi-mellin}; we recall that its size is
bounded independently of $q$. For a Dirichlet character
$\chi\notin X_q$, we denote by $\theta_{q,\chi} \in K^{\sharp}$ (or simply
$\theta_{\chi}$) the conjugacy class associated to the Mellin
transform of $\mcF_{q}$ at $\chi$.

\begin{theorem}\label{nonvanishing+anglegeneral}
  With assumptions as above, let $A\subset K^{\sharp}$ be a measurable
  set with non-empty interior. Then
$$
\liminf_{q\to +\infty} \frac{1}{\vphis(q)} |\{\chi\notin X_q\,\mid\,
|L(f\otimes\chi,\demi)|\geq (\log q)^{-1}\text{ and } \theta_{\chi}\in
A\}|>0.
$$
\end{theorem}

\begin{proof}
  Let $d\geq 1$ be the order of the finite group of finite-order
  characters of $K$.  Let $\phi_0\colon K^{\sharp}\to [0,1]$ be a
  non-zero continuous function supported in an open subset contained
  in $A$. Let further $\phi$ be a finite linear combination of
  characters of irreducible representations of $K$ such that
  $\|\phi-\phi_0\|<\delta$, where $\delta>0$ will be specified later;
  such a function exists by the Peter-Weyl Theorem.
\par
Fix $q$ so that $\mcF=\mcF_q$ is defined. Let $L=q^{\lambda}$ with
\begin{equation}\label{eq-max-lambda}
0<\lambda<\min\Bigl(\frac{1/2-\theta}{2d},\frac{1}{360}
\Bigr)
\end{equation}
and consider the mollifier
$$
M(f\otimes\chi,\bfx_{L})=\sum_{\ell \leq
  L}\frac{x_{\ell}}{\sqrt{\ell}}\chi(\ell)
$$
as in Section~\ref{sec-mollif} (see~(\ref{defM}) and~(\ref{xldef})).
Let
$$
\mcL=\frac{1}{\vphis(q)} \sum_{\chi\notin
  X_q}\phi_0(\theta_{\chi})L(f\otimes \chi,\demi)
M(f\otimes\chi,\bfx_L).
$$
We then have
$$
\mcL=\mcL'+O((\log q)^{-1/2})
$$
by Lemma \ref{lemma81}, where
$$
\mcL'=\frac{1}{\vphis(q)} \sum_{\substack{\chi\notin X_q\\ |L(f\otimes
    \chi,1/2)|\geq (\log q)^{-1}}}\phi_0(\theta_{\chi})L(f\otimes
\chi,\demi) M(f\otimes\chi,\bfx_L).
$$
On the other hand, we have
$$
|\mcL'|^2 \leq \mathcal{N}\mcQ,
$$
where
$$
\mathcal{N}=\frac{1}{\vphis(q)} |\{\chi\notin X_q\,\mid\,
|L(f\otimes\chi,\demi)|\geq (\log q)^{-1}\text{ and } \theta_{\chi}\in
A\}
$$
(since $\phi_0(\theta_{\chi})\not=0$ implies that
$\theta_{\chi}\in A$) and
\begin{align*}
  \mcQ&=\frac{1}{\vphis(q)} \sum_{\chi\notin X_q}
        |\phi_0(\theta_{\chi})|^2|L(f\otimes \chi,\demi)|^2
        |M(f\otimes\chi,\bfx_L)|^2\\
      &\leq \frac{1}{\vphis(q)} \sums_{\chi\mods
        q}|L(f\otimes \chi,\demi)|^2 |M(f\otimes\chi,\bfx_L)|^2\ll 1
\end{align*}
by Proposition~\ref{prop-moment12} since $\lambda<\expoLmax$.
\par
Hence it suffices to find a lower  bound for $\mcL'$. 
Let
$$
\mcL''=\frac{1}{\vphis(q)} \sum_{\chi\notin
  X_q}\phi(\theta_{\chi})L(f\otimes \chi,\demi)
M(f\otimes\chi,\bfx_L).
$$
Then 
$$
|\mcL''-\mcL|\leq \|\phi-\phi_0\|_{\infty} \times \frac{1}{\vphis(q)}
\sums_{\chi\mods q}|L(f\otimes \chi,\demi) M(f\otimes\chi,\bfx_L)| \ll
\|\phi-\phi_0\|_{\infty}
$$
by Proposition~\ref{prop-moment12} again. Write
\begin{equation}\label{decomp100}
\phi(x)=\int_K \phi+\sum_{\pi\not=1} \what{\phi}(\pi)\Tr(\pi(x))
\end{equation}
where the sum ranges over a finite set of non-trivial irreducible
representations of $K$. Then, if $\delta$ is small enough, we have
$$
\what{\phi}(1)=\int_K\phi>0.
$$
The equality \eqref{decomp100}  decomposes $\mcL''$ into 
\begin{equation}\label{decomp101}
\mcL''= \mcL''_{\rm MT} + \mcL''_{\rm ET},
\end{equation}
where the main term is given by
\begin{align*}
\mcL''_{\rm MT} & = \frac{\what{\phi }(1)}{\vphis (q)}   \sum_{\chi \not\in X_q} L (\ftchi, \demi)  M(\ftchi, \bfx_L)\\
& = \frac{\what{\phi }(1)}{\vphis (q)}   \sum_{\chi \bmod q} L (\ftchi, \demi)  M(\ftchi, \bfx_L) + O(q^{\lambda/2-1/2+\varepsilon}).
\end{align*}
In the line above,  the error term is deduced from the fact that $\vert X_q\vert =O(1)$, from the classical bound of $L (\ftchi, \demi) $ and from \eqref{averagexl}. Finally  
by the computation in Section~\ref{sec-m1-moll} with $k=0$, we obtain the equality
$$
\mcL''_{\rm MT} =  {\what{\phi }(1)}   \bigl(1+o_\lambda (1)\bigr),
$$
as $q$ tends to infinity provided that $\lambda$ satisfies \eqref{eq-max-lambda}.

Returning to \eqref{decomp101}, we have the equality 
\begin{equation}\label{L''ETdef}
  \mcL''_{\rm ET} 
  = \frac{1}{\vphis(q)} \sum_{\pi\not=1}\what{\phi}(\pi) 
  \sum_{\chi\notin X_q} \Tr(\pi(\theta_{\chi}))L(f\otimes \chi,\demi)
  M(f\otimes\chi,\bfx_L). 
\end{equation}
 
\par
Fix $\pi\not=1$ in the sum. We have
\begin{multline}\label{eq-sumell}
  \frac{1}{\vphis(q)} \sum_{\chi\notin
    X_q}\Tr(\pi(\theta_{\chi}))L(f\otimes \chi,\demi)
  M(f\otimes\chi,\bfx_L) \\
  =  \frac{1}{\vphis(q)}\sum_{ \ell\leq L}\frac{x_{\ell}}{\ell^{1/2}}
  \sum_{\chi\notin X_q}
  \Tr(\pi(\theta_{\chi}))L(f\otimes \chi,\demi) \chi(\ell).
\end{multline}
  We
recall first that \eqref{averagexl}  implies
\begin{equation}\label{anotherbound}
\sum_{\ell\leq L}\frac{|x_{\ell}|}{\sqrt{\ell}}\ll L^{1/2+\eps}
\end{equation}
for any $\eps>0$ and, for the end of the proof, we distinguish three cases.
\par
\medskip
\par
\textbf{Case 1.} Let $d$ be the order of the finite group of characters
of finite order of $K$. There exists $a$ such that $1\leq a\leq q-1$
and such
that $\Tr(\pi(\theta_{\chi}))=\chi(a)$ for all $\chi\notin X_q$, and
moreover $a$ is a non-trivial $d$-th root of unity modulo $q$ since
$\pi$ is a non-trivial character of finite order of $K$. This
is the ``punctual'' case (2) of Theorem~\ref{th-katz}. 
By Theorem~\ref{thm-twistedfirstmoment}, we have
$$ 
\frac{1}{\vphis(q)} \sum_{\chi\notin X_q}
 L(f\otimes \chi,\demi) \chi(a\ell)
=\frac{\lf((\ov{a\l})_q)}{((\ov{a\l})_q)^{1/2}} +O(q^{-1/8+\eps})
$$
for any $\eps>0$.
\par
In particular that there are at most $d-1$ possible values of $a$.
Let $a$ be such a root of unity. Write $b=\ov{a}_q \in [1, q-1]$ with
the notation as in Theorem \ref{thm-twistedfirstmoment}. Since
$b^d\equiv 1\mods{q}$, and $b\not=1$, we have $b\geq q^{1/d}$.
\par
We write $m=(\ov{\ell})_q=(1+\alpha q)/\ell$ for some $\alpha\geq
0$. We then have 
$$
0\leq \alpha=\frac{\ell m-1}{q}<\frac{\ell m}{q}<\ell.
$$
Write further $b\alpha=\delta\ell+\rho$ where $0\leq \rho<\ell$. Then
$$
bm=b\frac{1+\alpha q}{\ell}=\frac{b+\rho q}{\ell}+\delta q,
$$
and since
$$
0<\frac{b+\rho q}{\ell}< \frac{q}{\ell}+\Bigl(1-\frac{1}{\ell}\Bigr)q=
q,
$$
we get
$$
(\ov{a\ell})_q=(bm)_q=\frac{b+\rho q}{\ell}\geq \frac{b}{\ell}\geq
q^{1/d-\lambda}\geq q^{1/(2d)}.
$$
Therefore the contribution of this representation to the first moment
is
$$
\ll q^{1/(2d)(-1/2+\theta)} \sum_{ \ell\leq L}
\frac{|x_{\ell}|}{\ell^{1/2}}\ll
q^{\lambda/2+1/(2d)(-1/2+\theta)+\eps}\to 0
$$
as $q\to +\infty$ by~(\ref{eq-max-lambda}).
\par
\medskip
\par
If we are not in Case 1, we denote by $\pi(\mcF_q)$ the Mellin sheaf
obtained from Theorem~\ref{th-katz}. Two more cases appear.
\par
\medskip
\par
\textbf{Case 2.} Assume that there exists $a$ such that
$1\leq a\leq q-1$ and $\pi(\mcF_q)$ is geometrically isomorphic to
$[x\mapsto a/x]^* \HYPK_2$, so that $\Tr(\pi(\theta_{\chi}))$ is
proportional to $\eps_{\chi}^{-2}\chi(a)$, with the proportionality
constant of modulus $1$ (see Lemma \ref{lm-convol}).  Then, up to such a constant of modulus $1$,
the sum~(\ref{eq-sumell}) is equal to
$$
\sum_{\ell\leq L}\frac{x_{\ell}}{\ell^{1/2}} \frac{1}{\vphis(q)}
\sum_{\chi\notin X_q} L(f\otimes \chi,\demi)
\eps_{\chi}^{-2}\chi(a\ell)= \sum_{\ell\leq
  L}\frac{x_{\ell}}{\ell^{1/2}} \frac{\lf((a\ell r)_q)}{(a\ell
  r)_q^{1/2}}+O(L^{1/2}q^{-1/8+\eps})
$$
for any $\eps>0$ by Corollary~\ref{cor-k=-2}, where $m=(a\ell r)_q$ is
the representative between $1$ and $q$ of the residue class of
$a\ell r$ modulo $q$.  If there doesn't exist $\ell_0$ such that
$1\leq \ell_0\leq L$ and $(a\ell_0 r)_q\leq L^2$, then, by \eqref{anotherbound},  we get
$$
\sum_{\ell\leq L}\frac{x_{\ell}}{\ell^{1/2}} \frac{\lf((a\ell
  r)_q)}{(a\ell r)_q^{1/2}} \ll \frac{1}{L^{1-\theta}}\sum_{ \ell\leq
  L}\frac{|x_{\ell}|}{\ell^{1/2}} \ll L^{-1/2+\theta}.
$$
If there does exist $\ell_0$ such that $1\leq \ell_0\leq L$ and
$1\leq m_0=(a\ell_0 r)_q\leq L^2$, then we get
$$
a\equiv m_0\ov{\ell_0 r}.
$$
We can write
$$
a=\frac{m_0+\alpha_0 q}{\ell_0 r}
$$
for some $\alpha_0\geq 0$.  If $\alpha_0=0$, then we have
$\ell_0r\mid m_0$ and $a\leq L^2$. Then $(a\ell r)_q=a\ell r$ for all
$\ell\leq L$, if $q$ is large enough. Thus
$$
\sum_{\ell\leq L}\frac{x_{\ell}}{\ell^{1/2}} \frac{\lf((a\ell
  r)_q)}{(a\ell r)_q^{1/2}}= \sum_{ \ell\leq
  L}\frac{x_{\ell}}{\ell^{1/2}} \frac{\lf(a\ell r)}{(a\ell r)^{1/2}}
\ll \exp(-c\sqrt{\log L})
$$
for some $c>0$, by Lemma~\ref{lm-decay-mollif}.
\par
Now assume that $\alpha_0\geq 1$. Let $\ell\leq L$.  Then
$$
a\ell r=\frac{m_0\ell}{\ell_0}+ \frac{\alpha_0 q\ell}{\ell_0}.
$$
If $\ell_0$ divides $\alpha_0\ell$, it follows that
$\ell_0\mid m_0\ell$ and
$$
(a\ell r)_q=\frac{m_0\ell}{\ell_0}.
$$
Otherwise, write $\ell=\beta\ell_0+\delta$ where
$1\leq \delta <\ell_0$. We get
$$
a\ell r= \beta m_0+\frac{m_0\delta}{\ell_0} +\beta\alpha_0 q +
\frac{\alpha_0 \delta q}{\ell_0}= \beta m_0+ \frac{\delta
  (m_0+\alpha_0 q)}{\ell_0}+\alpha_0\beta q.
$$
Now write $\alpha_0\delta=\gamma \ell_0+\rho$ with $0\leq
\rho<\ell_0$. We derive
$$
a\ell r=\beta m_0+ \frac{\delta m_0+\rho q}{\ell_0}+ \gamma
q+\alpha_0\beta q.
$$
We have $\rho\not=0$, since otherwise $\ell_0\mid \alpha_0\ell$. Since
$\rho<\ell_0$, for
$q$ large enough, we have
$$
\beta m_0+ \frac{\delta m_0+\rho q}{\ell_0}\leq
\Bigl(1-\frac{1}{\ell_0}\Bigr)q+O(L^3)<q,
$$ 
and hence
$$
(a\ell r)_q =\beta m_0+ \frac{\delta m_0+\rho q}{\ell_0}\geq   \frac{\delta m_0+\rho q}{\ell_0}\geq  \frac{q}{\ell_0} \geq qL^{-1}
$$
We conclude that
$$
\sum_{\ell\leq L}\frac{x_{\ell}}{\ell^{1/2}} \frac{\lf((a\ell
  r)_q)}{(a\ell r)_q^{1/2}} = \sum_{\substack{ \ell\leq
    L\\\ell_0\mid \alpha_0\ell}} \frac{x_{\ell}}{\ell^{1/2}}
\frac{\lf(m_0\ell/\ell_0)}{(m_0\ell/\ell_0)^{1/2}}+ O\Bigl(
\frac{L^{1-\theta+\eps}}{q^{1/2-\theta}} \Bigr).
$$
Write $\alpha_0=\alpha_1\alpha_2$ where
$\alpha_1\mid\ell_0^{\infty}$. Define
$\ell_1=\ell_0/(\ell_0,\alpha_1)$. Then $\ell_0\mid \alpha_0\ell$ if
and only if $\ell_1\mid \ell$.  Moreover, since this condition holds
for $\ell=\ell_1$, we have $\ell_0\mid m_0\ell_1$, which implies that
$(\ell_0,\alpha_1)\mid m_0$. Let $m_1=m_0/(\ell_0,\alpha_1)$.
Then, by applying Lemma \ref{lm-decay-mollif}, we have 
\begin{align*}
  \sum_{\ell\leq L}\frac{x_{\ell}}{\ell^{1/2}} \frac{\lf((a\ell
  r)_q)}{(a\ell r)_q^{1/2}}  
  &= \sum_{ \ell\leq
    L/\ell_1}\frac{x_{\ell\ell_1}}{(\ell\ell_1)^{1/2}}
    \frac{\lf(m_1\ell)}{(m_1\ell )^{1/2}} + O\Bigl(
    \frac{L^{1-\theta+\eps}}{q^{1/2-\theta}} \Bigr)\\
  &\ll \ell_1^{-1/2}\exp(-c\sqrt{\log
    (L/\ell_1)})\to 0
\end{align*}
as $q\to +\infty$.
\par
\medskip
\par
\textbf{Case 3.}  In the final case, let $t_\pi(x)$ be the trace
function of Theorem~\ref{th-katz}, so that
$\widetilde{t_\pi}(\chi)=\Tr(\pi(\theta_{\chi}))$ for $\chi\notin
X_q$. Since we are not in Case 2, the sheaf $\pi(\mcF_q)$ is not of the
type of Case (2) of Lemma~\ref{lm-convol}.
\par
For each individual character $\chi$ modulo $q$  for every positive $\varepsilon$, we have  the bound 
$$
|\widetilde{t_\pi}(\chi)L(f\otimes \chi,\demi) \chi(\ell)|\ll_{f,\pi,\varepsilon}
q^{3/8+\eps}
$$
by the subconvexity estimate of Blomer and Harcos~\cite{BH}*{Th.\ 2}
since $|\widetilde{t_\pi}(\chi)|\ll 1$. Hence we can add the characters
in $X_q$ to the sum and obtain
\begin{multline*}
  \frac{1}{\vphis(q)} \sum_{\chi\notin X_q}
  \Tr(\pi(\theta_{\chi}))L(f\otimes \chi,\demi) \chi(\ell)\\
  = \frac{1}{\vphis(q)}
  \sums_{\chi\mods{q}}\widetilde{t_\pi}(\chi)L(f\otimes \chi,\demi)
  \chi(\ell)+ O(q^{-5/8+\eps}).
\end{multline*}
\par
Since we are in Case (1) in Lemma~\ref{lm-convol}, we deduce by Theorem~\ref{th-m1-mellin},
$$
\frac{1}{\vphis(q)} \sum_{\chi\notin
  X_q}\Tr(\pi(\theta_{\chi}))L(f\otimes \chi,\demi) \chi(\ell)\ll
q^{-1/8+\eps}
$$
for any $\eps>0$, where the implied constant depends on $f$, $C$,
$\pi$ and $\eps$. Using these bounds
in~(\ref{eq-sumell}) 
we deduce that
$$
\frac{1}{\vphis(q)} \sum_{\chi\notin
  X_q}\Tr(\pi(\theta_{\chi}))L(f\otimes \chi,\demi)
M(f\otimes\chi,\bfx_L) \ll_{f,\pi,\eps} L^{1/2+\eps}q^{-1/8+\eps}\ra 0
$$
as $q\to +\infty$.
\par
Collecting the bounds from all the three cases above, we conclude (see \eqref{L''ETdef}) that under assumption \eqref{eq-max-lambda} 
$$\mcL''_{\mathrm{ET}}\ra 0$$
as $q\to +\infty$.
If we choose $\delta>0$ small enough, depending only
on $\phi_0$,  it follows from \eqref{decomp101} that
$$
\liminf_{q\to +\infty} |\mcL''|>0,
$$
hence the result.
\end{proof}


\chapter{Extreme values of twisted \texorpdfstring{$L$-functions}{L-functions}}
\label{ch-extreme}

\section{Introduction}

In this chapter we prove Theorems \ref{thm-extremal} and
\ref{thm-extremal2}, which establish the existence of very large
values of twisted $L$-functions. We fix $f$ as in Section~\ref{intro}.

More precisely, we will prove the following refined statements:

\begin{theorem}
\label{LargeValuesAngularSectorsTheorem}
Let $I\subseteq\Rr/\pi\Zz$ be an interval of positive measure. Then,
for every sufficiently large prime modulus $q$, there exist primitive
characters $\chi$ of conductor $q$ such that
$$
\big|L\big(f\otimes\chi,\tfrac12\big)\big| \geqslant
\exp\left(\Bigl(\frac1{\sqrt8}+o(1)\Bigr) \sqrt{\frac{\log q}{\log\log
      q}}\right)\quad\text{and}\quad \theta(\ftchi)\in I.
$$
In fact, for every
$3\leqslant V\leqslant\frac3{14}\sqrt{\log q/\log\log q}$, we have
\begin{multline*}
  \big|\big\{\chi\bmod q\,\mid\,
  \big|L\big(f\otimes\chi,\tfrac12\big)\big|\geqslant e^V\text{ and }\theta(\ftchi)\in I\big\}\big|\\
  \geqslant\frac{\varphi(q)}{\log^2q}\exp\bigg(-(12+o(1))\frac{V^2}{\log\big(\log
    q/(16V^2\log V)\big)}\bigg).
\end{multline*}
\end{theorem}

We can also consider a product of twisted $L$-functions of two
different cusp forms.

\begin{theorem}
\label{LargeValuesOfProductsTheorem}
Let $g\not=f$ be a fixed primitive cusp of conductor $r'$ and trivial
central character, holomorphic or
not. 
There exists a constant $C>0$ such that for every sufficiently large
prime modulus $q$, there exists a primitive character $\chi$ of
conductor $q$ that satisfies
$$
\big|L\big(f\otimes\chi,\tfrac12\big)
L\big(g\otimes\chi,\tfrac12\big)\big|\geqslant
\exp\left(\big(C+o(1)\big)\sqrt{\frac{\log q}{\log\log q}}\right).
$$
\end{theorem}

\begin{remark}
  The constant $C$ depends on $f$ and $g$ and is effective. In
  particular, ``generically'', we can take $C=(6\sqrt{10})^{-1}$ (see
  Remark~\ref{ConstantRemark}, which explains what is meant by
  generic). Note that we assumed that $g\not=f$, since otherwise the
  first theorem gives a stronger result.
\end{remark}

We prove Theorems~\ref{LargeValuesAngularSectorsTheorem} and
\ref{LargeValuesOfProductsTheorem} using Soundararajan's method of
resonators. We draw inspiration for
Theorem~\ref{LargeValuesAngularSectorsTheorem} from Hough's
paper~\cite{Hough}; however, our results are more modest (in that we
are unable to detect angles in $\Rr/2\pi\Zz$) due to our inability to
evaluate second moments twisted by powers of Gau\ss\ sums and more in
line with the previously available results on extreme values with
angular restrictions.

We develop the method of resonators in a form ready for use in general
arithmetic situations in
Section~\ref{ResonatorSection}. Section~\ref{EvaluationOfMomentsSection}
combines this input with our evaluations of moments of twisted
$L$-functions to prove asymptotics for moments of $L$-functions
twisted by resonator and amplifier polynomials and we then use the
results of Section~\ref{MomentsHeckeEigenvaluesSection} to evaluate
the resulting main terms and prove the existence of large
$L$-values. Theorem~\ref{LargeValuesAngularSectorsTheorem} is proved
in Section~\ref{ExtremeValuesAngularConstraintsSection}, while
Theorem~\ref{LargeValuesOfProductsTheorem} is proved in
Section~\ref{LargeValuesProductsSection}.

\section{Background on the resonator polynomial}
\label{ResonatorSection}

The resonator method, originally introduced by Soundararajan~\cite{Soundararajan2008}, is a flexible tool that has been used in many contexts including (for extreme values in the $t$-aspect) the entire Selberg class (see, for example, \cite{Pankowski2015}), subject as usual to the Ramanujan conjecture. The method itself is by now standard and relies on a specific multiplicative arithmetic function, the ``resonator sequence'', which can take slightly different forms depending on the range of large values aimed for.

We refer the reader to Section \ref{LargeValuesIntroSection} in the introduction for a general description of the resonator method for a family of forms $f\in\mathcal{F}_N$. This relies on the comparison of the sizes of quantities
\[ Q_1=\expect_N(|R(f)|^2)\quad\text{and}\quad Q_2=\expect_N\Bigl(|R(f)|^2L(f,\demi)\Bigr), \]
with the resonator polynomial $R(f)$ constructed from the ``resonator sequence’’ $\alpha(\ell)$ and the arithmetic factors such as $\lambda_f(\ell)$ or their variations. Following custom, in this chapter we denote such a resonator polynomial by $R(f)=\sum_{n\leqslant N}r(n)\lambda_f(n)$. (To avoid confusion, we remind the reader that our specific Theorems \ref{LargeValuesAngularSectorsTheorem} and \ref{LargeValuesOfProductsTheorem} concern the family of twisted forms $f\otimes\chi$ which are naturally indexed by primitive characters $\chi$ modulo $q$, and the above averages are over $\chi$, while $f$, $g$ are fixed forms; in particular a resonator such as $R(\chi)=\sum_{n\leqslant N}r(n)\lambda_f(n)\chi(n)$ is used.)

In each application, to obtain large values in a family of $L$-functions (or other arithmetic objects), two steps are required.
\begin{enumerate}
\item \label{step1} The first step is analysis of averages in the family that to some degree isolates the terms contributing to the main term (usually the diagonal terms). Opening the sum in $R(f)$, evaluating the averages in $Q_1$ and $Q_2$ asymptotically involves first executing averages of (products of) arithmetic factors and twisted moments of $L$-functions, or variations of these, in the given family.
\item \label{step2} The second step is application of the resonator method, with the specific resonator constructed so as to reflect the main term contributions (which typically involve arithmetic factors such as, in the context of Theorems~\ref{LargeValuesAngularSectorsTheorem} and \ref{LargeValuesOfProductsTheorem}, Hecke eigenvalues of the fixed form(s)). In this step, the sum over the resonator polynomial is executed, leading in the main terms to sums of the form \eqref{TruncationOK1} and \eqref{TruncationOK2} below, and then to a lower bound on the quotient $|Q_2|/Q_1$ as in Lemma~\ref{GainLemma} below. Owing to the multiplicative nature of the resonator, optimizing this lower bound is seen to heavily depend on average information on the arithmetic factors over the primes, such as \eqref{SecondMomentCondition} and \eqref{OmegaPrimeLowerBound} below and their variations.
\end{enumerate}

Step~\eqref{step1} is the key arithmetic input and heavily depends on the family of $L$-functions considered. In this section, we focus on step~\eqref{step2}, the application of the resonator method, and make two points: first, one only needs a fairly limited amount of information about the arithmetic factors, and, second, the machinery of the resonator method can be developed in abstract, with no reference to the specific family and relying only on fairly general assumptions about the arithmetic factors. While probably known to the experts, these facts do not seem to be in the literature in a ready-to-use form and we take the opportunity of this memoir to expose them here.

Soundararajan~\cite{Soundararajan2008} introduced two variants of the
resonator sequence, with each being more efficient depending on
whether one is aiming for the highest possible values afforded by the
resonator method or for many values of slightly smaller size. We
develop both variants abstractly in the two sections below.

\subsection{Extreme values range}
\label{ExtremeValuesRangeSubsection}
In the extreme values range, we use a resonator polynomial similar to that used by Soundararajan~\cite{Soundararajan2008} and Hough~\cite{Hough}, which involves the multiplicative function $r(n)$ supported on square-free numbers and defined at primes by
\begin{equation}
\label{DefinitionResonator}
r(p)=\begin{cases} \displaystyle\frac{L}{\sqrt{p}\log p}, &L^2\leqslant p\leqslant\exp(\log^2L),\\ 0,&\text{otherwise}, \end{cases}
\end{equation}
where $L$ is a large parameter. 

In this section, we prove the two key claims for the application of the resonator method in the extreme values range, Lemmas~\ref{SmallTailsLemma} and \ref{GainLemma}.

We consider two non-negative multiplicative arithmetic functions $\omega$, $\omega'$ satisfying the following conditions. 

{\em There exists  $a_{\omega},a'_{\omega'}>0$,
  $0<\delta,\delta'\leqslant 1$, such that for all $Y\geqslant 2X\geq
  4$, we have}
\begin{align}
\label{SecondMomentCondition}
&\sum_{X\leqslant p\leqslant Y}\frac{\omega(p)}{p\log p}\leqslant a_{\omega}\bigg(\frac1{\log X}-\frac1{\log Y}\bigg)+O_{\omega}\bigg(\frac1{\log^2X}\bigg),\\
\label{OmegaPrimeLowerBound}
&\sum_{X\leqslant p\leqslant Y}\frac{\omega'(p)}{p\log p}\geqslant a'_{\omega'}\bigg(\frac1{\log X}-\frac1{\log Y}\bigg)+O_{\omega'}\bigg(\frac1{\log^2X}\bigg),\\
\label{FourthMomentCondition}
&\sum_{p\leqslant X}\omega(p)^{\delta}\omega'(p)\ll_{\omega,\omega',\delta}X(\log X)^{\delta},\quad\sum_{p\leqslant X}\omega'(p)^{1+\delta'}\ll_{\omega',\delta'}X^{1+\delta'/2}.
\end{align}

For the first lemma, we actually require only the following very generous upper bound:
\begin{equation}
\label{OmegaPrimeNotCrazy}
\sum_{p\leqslant X}\frac{\omega'(p)}p\ll_{\omega'}\frac{e^{\sqrt{\log X}}}{\log X},
\end{equation}
which a consequence of \eqref{FourthMomentCondition}.

In the sequel the implied constant may depend on $\omega,\omega'$ although we will not always mention explicitly such dependency.

\begin{lemma}
\label{SmallTailsLemma}
Let the arithmetic function $r(n)$ be as in
\eqref{DefinitionResonator} and let $\omega(n)\geqslant 0$ be a
multiplicative arithmetic function satisfying
\eqref{SecondMomentCondition}.  Then, for every $N$ such that
$L\leqslant\sqrt{a_{\omega}^{-1}\log N\log\log N}$, we have
\begin{equation}
\label{TruncationOK1}
\sum_{n\leqslant N}r(n)^2\omega(n)=\big(1+o^{\star}(1)\big)\prod_p\big(1+r(p)^2\omega(p)\big),
\end{equation}
and if, additionally, $\omega'(n)\geqslant 0$ is a multiplicative arithmetic function satisfying \eqref{OmegaPrimeNotCrazy}, then
\begin{equation}
\label{TruncationOK2}
\sum_{\substack{nm\leqslant N\\ (n,m)=1}}\frac{r(n)^2r(m)\omega(n)\omega'(m)}{\sqrt{m}}=\big(1+o^{\star}(1)\big)\prod_p\left(1+r(p)^2\omega(p)+\frac{r(p)}{\sqrt{p}}\omega'(p)\right).
\end{equation}
Here the notation $o^{\star}(1)$ is a shortcut to $$O_{C}\big(N^{-C/(\log\log N)^3}\big)$$ for any $C>0$.
\end{lemma}

\begin{lemma}
\label{GainLemma}
Let the arithmetic function $r(n)$ be as in \eqref{DefinitionResonator} and let $\omega(n),\omega'(n)\geqslant 0$ be multiplicative arithmetic functions satisfying \eqref{OmegaPrimeLowerBound} and \eqref{FourthMomentCondition} (for some $0<\delta,\delta'\leqslant 1$.) Then
$$ \prod_p\bigg(1+\frac{r(p)\omega'(p)}{\sqrt{p}\big(1+r(p)^2\omega(p)\big)}\bigg)\gg\exp\bigg(\big(a'_{\omega'}+o^{\star}(1)\big)\frac{L}{2\log L}\bigg), $$
Here the notation $o^{\star}(1)$ is a shortcut to $$O\big(1/(\log L)^{\min(\delta,\delta')}\big).$$
\end{lemma}

The relevance of Lemma~\ref{GainLemma} is clear in the light of Lemma~\ref{SmallTailsLemma}: it gives a lower bound for the quotient of the right-hand sides of \eqref{TruncationOK2} and \eqref{TruncationOK1}.

\begin{rems}

\begin{enumerate}
\item The resonator method as originally formulated is a first moment method, but it can be adapted for applications to products of $L$-functions such as our Theorem~\ref{LargeValuesOfProductsTheorem}. For clarity, we prove the corresponding variation of \eqref{TruncationOK2} separately in Lemma~\ref{SmallTailsLemmaProducts} below, while \eqref{TruncationOK1} and Lemma~\ref{GainLemma} are ready to use in their current form.

\item In the original setup of the resonator method to obtain large values of $\zeta(\tfrac12+it)$ \cite{Soundararajan2008}, one takes $\omega=\omega'=1$, in which case $a_{\omega}=a'_{\omega'}=1$. Constant sequences $\omega$ and $\omega'$ are similarly appropriate for some other families (such as the family of quadratic characters or the family of holomorphic cusp forms of large weight in \cite{Soundararajan2008}). As a point of reference, in a situation like Theorems~\ref{LargeValuesAngularSectorsTheorem} and \ref{LargeValuesOfProductsTheorem} where the family consists of twists of a fixed cusp form $f$,  choices that could be of interest include $\omega(n)=1$, $\omega'(n)=|\lambda_f(n)|$ and $\omega(n)=\omega'(n)=|\lambda_f(n)|^2$; we discuss the specific choices for that application in Section~\ref{OurChoiceSection}. For now we stress that all of our conditions involve only averages of $\omega(p)$, $\omega'(p)$ over at least dyadic intervals (and in fact we only apply them in intervals much longer than dyadic). The conditions \eqref{SecondMomentCondition} and \eqref{OmegaPrimeLowerBound} in particular are only non-empty for $Y\gg X$ with a sufficiently large implied constant.

\item The error terms in \eqref{SecondMomentCondition}-- \eqref{FourthMomentCondition} are one choice that works, and other choices are possible; for example, any $o_{\omega'}(1/\log X)$ in \eqref{OmegaPrimeLowerBound} would suffice with an adjustment in the explicit $o^{\star}$-terms in Lemma~\ref{GainLemma}, and \eqref{SecondMomentCondition} can similarly be relaxed with a possibly adjusted size of $L$ (compare the critical computation \eqref{FinnickyStep} below). Often, it is possible to obtain \eqref{SecondMomentCondition} and \eqref{OmegaPrimeLowerBound} with no error term whatsoever by just changing the corresponding constant to $a_{\omega}+\eps$ and $a'_{\omega'}-\eps$; this need not harm the final extreme value result since one can always take $\eps\to 0$ at the very end.

\item Finally, the conditions \eqref{SecondMomentCondition} and \eqref{OmegaPrimeLowerBound} can be written simply as
  $$ \sum_{p\sim X}\omega(p)\ll \frac{X}{\log X} \ll\sum_{p\sim
      X}\omega'(p) $$ if one is not concerned about the precise values
  of the constants $a_{\omega}$ and $a'_{\omega'}$; however, these
  constants have a direct impact on the exponent in the final result
  (such as our Theorem~\ref{LargeValuesAngularSectorsTheorem}), so
  they can be of significance.
\end{enumerate}
\end{rems}

\begin{proof}[Proof of Lemma~\ref{SmallTailsLemma}]

First, we prove \eqref{TruncationOK1} by following~\cites{Soundararajan2008,Hough}. Using Rankin's trick with a suitable (soon to be chosen) $\alpha>0$, we have that
$$ \sum_{n>N}r(n)^2\omega(n)\leqslant N^{-\alpha}\sum_{n=1}^{\infty}n^{\alpha}r(n)^2\omega(n)\leqslant N^{-\alpha}\prod_p\left(1+p^{\alpha}r(p)^2\omega(p)\right). $$
Moreover, for $0\leqslant\alpha\ll 1/\log^2L$,
\begin{align*}
&\log\prod_p\big(1+p^{\alpha}r(p)^2\omega(p)\big)-\log\prod_p\big(1+r(p)^2\omega(p)\big)\\
&\qquad=\sum_p\log\left(1+\frac{(p^{\alpha}-1)r(p)^2\omega(p)}{1+r(p)^2\omega(p)}\right)\\
&\qquad\leqslant\alpha\sum_p\log p\cdot r(p)^2\omega(p)
+O\bigg(\alpha^2\sum_p\log^2p\cdot r(p)^2\omega(p)\bigg)
\end{align*}
Using the definition of the resonator sequence $r(p)$, \eqref{SecondMomentCondition}, and summation by parts, this quantity is seen to be
\begin{align*}
&=\alpha L^2\sum_{L^2\leqslant p\leqslant\exp(\log^2L)}\frac{\omega(p)}{p\log p}
+O\bigg(\alpha^2L^2\sum_{L^2\leqslant p\leqslant\exp(\log^2L)}\frac{\omega(p)}p\bigg)
\\
&\leqslant\alpha a_{\omega}\cdot\frac{L^2}{2\log L}+O_{\omega}\left(\alpha\frac{L^2}{\log^2L}+\alpha^2L^2\log\log L\right),
\end{align*}
For $0\leqslant\alpha\ll_{\omega}1/(\log^2L\log\log L)$, the second error term may be absorbed in $O_{\omega}(\alpha L^2/\log^2L)$. Given that $L\leqslant\sqrt{a_{\omega}^{-1}\log N\log\log N}$, this estimate is further
\begin{equation}
\label{FinnickyStep}
\begin{aligned}
&=\alpha a_{\omega}\cdot\frac{a_{\omega}^{-1}\log N\log\log N}{\log\log N+\log\log\log N+O_{\omega}(1)}+O_{\omega}\left(\alpha\frac{\log N}{\log\log N}\right)\\
&=\alpha\left(\log N-\frac{\log N\log\log\log N}{\log\log N}+O_{\omega}\left(\frac{\log N}{\log\log N}\right)\right).
\end{aligned}
\end{equation}
Combining everything, we have that
\begin{align*}
&\sum_{n>N}r(n)^2\omega(n)\\
&\qquad\leqslant\prod_p\big(1+r(p)^2\omega(p)\big)
\exp\left(-\alpha\frac{\log N\log\log\log N}{\log\log N}+O_{\omega}\left(\alpha\frac{\log N}{\log\log N}\right)\right).
\end{align*}
Picking, say, $\alpha=c/(\log^2L\log\log L)=(4c+o(1))/((\log\log N)^2\log\log\log N)$, we thus have
$$ \sum_{n>N}r(n)^2\omega(n)\leqslant\prod_p\big(1+r(p)^2\omega(p)\big)\exp\left(-\frac{C\log N}{(\log\log N)^3}\right), $$
for an arbitrary $C>0$ (simply by choosing an appropriate $c>0$). In particular,
\begin{align*}
\sum_{n\leqslant N}r(n)^2\omega(n)&=
\prod_p\big(1+r(p)^2\omega(p)\big)-\sum_{n>N}r(n)^2\omega(n)\\
&=\big(1+O\big(N^{-C/(\log\log N)^3}\big)\big)\prod_p\big(1+r(p)^2\omega(p)\big),
\end{align*}
completing the proof of \eqref{TruncationOK1}.

The proof of \eqref{TruncationOK2} is analogous. First of all,
$$ \sum_{\substack{n,m\geq 1\\(n,m)=1}}r(n)^2\omega(n)\cdot\frac{r(m)\omega'(m)}{\sqrt{m}}=\prod_p\left(1+r(p)^2\omega(p)+\frac{r(p)}{\sqrt{p}}\omega'(p)\right). $$
Further, for every $\alpha>0$,
\begin{align*}
&\sum_{\substack{nm>N\\ (n,m)=1}}\frac{r(n)^2r(m)\omega(n)\omega'(m)}{\sqrt{m}}\\
&\qquad\leqslant N^{-\alpha}\sum_{\substack{n,m\geq 1\\(n,m)=1}}\big(r(n)^2\omega(n)n^{\alpha}\big)\big(r(m)\omega'(m)m^{\alpha-1/2}\big)\\
&\qquad=N^{-\alpha}\prod_p\big(1+r(p)^2\omega(p)p^{\alpha}+r(p)\omega'(p)p^{\alpha-1/2}\big).
\end{align*}
Further, for every $0\leqslant\alpha\ll 1/\log^2L$,
\begin{align*}
&\log\prod_p\big(1+r(p)^2\omega(p)p^{\alpha}+r(p)\omega'(p)p^{\alpha-1/2}\big)-\log\prod_p\left(1+r(p)^2\omega(p)+\frac{r(p)\omega'(p)}{\sqrt{p}}\right)\\
&\qquad=\sum_p\log\bigg(1+\frac{(p^{\alpha}-1)\big(r(p)^2\omega(p)+r(p)\omega'(p)/\sqrt{p}\big)}{1+\big(r(p)^2\omega(p)+r(p)\omega'(p)/\sqrt{p}\big)}\bigg)\\
&\qquad\leqslant\alpha\sum_p\log p\left(r(p)^2\omega(p)+\frac{r(p)}{\sqrt{p}}\omega'(p)\right)\\
&\qquad\qquad+O\bigg(\alpha^2\sum_p\log^2p\bigg(r(p)^2\omega(p)+\frac{r(p)}{\sqrt{p}}\omega'(p)\bigg)\bigg).
\end{align*}
Using the definition of the resonator sequence $r(p)$, \eqref{SecondMomentCondition}, \eqref{OmegaPrimeNotCrazy}, and summation by parts, this quantity is seen to be
\begin{align*}
&=\alpha L^2\sum_{L^2\leqslant p\leqslant\exp(\log^2L)}\frac{\omega(p)}{p\log p}+\alpha L\sum_{L^2\leqslant p\leqslant\exp(\log^2L)}\frac{\omega'(p)}p\\
&\qquad +O\bigg(\alpha^2L^2\sum_{L^2\leqslant p\leqslant\exp(\log^2L)}\frac{\omega(p)}p+\alpha^2L\sum_{L\leqslant p\leqslant\exp(\log^2L)}\frac{\omega'(p)}p\log p\bigg)\\
&\leqslant\alpha a_{\omega}\cdot\frac{L^2}{2\log L}+O_{\omega,\omega'}\left(\alpha\frac{L^2}{\log^2L}+\alpha^2L^2\log\log L\right),
\end{align*}
As before, for $0\leqslant\alpha\ll_{\omega,\omega'}1/(\log^2L\log\log L)$, the second term is absorbed in the first one, and with $L\leqslant\sqrt{a_{\omega}^{-1}\log N\log\log N}$, the above is
$$ =\alpha\left(\log N-\frac{\log N\log\log\log N}{\log\log N}+O_{\omega,\omega'}\left(\frac{\log N}{\log\log N}\right)\right). $$
As above, with $\alpha=c/(\log^2L\log\log L)$ for a suitable $c>0$, this leads to the combined estimate
\begin{align*}
&\sum_{\substack{nm>N\\ (n,m)=1}}\frac{r(n)^2r(m)\omega(n)\omega'(m)}{\sqrt{m}}\\
&\qquad\leqslant\prod_p\left(1+r(p)^2\omega(p)+\frac{r(p)}{\sqrt{p}}\omega'(p)\right)\exp\left(-\frac{C\log N}{(\log\log N)^3}\right)
\end{align*}
for an arbitrary $C>0$. As a consequence,
\begin{align*}
&\sum_{\substack{nm\leqslant N\\(n,m)=1}}\frac{r(n)^2r(m)\omega(n)\omega'(m)}{\sqrt{m}}\\
&\qquad=\big(1+O\big(N^{-1/(\log\log N)^3}\big)\big)\prod_p\left(1+r(p)^2\omega(p)+\frac{r(p)}{\sqrt{p}}\omega'(p)\right),
\end{align*}
proving \eqref{TruncationOK2}.
\end{proof}

\begin{proof}[Proof of Lemma~\ref{GainLemma}]
Let
$$ \msl:=\prod_p\bigg(1+\frac{r(p)\omega'(p)}{\sqrt{p}\big(1+r(p)^2\omega(p)\big)}\bigg). $$
Using $1/(1+x)=1+O_{\delta}(x^{\delta})$ and $\log(1+x)=x+O_{\delta'}(x^{1+\delta'})$, which hold uniformly for all $x>0$ (including for trivial reasons possibly large values of $x$),
$$ \log\msl=\sum_p\left[\frac{r(p)\omega'(p)}{\sqrt{p}}+O_{\delta}\bigg(\frac{r(p)^{1+2\delta}\omega(p)^{\delta}\omega'(p)}{\sqrt{p}}\bigg)+O_{\delta'}\bigg(\frac{r(p)^{1+\delta'}\omega'(p)^{1+\delta'}}{p^{(1+\delta')/2}}\bigg)\right]. $$

Using \eqref{OmegaPrimeLowerBound}, \eqref{FourthMomentCondition}, and summation by parts, we find that
$$ \sum_p\frac{r(p)\omega'(p)}{\sqrt{p}}=L\sum_{L^2\leqslant p\leqslant\exp(\log^2L)}\frac{\omega'(p)}{p\log p}\geqslant a'_{\omega'}\frac{L}{2\log L}+O_{\omega'}\bigg(\frac{L}{\log^2L}\bigg), $$
as well as
\begin{equation*}
\begin{alignedat}{7}
&\sum_p\frac{r(p)^{1+2\delta}\omega(p)^{\delta}\omega'(p)}{\sqrt{p}}&&=L^{1+2\delta}\sum_{L^2\leqslant p\leqslant\exp(\log^2L)}\frac{\omega(p)^{\delta}\omega'(p)}{p^{1+\delta}\log^{1+2\delta}p}&&\ll_{\omega,\omega'\!,\delta}&&\frac{L}{(\log L)^{1+\delta}},\\
&\sum_p\frac{r(p)^{1+\delta'}\omega'(p)^{1+\delta'}}{p^{(1+\delta')/2}}&&=L^{1+\delta'}\sum_{L^2\leqslant p\leqslant\exp(\log^2L)}\frac{\omega'(p)^{1+\delta'}}{(p\log p)^{1+\delta'}}&&\ll_{\omega'\!,\delta'}&&\frac{L}{(\log L)^{1+\delta'}}.
\end{alignedat}
\end{equation*}
Combining everything, we obtain the statement of Lemma~\ref{GainLemma}.
\end{proof}

Finally we prove a variation of \eqref{TruncationOK2} that is useful in applying the method of resonators to products of several $L$-functions.

\begin{lemma}
\label{SmallTailsLemmaProducts}
Let the arithmetic function $r(n)$ be as in \eqref{DefinitionResonator}, let $\omega(n)\geqslant 0$ be a multiplicative arithmetic function satisfying \eqref{SecondMomentCondition}, and let $\omega'_1(n),\dots,\omega'_s(n)\geqslant 0$ be multiplicative arithmetic functions each satisfying \eqref{OmegaPrimeNotCrazy}. Then, for every $N\geq 20$ such that $L\leqslant\sqrt{a_{\omega}^{-1}\log N\log\log N}$,
\begin{equation}
\label{TruncationOK2Products}
\begin{aligned}
&\sum_{n\leqslant N}r(n)^2\omega(n)\mathop{\sum\dots\sum}_{\substack{m_1,\dots,m_s\leqslant N/n\\ (n,m_i)=1,\,\,(m_i,m_j)=1}}\prod_{i=1}^s\frac{r(m_i)\omega'_i(m_i)}{\sqrt{m_i}}\\
&\qquad=\big(1+o^{\star}(1)\big)\prod_p\bigg(1+r(p)^2\omega(p)+\frac{r(p)}{\sqrt{p}}\sum_{i=1}^s\omega'_i(p)\bigg),
\end{aligned}
\end{equation}
with $o^{\star}(1)=O_{C,\omega}\big(N^{-C/(\log\log N)^3}\big)$ for an arbitrary $C>0$.
\end{lemma}

\begin{proof}
The proof is a straightforward adaptation of the proof of \eqref{TruncationOK2}. Using Rankin's trick, we have that
\begin{align*}
  &\sum_{n\leqslant N}r(n)^2\omega(n)\mathop{\sum\dots\sum}_{\substack{m_1,\dots,m_s\leqslant N/n\\ (n,m_i)=1,\,\, (m_i,m_j)=1}}\prod_{i=1}^s\frac{r(m_i)\omega'_i(m_i)}{\sqrt{m_i}}\\
  &\qquad=\mathop{\sum\sum\dots\sum}_{\substack{n,m_1,\dots,m_s\geq 1\\ (n,m_i)=1,\,\,(m_i,m_j)=1}}r(n)^2\omega(n)\prod_{i=1}^s\frac{r(m_i)\omega'_i(m_i)}{\sqrt{m_i}}\\
  &\qquad\qquad
    +O\bigg(N^{-\alpha}\mathop{\sum\sum\dots\sum}_\stacksum{n,m_1,\dots,n_s\geq
    1}{(n,m_i)=1,\,\,(m_i,m_j)=1}r(n)^2\omega(n)n^{\alpha}\prod_{i=1}^sr(m_i)\omega'_i(m_i)m_i^{\alpha-1/2}\bigg).
\end{align*}
Using multiplicativity, the above expression equals
$$ \prod_p\bigg(1+r(p)^2\omega(p)+\frac{r(p)\omega'(p)}{\sqrt{p}}\bigg)+O\bigg(N^{-\alpha}\prod_p\big(1+r(p)^2\omega(p)p^{\alpha}+r(p)\omega'(p)p^{\alpha-1/2}\big)\bigg), $$
with $\omega'(p)=\sum_{i=1}^s\omega'_i(p)$. From this point on, we proceed as in the proof of \eqref{TruncationOK2} in Lemma~\ref{SmallTailsLemma} and conclude that, with the choice $\alpha=c/(\log^2L\log\log L)$ for a suitable $c>0$, the ratio of the error term to the main term is $O\big(N^{-C/(\log\log N)^3}\big)$; this in turn proves the lemma.
\end{proof}

\subsection{Many high values range}
\label{ManyHighValuesRangeSubsection}
Sections~\ref{ExtremeValuesRangeSubsection} and \ref{ManyHighValuesRangeSubsection} prepare  ground in general for two different applications of Soundararajan's resonator method (which are demonstrated in the two claims of Theorem~\ref{LargeValuesAngularSectorsTheorem}). The first of these, subject of section~\ref{ExtremeValuesRangeSubsection}, is to show the existence of some extremely high values of $L$-functions in a family. The second is to prove that many $L$-functions in the family attain high values well beyond the generic size (conjecturally in the sense of any power average) and only slightly below the extreme values range. Such results require a bit different resonator sequence, whose application we develop in abstract here.

Let $X_0>0$ be a large parameter (namely sufficiently large so that \eqref{UpperBoundsConditions}--\eqref{UpperBoundCondition} and \eqref{ConditionsConditions} below hold). Let $A>0$ be arbitrary, and let
$$ A_0=\max(A,X_0). $$
Similarly as in \cite{Soundararajan2008}, let $r(n)$ be a multiplicative arithmetic function supported on square-free numbers and defined at primes by
\begin{equation}
\label{DefinitionResonator2}
r(p)=\begin{cases} \displaystyle\frac{A}{\sqrt{p}}, &A_0^2\leqslant p\leqslant N^{c/A_0^2},\\ 0, &\text{otherwise},\end{cases}
\end{equation}
where $N>0$ is a large parameter, and $c>0$ is a suitable constant (its value will be controlled by \eqref{BasicEvaluationsClaim} in Lemma~\ref{ManyLargeValuesLemma}). Note that this resonator (which is optimized for the purpose of exhibiting many large values in a family of $L$-functions) differs somewhat from the one in \eqref{DefinitionResonator} and that it directly depends on $N$. Also note that the sequence $r(n)$ can only be non-empty for $A_0\leqslant\sqrt{(c+o^{\star}(1))\log N/\log\log N}$, with $o^{\star}(1)=O(\log\log\log N/\log\log N)$. Although the sequence $r(n)$ is different from the one in \eqref{DefinitionResonator}, we keep the same notation since some of the evaluations take literally the same form.

As in Section~\ref{ExtremeValuesRangeSubsection}, the sequence $r(n)$ will be combined with arithmetic factors $\omega(n)$ and $\omega'(n)$ in the particular application of the resonator method. We make the following assumptions on these sequences for all $Y\geqslant 2X$, $X\geqslant X_0$:
\begin{align}
\label{UpperBoundsConditions}
&\sum_{p\leqslant X}\frac{\omega(p)\log p}p\leqslant b_{\omega}\log X+O_{\omega}(1),
\qquad \sum_{p\leqslant X}\frac{\omega'(p)\log p}p=O_{\omega'}(\log X),\\
\label{LowerBoundCondition}
&\sum_{X\leqslant p\leqslant Y}\frac{\omega'(p)}p\geqslant b'_{\omega'}\log\frac{\log Y}{\log X}+O_{\omega'}\left(\frac1{\log X}\right),\\
\label{ErrorControlConditions}
&\sum_{p\leqslant X}\omega(p)\omega'(p)=O_{\omega,\omega'}\bigg(\frac{X}{\log X}\bigg),\quad\sum_{p\leqslant X}\omega'(p)^2=O_{\omega'}\bigg(\frac{X^{3/2}}{\log X}\bigg),\\
\label{UpperBoundCondition}
&\sum_{X\leqslant p\leqslant Y}\frac{\omega(p)}{p}\leqslant b_{\omega 2}\log\frac{\log Y}{\log X}+O_{\omega}\bigg(\frac1{\log X}\bigg).
\end{align}
We remark that, of the two upper bounds in \eqref{UpperBoundsConditions}, the first one easily follows from
\eqref{UpperBoundCondition} but perhaps with a suboptimal value of $b_{\omega}$, while the second one would follow from a sharpened form of the second condition in \eqref{ErrorControlConditions} $O_{\omega'}(X/\log X)$, in which the latter would be typically expected. We keep \eqref{UpperBoundsConditions} to get the tightest constants and minimal conditions.

Analogously as in Lemmas~\ref{SmallTailsLemma} and \ref{GainLemma}, the following statement summarizes the resonator-related inputs into obtaining a large number of high values.
\begin{lemma}
\label{ManyLargeValuesLemma}
Let the arithmetic function $r(n)$ be as in \eqref{DefinitionResonator2}. Then:
\begin{enumerate}
\item \label{BasicEvaluationsClaim}
If multiplicative arithmetic functions $\omega(n),\omega'(n)\geqslant 0$ satisfy \eqref{UpperBoundsConditions}, then the basic evaluations \eqref{TruncationOK1} and \eqref{TruncationOK2} hold for every
\begin{equation}
\label{ConditionsConditions}
c<b_{\omega}^{-1},\quad X_0\gg_{\omega,\omega'}(1-cb_{\omega})^{-1},\quad 0<A\ll_{\omega,\omega',c}\sqrt{\log N},
\end{equation}
with $o^{\star}(1)=O\big(\exp(-\tilde{\delta}A_0^2)\big)$ for some fixed $\tilde{\delta}>0$ depending on $\omega$, $\omega'$, $c$ only.
\item \label{LowerBoundClaim}
If multiplicative arithmetic functions $\omega(n),\omega'(n)\geqslant 0$ satisfy \eqref{LowerBoundCondition} and \eqref{ErrorControlConditions}, then
$$ \prod_p\bigg(1+\frac{r(p)\omega'(p)}{\sqrt{p}\big(1+r(p)^2\omega(p)\big)}\bigg)\gg\exp\bigg(Ab'_{\omega'}\log\frac{c\log N}{2A_0^2\log A_0}+o^{\star}(A)\bigg), $$
with $o^{\star}(A)=O_{\omega,\omega'}(A/\log A_0)$.
\item \label{UpperBoundOnMomentsOfr}
  For every multiplicative function $a(n)$ and $\omega(n)=|a(n)|^2$,
  we have for every integer $K\geq 1$ and $N\leqslant q^{1/K}$
$$ \frac1{\varphi(q)}\sum_{\chi\bmod q}\bigg|\sum_{n\leqslant N}r(n)a(n)\chi(n)\bigg|^{2K}
\leqslant \prod_p\big(1+r(p)^2\omega(p)\big)^{K^2}. $$
If $\omega(n)$ satisfies \eqref{UpperBoundCondition}, then
$$ \prod_p\big(1+r(p)^2\omega(p)\big)\ll\exp\Bigg(A^2b_{\omega 2}\log\frac{c\log N}{2A_0^2\log A_0}+O_{\omega}\bigg(\frac{A^2}{\log A_0}\bigg)\bigg). $$
\end{enumerate}
\end{lemma}

\begin{proof}
Claim \eqref{BasicEvaluationsClaim} is proved analogously as Lemma~\ref{SmallTailsLemma}. We prove the basic evaluation \eqref{TruncationOK2} by using Rankin's trick as in the proof of Lemma~\ref{SmallTailsLemma}. Critically, for every $0\leqslant\alpha\ll A_0^2/\log N$, we estimate using~\eqref{UpperBoundsConditions}
\begin{align*}
&\sum_p\log\bigg(1+\frac{(p^{\alpha}-1)\big(r(p)^2\omega(p)+r(p)\omega'(p)/\sqrt{p}\big)}{1+\big(r(p)^2\omega(p)+r(p)\omega'(p)/\sqrt{p}\big)}\bigg)\\
&\qquad\leqslant\alpha\sum_p\log p\bigg(r(p)^2\omega(p)+\frac{r(p)}{\sqrt{p}}\omega'(p)\bigg)\\
&\qquad\qquad+O\bigg(\alpha^2\sum_p\log^2p\bigg(r(p)^2\omega(p)+\frac{r(p)}{\sqrt{p}}{\omega'(p)}\bigg)\bigg)\\
&\qquad=\alpha A^2\sum_{A_0^2\leqslant p\leqslant N^{c/A_0^2}}\frac{\omega(p)\log p}p+\alpha A\sum_{A_0^2\leqslant p\leqslant N^{c/A_0^2}}\frac{\omega'(p)\log p}p\\
&\qquad\qquad+O\bigg(\alpha^2A^2\sum_{A_0^2\leqslant p\leqslant N^{c/A_0^2}}\frac{\omega(p)\log^2p}p+\alpha^2A\sum_{A_0^2\leqslant p\leqslant N^{c/A_0^2}}\frac{\omega'(p)\log^2p}p\bigg)\\
&\qquad\leqslant cb_{\omega}\frac{\alpha A^2}{A_0^2}\log N+O_{\omega,\omega'}\left(\alpha A^2+\frac{\alpha A}{A_0^2}\log N+\frac{\alpha^2A^2}{A_0^4}\log^2N\right).
\end{align*}
Recall that $c<b_{\omega}^{-1}$.  Choosing $\alpha=\delta_{\omega,\omega',c} A_0^2/\log N$ for a sufficiently small $\delta_{\omega,\omega',c}>0$, in light of our conditions \eqref{ConditionsConditions} the above quantity is seen to be $\leqslant(1-\delta')\alpha\log N$ for some (fixed and depending on $\omega$, $\omega'$, $c$ only) $\delta'>0$. Thus, upon application of Rankin's trick,
$$ \sum_{\substack{nm>N\\(n,m)=1}}\frac{r(n)^2r(m)\omega(n)\omega'(m)}{\sqrt{m}}\leqslant N^{-\delta'\alpha}\prod_p\bigg(1+r(p)^2\omega(p)+\frac{r(p)}{\sqrt{p}}\omega'(p)\bigg), $$
which in turn suffices to prove \eqref{TruncationOK2} in claim \eqref{BasicEvaluationsClaim}. The basic evaluation \eqref{TruncationOK1} follows analogously simply by omitting the missing terms in the above argument.

In claim \eqref{LowerBoundClaim}, we simply compute using \eqref{LowerBoundCondition}, \eqref{ErrorControlConditions}, and summation by parts,
$$ \sum_p\frac{r(p)\omega'(p)}{\sqrt{p}}=A\sum_{A_0^2\leqslant p\leqslant N^{c/A_0^2}}\frac{\omega'(p)}{p}\geqslant Ab'_{\omega'}\log\frac{c\log N}{2A_0^2\log A_0}+O_{\omega'}\left(\frac{A}{\log A_0}\right), $$
as well as
\begin{alignat*}{7}
&\sum_p\frac{r(p)^3\omega(p)\omega'(p)}{\sqrt{p}}&&=A^3\sum_{A_0^2\leqslant p\leqslant N^{c/A_0^2}}\frac{\omega(p)\omega'(p)}{p^2}&&\ll_{\omega,\omega'}&&\frac{A^3/A_0^2}{\log A_0},\\
&\sum_p\frac{r(p)^2\omega'(p)^2}p&&=A^2\sum_{A_0^2\leqslant p\leqslant N^{c/A_0^2}}\frac{\omega'(p)^2}{p^2}&&\ll_{\omega'}&&\frac{A^2/A_0}{\log A_0}.
\end{alignat*}

Finally we prove the claim \eqref{UpperBoundOnMomentsOfr}. By orthogonality, the condition $N\leqslant q^{1/K}$, multiplicativity, and the fact that $r(n)$ is supported on square-free integers, we have that
\begin{align*}
&\frac1{\varphi(q)}\sum_{\chi\bmod q}\bigg|\sum_{n\leqslant N}r(n)a(n)\chi(n)\bigg|^{2K}\\
&\qquad=\sum_{\substack{n_1,\dots,n_{2K}\leqslant N\\ n_1\cdots n_K=n_{K+1}\cdots n_{2K}}}r(n_1)\cdots r(n_{2K})a(n_1)\cdots a(n_K)\overline{a(n_{K+1})}\cdots\overline{a(n_{2K})}\\
&\qquad\leqslant\prod_p\bigg(\sum_{k=0}^K\binom{K}{k}^2r(p)^{2k}\omega(p)^k\bigg)\\
&\qquad\leqslant\prod_p\bigg(\sum_{k=0}^{K^2}\binom{K^2}kr(p)^{2k}\omega(p)^k\bigg)
=\prod_p\big(1+r(p)^2\omega(p)\big)^{K^2}.
\end{align*}
Using \eqref{UpperBoundCondition}, we easily find that
\begin{align*}
\log\prod_p\big(1+r(p)^2\omega(p)\big)
&\leqslant\sum_p r(p)^2\omega(p)=A^2\sum_{A_0^2\leqslant p\leqslant N^{c/A_0^2}}\frac{\omega(p)}p\\
&\leqslant A^2b_{\omega 2}\log\frac{c\log N}{2A_0^2\log A_0}+O_{\omega}\bigg(\frac{A^2}{\log A_0}\bigg).
\end{align*}
This completes the proof of Lemma~\ref{ManyLargeValuesLemma}.
\end{proof}

\section{Evaluation of the moments}
\label{EvaluationOfMomentsSection}

In addition to the setup of the resonator method, the crucial input
for an application of this method is the evaluation of the first
moment twisted by the square of the resonator polynomial. In this
section, we complete this and associated steps for the family of
twisted $L$-functions $L(f\otimes\chi,s)$.

\subsection{Moment evaluations}

In this section, we present evaluations of the twisted first and second moments in the form in which they will be used in the application of the resonator method.

\begin{lemma} There is an absolute constant $A\geq 0$ such that for any $N< q$, any integers $1\leq n_1,n_2\leq N$ and any $\kappa\in\Zz$, the twisted first moment $\mcl(f;n_1r^{\kappa}\bar{n}_2,2\kappa)$ defined in \eqref{firstmom} satisfies
\begin{multline}
\label{TwistedFirstMomentEvaluationSpecialized}
\mcl(f;n_1r^{\kappa}\bar{n}_2,2\kappa)
=\delta_{\begin{subarray}{l}\kappa=0,\\ n_2=n_1m\end{subarray}}\frac{\lambda_f(m)}{\sqrt{m}}+\delta_{\begin{subarray}{l}\kappa=-1,\\ n_1=n_2m\end{subarray}}\frac{{\lambda_f(m)}}{\sqrt{m}}\\
+O_{f,\eps,A}\Big((|\kappa|+1)^Aq^{\eps}\big(q^{-1/8}+(q/N)^{\theta-1/2}\big)\Big).
\end{multline}
In the first term of the right-hand side the equality $n_2=n_1m$ means
that the term is zero unless $n_1$ divides $n_2$ (and the quotient is
defined as $m$) and similarly for the second term.
\end{lemma}

\begin{proof}
  By Corollary \ref{cor-k=-2},
$$
\mcl(f;n_1r^{\kappa}\bar{n}_2,2\kappa)=
O_{f,\eps,A}\Big((|\kappa|+1)^Aq^{-1/8+\eps}\Big)
$$
unless $\kappa=0$ or $-1$, in which case the additional terms
\begin{equation}\label{TwistedFirstMomentEvaluation}
  \delta_{\kappa=0}\frac{\lambda_f(a)}{\sqrt{a}}
  \quad  \hbox{ or }\quad
  \delta_{\kappa=-1}\eps(f)\frac{\lambda_f(b)}{\sqrt{b}}
\end{equation}
appear, where
$$ 
a=(\ov{n_1\bar{n}_2})_q
$$
is the representative in $[1,q]$ of the congruence class
$\ov{n_1\bar{n}_2}$ modulo $q$, and
$$
b=(n_1\bar{n}_2)_q
$$ 
is the representative of the congruence class $n_1\bar{n_2}$ modulo
$q$.

Assume that $\kappa=0$. Then the congruence
$a\equiv \ov{n_1\bar{n}_2}\mods{q}$ implies either that $n_2=n_1a$ (so
$n_1\mid n_2$) or that $n_1a>q$. In the second case, we have $a>q/N$,
and the first term of \eqref{TwistedFirstMomentEvaluation} is
$\ll (q/N)^{-(1/2-\theta+\eps)}$ for any $\eps>0$.

Assume that $\kappa=-1$. Then the congruence
$b\equiv n_1\bar{n}_2\mods{q}$ implies similarly either that
$n_1=n_2b$, or that $b>q/N$, in which case the second term of
\eqref{TwistedFirstMomentEvaluation} is
$\ll (q/N)^{-(1/2-\theta+\eps)}$ for any $\eps>0$. The lemma follows.
\end{proof}

Consider now two \emph{distinct} primitive cusp forms $f$ and $g$ of
signed level $r$ and $r'$ respectively, with trivial central character. Let
$q$ be a prime not dividing $rr'$. We refer to
Section~\ref{sec-secondmoment} for the definition of some of the
quantities below.  We recall Convention~\ref{def-conv} concerning the
signed   level of cusp forms. As in Chapter~\ref{ch-second}, we
write $\delta=(r,r')\geq 1$ and $|r|=\rho\delta$, $|r'|=\rho'\delta$.

We define arithmetic functions $\lambda_f^{\ast}$ and
$\lambda_{g}^{\ast}$ such that they are supported on squarefree
integers and satisfy
\begin{equation}\label{lambdaastdef}
  \lambda_f^{\ast}(p)=\lambda_f(p)-\lambda_{g}(p)/p\hbox{ and }
  \lambda^\ast_{g}(p)=\lambda_{g}(p)-\lambda_f(p)/p.
\end{equation}

We note that these functions depend on both $f$ and $g$, and that
$\lambda_f^{\ast}$ and $\lambda_{g}^{\ast}$ satisfy
\eqref{LambdaStarAdjustment}, i.e.
$$
\lambda_f^{\ast}(p)=\lambda_f(p)+O(p^{\theta-1}),\quad
\lambda_{g}^{\ast}(p)=\lambda_{g}(p)+O(p^{\theta-1}).
$$
In particular, Corollary \ref{corlambdaast} applies to them.

\begin{lemma}
\label{SecondMomentEvaluationForProducts}
For any integers $1\leq \ell,\ell'\leqslant L\leqslant q^{1/2}$ such
that $\ell\ell'$ is squarefree and coprime to $rr'$, the twisted second
moment $\mathcal{Q}^{\pm}(f,g;\ell,\ell')$ defined in Sections
\ref{sec-2ndmoment-intro} and \ref{sec-secondmoment} satisfies
\begin{equation}\label{Qspe}
  \mathcal{Q}^\pm(f,g;\ell,\ell')=\MT^\pm(f,g;\ell,\ell')+
  O\big(L^{3/2}q^{-1/144}\big),
\end{equation}
where
\begin{multline*}
  \MT^{\pm}(f,g;\ell,\ell')= \frac{1}{2}L^{\ast}(f\otimes g, 1)\Bigl(
  \frac{\lambda_f^{\ast}(\l')\lambda_g^{\ast}(\l)}{(\l\l')^{1/2}}\\
  + \eps(f)\eps(g)
  \frac{\lambda_f(\rho)\lambda_g(\rho')\lambda_f^{\ast}(\l)
    \lambda_g^{\ast}(\l')}{(\rho\rho'\l\l')^{1/2}}\Bigr) +
  O\big(q^{-1/2+\eps}\big).
\end{multline*}
\end{lemma}

\begin{proof} 
  By the argument in Section~\ref{sec-secondmoment} (see also
  Proposition~\ref{pr-mt}), we obtain the asymptotic formula
  \eqref{Qspe} with main term $\MT^{\pm}(f,g;\l,\l')$ given by
  \eqref{sun1}, \eqref{sun2}, \eqref{sun3}, \eqref{sun4}, namely
\begin{multline*}
  \MT^{\pm}(f,g;\l,\l')= \frac{1}{2} \intc_{(2)}\mcL^{\pm}_{\infty}(\demi+u)
  \frac{D(1+2u;\l',\l)}{(\l\l')^{1/2+u}} G(u) (q^{2}|rr'|)^{u} \frac{du}{u}
  \\
  +\frac{\eps(f)\eps(g)\lambda_f(\rho)\lambda_g(\rho')}{2(\rho\rho')^{1/2}}
  \intc_{(2)}\mcL^{\pm}_{\infty}(\demi+u)
  \frac{D(1+2u;\l,\l')}{(\l\l')^{1/2+u}} G(u) (q^{2}|rr'|)^{u} \frac{du}{u},
\end{multline*}
where
$$
\mcL^{\pm}_{\infty}(s)=\frac{L_\infty(f,\pm,s)}
{L_\infty(f,\pm,\frac12)}
\frac{L_\infty(g,\pm,s)}{L_\infty(g,\pm,\frac12)}
$$
and $D(s;\l,\l')$ is the Dirichlet series
$$
D(s;\l,\l')=\sum_{n\geq 1}\frac{\lambda_f(\ell n)\lambda_g(\ell'
  n)}{n^s},
$$
which is absolutely convergent for $\Reel(s)>1$.  

Since $\l$ and $\l'$ are squarefree and coprime, we have by
multiplicativity the formula
$$
D(s;\l,\l')=\prod_{p\nmid\l\l'}L^{\ast}_p(f\otimes g,s)\prod_{p\mid
  \l}A_p(f,g;s)\prod_{p\mid\l'}A_p(g,f;s)
$$
where
$$
A_p(f,g;s)=\sum_{k\geq
  0}\frac{\lambda_f(p^{k+1})\lambda_g(p^{k})}{p^{ks}}.
$$
Using the Hecke relation, we obtain the relation
$$
A_p(f,g;s)= \lambda_f(p)L_p^{\ast}(f\otimes
g,s)-\frac{1}{p^s}A_p(g,f;s).
$$
Applying it twice, this leads to the formula
$$
A_p(f,g;s)=(1-p^{-2s})^{-1}\Bigl(\lambda_f(p)-\frac{\lambda_g(p)}{p^s}\Bigr)
L^{\ast}_p(f\otimes g,s).
$$
It follows that
$$
D(s;\l,\l')=L^{\ast}(f\otimes g,s)\prod_{p\mid \l\l'}(1-p^{-2s})^{-1}
\prod_{p\mid \l} \Bigl(\lambda_f(p)-\frac{\lambda_g(p)}{p^s}\Bigr)
\prod_{p\mid\l'} \Bigl(\lambda_g(p)-\frac{\lambda_f(p)}{p^s}\Bigr).
$$

Now, moving the contour of integration to
$\Re u=-\tfrac14+\frac12\eps$, and recalling that $f\not=g$, so that $L^{\ast}(f\otimes g,s)$
is holomorphic inside the contour, we obtain by the residue theorem
the formula
\begin{multline*}
  \MT^{\pm}(f,g;\ell,\ell')= \frac{1}{2}L^{\ast}(f\otimes g, 1)\Bigl(
  \frac{\lambda_f^{\ast}(\l')\lambda_g^{\ast}(\l)}{(\l\l')^{1/2}}\\
  + \eps(f)\eps(g)
  \frac{\lambda_f(\rho)\lambda_g(\rho')\lambda_f^{\ast}(\l)
    \lambda_g^{\ast}(\l')}{(\rho\rho'\l\l')^{1/2}}\Bigr) +
  O\big(q^{-1/2+\eps}\big)
\end{multline*}
for any $\eps>0$, after picking the simple pole at $u=0$.
\end{proof}

\subsection{Asymptotics involving the resonator polynomial}
Let $r(n)$ be one of the following two resonator sequences:
\begin{equation}
\label{ChoicesForr}
\begin{aligned}
&\text{Let $L$ be a large parameter, and let $r(n)$ be as in \eqref{DefinitionResonator}, \emph{or}}\\
&\text{Let $N$ be a large parameter, let $A,c>0$, and let $r(n)$ be as in \eqref{DefinitionResonator2}.}
\end{aligned}
\end{equation}
The values of $L$, $A$, $c$ will eventually be restricted by the
conditions in Section~\ref{ResonatorSection}, but for now we leave
them general. We also set, in each case respectively,
\begin{equation}
\label{oStar1}
\begin{aligned}
&\text{$o^{\star}(1)=O_C\big(N^{-C/(\log\log N)^3}\big)$ with an arbitrary $C>0$, \emph{or}}\\
&\text{$o^{\star}(1)=O_{\delta}(e^{-\delta A^2})$ for some suitable fixed $\delta>0$, respectively.}
\end{aligned}
\end{equation}

Let $a_f(n)$ be a multiplicative arithmetic function supported on square-free positive integers such that
\begin{equation}
\label{SignConditionOnafn}
\sgn a_f(n)=\sgn\lambda_f(n)\quad\text{whenever }a_f(n)\lambda_f(n)\neq 0.
\end{equation}
For example, we could pick $a_f(n)=\mu^2(n)\lambda_f(n)$, or $a_f(n)=\mu^2(n)\sgn\lambda_f(n)$. For practical purposes, we only need to be concerned with defining $a_f(n)$ for $n$ such that $r(n)\neq 0$. Define
\begin{equation}
\label{ResultingOmegas}
\omega(n)=|a_f(n)|^2,\quad \omega'(n)=\overline{a_f(n)}\lambda_f(n);
\end{equation}
in view of \eqref{SignConditionOnafn}, $\omega$, $\omega'$ are non-negative multiplicative functions.

For every Dirichlet character $\chi$ modulo $q$, we define our resonator polynomial by
\begin{equation}
\label{DefinitionResonatorPolynomial}
R(\chi)=\sum_{n\leqslant N}r(n)a_f(n)\chi(n).
\end{equation}
We also recall the definition of the argument (cf.\ \eqref{phaseformula})
\begin{equation}
\label{DefinitionSign}
e^{i\theta(\ftchi)}:=\begin{cases} L\big(f\otimes\chi,\tfrac12\big)/\big|L\big(f\otimes\chi,\tfrac12\big)\big|, &L\big(f\otimes\chi,\tfrac12\big)\neq 0,\\ 1,&\text{else}. \end{cases}
\end{equation}
and the formula
\begin{equation}\label{argrecall}
e^{2i\theta(\ftchi)}=\eps({f\otimes\chi})=\eps(f)\chi(r)\eps_\chi^2.	
\end{equation}

To exhibit the desired large values of $L(f\otimes\chi,\tfrac12)$ with
$\theta(\ftchi)$ in desired angular segments, we will evaluate the
following two character averages.

\begin{lemma}
\label{NormalizingWeightLemma}
Let $q$ be a prime modulus, let $N\leqslant q$, let $r(n)$ be as in
\eqref{ChoicesForr}, let $a_f(n)$ be an arbitrary multiplicative
function supported on square-free integers satisfying
\eqref{SignConditionOnafn}, and, for every primitive character $\chi$
of conductor $q$, let $R(\chi)$ be as in
\eqref{DefinitionResonatorPolynomial}.

Assume that $r(n)$ and the multiplicative function
$\omega(n)=|a_f(n)|^2$ satisfy the basic
evaluation~\eqref{TruncationOK1}. Then, with $o^{\star}(1)$ as in
\eqref{oStar1},
$$ \frac1{\vphis(q)}\sums_{\chi\bmod q}|R(\chi)|^2=\big(1+o^{\star}(1)+O(N/q)\big)\prod_p\big(1+r(p)^2\omega(p)\big). $$
\end{lemma}

\begin{proof}
By orthogonality of characters and the Cauchy--Schwarz inequality,
\begin{multline*}
  \frac1{\vphis(q)}\sums_{\chi\bmod q}|R(\chi)|^2
  =\frac1{\vphis(q)}\sum_{\chi\bmod q}\bigg|\sum_{n\leqslant
    N}r(n)a_f(n)\chi(n)\bigg|^2
  \\
  -\frac1{\vphis(q)}\bigg|\sum_{n\leqslant N}r(n)a_f(n)\bigg|^2
=\big(1+O(N/q)\big)\sum_{n\leqslant N}r(n)^2\omega(n),
\end{multline*}
since $N\leqslant q$. Applying the basic evaluation \eqref{TruncationOK1},
we have that
$$ \sum_{n\leqslant N}r(n)^2\omega(n)=\big(1+o^{\star}(1)\big)\prod_p\big(1+r(p)^2\omega(p)\big), $$
and this in turn yields Lemma~\ref{NormalizingWeightLemma}.
\end{proof}

In the next lemma, we consider functions $\psi\colon \Rr/2\pi\Zz\to
\Cc$ which are $\pi$-anti-periodic, namely that satisfy
$\psi(\theta+\pi)=-\psi(\theta)$ for all $\theta$.
The Fourier expansion of such a function has the form
$$
\psi(\theta)=\sum_{\kappa\in\Zz}\hat\psi(2\kappa+1)e^{i(2\kappa+1)\theta}.
$$
We set 
$$
I(\psi):=\frac1\pi\int_{\Rr/2\pi}\psi(\theta)\cos(\theta)d\theta=
\hat\psi(1)+\hat\psi(-1).
$$
If $\psi$ is smooth, then for any integer   $B\geq 0$, we denote
the $B$-Sobolev norm of $\psi$ by
$$
\|\psi\|_B= \sum_{\kappa\in\Zz}(|\kappa|+1)^B\,
|\hat{\psi}(1+2\kappa)|.
$$

\begin{lemma}
\label{TwistedFirstMomentLemma}
Let $q$ be a prime modulus, let $N< q$, let $r(n)$ be as in
\eqref{ChoicesForr}, and let $a_f(n)$ be an arbitrary multiplicative
function supported on square-free integers satisfying
\eqref{SignConditionOnafn}. For every primitive character $\chi$ of
conductor $q$, let $R(\chi)$ be as in
\eqref{DefinitionResonatorPolynomial}.

Let $\psi$ be a smooth $\pi$-anti-periodic function.

Assume that $r(n)$ and the multiplicative functions $\omega,\omega'$
given in \eqref{ResultingOmegas} satisfy the basic
evaluation~\eqref{TruncationOK2}. Then, for a sufficiently large
absolute $B\geq 0$ and with $o^{\star}(1)$ as in \eqref{oStar1}, we
have
\begin{multline*}
\frac1{\vphis(q)}\sums_{\chi\bmod q}|R(\chi)|^2\big|L\big(f\otimes\chi,\tfrac12\big)\big|\psi(\theta(\ftchi))\\
 =I(\psi)\big(1+o^{\star}(1)\big)\cdot\prod_p\bigg(1+r(p)^2\omega(p)+\frac{r(p)}{\sqrt{p}}\omega'(p)\bigg)\\
+O_{f,\eps,B}\Big(q^{\eps}N\big(q^{-1/8}+(q/N)^{\theta-1/2}\big)\|\psi\|_B\cdot\prod_p\big(1+r(p)^2\omega(p)\big)\Big).
\end{multline*}
\end{lemma}

\begin{proof}
The quantity to evaluate is equal to
\begin{multline*}
  \sum_{n_1,n_2\leqslant
    N}r(n_1)r(n_2)a_f(n_1)\overline{a_f(n_2)}\sum_{\kappa\in\Zz}\hat{\psi}({2\kappa+1})\\
  \times\frac1{\vphis(q)}\sums_{\chi\bmod q}L\big(f\otimes\chi,\tfrac12\big)e^{i2\kappa\theta(\ftchi)}\chi(n_1\bar{n}_2)\\
  =\sum_{n_1,n_2\leqslant N}r(n_1)r(n_2)a_f(n_1)\overline{a_f(n_2)}\sum_{\kappa\in\Zz}\hat{\psi}(2\kappa+1)
  \eps(f)^{\kappa}\mcl(f;n_1r^{\kappa}\bar{n}_2,2\kappa)
\end{multline*}
by \eqref{argrecall}.

Using \eqref{TwistedFirstMomentEvaluationSpecialized},
\eqref{ResultingOmegas}, and keeping in mind that the resonator
sequence $r(n)$ is supported on square-free integers, there exists an
constant $B\geq 0$ such that this is in turn equal to
\begin{multline*}
  (\hat{\psi}(1)+\hat{\psi}({-1}))\sum_{\substack{nm\leqslant N\\ (n,m)=1}}\frac{r(n)^2r(m)\omega(n)\omega'(m)}{\sqrt{m}}\\
  +O_{f,\eps,A}\bigg(q^{\eps}\big(q^{-1/8}+(q/N)^{\theta-1/2}\big)
  \|\psi\|_B
  \bigg(\sum_{n\leqslant N}r(n)|a_f(n)|\bigg)^2\bigg).
\end{multline*}

Applying the basic evaluation~\eqref{TruncationOK2} to the double $(n,m)$-sum, we have that
$$\sum_{\substack{nm\leqslant N\\(n,m)=1}}\frac{r(n)^2r(m)\omega(n)\omega'(m)}{\sqrt{m}}
=\big(1+o^{\star}(1)\big)\prod_p\bigg(1+r(p)^2\omega(p)+\frac{r(p)}{\sqrt{p}}\omega'(p)\bigg). $$
Finally, by the Cauchy--Schwarz inequality,
\begin{equation}
\label{CauchySchwarzForRemainders}
\bigg(\sum_{n\leqslant N}r(n)|a_f(n)|\bigg)^2
\leqslant N\sum_{n\leqslant N}r(n)^2\omega(n)\ll 
N\prod_p\big(1+r(p)^2\omega(p)\big).
\end{equation}
Lemma~\ref{TwistedFirstMomentLemma} follows by combining these
estimates.
\end{proof}

We now turn our attention to large values of the product of twisted
$L$-functions of two \emph{distinct} primitive cusp forms $f$ and $g$
of signed levels $r$ and $r'$. We use the same notation as before, including
$\rho$ and $\rho'$.

We begin with an auxiliary lemma.

\begin{lemma}\label{lm-aux}
  With notation as above, there exists a squarefree integer $u\geq 1$
  coprime to $rr'$ such that
$$
 \lambda_g(u)+ \eps(f)\eps(g)
  \frac{\lambda_f(\rho)\lambda_g(\rho')\lambda_f(u)}
  {(\rho\rho')^{1/2}}\not=0.
$$
\end{lemma}

\begin{proof}
  If $\rho$ or $\rho'$ is not $1$, then this holds for $u=1$ (see
  Proposition~\ref{pr-mt}). Otherwise, we need to find $u\geq 1$
  squarefree and coprime to $rr'$ such that
$$
\lambda_g(u)+ \eps(f)\eps(g)\lambda_f(u)\not=0,
$$
and the existence of a prime $u$ with this property follows from
Rankin-Selberg theory and multiplicity one.
\end{proof}

\begin{remark}
We need to involve $u$ in the resonator method, because otherwise it
could be the case that
$$
\frac1{\vphis(q)}\sums_{\chi\bmod q}|R(\chi)|^2
L\big(f\otimes\chi,\tfrac12\big)L\big(g\otimes\bar{\chi},\tfrac12\big)
$$
is zero because of the cancellation between a character and its
conjugate, although the individual terms have no reason to vanish, or
their product to be small (see the last part of
Theorem~\ref{th-second-final}). In that case, the resonator method
would not apply. However, if $u\not=1$, we consider instead
$$
\frac1{\vphis(q)}\sums_{\chi\bmod q}|R(\chi)|^2
L\big(f\otimes\chi,\tfrac12\big)L\big(g\otimes\bar{\chi},\tfrac12\big)\chi(u)
$$
where the symmetry between $\chi$ and $\bar{\chi}$ is broken, leading
to a non-trivial sum.
\end{remark}

We now fix an integer $u$ as given by Lemma~\ref{lm-aux}.

We assume given a multiplicative function $\varpi(n)$, supported on
squarefree positive integers coprime to $urr'$, such that
$$
\sgn\varpi(n)=\sgn\lambda_f(n)=\sgn\lambda_{g}(n) \quad\text{whenever
}\varpi(n)\lambda_f(n)\lambda_{g}(n)\neq 0.  
$$
This gives rise to the non-negative multiplicative functions
$\omega,\omega'_1,\omega'_2$ and the resonator polynomial $R(\chi)$,
defined by
\begin{gather}
\label{DefOmegaOmegaPrime2}
\omega(n)=|\varpi(n)|^2,\quad\omega'_1(n)=\overline{\varpi(n)}\lambda_f^{\ast}(n),\quad \omega_2'(n)=\overline{\varpi(n)}\lambda_{g}^{\ast}(n),\\
\label{ResonatorPolynomial2}
R(\chi)=\sum_{n\leqslant N}r(n)\varpi(n)\chi(n),
\end{gather}
where $\lambda_f^{\ast}(n)$ and $\lambda_{g}^{\ast}(n)$ are the
multiplicative functions defined before
Lemma~\ref{SecondMomentEvaluationForProducts}.

\begin{lemma}
\label{SecondMomentWithResonatorLemma}
With notation as above, assume that $N\leq q^{1/2}$. 
Assume that $r(n)$ and the multiplicative functions
$\omega,\omega'_1,\omega'_2$ defined in \eqref{DefOmegaOmegaPrime2}
satisfy the basic evaluation \eqref{TruncationOK2Products}. Then, with
$\omega'(p)=\omega'_1(p)+\omega'_2(p)$ and with $o^{\star}(1)$ as in
\eqref{oStar1}, we have
\begin{multline*}
  \frac1{\vphis(q)}\sums_{\chi\bmod q}|R(\chi)|^2L\big(f\otimes\chi,\tfrac12\big)L\big(g\otimes\bar{\chi},\tfrac12\big)\chi(u)\\
  =L^{\ast}(f\otimes g,1)\Bigl(\nu+o^{\star}(1)\Bigr)
  \prod_p\bigg(1+r(p)^2\omega(p)+\frac{r(p)\omega'(p)}{\sqrt{p}}\bigg)\\
  +O\bigg(N^{5/2}q^{-1/144}\prod_p\big(1+r(p)^2\omega(p)\big)\bigg),
\end{multline*}
where $\nu\not=0$.
\end{lemma}

\begin{proof}
  Applying the definition or $R(\chi)$ and
  Lemma~\ref{SecondMomentEvaluationForProducts}, it follows that 
\begin{multline*}
  \frac1{\vphis(q)}\sums_{\chi\bmod
    q}|R(\chi)|^2L\big(f\otimes\chi,\tfrac12\big)L\big(g\otimes\bar{\chi},\tfrac12\big)
  \\= \frac1{\vphis(q)}\sum_{n_1,n_2\leqslant
    N}r(n_1)\overline{r(n_2)}\varpi(n_1)\overline{\varpi(n_2)}
  \sums_{\chi\bmod
    q}L\big(f\otimes\chi,\tfrac12\big)L\big(g\otimes\bar{\chi},
  \tfrac12\big)\chi(un_1\bar{n}_2)
  \\
  =\frac{1}{\sqrt{u}} L^{\ast}(f\otimes
  g,1)\Bigl(X+\eps(f)\eps(g)Y\Bigr)+O\bigg(N^{3/2}q^{-1/144}\bigg(\sum_{n\leqslant
    N}r(n)|\varpi(n)|\bigg)^2\bigg)
\end{multline*}
where
$$
X=\lambda_g(u)\sum_{n_1,n_2\leq
  N}r(n_1)\overline{r(n_2)}\varpi(n_1)\overline{\varpi(n_2)}
\frac{\lambda^{\ast}_f(n_2)\lambda^{\ast}_g(n_1)}{(n_1n_2)^{1/2}}
$$
and 
$$
Y=\lambda_f(u)\frac{\lambda_f(\rho)\lambda_g(\rho')}{(\rho\rho')^{1/2}}
\sum_{n_1,n_2\leq
  N}r(n_1)\overline{r(n_2)}\varpi(n_1)\overline{\varpi(n_2)}
\frac{\lambda^{\ast}_f(n_1)\lambda^{\ast}_g(n_2)}{(n_1n_2)^{1/2}}
$$
otherwise. Using the Cauchy--Schwarz inequality and the
multiplicativity of $\omega(n)$ to estimate the resulting sum by a
product over primes (as in \eqref{CauchySchwarzForRemainders}), we see
that the error term is
$$
\ll N^{5/2}q^{-1/144}\prod_p\big(1+r(p)^2\omega(p)\big).
$$
\par
We write
\begin{multline*}
  X=\lambda_g(u)\sum_{n\leqslant N}|r(n)|^2|\varpi(n)|^2
  \\
  \mathop{\sum\sum}_{\substack{m_1,m_2\leqslant N/n\\
      (n,m_1)=(n,m_2) =(m_1,m_2)=1}}
  \frac{r(m_1)\varpi(m_1)\lambda^{\ast}_{g}(m_1)r(m_2)\overline{\varpi(m_2)}\lambda^{\ast}_f(m_2)}{(m_1m_2)^{1/2}}
\end{multline*}
and similarly for $Y$. By the basic
evaluation~\eqref{TruncationOK2Products}, we obtain
$$
X=\lambda_g(u)\big(1+o^{\star}(1)\big)\prod_p\bigg(1+r(p)^2\omega(p)+
\frac{r(p)\omega'(p)}{\sqrt{p}}\bigg),
$$
and similarly
$$
Y=\lambda_f(u)\frac{\lambda_f(\rho)\lambda_g(\rho')}{(\rho\rho')^{1/2}}
\big(1+o^{\star}(1)\big)\prod_p\bigg(1+r(p)^2\omega(p)+
\frac{r(p)\omega'(p)}{\sqrt{p}}\bigg),
$$
(since the function
$\omega'$ plays the same role in both cases). The result follows, with
$\nu\not=0$ by the defining property of
$u$ given by Lemma~\ref{lm-aux}.
\end{proof}


\subsection{Asymptotics involving an amplifier}
In lower ranges for $V$ in
Theorem~\ref{LargeValuesAngularSectorsTheorem}, we will be using an
amplifier instead of a resonator polynomial. In this section, we
prove moment asymptotics that will be useful in this treatment.

We may write
\begin{equation}
\label{DecompositionOfLSeries}
\sum_{\ell=1}^{\infty}\mu^2(\ell)\frac{|\lambda_f(\ell)|^2}{\ell^s}=L(f\otimes f,s)G_f(s),
\end{equation}
where $G_f(s)$ is a certain Euler product absolutely convergent for $\Re(s)>\tfrac12$. The Dirichlet series on the left thus has a simple pole at $s=1$, and we write
\begin{equation}
\label{Definitioncf}
c_f=\res_{s=1}L(f\otimes f,s)\cdot G_f(1)>0.
\end{equation}
In view of \eqref{DecompositionOfLSeries} and \eqref{Definitioncf}, we have the asymptotic
\begin{equation}
\label{Asymptotics}
\sum_{\ell\leqslant L}\mu^2(\ell)\frac{|\lambda_f(\ell)|^2}{\ell}=c_f\log L+O_f(1).
\end{equation}

Let $L\leqslant q$ and
\begin{equation}
\label{DefinitionAmplifier}
A_f(\chi)=\sum_{\ell\leqslant L}\frac{\lambda_f(\ell)}{\sqrt{\ell}}\mu^2(\ell)\chi(\ell).
\end{equation}
We will prove the following two claims.

\begin{lemma}
\label{AmplifierNormalizingWeightLemma}
Let $q$ be a prime modulus, let $L\leqslant q$, and, for every primitive character $\chi$ of conductor $q$, let $A_f(\chi)$ be as in~\eqref{DefinitionAmplifier}. Then, with $c_f>0$ as in \eqref{Definitioncf}, 
\begin{alignat}{3}
\label{SecondMomentAmplifierOnly}
\frac1{\vphis(q)}\sums_{\chi\bmod q}|A_f(\chi)|^2&=c_f\log  L+O_f(1),&&\text{for }L\leqslant q;\\
\label{FourthMomentAmplifierOnly}
\frac1{\vphis(q)}\sums_{\chi\bmod q}|A_f(\chi)|^4&\leqslant c_f^4\log^4L+O_f(\log^3L),&\quad&\text{for }L\leqslant q^{1/2}.
\end{alignat}
\end{lemma}

\begin{proof}
By orthogonality, asymptotic~\eqref{Asymptotics}, and the Cau\-chy--Schwarz inequality,
\begin{align*}
\frac1{\vphis(q)}\sums_{\chi\bmod q}|A_f(\chi)|^2
&=\sum_{\ell\leqslant L}\mu^2(\ell)\frac{|\lambda_f(\ell)|^2}{\ell}-\frac1{\vphis(q)}\bigg|\sum_{\ell\leqslant L}\mu^2(\ell)\frac{\lambda_f(\ell)}{\sqrt{\ell}}\bigg|^2\\
&=c_f\log L+O_f(1).
\end{align*}
Similarly,
\begin{align*}
\frac1{\vphis(q)}\sums_{\chi\bmod q}|A_f(\chi)|^4
&=M_4(f,L)-\frac1{\vphis(q)}\bigg|\sum_{\ell\leqslant L}\mu^2(\ell)\frac{\lambda_f(\ell)}{\sqrt{\ell}}\bigg|^4\\
&=M_4(f,L)+O_f\big(\log^4L/q\big),
\end{align*}
where
\begin{align*}
|M_4(f,L)|
&=\bigg|\sum_{\substack{\ell_1\ell_2=\ell_3\ell_4\\\ell_i\leqslant L}}\mu^2(\ell_1)\mu^2(\ell_2)\mu^2(\ell_3)\mu^2(\ell_4)\frac{\lambda_f(\ell_1)\lambda_f(\ell_2)\overline{\lambda_f(\ell_3)\lambda_f(\ell_4)}}{\sqrt{\ell_1\ell_2\ell_3\ell_4}}\bigg|\\
&\leqslant\bigg(\sum_{\ell\leqslant L}\mu^2(\ell)\frac{|\lambda_f(\ell)|^2}{\ell}\bigg)^4=c_f^4\log^4L+O_f(\log^3L).
\qedhere
\end{align*}
\end{proof}

\begin{lemma}
\label{AmplifierTwistedMomentLemma}
Let $\psi:\Rr/2\pi\Zz\to\Cc$ be a $\pi$-anti-periodic smooth
function. Let $q$ be a prime modulus. Let $0<\theta<\tfrac12$ be an
admissible exponent toward the Ramanujan--Petersson conjecture for
$f$, and let $0<\eta<\tfrac14$ be such that
\begin{equation}
\label{ConditionOnEta}
(1-\eta)(\theta-\tfrac12)+\tfrac12\eta<0.
\end{equation}
Let $L=q^{\eta}$, and, for every primitive character $\chi$ of conductor $q$, let $A_f(\chi)$ be as in \eqref{DefinitionAmplifier}, and let $\theta(\ftchi)$ be as in \eqref{DefinitionSign}. Then, with $c_f>0$ as in \eqref{Definitioncf},
$$ \frac1{\vphis(q)}\sums_{\chi\bmod q}\big|L\big(\tfrac12,f\otimes\chi\big)\big|\overline{A_f(\chi)}\psi(\theta(\ftchi))=c_f\hat{\psi}(1)\log L+O_{f	,\eta}(1). $$
\end{lemma}

Regarding condition \eqref{ConditionOnEta}, we remark that any $\theta<\tfrac12$ is sufficient to obtain this inequality for some $\eta>0$, which is all we really need. On the other hand, $\theta<\tfrac13$ is known, so any $\eta<\tfrac14$ will be acceptable for this condition.

\begin{proof}
Using the evaluation \eqref{TwistedFirstMomentEvaluation}, we have similarly as in Lemma \ref{TwistedFirstMomentLemma} 
\begin{align*}
\frac1{\vphis(q)}\sums_{\chi\bmod q}&\big|L\big(\tfrac12,f\otimes\chi\big)\big|\overline{A_f(\chi)}\psi(\theta(\ftchi))\\
&=\sum_{\ell\leqslant L}\frac{\overline{\lambda_f(\ell)}}{{\ell}^{1/2}}\mu^2(\ell)\sum_{\kappa\in\Zz}\hat{\psi}({1+2\kappa})\varepsilon(f)^{\kappa}\mcl(f;\bar{\ell}r^\kappa,2\kappa)\\
&=\hat{\psi}(1)\sum_{\ell\leqslant L}\mu^2(\ell)\frac{|\lambda_f(\ell)|^2}{\ell}+\hat\psi(-1)\eps(f)\\
&\qquad+O_{f,\eps,A}\bigg(q^{\eps}\big(q^{-1/8}+(q/L)^{\theta-1/2}\big)\|\psi\|_A\bigg(\sum_{\ell\leqslant L}\frac{|\lambda_f(\ell)|}{\sqrt{\ell}}\bigg)\bigg).
\end{align*}
Using the asymptotic~\eqref{Asymptotics} and keeping in mind the
condition~\eqref{ConditionOnEta}, this equals
\begin{gather*}
  c_f\hat{\psi}(1)\log L+O_{f,\eps}\Big(1+q^{\eps}\big(q^{-1/8}+(q/L)^{\theta-1/2}\big)L^{1/2}\Big)\\
  =c_f\hat{\psi}(1)\log L+O_{f,\eta}(1).  \qedhere
\end{gather*}
\end{proof}

\section{Extreme values with angular constraints}
\label{ExtremeValuesAngularConstraintsSection}
By comparing the main terms in Lemmas~\ref{NormalizingWeightLemma} and
\ref{TwistedFirstMomentLemma}, we see that if $N$ is not too large
compared to $q$, we can obtain values of $L(f\otimes\chi,\tfrac12)$
with $\psi(\theta(\ftchi))>0$ as large as the quotient of these main
terms. This quotient has a lower bound provided by
Lemma~\ref{GainLemma} or
Lemma~\ref{ManyLargeValuesLemma}~\eqref{LowerBoundClaim} (depending on
which $r(n)$ is used), which in turn depends on the arithmetic
sequence $a_f(n)$ used in the construction of the resonator polynomial
$R(\chi)$ in \eqref{DefinitionResonatorPolynomial}, subject to the
sign condition \eqref{SignConditionOnafn}. 


In this section, we construct an essentially optimal sequence $a_f(n)$
for exhibiting extreme values of $L(f\otimes\chi,\tfrac12)$, verify
that it is allowable for Lemmas~\ref{SmallTailsLemma} and
\ref{GainLemma}, and then use it to prove the extreme values claim of
Theorem~\ref{LargeValuesAngularSectorsTheorem}.

\subsection{Choice of the resonator polynomial}
\label{OurChoiceSection}

For the purpose of exhibiting extreme values of
$L(f\otimes\chi,\tfrac12)$ in
Theorem~\ref{LargeValuesAngularSectorsTheorem}, we use the resonator
sequence $r(n)$ given by \eqref{DefinitionResonator} that is studied
in Section~\ref{ExtremeValuesRangeSubsection}. Construction of the
resonator polynomial $R(\chi)$ rests on multiplicative arithmetic
factors $a_f(n)$, subject to the sign condition
\eqref{SignConditionOnafn}. From these, we defined
(see \eqref{ResultingOmegas})
\begin{equation}
\label{OmegaOmegaPrimeMustSatisfy}
\omega(n)=|a_f(n)|^2,\quad\omega'(n)=\overline{a_f(n)}\lambda_f(n)
\end{equation}
for square-free $n$.

There are \emph{a priori} many reasonable choices of arithmetic
factors $a_f(n)$ satisfying the sign condition
\eqref{SignConditionOnafn}.  For a moment, we put aside the issue of
actually verifying
conditions~\eqref{SecondMomentCondition}--\eqref{FourthMomentCondition},
and consider the question of optimizing the choice of $\omega(n)$ and
$\omega'(n)$. To get the highest possible lower bound in
Lemma~\ref{GainLemma} for a given $N$ (whose allowable size is in turn
dictated by computations unrelated to the specific application of the
resonator method), one chooses
$$L=\sqrt{a_{\omega}^{-1}\log N\log\log N}$$ in
Lemma~\ref{SmallTailsLemma} and thus obtains in Lemma~\ref{GainLemma}
a lower bound of the shape
$$ 
\exp\bigg(\Bigl(\frac{a'_{\omega'}}{\sqrt{a_{\omega}}}+o(1)\Bigr)
\sqrt{\frac{\log N}{\log\log N}}\bigg). 
$$ 
Maximizing the ration $a'_{\omega'}/\sqrt{a_{\omega}}$ (keeping in
mind the conditions \eqref{SecondMomentCondition},
\eqref{OmegaPrimeLowerBound}, and the definition
\eqref{OmegaOmegaPrimeMustSatisfy}) is tantamount to asymptotically
maximizing the ratio
\begin{equation}
\label{MaximizeThisRatio}
\bigg|\sum_{X\leqslant p\leqslant Y}\frac{\overline{a_f(p)}\lambda_f(p)}{p\log p}\bigg|^2\bigg/\sum_{X\leqslant p\leqslant Y}\frac{|a_f(p)|^2}{p\log p}.
\end{equation}
By the Cauchy--Schwarz inequality, we see that the choice
\begin{equation}
\label{OurChoiceNew}
a_f(n)=\mu^2(n)\lambda_f(n),\quad \omega(n)=\omega'(n)=\mu^2(n)|\lambda_f(n)|^2
\end{equation}
is actually essentially optimal. 

It remains to verify that
conditions~\eqref{SecondMomentCondition}--\eqref{FourthMomentCondition}
are satisfied for this choice. This is the content of the following
lemma, which is a special case of Corollary \ref{PNTsymandpower}.
 
\begin{lemma}
\label{PNTNecessaryForFirstClaim}
For any primitive cusp form $f$ with trivial central character, we
have for $4\leq 2X\leq Y$
\begin{gather*}
  \sum_{p\leqslant X}\frac{|\lambda_f(p)|^2\log p}{p}=\log X+O_f(1),\\
  \sum_{X\leqslant p\leqslant Y}\frac{|\lambda_f(p)|^2}{p\log p}=\left(\frac1{\log X}-\frac1{\log Y}\right)+O_f\left(\frac1{\log^2X}\right),\\
  \sum_{X\leqslant p\leqslant Y}\frac{|\lambda_f(p)|^2}{p}=\log\Bigl
  (\frac{\log Y}{\log X}\Bigr)+O_f\left(\frac1{\log X}\right),\\
  \sum_{p\leqslant X}|\lambda_f(p)|^4\ll_f\frac{X}{\log X}.
\end{gather*}
\end{lemma}

This lemma shows that the multiplicative arithmetic
functions 
$$
\omega(n)=\omega'(n)=\mu^2(n)|\lambda_f(n)|^2
$$ 
do satisfy the
conditions~\eqref{SecondMomentCondition}--\eqref{FourthMomentCondition}
and \eqref{UpperBoundsConditions}--\eqref{UpperBoundCondition} with
$\delta=\delta'=1$ and
$$ a_{\omega}=a'_{\omega'}=1,\quad b_{\omega}=b_{\omega
  2}=b'_{\omega}=1. $$

\subsection{The extreme values claim in
  Theorem~\ref{LargeValuesAngularSectorsTheorem}}
\label{ProofLowerBoundSection}

In this section, we prove the first part of
Theorem~\ref{LargeValuesAngularSectorsTheorem}, which is concerned
with extreme values of $L(f\otimes\chi,\frac12)$.

We use a resonator polynomial \eqref{DefinitionResonatorPolynomial},
with the resonator sequence as in \eqref{DefinitionResonator}, and
arithmetic factors $a_f(n)$ as in \eqref{OurChoiceNew}. According to
the previous section, the resulting multiplicative arithmetic
functions $$\omega(n)=\omega'(n)=\mu^2(n)|\lambda_f(n)|^2$$ satisfy
the
conditions~\eqref{SecondMomentCondition}--\eqref{FourthMomentCondition}
with $a_{\omega}=a'_{\omega'}=1$. According to
Lemma~\ref{SmallTailsLemma}, $\omega$ and $\omega'$ satisfy the basic
evaluations~\eqref{TruncationOK1} and \eqref{TruncationOK2}.

As in Lemma~\ref{TwistedFirstMomentLemma}, choose an arbitrary smooth $\pi$-anti-periodic function
$\psi:\Rr/2\pi\to\Cc$ such that
\begin{equation}
\label{AngleCutoffFunction}
\supp \psi\cap \left]-\tfrac{\pi}2,\tfrac{\pi}2\right[\subseteq I,\quad \psi|_I\geqslant 0,\quad \int_I\psi(\theta)\,\text{d}\theta=1.
\end{equation}
In particular, we have then $I(\psi)>0$.

Fix an arbitrary $\delta>0$, and apply Lemmas~\ref{NormalizingWeightLemma} and \ref{TwistedFirstMomentLemma} with
$$ N=q^{1/8-\delta} $$
and $L=\sqrt{\log N\log\log N}$. Using the available estimate $\theta<\frac{5}{14}$, we have that $N(q/N)^{\theta-1/2}=q^{(\theta-1/2)(7/8+\delta)+1/8-\delta}<q^{-(8/7)\delta}<q^{-\delta}$. Therefore, Lemmas~\ref{NormalizingWeightLemma} and \ref{TwistedFirstMomentLemma} give
$$
\frac1{\vphis(q)}\sums_{\chi\bmod
  q}|R(\chi)|^2\sim\prod_p\big(1+r(p)^2|\lambda_f(p)|^2\big),
$$
and
\begin{align*}
  &\frac1{\vphis(q)}\sums_{\chi\bmod q}|R(\chi)|^2\big|L\big(f\otimes\chi,\tfrac12\big)\big|\psi(\theta(f\otimes\chi))\\
  &\qquad=I(\psi)\big(1+o(1))\cdot\prod_p\bigg(1+r(p)^2|\lambda_f(p)|^2+\frac{r(p)}{\sqrt{p}}|\lambda_f(p)|^2\bigg)\\
  &\qquad\qquad+O\Big(q^{-\delta+\eps}\cdot\prod_p\big(1+r(p)^2|\lambda_f(p)|^2\big)\Big)\\
  &\qquad\sim I(\psi)\msl\,\prod_p\big(1+r(p)^2|\lambda_f(p)|^2\big),
\end{align*}
with implicit constants depending on $f$, $\delta$, $\eps$, and $\psi$, and
$$ \msl=\prod_p\bigg(1+\frac{r(p)|\lambda_f(p)|^2}{\sqrt{p}\big(1+r(p)^2|\lambda_f(p)|^2\big)}\bigg). $$
It follows that, for sufficiently large $q$, there exists at least one
primitive character $\chi$ of conductor $q$ such that
$\psi(\theta(f\otimes\chi))>0$ (and so \emph{a fortiori}
$\theta(f\otimes\chi)\in I$) 
and
$$ \big|L\big(f\otimes\chi,\tfrac12\big)\big|\gg\msl. $$

Finally, applying Lemma~\ref{GainLemma}, and keeping in mind the
present choices $$L=\sqrt{\log N\log\log N}$$ and $N=q^{1/8-\delta}$ (made below \eqref{AngleCutoffFunction})
and $a'_{\omega'}=1$, we obtain a lower bound
$$ \msl\gg\exp\bigg(\big(1+o^{\star}(1)\big)\frac{L}{2\log L}\bigg)=\bigg(\big(\tfrac1{\sqrt8}-\delta+o^{\star}(1)\big)\sqrt{\frac{\log q}{\log\log q}}\bigg), $$
with $o^{\star}(1)=O(\log\log\log q/\log\log q)$. The omega-statement
of Theorem~\ref{LargeValuesAngularSectorsTheorem} follows since we may
take $\delta>0$ as small as we please.

\subsection{Many large values}
\label{ManyLargeValuesSection}
In this section, we prove the second part of Theorem~\ref{LargeValuesAngularSectorsTheorem} about the number of primitive characters $\chi$ of conductor $q$ such that $|L(f\otimes\chi,\tfrac12)|\geqslant e^V$ for a sizable $V$. The argument is an adaptation of that in \cite{Soundararajan2008}; here we focus on the specific requirements on the sequences $\omega(n)$ and $\omega'(n)$ and on the few aspects that require some modification (such as the treatment of moderately large $V$).

Let $c_f>0$ be as in \eqref{Definitioncf}, and let $0<\eta<\tfrac14$
satisfy the condition \eqref{ConditionOnEta}. Choose an arbitrary
smooth $\pi$-anti-periodic function $\psi:\Rr/2\pi\Zz\to\Cc$ as in
\eqref{AngleCutoffFunction}. We have then, $|\hat{\psi}(1)|>0$. Let
$\tilde{c}=\tilde{c}_{f,\psi,\eta}>0$ be an arbitrary constant such
that
\begin{equation}
\label{OurChoicecfgeta}
\tilde{c}<\sqrt{c_f\eta}\,
\frac{|\hat{\psi}(1)|}
{\,\,2\|\psi\|_{\infty}}.
\end{equation}
We consider two cases, depending on $V$.
\par
\medskip
\par
\textbf{Case 1}. \textit{The range
  $V\leqslant\frac12\log\log q+\log\tilde{c}$}.  In this range, the
second part of Theorem~\ref{LargeValuesAngularSectorsTheorem} states
that $|L(f\otimes\chi,\frac12)|$ achieves moderately high values for a
very large number of $\chi$. This is in a sense a complementary range;
instead of the method of resonators, we prove
Theorem~\ref{LargeValuesAngularSectorsTheorem} by a comparison of
moments, including the amplified first moment as follows.

Let $L=q^{\eta}$, and, for every primitive character $\chi$ of
conductor $q$, let the amplifier $A_f(\chi)$ be as in
\eqref{DefinitionAmplifier}. Then, according to
Lemma~\ref{AmplifierTwistedMomentLemma},
\begin{align*}
  I_{1,f,\psi,L}(q)
  &=\frac1{\vphis(q)}\sums_{\chi\bmod
    q}\big|L(\tfrac12,f\otimes\chi\big)\big|\overline{A_f(\chi)}\psi(\theta({f\otimes\chi}))\\
  &=c_f\eta\hat{\psi}(1)\log q+O_{f,\psi,\eta}(1).
\end{align*}

Note that, in this range, $e^V\leqslant\tilde{c}\sqrt{\log q}$. We
split the sum
$$
I_{1,f,g,L}(q)
=I_{1,f,\psi,L}^0(q)+I_{1,f,\psi,L}^{+}(q)
$$
where $I_{1,f,\psi,L}^0$ restricts to those $\chi$ such that
$|L(f\otimes\chi,\frac12)|\leqslant\tilde{c}\sqrt{\log q}$.
By the Cauchy--Schwarz inequality, we have
\begin{align*}
  |I_{1,f,\psi,L}^0(q)|
  &\leqslant\tilde{c}\sqrt{\log
    q}\|\psi\|_{\infty}\ \frac1{\vphis(q)}
    \sums_{|L(f\otimes\chi,1/2)|\leqslant \tilde{c}\sqrt{\log q}}
    |A_f(\chi)|\\
  &\leqslant\tilde{c}\sqrt{\log q}\|\psi\|_{\infty}\bigg(\frac1{\vphis(q)}\sums_{\chi\bmod q}|A_f(\chi)|^2\bigg)^{1/2}.
\end{align*}
Using \eqref{SecondMomentAmplifierOnly} of
Lemma~\ref{AmplifierNormalizingWeightLemma} and recalling that
$\tilde{c}$ satisfies \eqref{OurChoicecfgeta}, we deduce that
\begin{align*}
  |I_{1,f,\psi,L}^0(q)|
  &\leqslant \tilde{c}\sqrt{c_f\eta}\|\psi\|_{\infty}\log q+O_{f,\psi}(1)\\
  &\leqslant\tfrac12c_f|\hat{\psi}(1)|\eta\log q\leqslant\tfrac12|I_{1,f,\psi,L}(q)|
\end{align*}
for sufficiently large $q$. This shows that, for sufficiently large $q$,
\begin{equation}
\label{MostOfItIsLeft}
|I_{1,f,\psi,L}^{+}(q)|\geqslant\tfrac12 |I_{1,f,\psi,L}(q)|=\tfrac12c_f\eta|\hat{\psi}(1)|\log q+O_{f,\psi,\eta}(1).
\end{equation}

On the other hand, using H\"older's inequality, we estimate
\begin{multline}
\label{ResultOfHolder}
|I_{1,f,\psi,L}^{+}(q)| \leqslant \|\psi\|_{\infty}\frac1{\vphis(q)}
\sums_{\substack{|L(f\otimes\chi,1/2)|>\tilde{c}\sqrt{\log q}\\
    g(\theta(\ftchi))\neq
    0}}\big|L\big(f\otimes\chi,\tfrac12\big)\big||A_f(\chi)|
\\
\leqslant \|\psi\|_{\infty} \bigg(\frac1{\vphis(q)}\sums_{\chi\bmod q}
\big|L\big(f\otimes\chi,\tfrac12\big)\big|^2\bigg)^{1/2}
\bigg(\frac1{\vphis(q)} \sums_{\chi\bmod q}|A_f(\chi)|^4\bigg)^{1/4}
\\
\times\bigg (\frac1{\vphis(q)}\Big|\big\{\chi\mods{q}\,\mid\,
\big|L\big(\tfrac12,f\otimes\chi\big)\big|>\tilde{c}\sqrt{\log
  q},\,\,\theta(\ftchi)\in I\big\}\Big|\bigg)^{1/4}.
\end{multline}
\par
Combining \eqref{MostOfItIsLeft}, \eqref{ResultOfHolder}, the second
moment evaluation of Theorem~\ref{th-second-final}, and
\eqref{FourthMomentAmplifierOnly} of
Lemma~\ref{AmplifierNormalizingWeightLemma}, we conclude that, for
sufficiently large $q$,
$$
\Big|\big\{\chi\mods{q}\,\mid\,
\big|L\big(f\otimes\chi,\tfrac12\big)\big|>\tilde{c}\sqrt{\log
  q},\,\theta(\ftchi)\in
I\big\}\Big|\gg_{f,\psi,I}\frac{\vphis(q)}{\log^2q},
$$
which more than suffices for the second part of
Theorem~\ref{LargeValuesAngularSectorsTheorem} for
$$
3\leqslant V\leqslant\tfrac12\log\log q+\log\tilde{c}
$$
and any $\eta<\tfrac14$.

\begin{remark}
  In place of H\"older's inequality above, one could use the
  Cauchy--Schwarz inequality and then estimate from above the
  amplified second moment for sufficiently small $\eta>0$; this would
  yield a lower bound of the same form save for the numerical values
  of various constants. We chose the above treatment which is softer
  and perhaps more universally applicable.
\end{remark}

\par
\medskip
\par
\textbf{Case 2}. \textit{The range
  $\tfrac12\log\log q+\log\tilde{c}\leqslant V\leqslant
  \tfrac3{14}\sqrt{\log q/\log\log q}$}.  In this principal range, we
use the resonator method to prove
Theorem~\ref{LargeValuesAngularSectorsTheorem}, proceeding analogously
as in \cite{Soundararajan2008}. We will be using a resonator sequence
$r(n)$ of type \eqref{DefinitionResonator2}, as studied in
Section~\ref{ManyHighValuesRangeSubsection}. For the arithmetic
factors $a_f(n)$ in the resonator polynomial $R(\chi)$, we make the
same choice as in \eqref{OurChoiceNew}, namely
\begin{equation}
\label{ChoiceOfaf}
a_f(n)=\mu^2\lambda_f(n),\quad \omega(n)=\omega'(n)=\mu^2(n)|\lambda_f(n)|^2,
\end{equation}
which satisfies the sign condition~\eqref{SignConditionOnafn}. As in
Section~\ref{OurChoiceSection}, this choice is essentially optimal: an
inspection of \eqref{SpecificChoiceOfA} shows that in generic ranges
it allows a choice $A\approx V/(b'_{\omega'}\log Q)$ and eventually to
the lower bound
$$ \gg_{f,\psi,c_1}\vphi^{\ast}(q) \exp\left(-12(b_{\omega 2}/b^{\prime 2}_{\omega'})\big(V^2/\log Q\big)\right) $$
in \eqref{ResultingLowerBoundForNumberOfChi}. Minimizing the constant
in this estimate is tantamount to asymptotically maximizing the same
ratio \eqref{MaximizeThisRatio} as in Section \ref{OurChoiceSection},
and by the Cauchy--Schwarz inequality leads to the same asymptotically
optimal choice \eqref{ChoiceOfaf}.

We have verified after Lemma~\ref{PNTNecessaryForFirstClaim} that the
choice \eqref{ChoiceOfaf} satisfies conditions
\eqref{UpperBoundsConditions}--\eqref{UpperBoundCondition} with
$b_{\omega}=b_{\omega 2}=b'_{\omega'}=1$.

Fix an arbitrary $\delta>0$ (which will be chosen suitably small under~\eqref{SpecificChoiceOfA}), and as in Section~\ref{ProofLowerBoundSection} let
$$ N=q^{1/8-\delta}. $$
Using the available estimate $\theta<\frac{5}{14}$, we have that $N(q/N)^{\theta-1/2}<q^{-\delta}$ as in Section \ref{ProofLowerBoundSection}. Further, set
$$ c<1 \quad\text{and}\quad X_0\gg_f (1-c)^{-1},\,\, A\ll_{f,c}\sqrt{\log N}\,\,\text{ as in \eqref{ConditionsConditions}}, $$
where $A$ will be suitably chosen later, $c<1$ will be chosen suitably close to $1$ under~\eqref{SpecificChoiceOfA}, and we additionally take $X_0$ sufficiently large (depending on $f$, $c$ only) so that the term $o^{\star}(1)=O\big(e^{-\tilde{\delta}X_0^2}\big)$ in Lemma~\ref{ManyLargeValuesLemma}\eqref{BasicEvaluationsClaim} is $\leqslant\tfrac1{10}$. With these choices, let the resonator sequence $r(n)$ be as in \eqref{DefinitionResonator2} and the arithmetic factors $a_f(n)$ and the resulting multiplicative functions $\omega(n),\omega'(n)$ be as in \eqref{ChoiceOfaf}, and define the resonator polynomial $R(\chi)$ as in \eqref{DefinitionResonatorPolynomial}. Let
\begin{align*}
M_1(q)&=\frac1{\vphis(q)}\sums_{\chi\bmod q}|R(\chi)|^2,\\
M_{2,f,\psi}(q)&=\frac1{\vphis(q)}\sums_{\chi\bmod q}|R(\chi)|^2\big|L\big(f\otimes\chi,\tfrac12\big)\big|\psi(\theta(\ftchi)).
\end{align*}

According to Lemma~\ref{ManyLargeValuesLemma}\eqref{BasicEvaluationsClaim}, the basic evaluations \eqref{TruncationOK1} and \eqref{TruncationOK2} hold. Lemmas~\ref{NormalizingWeightLemma} and \ref{TwistedFirstMomentLemma} then give
\begin{equation}
\label{MomentsEvaluatedReadyToBeCompared}\begin{aligned}
  M_1(q)&=\big(1+o^{\star}(1)+O(q^{-7/8})\big)\cdot NW,\\
  M_{2,f,\psi}(q)&=\big[I(\psi)(1+o^{\star}(1))\msa+O_{f,\psi,\eps}(q^{-\delta+\eps})
  \big]\cdot NW,
\end{aligned}
\end{equation}
where $o^{\star}(1)=O\big(e^{\tilde{\delta}A_0^2}\big)$ and, according
to Lemma~\ref{ManyLargeValuesLemma}\eqref{LowerBoundClaim} and
\eqref{UpperBoundOnMomentsOfr},
\begin{equation}
\label{MsaLowerBound}
\begin{aligned}
  NW&=\prod_p\big(1+r(p)^2|\lambda_f(p)|^2\big)\\
  &\qquad\leqslant\exp\bigg(A^2\log\frac{c\log N}{2A_0^2\log A_0}+O_f\bigg(\frac{A^2}{\log A_0}\bigg)\bigg),\\
  \msa&=\prod_p\bigg(1+\frac{r(p)|\lambda_f(p)|^2}{\sqrt{p}(1+r(p)^2|\lambda_f(p)|^2)}\bigg)\\
  &\qquad\geqslant\exp\bigg(A\log\frac{c\log N}{2A_0^2\log
    A_0}+O_f\bigg(\frac{A}{\log A_0}\bigg)\bigg).
\end{aligned}
\end{equation}
In particular, since our choice of $X_0$ ensures that $o^{\star}(1)\leqslant\frac1{10}$ in \eqref{MomentsEvaluatedReadyToBeCompared}, we have that for sufficiently large $q$,
\begin{equation}
\label{Combine1}
\tfrac45I(\psi)\msa M_1(q)\leqslant M_{2,f,\psi}(q).
\end{equation}

Let $C_f$ denote an implicit constant sufficient for both asymptotics in \eqref{MsaLowerBound}.
In the following claim, we now specify our choice of the parameter $A$.

\begin{claim}
With suitable $\delta>0$ and $0<c<1$, for sufficiently large $q$ we can choose $A\ll_{f,c}\sqrt{\log N}$ satisfying
\begin{equation}
\label{SpecificChoiceOfA}
\begin{aligned}
&A\leqslant\frac{V}{\log Q}, \quad Q=\frac{c\log N}{2V^2\log V}\qquad\text{such that}\\
&A\log\frac{c\log N}{2A_0^2\log A_0}=V+\log\frac{2\|\psi\|_{\infty}}{I(\psi)}+C_f\frac{A}{\log A_0}.
\end{aligned}
\end{equation}
\end{claim}

For easier reading, we postpone the proof of this technical claim to the end of this section. With our choice of $A$ satisfying \eqref{SpecificChoiceOfA}, we find that
\begin{align*}
  I(\psi)\msa&\geqslant\exp\left(A\log\frac{c\log N}{2A_0^2\log A_0}+\log I(\psi)-C_f\frac{A}{\log A_0}\right)\\
             &=\exp\big(V+\log(2\|\psi\|_{\infty})\big)=2\|\psi\|_{\infty}e^V.
\end{align*}

Then, separating the summands in $M_{2,f,\psi}(q)$ according to
whether we have $L(f\otimes\chi,\frac12)\leqslant e^V$ or not, we can
write
\begin{equation}
\label{Combine2}
M_{2,f,\psi}(q)=M_{2,f,\psi}^0(q)+M_{2,f,\psi}^{+}(q),
\end{equation}
where
\begin{equation}
\label{Combine3}
|M_{2,f,\psi}^0(q)|
\leqslant e^V\|\psi\|_{\infty}\frac{1}{\vphis(q)}\sums_{\chi:|L(f\otimes\chi,1/2)|\leqslant e^V}|R(\chi)|^2
\leqslant \tfrac12 I(\psi)\msa M_1(q).
\end{equation}
Combining \eqref{Combine1}, \eqref{Combine2}, \eqref{Combine3}, and
H\"older's inequality, we then deduce that
\begin{multline*}
  \frac3{10}I(\psi)\msa M_1(q)\leqslant |M_{2,f,\psi}^{+}(q)|
  \leqslant\|\psi\|_{\infty}\frac1{\vphis(q)}
  \sums_{\substack{|L(f\otimes\chi,1/2)|>e^V,\\
      \psi(\theta(\ftchi))\neq 0}}
  |R(\chi)|^2\big|L\big(f\otimes\chi,\tfrac12\big)\big|\\
  \leqslant\|\psi\|_{\infty}\bigg(\frac1{\vphis(q)} \sums_{\chi\bmod
    q}|R(\chi)|^8\bigg)^{1/4} \bigg(\frac1{\vphis(q)}
  \sums_{\chi\bmod q}\big|L\big(f\otimes\chi,\tfrac12\big)\big|^2\bigg)^{1/2}\\
  \times\bigg(\frac1{\vphis(q)}\Big|
  \big\{\chi\mods{q}\,\mid\,\big|L\big(f\otimes\chi,\tfrac12\big)\big|>e^V,\,\theta(\ftchi)\in
  I\big\}\Big|\bigg)^{1/4}.
\end{multline*}
According to part \eqref{UpperBoundOnMomentsOfr} of Lemma~\ref{ManyLargeValuesLemma}, and using
our evaluation of the second moment of twisted $L$-functions, we
conclude that
\begin{multline*}
  \Big|\big\{\chi\mods{q}\,\mid\,
  \big|L\big(f\otimes\chi,\tfrac12\big)\big|>e^V,\,\theta(\ftchi)\in
  I\big\}\Big| \gg_{f,\psi}\frac{\vphis(q)}{\log^2q}\big(\msa
  M_1(q)\big)^4\frac{1}{NW^{16}}
  \\
  \gg_{f,\psi}\frac{\vphis(q)}{\log^2q}\exp\bigg((4A-12A^2)\log\frac{c\log
    N}{2A_0^2\log A_0}-C_f\frac{4A+12A^2}{\log A_0}\bigg).
\end{multline*}
With our choice \eqref{SpecificChoiceOfA}, the right-hand side of this estimate is
\begin{align}
  &\gg_{f,\psi,c}\frac{\vphis(q)}{\log^2q}\exp\big((4-12A)(V+M_{\psi})-12C_fA^2/\log A_0\big)
    \nonumber  \\
  &
    \gg_\psi\vphis(q)\exp\bigg(-12(1+o^{\star}(1))\frac{V^2}{\log
    Q}\bigg),
\label{ResultingLowerBoundForNumberOfChi}
\end{align}
with
$o^{\star}(1)=O_{f,\psi}(1/V+1/(\log Q\log A_0))=O_{f,\psi}(1/\log\log
q)$.
This concludes the proof of Case 2, hence of the theorem. It only remains to prove
the technical claim.
\begin{proof}[Proof of Claim]
Let $A_1=V/\log Q$, $M_{\psi}=\log(2\|\psi\|_{\infty}/I(\psi))$, and
$$ \phi(A):=A\left(\log\frac{c\log N}{2A_0^2\log A_0}-\frac{C_f}{\log A_0}\right). $$
We will verify that, for sufficiently large $q$, $\phi(A_1)>V+M_{\psi}$; this is tedious but not difficult. Note that $2.72<\tfrac{196}{72}c(1-8\delta)<Q\leqslant c\log N$ for a sufficiently small $\delta>0$ and a $c<1$ suitably close to $1$, and in particular we have, with an absolute $\delta_{\mathrm{ll}}>0$, 
\begin{gather*}
1<\log Q,\ A_1<V,\ X_0<V,\ 0<\delta_{\mathrm{ll}}\leqslant\log\log Q,\\
 \log Q\leqslant\log\log q+O(1)<(2+\delta)V	
\end{gather*}
 for sufficiently large $q$. 

Let $X_1=\max(X_0,M_{\psi}/\delta_{\mathrm{ll}})$.
If $A_1\geqslant X_0$, then $$A_0=A_1, 2A_0^2\log A_0\leqslant 2V^2\log V/\log^2Q,$$
 and thus
$$ \phi(A_1)\geqslant\frac{V}{\log Q}\left(\log Q+2\log\log Q-\frac{C_f}{\log A_0}\right)=V+\frac{V}{\log Q}\left(2\log\log Q-\frac{C_f}{\log A_0}\right).$$

This is in particular the case if $V\geqslant (1+\delta)X_1\log\log q$, when for sufficiently large $q$, $A_1\geqslant X_1$. If $\log\log Q\geqslant C_f/\log X_0$, then the second term exceeds $\delta_{\mathrm{ll}}V/\log Q=\delta_{\mathrm{ll}}A_1\geqslant M_{\psi}$, and we are done. If $\log\log Q<C_f/\log X_0$, we have that $V^2\log V\gg_{C_f,X_0,c}\log N$ and thus $\log A_1\geqslant(\frac12-o_{C_f,X_0,c}(1))\log\log q$, and for sufficiently large $q$ the second term exceeds $V(\log\log Q/\log Q)\gg_{C_f,X_0}V>M_{\psi}$.

For $V\leqslant (1+\delta)X_1\log\log q$, we have that $\log Q=(1+o(1))\log\log q$. If $A_1\geqslant X_0$, then the above lower bound holds, with the second term $\geqslant (2+\delta)^{-1}((2+o(1))\log\log\log q-C_f/\log X_0)>M_{\psi}$ for sufficiently large $q$. Otherwise $A_0=X_0$ and 
$$ \phi(A_1)\geqslant V+\frac{V}{\log Q}\left(2\log(V/X_0)-\frac{C_f}{\log X_0}\right), $$
and the second term is again $\geqslant(2+\delta)^{-1}((2+o(1))\log\log\log q-C_f/\log X_0)>M_{\psi}$ for sufficiently large $q$.

On the other hand, for $Q_0=(c\log N)/(2X_0^2\log X_0)$ and $A_2=(V+M_{\psi})/\log Q_0$, we clearly have that $\phi(A_2)\leqslant V+M_{\psi}$. Thus the existence of $A\in [A_2,A_1]$ satisfying \eqref{SpecificChoiceOfA} follows simply by continuity. Note that $A\leqslant V/\log Q<V$ guarantees the required condition $A\ll_{f,c}\sqrt{\log N}$ for sufficiently large $q$.
\end{proof}

\section{Large values of products}
\label{LargeValuesProductsSection}
In this section, we prove
Theorem~\ref{LargeValuesOfProductsTheorem}. With the resonator
sequence of the form~\eqref{DefinitionResonator},
Section~\ref{SubsectionProductsResonator} is devoted to the
construction of arithmetic factors and verification that they satisfy
conditions for the application of the resonator method. The proof of
Theorem~\ref{LargeValuesOfProductsTheorem} then follows in
Section~\ref{SubsectionProductsProof}.

\marginpar

\subsection{Choice of the resonator polynomial}
\label{SubsectionProductsResonator}

Let $\lambda^\ast_{f}$ and $\lambda^\ast_{g}$ be the multiplicative functions supported on squarefree integers and defined by \eqref{lambdaastdef} and let 
$$ \mathcal{G}:=\big\{n\geq 1\,\mid\,
\lambda^{\ast}_f(n)\lambda^{\ast}_g(n)\not=0,\
\sgn(\lambda^{\ast}_f(n))=\sgn(\lambda^{\ast}_g(n)\big\}. $$

We construct a multiplicative arithmetic function $\varpi(n)$ supported on square-free positive integers, subject to the condition
\begin{equation}
\label{SignConditionOnVarpi}
\begin{cases} \varpi(n)=0,&n\not\in\mathcal{G},\\ \varpi(n)=0\text{ or }\sgn\varpi(n)=\sgn\lambda_f^{\ast}(n)=\sgn\lambda_g^{\ast}(n), &n\in\mathcal{G}.\end{cases}
\end{equation}
This multiplicative function $\varpi(n)$ is entirely determined by the sequence of values $\varpi(p)$.
We base our construction of the sequence $\varpi(p)$ on the simple observation that, for any two $x,y\in\Rr$,
$$ xy(x+y)^2=0\quad\text{or}\quad\sgn(xy)=\sgn\big(xy(x+y)^2\big). $$
In particular,
$$ \lambda^{\ast}_f(p)\lambda^{\ast}_g(p)\big(\lambda_f^{\ast}(p)+\lambda_g^{\ast}(p)\big)^2>0\,\Longrightarrow\,p\in\mathcal{G}. $$
We define
\begin{equation}
\label{DefinitionOfVarpi}
\varpi(p)=\begin{cases} \sgn(\lambda^{\ast}_f(p))\lambda^{\ast}_f(p)\lambda^{\ast}_g(p)\big(\lambda_f^{\ast}(p)+\lambda_g^{\ast}(p)\big), &p\in\mathcal{G},\\ 0,&p\not\in\mathcal{G}. \end{cases}
\end{equation}

Then, the multiplicative arithmetic function $\varpi(n)$ satisfies the sign condition~\eqref{SignConditionOnVarpi}, and hence
\begin{equation}
\label{DefinitionOmegaOmegaPrime3}
\omega(n)=|\varpi(n)|^2,\quad \omega'(n)=\varpi(n)\prod_{p\mid n}\big(\lambda^{\ast}_f(p)+\lambda^{\ast}_g(p)\big)\geqslant 0
\end{equation}
are non-negative multiplicative functions supported within $\mathcal{G}$. In the following lemma, we verify that $\omega$, $\omega'$ satisfy the conditions for application of the resonator method.

\begin{lemma}
\label{VerificationOfTheProductResonator}
Let $f$, $g$ be two primitive cusp forms of signed levels $r$ and $r'$
respectively and trivial central character, and let $\varpi$ be the
multiplicative arithmetic function supported on square-free integers
and defined on primes by \eqref{DefinitionOfVarpi}; in particular,
$\varpi$ satisfies \eqref{SignConditionOnVarpi}.

Then, for every $\tilde{\delta}>0$ there exists an
$X_0=X_0(\tilde{\delta},f,g)$ such that, for every $X\geqslant X_0$
and every $Y\geqslant 2X$, the non-negative multiplicative functions
$\omega,\omega'\geqslant 0$ supported on square-free integers and
defined by \eqref{DefinitionOmegaOmegaPrime3} satisfy conditions
\eqref{SecondMomentCondition}--\eqref{FourthMomentCondition} with
$\delta=\tfrac14$, $\delta'=\tfrac13$, and
\begin{equation}
\label{ValuesForAVarpis}
\begin{aligned}
  a_{\omega}&=n_{4,2}+2n_{3,3}+n_{2,4}+4\tilde{\delta},\\
  a'_{\omega'}&=n_{3,1}+2n_{2,2}+n_{1,3}-4\tilde{\delta},
\end{aligned}
\end{equation} 
where $n_{k,k'}$ are the non-negative integers defined in
Corollary~\ref{PNTsymandpower}. We have
$$
n_{3,1}+2n_{2,2}+n_{1,3}\geq 2n_{2,2}\geq 2.
$$
\end{lemma}


\begin{proof}
  We use Corollary~\ref{corlambdaast} with various values of the
  parameters $(k,k')$ and the given $\tilde{\delta}>0$ to verify that
  \eqref{SecondMomentCondition}--\eqref{FourthMomentCondition} are
  satisfied with the stated values of the parameters $a_{\omega}$ and
  $a'_{\omega'}$. Choose $X_0 = X_0(\tilde{\delta},f,g)\geq 4$ so that
$$
\sum_{x\leq p\leq y}\frac{\lf(p)^k\lamg(p)^{k'}}{p\log p} \geq \big(
  n_{k,k'} -\tilde{\delta}\big)\left(\frac{1}{\log x}-\frac{1}{\log
  y}\right)
$$ 
and
$$ \Big| \sum_{x\leq p\leq y}\frac{\lf(p)^k\lamg(p)^{k'}}{p\log p}\Big|
\leq \big( n_{k,k'} +\tilde{\delta}\big)\left(\frac{1}{\log
  x}-\frac{1}{\log y}\right)
$$ 
for $1\leq k$, $k'\leq 4$, and for $y\geq 2x \geq 2X_0\geq 4$, which
is possible by \eqref{powerresonator}.

First, we find that for $Y\geq 2X\geq 42$, we have
\begin{align*}
  \sum_{X\leqslant p\leqslant Y}\frac{\omega(p)}{p\log p}
  &\leq \sum_{X\leqslant p\leqslant Y}\frac{\lambda_f^{\ast}(p)^2\lambda_g^{\ast}(p)^2(\lambda_f^{\ast}(p)^2+2\lambda_f^{\ast}(p)\lambda_g^{\ast}(p)+\lambda_g^{\ast}(p)^2)}{p\log p}\\
  &\leqslant \big(n_{4,2}+2n_{3,3}+n_{2,4}+4\tilde{\delta}\big)\left(\frac1{\log X}-\frac1{\log Y}\right).
\end{align*}
This verifies \eqref{SecondMomentCondition} with
$a_{\omega}=n_{4,2}+2n_{3,3}+n_{2,4}+4\tilde{\delta}$.

We proceed to the proof of \eqref{OmegaPrimeLowerBound}. Keeping in
mind
that
$$
\lambda_f^{\ast}(p)\lambda_g^{\ast}(p)(\lambda_f^{\ast}(p)
+\lambda_g^{\ast}(p))^2\leqslant 0
$$ 
if $p\not\in\mathcal{G}$, we deduce that
\begin{align*}
\sum_{X\leqslant p\leqslant Y}\frac{\omega'(p)}{p\log p}
&=\sum_{\substack{X\leqslant p\leqslant Y\\p\in\mathcal{G}}}\frac{\lambda_f^{\ast}(p)\lambda_g^{\ast}(p)\big(\lambda_f^{\ast}(p)+\lambda_g^{\ast}(p)\big)^2}{p\log p}\\
&\geqslant\sum_{X\leqslant p\leqslant Y}\frac{\lambda_f^{\ast}(p)\lambda_g^{\ast}(p)(\lambda_f^{\ast}(p)^2+2\lambda_f^{\ast}(p)\lambda_g^{\ast}(p)+\lambda_g^{\ast}(p)^2)}{p\log p}\\
&\geqslant \big(n_{3,1}+2n_{2,2}+n_{1,3}
  -4\tilde{\delta}\big)\left(\frac1{\log X}-\frac1{\log Y}\right)
\end{align*}
for $Y\geq 2X\geq 4$.  This verifies \eqref{OmegaPrimeLowerBound} with
$a'_{\omega'}=n_{3,1}+2n_{2,2}+n_{1,3}-4\tilde{\delta}\geq
2-4\tilde{\delta}.$

Finally, we check that~\eqref{FourthMomentCondition} holds. Recalling
the choice $\delta=\tfrac14$ and using a simple dyadic subdivision, we
first find that
\begin{align*}
  \sum_{p\leqslant X}\omega(p)^{\delta}\omega'(p)
  &\leqslant\sum_{p\leqslant X}\big|\lambda_f^{\ast}(p)\lambda_g^{\ast}(p)\big|^{1+2\delta}\big(|\lambda_f^{\ast}(p)|+|\lambda_g^{\ast}(p)|\big)^{2+2\delta}\\
  &\ll\big(n_{4,2}+n_{4,0}+n_{2,4}+n_{0,4}+4\tilde{\delta}\big)\frac{X}{\log X},
\end{align*}
with an absolute implied constant. Similarly recalling that
$\delta'=\tfrac13$, we obtain
\begin{align*}
  \sum_{p\leqslant X}\omega'(p)^{1+\delta'}
  &\leqslant\sum_{p\leqslant X}\big|\lambda_f^{\ast}(p)\lambda_g^{\ast}(p)\big|^{1+\delta'}\big(|\lambda_f^{\ast}(p)|+|\lambda_g^{\ast}(p)|\big)^{2+2\delta'}\\
  &\ll\big(n_{4,2}+n_{4,0}+n_{2,4}+n_{0,4}+4\tilde{\delta}\big)\frac{X}{\log X}.
    \qedhere
\end{align*}
\end{proof}

\subsection{Proof of Theorem~\ref{LargeValuesOfProductsTheorem}} 

In this section, we complete the proof of
Theorem~\ref{LargeValuesOfProductsTheorem}.
\label{SubsectionProductsProof}

\begin{proof}
We use a resonator polynomial \eqref{ResonatorPolynomial2}, with the resonator sequence $r(n)$ as in \eqref{DefinitionResonator}, and arithmetic factors $\varpi(n)$ chosen as in \eqref{DefinitionOfVarpi} in Section~\ref{SubsectionProductsResonator}. Let multiplicative functions $\omega$, $\omega'_1$, $\omega'_2$, and $\omega'$ be as in \eqref{DefOmegaOmegaPrime2} and \eqref{DefinitionOmegaOmegaPrime3}. Using Lemma~\ref{VerificationOfTheProductResonator}, we have that $\omega'_1$ and $\omega'_2$ satisfy \eqref{OmegaPrimeNotCrazy}, while $\omega$ and $\omega'$  satisfy the conditions~\eqref{SecondMomentCondition}--\eqref{FourthMomentCondition} with $a_{\omega}$, $a'_{\omega'}$ as in \eqref{ValuesForAVarpis}.

Fix an arbitrary $\delta>0$, and apply Lemmas~\ref{NormalizingWeightLemma} and \ref{SecondMomentWithResonatorLemma} with
$$ N=q^{1/360-\delta} $$ and
$L=\sqrt{a_{\omega}^{-1}\log N\log\log N}$. According to
Lemma~\ref{SmallTailsLemma}, $\omega$ satisfies the basic
evaluation~\eqref{TruncationOK1}; according to
Lemma~\ref{SmallTailsLemmaProducts}, $\omega$ and $\omega'$ satisfy
the basic evaluation~\eqref{TruncationOK2Products}. In turn,
Lemmas~\ref{NormalizingWeightLemma} and
\ref{SecondMomentWithResonatorLemma} give
$$
  \frac1{\vphis(q)}\sums_{\chi\bmod
    q}|R(\chi)|^2\sim\prod_p\big(1+r(p)^2\omega(p)\big)
$$
and
\begin{multline*}
  \frac1{\vphis(q)}\sums_{\chi\bmod q}|R(\chi)|^2L\big(f\otimes\chi,\tfrac12\big)L\big(g\otimes\bar{\chi},\tfrac12\big)\\
  =L^{\ast}(f\otimes g,1)(\nu+o^{\star}(1)\big)
  \prod_p\bigg(1+r(p)^2\omega(p)+\frac{r(p)\omega'(p)}{\sqrt{p}}\bigg)\\
  +O\bigg(N^{5/2}q^{-1/144}\prod_p\big(1+r(p)^2\omega(p)\big)\bigg),
\end{multline*}
where $\nu\not=0$ depends only on $f$ and $g$, with implicit constants
depending on $(f,g,\delta,\eps)$. According to Lemma~\ref{GainLemma},
\begin{align*}
&\prod_p\bigg(1+r(p)^2\omega(p)+\frac{r(p)\omega'(p)}{\sqrt{p}}\bigg)\\
&\qquad\gg\exp\bigg(\big(a'_{\omega'}+o^{\star}(1)\big)\frac{L}{2\log L}\bigg)\prod_p\big(1+r(p)^2\omega(p)\big).
\end{align*}
Since $L^{\ast}(f\otimes g,1)\not=0$ (Lemma~\ref{lm-ast}), it follows
that there exists at least one primitive character $\chi$ of conductor
$q$ such that
$$ \big|L\big(f\otimes\chi,\tfrac12\big)L\big(g\otimes\bar{\chi},\tfrac12\big)\big|\gg\exp\bigg(\big(\tfrac1{6\sqrt{10}}\sqrt{1-360\delta}+o^{\star}(1)\big)\frac{a'_{\omega'}}{\sqrt{a_{\omega}}}\sqrt{\frac{\log q}{\log\log q}}\bigg), $$
with $o^{\star}(1)=O(\log\log\log q/\log\log q)$. Since we may take
$\delta$ and $\tilde{\delta}$ (implicit in $a_{\omega}$ and
$a'_{\omega'}$) as small as we wish, this proves the first statement
of Theorem~\ref{LargeValuesOfProductsTheorem} with the constant
\begin{equation}
\label{ExponentObtained}
C_{f,g}:=\frac1{6\sqrt{10}}C_{f,g}^0,\quad C_{f,g}^0:=\frac{n_{f,g,3,1}+2n_{f,g,2,2}+n_{f,g,1,3}}{(n_{f,g,4,2}+2n_{f,g,3,3}+n_{f,g,2,4})^{1/2}}
\end{equation}
in the exponent. It is clear that there is an absolute lower bound $C>0$ for $n_{f,g}^0$.
\end{proof}

\begin{remark}
\label{ConstantRemark}
It is clear that, in determining $C>0$, it suffices to consider the case when $f$ and $g$ are not scalar multiples of each other, for otherwise Theorem~\ref{LargeValuesOfProductsTheorem} follows, for example, from Theorem~\ref{LargeValuesAngularSectorsTheorem} (with a better exponent).

In a generic situation, where neither $f$ nor $g$ are of polyhedral
type (in particular, $\sym^k\pi_f$, $\sym^k\pi_g$ are cuspidal for all
$k\leqslant 4$) and if $\sym^k\pi_f\not\simeq\sym^k\pi_g$ for every
$k\leqslant 4$, then $$C_{f,g}^0=(0+2+0)/(2+0+2)^{1/2}=1$$ and
consequently $$C_{f,g}=\tfrac1{6\sqrt{10}}$$ in
\eqref{ExponentObtained} and
Theorem~\ref{LargeValuesOfProductsTheorem}.

Any of the terms in \eqref{ExponentObtained}, including $n_{f,g,k,k'}$
when $k$ and $k'$ are not both even, can take values larger than the
generic ones, for several distinct reasons: first, some of
$\sym^k\pi_f$, $\sym^k\pi_g$ might not be cuspidal, and the
classification of their isobaric components is quoted in
Section~\ref{ssec-symk}; second, it is possible to have
$\sym^k\pi_f\simeq\sym^k\pi_g$ if $f$ and $g$ are character twists of
each other (necessarily by a quadratic character due to the trivial
central character); and, third, $\sym^3\pi_f\simeq\sym^3\pi_g$ can
happen even if $f$ and $g$ are not character twists of each other (see
Ramakrishnan's paper~\cite{Ramakrishnan2015}; for this case, while the
known examples arise from icosahedral representations, and are
conjectured to be exhaustive, this is not known unconditionally).

Thus, in most cases, cusp forms $f$ and $g$ for which this happens can
be explicitly classified, and then the constant $C_{f,g}$ can probably
be improved by using custom-made arithmetic factors; however, since
such a classification is not available in at least one of the cases,
and since getting a tight universal lower bound for our $C_{f,g}^0$
involves an uninspiring case-by-case computation, we are satisfied
simply with stating the existence of such a lower bound.
\end{remark}


\chapter{Upper bounds for the analytic rank}\label{analyticrank}

\section{Introduction}

In this chapter, we prove Theorem \ref{thm-exponentialdecay}.  We
again fix $f$ as in Section~\ref{intro}, and we recall the statement.

\begin{theorem} \label{thmexpbound} There exist constants $R\geq 0$, $c>0$ such that
\begin{equation}\label{expbound}
  \frac{1}{\vphis(q)}\sums_{\chi\mods q}\exp(c\rk_{an}(f\otimes\chi))\leq \exp(cR)
\end{equation}
for all primes $q$.
\end{theorem}

The proof follows the method of Heath-Brown and Michel, who
established a version of Theorem \ref{thm-exponentialdecay} for the
analytic rank in the family of Hecke $L$-functions of primitive
holomorphic cusp forms of weight $2$ and level $q$ (\cite{HBMDMJ}*{Thm 0.1 \&
Cor 0.2}). This method is robust and general and could be
axiomatized (using the definition of families of $L$-functions as
provided in \cites{KoFam,SST}); we will merely indicate where to
modify the original argument of \cite{HBMDMJ}*{\S 2, p.\ 497}.

\section{Application of the explicit formula}

The basic principle is to use the explicit formula of Weil
(Proposition~\ref{Weilexplicit}) to bound the analytic rank by a sum
over the primes.

Let $\phi$ be a smooth non-negative function, compactly supported in
$[-1,1]$. We denote by 
$$
\what{\phi}(s)=\int_{\Rr}\phi(t)e^{st}dt
$$ 
\label{pg-phi} its Fourier-Laplace transform, which is an entire function of
$s\in\Cc$. 

In this chapter, we assume that such a function $\phi$ is chosen once
and for all, with the properties that $\what{\phi}(0)=1$ and
$\Reel(\what{\phi}(s))\geq 0$ for all $s\in\Cc$ such that
$|\Reel(s)|\leq 1$. (The existence of such functions is standard, see,
e.g.,~\cite{IwKo}*{Prop.\ 5.55}.)

\begin{proposition}\label{pr-anrk}
  Let $\xi>1/10000$ be some parameter. We have the inequality
\begin{equation}\label{rankupperboundWeil}
  \xi \rk_{an}(f\otimes\chi)\leq 
  2\phi(0)\log q-S(f\otimes\chi)-S(f\otimes\ov\chi)-
  2\,\xi\,\Xi(f\otimes\chi)+O_{\phi,f}(\xi)	
\end{equation}
where
$$S(f\otimes\chi)=\sum_{p}\frac{\chi(p)\lf(p)\log p}{p^{1/2}}\phi\Bigl(\frac{\log p}{\xi}\Bigr)$$
and
$$\Xi(f\otimes\chi)=\sum_{\Re(\rho-\frac12)\geq\frac1\xi}\Re\Bigl\{\hat\phi\Bigl(\xi(\rho-\frac12) \Bigr)\Bigr\}$$
where $\rho$ ranges over the non-trivial zeros of $L(\ftchi,s)$.
\end{proposition}

\begin{proof} 
  We apply \eqref{eq-Weil} to the function
$$
\vphi(y)=\frac{1}{\sqrt{y}}\phi\Bigl(\frac{\log y}\xi\Bigr),
$$
with Mellin transform
$$
\wtilde\vphi(\rho)=\what\phi\Bigl(\xi\Bigl(\rho-\frac{1}2\Bigr)\Bigr).
$$
On the side of the sum over powers of primes, we easily get
$$
\sum_{l\geq
  2}\sum_{p}\frac{\chi(p)^l\Lambda_f(p^l)}{p^{l/2}}\phi\Bigl(\frac{l\log
  p}{\xi}\Bigr)\ll \xi
$$
by distinguishing the case $l=2$ (for which one applies
Corollary~\ref{PNTsymandpower} after noting that
$\Lambda_f(p^2)=(\lambda_{\symf}(p)-1)\log p$ for $p\nmid r$) and the case
$l\geq 3$ (when the series can be extended to all primes and converges
absolutely). The same bound holds for the corresponding sum with
$\bar{\chi}$. Then $S(f\otimes\chi)+S(f\otimes\ov\chi)$ is the
contribution of the primes themselves to the explicit formula.
\par
On the side of the zeros, the assumption on the test function shows
that the contribution of any subset of the zeros of
$L(f\otimes\chi,s)$ may be dropped by positivity from the explicit
formula to obtain an upper bound as in the statement of the
proposition.
\end{proof}


\begin{remark}
  Note that from \eqref{rankupperboundWeil}, by taking $\xi=1$ and a
  suitable $\phi$ (see \cite[p.217 Example]{mestre}), one obtains the pointwise bound
\begin{equation}\label{eqrankbound}
\rk_{an}(\ftchi)\leq 2\log q+O_f(1).	
\end{equation}

\end{remark}

\subsection{Bounds for moments of the analytic rank}

Theoerem \ref{thmexpbound} is a consequence of the following
proposition which bound the moments of the analytic ranks:

\begin{proposition}\label{propgoal-rk} There exists an absolute
constant $C$ such that for all (sufficiently large) primes $q$ and for all integers
$k\geq 1$, one has
\begin{equation}\label{eq-goal-rk}
\frac{1}{\vphis(q)}\sums_{\chi\mods{q}} \rk_{an}(f\otimes\chi)^{2k}\ll_f
(Ck)^{2k}.
\end{equation}
\end{proposition}

Assuming \eqref{eq-goal-rk}, we deduce \eqref{expbound}: fix $A>0$ such $AC<1/3$; then
$$
\sum_{k\geq 0} \frac{(ACk)^{k}}{k!}<\infty,
$$
by Stirling's formula, and we therefore obtain
$$
\frac{1}{\vphis(q)}\sums_{\chi\mods{q}}
\exp(A\rk_{an}(f\otimes\chi))<\infty
$$
for all primes $q$, as desired.\qed

It remains to prove Proposition \ref{propgoal-rk}.

\subsection{Reduction to a mean square estimate}

In this section we reduce the proof of Proposition \ref{propgoal-rk} to a "second moment" upper bound (Theorem \ref{THM-MOMENT2UPPERBOUND} below) .

To prove Proposition \ref{propgoal-rk} we observe first that because of \eqref{eqrankbound}, we may assume that
$$k\leq \log(q/2).$$
Now for any $k\in[1,\log(q/2)]$, we set
$$\xi=\frac{\log(q/2)}k\geq 1.$$
By Proposition~\ref{pr-anrk}, it is enough to prove that there exists a constant $C>0$, depending only on $\phi$, such that
\begin{align}\label{S2kbound}
  \frac{1}{\vphis(q)}\sums_{\chi\mods
  q}|S(f\otimes\chi)|^{2k}\ll_{f,\phi} (Ck\xi)^{2k}	\\
  \label{Xi2kbound}
  \frac{1}{\vphis(q)}\sums_{\chi\mods
  q}|\Xi(f\otimes\chi)|^{2k}\ll_{f,\phi} (Ck)^{2k}.
\end{align}

We can quickly deal with the first bound as in~\cite{HBMDMJ}*{\S 2.1
(6)}. Since $\exp(\xi)<q^{1/k}$, for any prime numbers
$p_i\leq \exp(\xi)$ for $1\leq i\leq 2k$, we have the equivalence
$$
p_1\cdots p_k\equiv p_{k+1}\cdots p_{2k}\, (\bmod q)
\Longleftrightarrow p_1\cdots p_k= p_{k+1}\cdots p_{2k},$$ hence the
left-hand side of~(\ref{S2kbound}) is bounded by
$$
\frac{\vphi(q)}{\vphis(q)}
\sumsum_{p_1,\cdots,p_{2k}}\prod_{i=1}^{2k}\frac{\lf(p_i)\log
  p_i}{p_i^{1/2}}\phi\Bigl(\frac{\log p_i}
{\xi}\Bigr)\delta_{p_1\cdots p_k=p_{k+1}\cdots p_{2k}},
$$
which is
\begin{multline*}
  \leq \frac{\vphi(q)}{\vphis(q)} k!\, \underset{p_1,\dots,p_{k}}{\sum
    \cdots\sum}\ \prod_{i=1}^{k}\frac{\lf(p_i)^2\log^2
    p_i}{p_i}\phi\Bigl(\frac{\log p_i}
  {\xi}\Bigr)^2\\
  =\frac{\vphi(q)}{\vphis(q)} k
  !\,\Bigl(\,\sum_{p}\frac{\lf(p)^2\log^2 p}{p}\phi\Bigl(\frac{\log p}
  {\xi}\Bigr)^2\,\Bigr)^k
  \ll (Ck\xi)^{2k}
\end{multline*}
by Corollary~\ref{PNTsymandpower}.  This proves \eqref{S2kbound}.

The control of the sum over the zeros in~(\ref{Xi2kbound}), is achieved by reducting to a second
moment estimate. This reduction follows  general principles and is explained in~\cite{HBMDMJ}*{Th. 0.4, \S 2.2}. 

Let $0<\lambda<\expoLmax$ be fixed. Define
\begin{equation}\label{chooseL}
L=q^\lambda
\end{equation}
and for $x \geq 0$, let
$$
P(x)=\begin{cases}2x&\hbox{ if }0\leq x\leq 1/2\\
1&	\hbox{ if }1/2\leq x\leq 1.
\end{cases}
$$
Define furthermore $$\bfx_L:=(x_l)_{l\leq L},$$ where the $x_l$ are
defined in \eqref{xldef}. Let then $M(\ftchi,\bfx_L)$ be as
in~\eqref{defM}. The reduction step mentioned above (which relies in
particular on an important lemma of Selberg, see~\cite{HBMDMJ}*{Lemma
1.1}) shows that~(\ref{Xi2kbound}) follows from:

\begin{theorem}\label{THM-MOMENT2UPPERBOUND} 
  For every $0 < \lambda <1/360$ there exists
  $\eta =\eta (\lambda) >0$, such that for any prime $q\geq 2$ and 
  any $\sigma-\frac12\geq-\frac{1}{\log q}$, we have 
$$
\frac{1}{\vphis(q)}\sums_{\chi\mods q}\Bigl\vert
L(\ftchi,s)M(\ftchi,\bfx_L)-1\Bigr\vert^2\ll
|s|^{O(1)}q^{-\eta(\sigma-\frac12)},
$$
where~$L=q^{\lambda}$.
\end{theorem}

\section{Proof of the mean-square estimate}

The rest of this chapter is devoted to the proof of
Theorem~\ref{THM-MOMENT2UPPERBOUND}. 

\subsection{Application of the twisted second moment formula}

By definition of the mollifier $M(\ftchi,\bfx_L)$, we have for
$\Re s=\sigma\geq 2$ and any $\eps>0$ the equality
$$L(\ftchi,s)M(\ftchi,\bfx_L)=1+\sum_{m>L^{1/2}}\frac{(\lf \cdot \chi*\bfx_L) (m)}{m^s}=1+O_\eps(L^{-\frac12(\sigma-1)+\eps}),$$
where $(\lf \cdot \chi*\bfx_L) (m)$ denotes the multiplicative convolution
$$(\lf \cdot \chi*\bfx_L) (m):=\sum_\stacksum{n\ell=m}{\ell\leq L}\lf(n)\chi(n)x_\ell.$$
From the definition of the $(x_\ell)_{\ell\leq L}$ (cf. \eqref{xldef}) we see that for such $s$ we have
$$\frac{1}{\vphis(q)}\sums_{\chi\mods q}\Bigl\vert\, L(\ftchi,s)M(\ftchi,\bfx_L)-1\,\Bigr\vert^2\ll q^{-\eta(\sigma-\frac12)}$$
for some absolute $\eta>0$. This suffices to establish
Theorem~\ref{THM-MOMENT2UPPERBOUND} for $\sigma \geq 2$.

By the Phragmen--Lindel\"of convexity argument for subharmonic
functions, it is then sufficient to show that
\begin{equation}\label{eqsufficientsecond}
	\frac{1}{\vphis(q)}\sums_{\chi}\Bigl\vert \,
L(\ftchi,s)M(\ftchi,\bfx_L)\, \Bigr\vert^2\ll |s|^{O(1)}
\end{equation}
for $$\sigma=\frac12-\frac{1}{\log q}.$$
For this we will use the results of Chapter \ref{ch-second}.

 To establish \eqref{eqsufficientsecond}, we
decompose the sum along even and odd characters. In the sequel, we
will evaluate in detail only the contribution of the even characters
(multiplied by $2$), namely
$$
\mcQ^+(f,s;\bfx_L)=\frac{2}{\vphis(q)}\sump_{\chi \text{
    non-trivial}}\Bigl\vert \,L(\ftchi,s)M(\ftchi,\bfx_L)\Bigr\vert^2
$$
since the treatment of the contribution of odd characters is entirely similar.

We recall that $(x_{\ell})_{\ell\leq L}$ is supported on integers coprime to $r$.
By Theorem \ref{th-second-final}, we have
$$
\mcQ^+(f,s;\bfx_L)=\MT^+(f,s;\uple{x}_L)+\ET
$$ 
where, with notations of Chapter \ref{ch-second}, the main term is given by (cf. \eqref{MTsecondmoment})
$$
\MT^+(f,s;\bfx_L)=\sum_{d\geq
  1}\sum_{(\l_1,\l_2)=1}\frac{x_{d\l_1}\ov{x_{d\l_2}}}{d^{2\sigma}\l_1^s{\l_2}^{\ov
    s}}\MT^+(f,s;\l_1,\l_2)
$$
and the error term is bounded by
\begin{displaymath}
\begin{split}
\ET
&\ll_\eps
|s|^{O(1)}\Bigl(q^{-1} + 
\sum_{d}\sum_{(\l_1,\l_2)=1}\frac{|x_{d\l_1}x_{d\l_2}|}{d(\l_1\l_2)^{1/2}}L^{\expoL}
q^{-\expoq+\eps}\Big)\\
&\ll |s|^{O(1)}L^{\expoLt}q^{-\expoq+2\eps}
\end{split}
\end{displaymath}
(see  \eqref{errorterm} in Section \ref{moll-sec-mom} for a similar bound).

Before proceeding further we simplify some notations: we set
$$
L_{\infty}(s):=L_{\infty}(f,s)$$
and\label{pg-r}
$$
R(\l_1,\l_2,s):=\sum_{n\geq 1}\frac{\lf(\l_1
  n)\lf(\l_2n)}{(\l_1\l_2n^2)^{s}}.
$$
Note that we already encountered this function in earlier sections 
since, with the notation of
\eqref{Lfgdef}, we have the equality
$$
R(\l_1, \l_2,s)= L \bigl(f\times f, 1, s-\demi; \l_1,
\l_2\bigr).
$$
We can then write $\MT^+(f,s;\l_1,\l_2)$ in the form
$$
\MT^+(f,s;\l_1,\l_2)=\frac{1}{2}M(s,\l_1,\l_2)+ \frac{1}{2}
\eps(f,+,s) M(1-s,\l_1,\l_2),
$$
where
\begin{equation}\label{eq-m1}
  M(s,\l_1,\l_2)=\intc_{(2)}
  \frac{L_\infty(s+u)^2}{L_\infty(s)^2}R(\l_1,\l_2,s+u) G(u)(q^2 |r|)^{u}
  \frac{du}{u}.
\end{equation}
\par
We rename  $s$ into $s_0=\sigma_0+it_0$. In order to prove Theorem \ref{THM-MOMENT2UPPERBOUND}  it suffices to prove the following estimate:

\begin{proposition}\label{11.1} There exist two constants $C_1$ and $C_2$, such that,   for  every prime $q$, for every
   $s_0=\sigma_0+it_0$ satisfying
\begin{equation}\label{=}
\sigma_0=\frac12\pm\frac1{\log q}\text{ and } t_0 \text{\, real},
\end{equation}
 we have the inequality
\begin{equation}\label{eq-goal-mt} \Bigl\vert
\sum_{d\geq
  1}\sum_{(\l_1,\l_2)=1}\frac{x_{d\l_1}\ov{x_{d\l_2}}}{d^{2\sigma_0}\l_1^{s_0}
  \l_2^{\bar{s}_0}}M(s_0,\l_1,\l_2)\Bigr\vert \leq C_1\,|s_0|^{C_2}.
\end{equation}
\end{proposition}

We emphasize that $\sigma_0$ may be taken to be $<\demi$ in this result.

\begin{remark}
  The upper bound \eqref{eqsufficientsecond} for the mollified second
  moment has many similarities with the evaluation of the mollified
  second moment at $1/2$ discussed Section \ref{moll-sec-mom}. A chief
  difference, is that in that section, we were looking for a
  asymptotic formula (see Proposition \ref{prop-moment12}) while here,
  an upper bound of the correct order of magnitude his sufficient for
  our purpose; another important difference is that the complex
  variable $s$, while close to the critical line, is not necessarily
  located near the central point $1/2$, but range along the whole
  critical line. As we will see below, this significantly complicates
  the evaluation of the main term (see also Remark
  \ref{endremarkchap8}).
\end{remark}

\subsection{Beginning of the proof of Proposition \ref{11.1}}
For the proof of the result, we consider $s_0$ fixed and write simply
$M(\l_1,\l_2)=M(s_0,\l_1,\l_2)$. 

For $d\l\leq L$, the definition of $P$ and the standard formula
$$
\intc_{(2)}y^v\frac{dv}{v^2}=
\begin{cases}\log y&\text{if $y\geq 1$},\\
0&\text{ if $0<y\leq 1$}.	
\end{cases}
$$
show that we have
$$
P\Bigl(\frac{\log (L/d\l)}{\log
  L}\Bigr)=\intc_{(2)} H_L(v)(d\l)^{-v}\frac{dv}{v}
$$
where $H_L$ is the entire function defined by
\begin{equation*}
H_L(v)=2\,\frac{L^{v/2}(L^{v/2}-1)}{v\,\log L}
\end{equation*}
for $v\not=0$ and $H_L(0)=1$. 
\par
We insert this integral in the left-hand side of~(\ref{eq-goal-mt}),
obtaining (see~(\ref{eq-m1})) the formula
\begin{multline}\label{eq-mt2} 
{\mathcal M}= \mathcal M (s_0):=
\sum_{d\geq 1}\sum_{(\l_1,\l_2)=1}
\frac{x_{d\l_1}\ov{x_{d\l_2}}}{d^{2\sigma_0}\l_1^{s_0}\l_2^{\ov{s_0}}}
M(s_0,\l_1,\l_2)\\
=\frac{1}{(2\pi i)^3} \int_{(2)}\int_{(2)}\int_{(2)}
\frac{L_\infty(s_0+u)^2}{L_\infty(s_0)^2}G(u)\\
\times L(s_0,s_0,\bar{s}_0,u,v,w) H_L(v)H_L(w)(q^2|r|)^{u}
\frac{du}{u}\frac{dv}{v}\frac{dw}{w}
\end{multline}
where the auxiliary function $L$ is given by~(\ref{bigLdef}), namely
$$
L(s,z,z',u,v,w)= \sumsum_{\stacksum{d,\l_1,\l_2,n}{(\l_1,\l_2)=(d\l_1\l_2, r) = 1}}
\frac{\mu_f(d\l_1)\lf(\l_1 n)\mu_f(d\l_2)\lf(\l_2n)}
{\l_1^{s+z+u+v}{\l_2}^{s+z'+u+w} {d}^{z+z'+v+w}n^{2s+2u}}.	
$$
With this definition, based on integrals, our purpose (see
\eqref{eq-goal-mt}) is to prove the inequality
\begin{equation}\label{157}
\bigl\vert \mathcal M \bigr\vert \leq C_1 \vert s_0\vert^{C_2},
\end{equation}
for some absolute $C_1$ and $C_2$, uniformly for $s_0$ satisfying \eqref{=}.

To prove \eqref{157}, we proceed by shifting the three contours slightly to the left of the product of lines
$$\Re u=\Re v=\Re w=0.$$
In the sequel we decompose the complex variables $u$, $v$  and $w$ into their real and imaginary parts
as
$$
u=\sigma_u +it_u,\ v= \sigma_v+it_v,\ w= \sigma_w+it_w.
$$
It will also be useful to set
$$
\mcL=({\log q})^{-1}.
$$

We will need estimates for the various factors in the integral
\refs{eq-mt2}.  We start with $H_L(v)$.

 \begin{lemma}\label{boundforHL}
   Let $0< \lambda < 1/360$ and $B>0$ be two constants. Let $L$ be
   defined by \eqref{chooseL}.  Then there exists a constant $C_3$
   depending only on $\lambda$ and $B$, such that uniformly for
   $ \sigma_v \leq B \mcL$, we have the inequality
\begin{equation*}
  \bigl| H_L(v)\bigr| \leq C_3\, L^{\sigma_v/2}\,\min\Bigl(1,\frac{\mcL}{|v|}\Bigr).
\end{equation*}
\end{lemma}

\begin{proof} 
  This is an easy combination of the two bounds
  $L^{v/2}-1 =O (\vert v \vert \log L)$, valid uniformly for
  $\vert v \vert \log q\leq 1$, and $L^{v/2} -1=O_{B} (1)$, valid
  uniformly for $\sigma_v \leq B \mcL$.
\end{proof}

Next we provide bounds for the Gamma factors (see \eqref{defLinfty}
for the definition of $L_\infty$).
\begin{lemma}\label{Gammabound} 
Let $G(u)$ be the function defined in \eqref{Gdef}. Then there exists a constant $\alpha_f$ depending only on $f$, such that, uniformly for $s_0$ satisfying \eqref{=} and 
for $u=\sigma_u+it_u$ with  $\sigma_u\in [-1/4,2]$  and $t_u$ real, we have the bound
$$\Bigl(\frac{L_\infty(s_0+u)}{L_\infty(s_0)}\Bigr)^2G(u) \ll_f (1+|t_0|)^{\alpha_f} e^{-3\pi{|t_u|}}.$$
\end{lemma}
\begin{proof} In both cases ($f$ holomorphic or not), we have the equalities (see the definition \eqref{defLinfty})
$$
L_\infty(s)=\xi_f \,\pi^{-s} \prod_{i=1,2}\Gamma \Bigl(\frac{s+\mu_{f,i}}2\Bigr) 
$$
where  $\xi_f=1$ if $f$ is a Hecke--Maa{\ss} form, and $\xi_f= \pi^{-1/2} 2^{(k-3)/2}$ if  $f$ is holomorphic with weight $k$. Furthermore, the $\mu_{f,i}$ are the archimedean Langlands parameters of the automorphic representation attached to $f$ as in Section \ref{hecke}, i.e.
$$\mu_{f,1}=-\frac{k-1}2,\ \mu_{f,2}=-\frac{k}2$$
 if $f$ is holomorphic of weight $k\geq 2$ and 
 $$\mu_{f,1}=\frac{1- \kappa_f}{2}+ it_f,\ \mu_{f,2}=\frac{1- \kappa_f}{2}-it_f$$
 if $f$ is a Maa{\ss} form with Laplace eigenvalue  $\lambda_f(\infty)=(\frac12+it_f)(\frac12-it_f)$ and parity $\kappa_f \in \{\pm 1\}$. 
This implies that, in both cases, we have
$$
\frac{1}{16} \leq \Re \Bigl(\frac{s_0 +\mu_{f,i} +u}{2} \Bigr)\ll_f 1,
$$
under the assumptions of Lemma \ref{Gammabound}.

Decompose $\mu_{f,i}$ as $\mu_{f,i} =\sigma_{f,i}+i t_{f,i}.$ Then by
Stirling's formula \cite{GR}*{formula 8.328, page 895} we have for
$1/5\leq \sigma\leq 3$
\begin{align*}
  L_{\infty}(s) &\asymp
                  (1+|t+t_{f,1}|)^{\frac{\sigma}{2}+\frac{\sigma_{f,1}}{2}-\frac
                  12} e^{-\frac\pi4|t+t_{f,1}|}\times
                  (1+|t+t_{f,2}|)^{\frac{\sigma}{2}+\frac{\sigma_{f,2}}{2}-\frac 12}
                  e^{-\frac\pi4|t+t_{f,2}|} \\
                & \asymp_f
                  (1+|t|)^{\sigma+\frac{\sigma_{f,1}+\sigma_{f,2}}{2}-1}
                  e^{-\frac\pi2|t|}.
\end{align*}
Therefore, since for $\sigma_u\in [-1/4, 2]$ we have
$\Re (s_0 +u) \in [1/5,3]$, we deduce the inequality
$$
\frac{L_\infty(s_0+u)^2}{L_\infty(s_0)^2}G(u)\ll_f
\frac{(1+|t_0+t_u|)^{2\sigma_0-2+\sigma_{f,1}+\sigma_{f,2}+2\sigma_u}}{(1+|t_0|)^{2\sigma_0-2+\sigma_{f,1}+\sigma_{f,2}}}e^{-\pi(|t_0+t_u|-|t_0|)}\,
e^{-4\pi|t_u|}.
$$
\par
To control the size of the numerator of the above fraction, we will
use either the lower bound $1 + \vert t_u +t_0 \vert \geq 1$ or the
upper bound
$1+ \vert t_u + t_0 \vert \leq (1 +\vert t_u \vert) (1 + \vert t_0
\vert)$ according to the sign of the exponent and we will consider two
cases:
\par
\smallskip
\par
\textbf{Case 1.} For $|t_u|\geq |t_0|$, using $|t_u+t_0|\geq
|t_u|-|t_0|$, we have
$$
\frac{L_\infty(s_0+u)^2}{L_\infty(s_0)^2}G(u)\ll (1+\vert
t_u\vert)^{\alpha_f}e^{\pi|t_u|} e^{-4\pi|t_u|}= (1+\vert
t_u\vert)^{\alpha_f}e^{-3\pi|t_u|}
$$
for some absolute $\alpha_f\geq 0$.
\par
\textbf{Case 2.} For $|t_u|\leq |t_0|$, using the inequality
$ |t_u+t_0|\geq |t_0|-|t_u|$, we get
$$
\frac{L_\infty(s_0+u)^2}{L_\infty(s_0)^2}G(u)
\ll_f   (1+\vert t_0 \vert)^{\alpha_f} \ e^{\pi |t_u|}\, e^{-4\pi|t_u|}
\ll_f(1+|t_0|)^{\alpha_f} e^{-3\pi|t_u|}
$$
again for some absolute constant $\alpha_f\geq 0$.
\end{proof}

As in \eqref{defT(s)} we now denote by
 \begin{equation*}
 T(s)=L(f\otimes f,s)=\zeta(s)L(\symf,s)
\end{equation*}
the Rankin-Selberg $L$-function of $f$, and by
$$
T_p(s)=L_p(f\otimes f,s)=\zeta_p(s)L_p(\symf,s)
$$
its local factor at $p$.

The analytic properties of $T(s)$ have been reviewed in
Section~\ref{sectionSymRS}. Recall in particular that $T(s)$ is
holomorphic on $\Cc-\{1\}$ and has a simple pole at $s=1$; its residue
there is denoted $\kappa_f$. Lemma~\ref{lm-factor} implies that there
exists $\eta>0$ and an analytic continuation and factorization of
$L(s,z,z',u,v,w)$ of the form
$$
L(s,z,z',u,v,w)=\frac{T(2s+2u)T(z+z'+v+w)}{T(s+z+u+v)T(s+z'+u+w)}
D(s,z,z',u,v,w).
$$
in the region $\mcR({\eta})\subset \Cc^6$ defined by the inequalities
\begin{align*}
  \Re s  > \frac{1}{2}-\eta, \quad  \Re z > \frac{1}{2}-\eta, \quad \Re z' > \frac{1}{2}-\eta,\\
  \Re  u >-\eta, \quad  \Re v >-\eta, \quad  \Re w >-\eta,
\end{align*}
where $D(s,z,z',u,v,w)$ is holomorphic and bounded on $\mcR(\eta)$.

\subsection{Study of $\mathcal M$} 

We now start the proof of \eqref{157}, which will eventually prove Proposition
\ref{11.1}. 

Let us recall that $\mcM$ is an integral in three variables (varying
along vertical lines) whose integrand contains factors involving
translates of Riemann's zeta function $\zeta(s)$ and of the symmetric
square $L$-function $L(\sym^2 f,s)$. The strategy is similar to that
of Section \ref{moll-sec-mom}: we are going to shift contours to the
left so that the integrals along the new contours contribute as error
terms and to evaluate the residues of the poles that we have met in
the process. That would be reasonably easy to do under the {\em
  Generalised Riemann Hypothesis}; however to obtain unconditional
results, we need to use the standard Hadamard-de la Vall\'ee-Poussin
zero free region. The proof is a bit tedious so we give and overview
of what is coming up.

\begin{enumerate}
\item We first reduce (up to an admissible error term) to evaluating a
  truncated version of $\mcM$ in which the imaginary parts of the
  variables are bounded by a power of $\log q$ and the real part is
  such that the arguments at which the various $L$-functions are
  evaluated, stay slightly to the right of the critical strip: this is
  the content of Lemma \ref{lemmaredchap8}.
\item We shift the contour of one of the variables to the left so that
  one still remains on the right of the zero free region of any
  $L$-function involved in the denominator. We show that the resulting
  integral contribute a negligible error term and it remains to deal
  with the contributions of the poles encountered in the process: this
  is the content of Section \ref{sec-contourshift} ending with
  \eqref{eqendofshift}.
\item In Section \ref{sec-descrresidues} we describe the contributions
  of the two poles (these are integrals in two variables) with the aim
  to bounding them. The target bound is \eqref{424}. We focus on one
  integral, the treatment of the other being entirely similar.
\item In Section \ref{secI0transform} we perform a contour shift on
  one of the remaining variables along the same line as in Step
  (2). The outcome is the integral along the new contour $J_1$ and two
  contributions from poles met in the process $J_2,J_3$. The target
  bound for any of these terms is \eqref{J1J2J3<}.
\item We prove \eqref{J1J2J3<} for $J_1$ in Sections \ref{secJ11} and
  \ref{secJ12}.
\item The proofs of \eqref{J1J2J3<} for $J_2$ and $J_3$ (which are
  integrals in one variable) are spread over Sections \ref{secJ231},
  \ref{secJ232} and \ref{secJ233}. Unlike Section \ref{moll-sec-mom}
  we don't need to perform a final contour shift and evaluate the
  residue: as we only need an upper bound, we simply split the
  integral into pieces and apply different bounds depending on the
  position of the piece with respect to the other parameters.

\end{enumerate}

We start with Step (1) which is a reduction to another estimate for
the following truncated triple integral: for $C\geq 1$ and
\begin{equation}\label{defV0}
  V_0:=(\log q)^{C},
\end{equation} we define
\begin{multline}\label{defM(V)}
  \mathcal M ( V_0):= \frac{1}{(2\pi i)^3} \int_{\substack{\, \\ \, \\
      (3\mcL)\\ |t_w|\leq 2V_0}} \int_{\substack{\, \\ \, \\ (3\mcL)\\
      |t_v|\leq V_0}} \int_{\substack{\, \\ \, \\ (3\mcL)\\ |t_u|\leq
      V_0}}
  \frac{L_\infty(s_0+u)^2}{L_\infty(s_0)^2}\\
  \times   \frac{T(2s_0+2u)T(2\sigma_0+v+w)}{T(2s_0+u+v)T(2\sigma_0+u+w)}\\
  \times E(s_0,u,v,w) G(u) H_L(v)H_L(w)(q^2|r|)^{u}
  \frac{du}{u}\frac{dv}{v}\frac{dw}{w}
\end{multline}

\begin{lemma}\label{lemmaredchap8} The bound \eqref{157} follows from
  the following bound: for any $C\geq 1$ we have uniformly for
  $\vert t_0\vert \leq \log q$
\begin{equation}\label{347}
\vert \mathcal M(V_0) \vert \ll \vert s_0\vert^{C'},
\end{equation}
where $C'$ depends on $C$ and the implicit constant on $f$ and $C$
\end{lemma}

For the proof of Lemma \ref{lemmaredchap8} and later arguments will
repeatedly use the following Lemma which combines (a special case of)
Corollary \ref{cor29} and the Phragmen-Lindel\"of principle:
\begin{lemma}\label{ZFR} There exists two constants  $c=c_f>0$ and $A^*=A^*_f\geq 0$ such that 
\begin{itemize}
\item For $s=\sigma+it$ in the region
\begin{equation}\label{defZFR}
\sigma\geq -\frac{c}{\log(2+|t|)},
\end{equation}
we have  $T(1+s)\not= 0$ and the inequalities
\begin{equation}\label{ineqZFR}
\log^{-A^*}(2+|s|)\ll \frac{s}{1+s}T(1+s)\ll \log^{A^*}(2+|s|).
\end{equation}

\item For $s=\sigma +it$ such that $\sigma \geq -1/2$, we have the inequality
\begin{equation}\label{gen-ineq}
\Bigl\vert \,\frac{s}{s+1}T(1+s) \,\Bigr\vert \ll \max\bigl( 1, (1+|s|)^{\max(0,	4({1/2-\sigma}))+\eps}\bigr).
\end{equation}
for any $\eps>0$ where the constants implied depend only on $f$ and $\eps$. 
\end{itemize}
\end{lemma}

\proof(of Lemma \ref{lemmaredchap8}) Recall that $\mathcal M$ is defined in \eqref{eq-mt2}. We use Lemma
\ref{lm-factor} and set $$E(s,u,v,w)=D(s,s,\bar{s},u,v,w).$$ With
these notations, the function $L$ in \eqref{eq-mt2} can be written as
\begin{equation*}
  L(s_0,s_0,\bar{s}_0,u,v,w)=
  \frac{T(2s_0+2u)T(2\sigma_0+v+w)}{T(2s_0+u+v)T(2\sigma_0+u+w)}
  E(s_0,u,v,w).
\end{equation*}

We first shift the three lines of integration  in~(\ref{eq-mt2}) to
$$
\Re(u)=\Re(v)=\Re(w)=3\mcL.
$$
There is no pole encountered in this shift, so that the triple
integral $\mathcal M$, defined in \eqref{eq-mt2}, satisfies the
equality
\begin{multline*}
\mathcal M =  \frac{1}{(2\pi i)^3} \int_{(3\mcL)}\int_{(3\mcL)}\int_{(3\mcL)}
  \frac{L_\infty(s_0+u)^2}{L_\infty(s_0)^2}\cdot
  \frac{T(2s_0+2u)T(2\sigma_0+v+w)}{T(2s_0+u+v)T(2\sigma_0+u+w)}\\
\times   E(s_0,u,v,w)
  G(u) H_L(v)H_L(w)(q^{2}|r|)^u \frac{du}{u}\frac{dv}{v}\frac{dw}{w}.
\end{multline*}
\par
First, using straightforwardly Lemma \ref{lm-factor} to bound the
$E$--function, Lemma \ref{ZFR} (inequality \eqref{ineqZFR}) to bound
the $T$--functions or their inverses, Lemma \ref{Gammabound} to bound
the $L_\infty$ and $G$-factors, and Lemma \ref{boundforHL} for the
$H_L$--functions, we can already deduce the rough bound
\begin{equation}\label{eq-trivial}
  \mathcal M\ll |s_0|^{O(1)}(\log q)^{O(1)}.
\end{equation}
In particular, in order  to prove~(\ref{157}), we may now assume
that
\begin{equation}\label{cond0}
|t_0|\leq \log q.
\end{equation}
\par
This being done we  consider the integral  truncated in the variable $u$
\begin{multline*}
 \mathcal M_0( V):=
  \frac{1}{(2\pi i)^3} \int_{\substack{\, \\ \, \\ (3\mcL)\\ |t_w|\leq
    2V}}  \int_{\substack{\, \\ \, \\ (3\mcL)\\ |t_v|\leq
    V}}\int_{(3\mcL)}
  \frac{L_\infty(s_0+u)^2}{L_\infty(s_0)^2}\\
  \times
  \frac{T(2s_0+2u)T(2\sigma_0+v+w)}{T(2s_0+u+v)T(2\sigma_0+u+w)}\\
  \times E(s_0,u,v,w) G(u) H_L(v)H_L(w)(q^2 |r|)^{u}
  \frac{du}{u}\frac{dv}{v}\frac{dw}{w},
\end{multline*}
where $V\geq 2$ is some parameter.  Using the same lemmas as in the
proof of \eqref{eq-trivial}, we obtain the equality
\begin{equation}\label{308} 
  \mathcal M =\mathcal M_0 (V)+O\Bigl( \frac{(|s_0|\log q)^{C_4}}{V^{1/2}}\Bigr),
\end{equation}
for some absolute constant $C_4\geq 0$. In view of the inequality
\eqref{157}, the error term in \eqref{308} is admissible if we fix the
value of $V$ to be
\eqref{defV0}
for a sufficiently large constant $C\geq 2C_4$.

By the same techniques which led to \eqref{308} (particularly the
decay at infinity of the functions $H_L (v)/v$ and $H_L(w)/w$, see
Lemma \ref{boundforHL}), we approximate $\mathcal M_0 (V_0)$ by
$\mathcal M(V_0)$ with an admissible error.  By combining with
\eqref{308}, we finally obtain the equality
\begin{equation*}
\mathcal M =\mathcal M (V_0) + O\bigl( \vert s_0\vert^{C_4}\bigr),
\end{equation*}
where $C_4$ is some absolute constant, where $t_0$ satisfies
\eqref{cond0} and where $V_0$ is defined by \eqref{defV0}, with a
sufficiently large $C$.
Lemma \ref{lemmaredchap8} follows from \eqref{347}.\qed

\subsection{Shifting the contours of integration} \label{sec-contourshift}

In the $u$--plane we consider the vertical segment:
$$
\gamma_u:= \bigl\{ u\in \Cc\,\mid\, \sigma_u =3 \mcL, \ \vert
t_u \vert \leq V_0\bigr\},
$$
and the curve
$$
\Gamma_u :=\bigl\{ u \in \Cc\,\mid\,
\sigma_u=-\frac{c_f}{\log(V_0^3+|t_u|)},\ \vert t_u \vert \leq
V_0\bigr\},
$$
where $c_f$ is the constant appearing in Lemma \ref{ZFR}.  We also introduce two horizontal segments
$$
S_u:= \bigl\{ u\in \Cc\,\mid\, -\frac{c_f}{\log(V_0^3+V_0)}\leq \sigma_u \leq 3 \mcL,\ t_u =V_0\bigr\},
$$
and its conjugate $\overline{S_u}$.  The hypothesis \eqref{cond0} and
Lemma \ref{ZFR} imply that there is no zero of the function
$$u \mapsto T(2s_0+u+v)T(2\sigma_0+u+w),
$$
in the interior of the curved rectangle $\mathcal R_u$ with edges
$\gamma_u$, $S_u$, $\Gamma_u $ and $\overline {S_u}$, when the
variables $v$ and $w$ belong to the paths of integration appearing in
the definition \eqref{defM(V)} of $\mathcal M (V_0)$.

Furthermore, when   $u$ belongs to  $S_u\cup \Gamma_u\cup \overline{S_u}$ and when $v$ and $w$ are as above,  the four numbers 
$$2s_0+2u-1,\ 2\sigma_0+v+w-1,\ 2s_0+u+v-1,\ 2\sigma_0+u+w-1$$
all satisfy the lower bound \eqref{defZFR}.  Finally, the modulus of
these four numbers is also not too small, namely they are
$\gg 1/(\log q)$. We then apply \eqref{ineqZFR} in the condensed
form
$$
  \frac{T(2s_0+2u)T(2\sigma_0+v+w)}{T(2s_0+u+v)T(2\sigma_0+u+w)}\ll (\log q)^{O(1)},
$$
uniformly for $u$, $v$ and $w$ as above and $t_0$ satisfying
\eqref{cond0}.

To shorten notation we rewrite $\mathcal M ( V_0)$ into the form
\begin{equation}\label{defT}
\mathcal M ( V_0):=
  \frac{1}{(2\pi i)^3}  \int_{\substack{\, \\ \, \\ (3\mcL)\\ |t_w|\leq
    2V_0}} \int_{\substack{\, \\ \, \\ (3\mcL)\\ |t_v|\leq
    V_0}} \int_{\substack{\, \\ \, \\ (3\mcL)\\ |t_u|\leq
    V_0}} \mathcal T (s_0,u,v,w) \frac{du}{u}\frac{dv}{v}\frac{dw}{w},
\end{equation}
From this definition of the $\mathcal T$--function and bounding the $E$--function by Lemma \ref{lm-factor},
we deduce the following bound where the variables are now separated
\begin{multline}\label{001}
  \frac{1}{(2\pi i)^3}  \int_{\substack{\, \\ \, \\ (3\mcL)\\ |t_w|\leq
    2V_0}} \int_{\substack{\, \\ \, \\ (3\mcL)\\ |t_v|\leq
    V_0}} \int_{u\in S_u\cup \Gamma_u\cup \overline{S_u}} \mathcal T (s_0,u,v,w) \frac{du}{u}\frac{dv}{v}\frac{dw}{w} \\
  \ll (\log q)^{O(1)}\Bigl( \int_{u\in S_u\cup \Gamma_u\cup
    \overline{S_u}} q^{\sigma_u}\ \Bigl\vert
  \frac{L_\infty(s_0+u)^2}{L_\infty(s_0)^2}G(u)\Bigr\vert \cdot
  \Bigl\vert \frac{du}{u}\Bigr\vert\Bigr)\\
  \times \Bigl( \int_\stacksum{(3\mcL)}{|t_v|\leq 2V_0} \Bigl\vert
  \frac{H_L (v)}{v} \Bigr\vert \,\vert d v \vert \Bigr)^2.
 \end{multline}
\par
To bound the integral $\int_{\Gamma_u}$ we exploit the fact that
$\sigma_u$ is negative and satisfies
$\vert \sigma_u \vert \gg 1/(\log \log q)$. When combined with Lemma
\ref{Gammabound}, we deduce the bound
$$
\int_{\Gamma_u} (\cdots) \ll \exp\Bigl(-c'_f \frac{\log q}{\log \log
  q}\Bigr)
$$ 
for some positive constant $c'_f$. To bound $\int_{S_u}$ and
$\int_{\overline{S_u}}$, we use the fact that $\vert t_u\vert $ is
large, that is $\vert t_u\vert =V_0$, to apply Lemma
\ref{Gammabound}. These remarks and easy computations lead to the
following bound
\begin{equation}\label{002}
  \int_{u\in S_u\cup \Gamma_u\cup \overline{S_u}}
  (\cdots)\ll (\log q)^{O(1)}\exp\Bigl(-d_f \frac{\log q}{\log \log q}\Bigr),
 \end{equation}
 where $d_f$ is some positive constant.  Furthermore, the inequality
\begin{equation}\label{003}
 \int_{\substack{\, \\ \, \\ (3\mcL)\\ |t_v|\leq
    V}}   \Bigl\vert \frac{H_L (v)}{v} \Bigr\vert \,\vert d v \vert \ll \log q
\end{equation}
is a direct consequence of Lemma  \ref{boundforHL}. It remains to
combine \eqref{001}, \eqref{002} and \eqref{003} to deduce the
inequality
\begin{multline}\label{eqendofshift}
  \frac{1}{(2\pi i)^3}\int_{\substack{\, \\ \, \\ (3\mcL)\\ |t_w|\leq
    2V_0}} \int_{\substack{\, \\ \, \\ (3\mcL)\\ |t_v|\leq
    V_0}}\int_{u\in S_u\cup \Gamma_u\cup \overline{S_u}} \mathcal T (s_0,u,v,w) \frac{du}{u}\frac{dv}{v}\frac{dw}{w} \\
  \ll \exp\Bigl(-\frac{d_f}{2}\cdot \frac{\log q}{\log \log q}\Bigr).
\end{multline}
This error term is negligible when compared with the right--hand side
of \eqref{347}. By the residue formula, we are reduced to proving that
the contribution of the residues of the poles which are inside the
curved rectangle $\mathcal R_u$ are also in modulus less than
$C_1 \vert s_0 \vert ^{C_2}.$


\subsection{Description of the residues}\label{sec-descrresidues}
During the contour shift from $\gamma_u$ to $S_u\cup \Gamma_u\cup \overline{S_u}$ we hit  exactly two  poles. They are both simple and located  at $u=0$ (from $1/u$) and at $u =1/2 -s_0$ (from the factor $T(2s_0+2u)$).  Let us denote by 
$I_0$ and $I_{1/2-s_0}$ the contribution of these residues to $\mathcal M( V_0)$. More precisely we have the equalities
\begin{multline}\label{defI0} I_0:=
  \frac{T(2s_0)}{(2\pi i)^2} \int_{\substack{\, \\ \, \\ (3\mcL)\\ |t_w|\leq
    2V_0}} \int_{\substack{\, \\ \, \\ (3\mcL)\\ |t_v|\leq
    V_0}}
  \frac{T(2\sigma_0+v+w)}{T(2s_0+v)T(2\sigma_0+w)} \\
  \times E(s_0,0,v,w) H_L(v)H_L(w)\frac{dv}{v}\frac{dw}{w}
\end{multline}
and 
\begin{multline*}
 I_{1/2-s_0}:=\frac{\kappa_fq^{1-2s_0}G(\demi-s_0)}{(2\pi i)^2(\demi-s_0)}
  \frac{L_\infty(\demi)^2}{L_\infty(s_0)^2}\times\\
 \int_{\substack{\, \\ \, \\ (3\mcL)\\ |t_w|\leq
    2V_0}} \int_{\substack{\, \\ \, \\ (3\mcL)\\ |t_v|\leq
    V_0}}
  \frac{T(2\sigma_0+v+w)}{T(s_0+\demi+v)T(\sigma_0-it_0+\demi+w)}\\
\times
  E(s_0,\demi-s_0,v,w) H_L(v)H_L(w)\frac{dv}{v}\frac{dw}{w}.
\end{multline*}

From the above discussions, it remains to prove that, uniformly for \eqref{cond0}, we have the inequalities
\begin{equation}\label{424}
\vert I_0\vert, \ \vert I_{1/2-s_0} \vert \leq C_1 \vert s_0 \vert^{C_2}.
\end{equation}
 We will concentrate on $I_0$, since the other bound is similar.
 \subsection{Transformation of $I_0$}\label{secI0transform} We return to the definitions
 \eqref{defI0} of $I_0$ and \eqref{=} of $\sigma_0$.   We define four
 paths in the $w$--plane
\begin{align*}
  \gamma_w&= \bigl\{ w \in \Cc\,\mid\, \sigma_w = 3 \mathcal L, \
            \vert t_w \vert \leq 2V_0 \bigr\},
  \\
  \Gamma_w&= \bigl\{ w \in \Cc\,\mid\, \sigma_w
            =1-2\sigma_0-\frac{c_f}{\log(V_0^3+|t_w|)}\ t_w  =2V_0\bigr\}\\
          &=\bigl\{ w\in\Cc\,\mid\, \sigma_w=
            \mp2\mcL-\frac{c_f}  {\log(V_0^3+|t_w|)},\ \vert t_w \vert \leq 2V_0\bigr\},
  \\
  S_w&= \bigl\{ w \in\Cc\,\mid\,  \mp2\mcL-\frac{c_f}
       {\log(V_0^3+2V_0)}  \leq \sigma_w\leq 3\mcL,\ t_w  =2V_0\bigr\}
\end{align*}
and its conjugate $\overline {S_w}$, where $c_f$ is the constant
appearing in Lemma \ref{ZFR}. These four paths define a curved
rectangle $\mathcal R_w$. Inside $\mathcal R_w$, the function
 $$
 w \mapsto \frac{T(2\sigma_0+v+w)}{T(2s_0+v)T(2\sigma_0+w)} E(s_0,0,v,w)
 $$
 has only one pole. It is simple and  is located at 
  $$
 w_v:=1-2\sigma_0-v = \mp2\mcL -v=(-3\mp 2)\mcL-it_v.$$
 It corresponds to the pole at $1$ of the numerator  $T(2 \sigma_0 +v+w)$. Remark that the rectangle $\mathcal R_w$ is defined in order to contain no zero of the function $w\mapsto T(2 \sigma_0 +w)$. The function to integrate with respect to $w$ 
 in \eqref{defI0} has another pole at $w=0$ and it is simple. By the residue formula, we have the equality
 \begin{eqnarray}\nonumber
I_0 &&=\ \frac{T(2s_0)}{(2\pi i)^2} \int_\stacksum{\sigma_v=3\mcL}{| t_v| \leq V_0}\int_{w\in S_w\cup\Gamma_w\cup \overline{S_w}}
\frac{T(2\sigma_0+v+w)}{T(2s_0+v)T(2\sigma_0+w)} \\\nonumber
&&\hskip 4cm\times E(s_0,0,v,w)
H_L(v)H_L(w)\frac{dw}{w}\frac{dv}{v}\\
 &&\ +\ \frac{T(2s_0)}{2\pi i} \int_\stacksum{\sigma_v=3\mcL}{ | t_v| \leq V_0}
\frac{L(\symf,1)}{T(2s_0+v)T(1-v)}
 E(s_0,0,v,\mp2\mcL -v)\nonumber\\
&&\hskip 4cm\times H_L(v)H_L(\mp2\mcL -v)\frac{dv}{v(\mp2\mcL -v)}\label{dinner1}\\
&&\  +\ \frac{T(2s_0)}{2\pi i} \int_\stacksum{\sigma_v=3\mcL}{ | t_v| \leq V_0}
\frac{T(2\sigma_0+v)}{T(2s_0+v)T(2\sigma_0)} E(s_0,0,v,0)
H_L(v)H_L(0)\frac{dv}{v} \nonumber\\
 &&:=\ \frac{T(2s_0)}{(2\pi i)^2} \Bigl( J_1+ 2 \pi i J_2 + 2 \pi i
J_3\Bigr).\nonumber
\end{eqnarray}

Hence, in order to prove \eqref{424}, it remains to prove the inequalities
\begin{equation}\label{J1J2J3<}
\bigl\vert T(2s_0)J_i \bigr\vert \leq  C_1\vert s_0 |^{C_2},
\end{equation}
for $i=1,$ $2$ and $3$, for some absolute $C_1$, $C_2$ and for any  $s_0 =\frac{1}{2} \pm\mcL +it_0$, with $ t_0$ satisfying 
\eqref{cond0}.  
\subsection{Dissection of $J_1$}\label{secJ11} We decompose $J_1$ into
\begin{equation}\label{J1=...}
J_1=J_{1,1}+J_{1,2} + J_{1,3}
\end{equation}
where $J_{1,1}$ corresponds to the contribution in the double integral defining $J_1$, of the $w$ in $\Gamma_w$, and $J_{1,2}$ (resp.\ $J_{1,3}$) corresponds
to the contribution of the $w$ in $S_w$ (resp.\ $w$ in $\overline{S_w}$).

For $\vert t_v\vert \leq V_0$ and $\vert t_w\vert =2 V_0$, we have $\vert 2 \sigma_0 +v+w -1\vert \geq 1$. Appealing once again 
to Lemma \ref{ZFR} to bound each of the three $T$--factors, we deduce the inequality
\begin{equation*}
J_{1,2} \ll (\log q)^{O(1)}\Bigl( \int_\stacksum{\sigma_v=3\mcL}{ | t_v| \leq V_0} \Bigl\vert \frac {H_L (v)}{v}\Bigr\vert \, dv\Bigr) \Bigl( \int_{w \in S_w} \, \Bigl\vert \frac{H_L (w)}{w}\Bigr\vert dw \Bigr),
\end{equation*}
uniformly for $t_0$ satisfying \eqref{cond0}. 
We now appeal to Lemma \ref{boundforHL}, which is quite efficient since $\vert t_w\vert =2 V_0$ is  large, to conclude by the inequality
\begin{equation}\label{J12<<}
J_{1,2} \ll (\log q)^{-10},
\end{equation}
by choosing $C$ sufficient large in the definition \eqref{defV0} of $V_0$. The same bound holds true for $J_{1,3}$ 
\subsection{Study of $J_{1,1}$}\label{secJ12}  To bound $J_{1,1}$ we will benefit from the fact that $\sigma_w$ is negative and not too
small, that is
\begin{equation}\label{ineqw}
\sigma_w <0 \text{ and } - \sigma_w \gg 1/( \log \log q) \text{ for } w \in \Gamma_w.
\end{equation}
Now remark that, for $w \in \Gamma_w$ and $v$ with $\sigma_v =3 \mcL,\ \vert t_v \vert \leq V_0$, we have the three lower bounds
$$
\vert 2 \sigma_0 +v+w -1 \vert \geq \vert \pm 2 \mcL +\sigma_v + \sigma_w\vert \gg 1/(\log \log q),
$$
$$
\Re (2 s_0 +v-1) = \pm 2 \mcL +\sigma_v \geq \mcL \geq - c_f/ \log (2 + \vert 2t_0 +t_v\vert),
$$
and 
$$
\Re ( 2 \sigma_0 +w-1) = \pm 2 \mcL -c_f/\log (V_0^3 +\vert t_w\vert)\geq -c_f/ \log (2 + \vert t_w\vert),
$$
for sufficiently large $q$. Appealing one more time to \eqref{ineqZFR} and Lemma \ref{lm-factor} to bound the $E$--function,
we deduce 
\begin{equation*}
J_{1,1} \ll (\log q)^{O(1)}\Bigl( \int_\stacksum{\sigma_v=3\mcL}{ | t_v| \leq V_0} \Bigl\vert \frac {H_L (v)}{v}\Bigr\vert \, dv\Bigr) \Bigl( \int_{w \in \Gamma_w} \, \Bigl\vert \frac{H_L (w)}{w}\Bigr\vert dw \Bigr),
\end{equation*}
and, finally by Lemma \ref{boundforHL} and the inequality \eqref{ineqw}, we arrive at the inequality
\begin{equation}\label{J11<<}
J_{1,1} \ll (\log q)^{-10}.
\end{equation}
Gathering \eqref{J1=...}, \eqref{J12<<} and \eqref{J11<<}, we  obtain the bound
\begin{equation}\label{J1<<}
J_1\ll (\log q)^{-10}.
\end{equation}
Finally, by the definition of $s_0$ and the assumption \eqref{cond0}, we deduce from \eqref{gen-ineq} the bound
\begin{equation}\label{T(2s0)}
T(2s_0) \ll \log q.
\end{equation}
Combining \eqref{T(2s0)} with \eqref{J1<<} we complete the proof of \eqref{J1J2J3<} for $i=1$. 

\subsection{A first bound for $J_2$ and $J_3$}\label{secJ231} Recall that these quantities are defined in \eqref{dinner1}. For $v$ such that
$\sigma_v =3 \mcL$ and $\vert t_v \vert \leq V_0$, we have the following lower bounds
$$
\Re (2s_0+v-1) =\pm 2\mcL +3 \mcL\geq -c_f/\log (2 +\vert 2t_0+t_v\vert),
$$
$$
\Re ((1-v)-1)= -\sigma_v =-3 \mcL \geq -c_f/\log (2 + \vert t_v\vert),
$$
$$
\Re (2\sigma_0+v-1) =\pm 2 \mcL +3 \mcL \geq -c_f/\log (2 +\vert t_v\vert).
$$
Furthermore, under the same conditions, we have
$$
\vert v\vert \asymp \vert \pm 2 \mcL -v \vert \asymp \mcL + \vert t_v\vert.
$$
These remarks, when inserted in Lemma \ref{ZFR} (inequality
\eqref{ineqZFR}) and Lemma \ref{boundforHL}, give the following bound
for $J_2$
\begin{multline*} 
  J_2 \ll \int_{| t_v | \leq V_0} \frac{\mcL +|2t_0+t_v|}{1+ \mcL
    +|2t_0+t_v|}\cdot
  \frac{ \mcL +|t_v| }{1+ \mcL +|t_v|}\\
  \times \log^{A^*}\bigl(2 +|2t_0+t_v| \bigr) \log^{A^*}\bigl(2 +|t_v|
  \bigr) \min\Bigl( 1, \frac{ \mcL}{\mcL +|t _v|}\Bigr)^2
  \frac{dt_v}{(\mcL +|t_v|)^2},
\end{multline*} 
which is simplified into
\begin{multline}\label{1boundJ2}
  J_2 \ll \log^{A^*}(2+\vert t_0\vert) \int_{| t_v | \leq V_0}
  \frac{\mcL
    +|2t_0+t_v|}{1  +|2t_0+t_v|}\\
  \times \frac{ \mcL^2}{1 +|t_v|}\cdot \log^{2A^*}\bigl(2 +|t_v|
  \bigr) \frac{dt_v}{(\mcL +|t_v|)^3}.
\end{multline}
Proceeding similarly for $J_3$, we have
\begin{multline*}
J_3\ll
  \int_{ | t_v| \leq V_0} \frac{1 +\mcL +|t_v|}{\mcL +|t_v|}\cdot \frac{ \mcL +| 2t_0 +t_v|}{1 +
  \mcL +|2t_0 +t_v|}\cdot \mcL \\
 \times  \log^{A^*} \bigl( 2  +|t_v|\bigl) \cdot  \log ^{A^*}\bigl( 2  +|2t_0 +t_v| \bigr)\cdot
 \min\Bigl( 1, \frac{ \mcL}{ \mcL +|t _v|}\Bigr) \, \frac{dt_v}{ \mcL +|t_v|},
\end{multline*}
which simplifies into
\begin{multline}\label{1boundJ3}
  J_3\ll \log^{A^*}(2+\vert t_0\vert) \int_{ | t_v| \leq V_0} \frac{
    \mcL +| 2t_0 +t_v|}{1 +
    |2t_0 +t_v|}\cdot \mcL ^2\\
  \times \log^{2A^*} \bigl( 2 +|t_v|\bigl) \ \frac{1+\vert t_v\vert}{(
    \mcL +|t_v|)^3} \, dt_v.
\end{multline}

\subsection{Bound for $J_2$ and $J_3$: the case $1 \leq \vert t_0 \vert \leq \log q$}\label{secJ232}  In that case, the inequality \eqref{gen-ineq} asserts the truth of the bound
$$
T(2s_0)\ll \vert s_0 \vert^{O(1)}.
$$
Hence, in order to prove \eqref{J1J2J3<} under the above restriction on $t_0$, it is sufficient to prove the inequality
\begin{equation}\label{night}
J_i \ll \vert s_0\vert^{O(1)} \text{ for } i=2,\ 3.
\end{equation}
We write $\mathcal{L}_0=\log(2+|t_0|)$.
\par
\smallskip
\par
\textbf{Estimate of $J_2$}. From \eqref{1boundJ2}, we deduce that
\begin{align*}
  J_2&\ll \mathcal{L}_0^{A^*}
       \int_{| t_v | \leq V_0} \log^{2A^*} (2+| t_v|)
       \frac{\mcL^2 }{( \mcL +|t _v|)^3}\ dt_v\\
     &= \mathcal{L}_0^{A^*}
       \Bigl( \int_{| t_v | \leq \mcL} +\int_{\mcL < | t_v| \leq 1} +\int_{1 <|  t_v| \leq V_0}\Bigr)\log^{2A^*} (2+| t_v|)\frac{\mcL^2 }{( \mcL +|t _v|)^3}\ dt_v\\
     &\ll \mathcal{L}_0^{A^*}  \bigl( 1 + 1+  \mcL^2\bigr)\ll | s_0|.
\end{align*}
This proves \eqref{night} for $J_2$.
\par
\smallskip
\par
\textbf{Estimate of $J_3$}. From \eqref{1boundJ3}, we deduce that
\begin{align*}
  J_3&
       \ll  \mcL^2\ \mathcal{L}_0^{A^*}
       \Bigl(
       \int_{| t_v| \leq \mcL} +\int_{\mcL < |t_v| \leq 1}
       +\int_{ 1\leq| t_v| \leq V_0} 
       \Bigr)
       \log^{2A^*} \bigl( 2 +| t_v|\bigl) \cdot   
       \frac{1 +|t_v|}{( \mcL +|t _v|)^3}dt_v\\
     & \ll  \mcL^2\ \mathcal{L}_0^{A^*}\ \bigl( \mcL^{-2} + \mcL^{-2} + 1
       \bigr)\ll \mathcal{L}_0^{A^*} \ll \vert s_0\vert.
\end{align*}
This proves  \eqref{night} for $J_3$.

\subsection{Bound for $J_2$ and $J_3$: the case $\vert t_0 \vert \leq
  1$} \label{secJ233}

In that case we have the inequality
$$
T(2s_0) \ll \frac{1}{\mcL+\vert t_0\vert}
$$
as a direct consequence of \eqref{gen-ineq}. Hence, in order to prove \eqref{J1J2J3<} under the above restriction on $t_0$, it is sufficient to prove the inequality
\begin{equation}\label{night1}
J_i \ll  \mcL +\vert t_0\vert \text{ for } i=2,\ 3.
\end{equation}
\par
\smallskip
\par
\textbf{Estimate of $J_2$}. We start from \eqref{1boundJ2}, which in
that case simplifies into
$$
J_2 \ll
 \int_{| t_v | \leq V_0}  \frac{\mcL +|2t_0+t_v|}{1  +|2t_0+t_v|}\cdot 
\frac{ \mcL^2}{1 +|t_v|}\cdot  \log^{2A^*}\bigl(2 +|t_v| \bigr)
 \frac{dt_v}{(\mcL +|t_v|)^3}.
$$
 
We split this integral in three ranges
$$
\vert t_v \vert \leq \mcL, \quad \mcL\leq \vert t_v \vert \leq
1,\quad\text{ and }\quad 1\leq \vert t_v \vert \leq V_0.
$$
We have
$$
\int_{\vert t_v \vert\leq \mcL}(\cdots) \ll \frac{\mcL +\vert
  t_0\vert}{1+ \vert t_0\vert}\int_{\vert t_v \vert \leq \mcL}
\mcL^{-1}\ dt_v \ll \mcL + \vert t_0\vert.
$$
For the second range, we have
\begin{align*}
\int_{\mcL\leq \vert t_v\vert \leq 1}(\cdots)
&\ll\mcL^2\int_{\mcL}^{1}\frac{\mcL+|t_0|+t_v}{t_v^3}{dt_v}
\\
&\ll
\mcL^2\Bigl(\frac{\mcL+|t_0|}{\mcL^{2}}+\mcL^{-1} \Bigr)\ll
\mcL+|t_0|
\end{align*}
and for the last one we get
\begin{multline*}
\int_{1\leq |t_v|\leq V_0}(\cdots)\ll 
\mcL^2 \int_{ \vert t_v \vert \geq 1}
\frac{\mcL+|t_v+2t_0|}{1+|t_v+2t_0|}\log^{2A^*}(2+ | t_v| )
\frac{dt_v}{\vert t_v\vert^4}\\
\ll \mcL^2\ll \mcL+|t_0|.
\end{multline*}
Gathering the three inequalities above, we complete the proof of
\eqref{night1} in the case $i=2$.
\par
\smallskip
\par
\textbf{Estimate of $J_3$}. In the case of $J_3$, we first simplify
\eqref{1boundJ3} into
$$
J_3\ll  \mcL^2
  \int_{ | t_v| \leq V_0}  \frac{ \mcL +| 2t_0 +t_v|}{1 +
  |2t_0 +t_v|}\cdot 
   \log^{2A^*} \bigl( 2  +|t_v|\bigl) \  \frac{1+\vert t_v\vert}{( \mcL +|t_v|)^3} \, dt_v,
$$
and we again split this integral in three parts, obtaining
\begin{align*}
  \mcL^2 \int_{\vert t_v \vert \leq \mcL}(\cdots)
  & \ll \mcL^2 \int_0^\mcL (\mcL +\vert t_0    \vert)\,\mcL^{-3}
    dt_v \ll \mcL +\vert
    t_0\vert,
  \\
  \mcL^2 \int_{\mcL \leq \vert t_v \vert \leq 1}(\cdots)
  &\ll \mcL^2\int_\mcL^1 (\mcL+ \vert t_0\vert+ t_v )
    \frac{dt_v}{t_v^3}\ll \mcL +\vert t_0\vert,\\
  \mcL^2 \int_{1\leq \vert t_v \vert \leq V_0}(\cdots)
  &\ll \mcL^2 \int_1^\infty \log^{2A^*} (2
    +t_v)\,\frac{dt_v}{t_v^2}\ll \mcL^2\ll \mcL +\vert t_0\vert.
\end{align*}
Gathering the three above inequalities, we complete the proof of
\eqref{night1} for $i=3$. The proof of \eqref{night1} is now complete.
Hence the proof of \eqref{157} is now complete, and so is the proof of
Proposition \ref{11.1}.

\begin{remark}\label{endremarkchap8}The informed reader will have
  noticed that the proof presented here is slightly different from
  similar second moment estimates found in other works (for
  instance,~\cite[Propositions 4 and 5]{KMDMJ}). These other arguments
  (following earlier ideas of Selberg) made key use of the positivity
  of certain complicated terms to avoid having to evaluate them too
  precisely (see~\cite[(67), (68)]{KMDMJ}). In the present work
  -- precisely the proof of Proposition \ref{11.1} -- positivity is not
  so evident so we have to estimate the corresponding sums directly.
\end{remark}

\chapter{A conjecture of Mazur-Rubin concerning modular symbols}
\label{ch-modular}

\section{Introduction}

In this chapter, we assume that $f$ is a holomorphic primitive cusp
form of weight $2$ and level $r$.  We recall that for $q\geq 1$ and
$(a,q)=1$, the modular symbol $\langle a/q\rangle_f$ is defined by
$$
\msym{\frac aq}_f=2\pi i\int_{i\infty}^{a/q}f(z)dz =2\pi
\int_{0}^{\infty}f\Bigl(\,\frac{a}q+iy\,\Bigr)\,dy
$$
and that it only depends on the congruence class $a\mods q$.

In this chapter we investigate some correlation properties of the family
$$
\msym{\frac aq}_f,\quad a\in(\Zz/q\Zz)^\times
$$
when $q$ is a prime number. In particular, we will prove
Theorem~\ref{th-MaRu} concerning the variance of modular symbols.

Our main ingredient is the Birch-Stevens formula that relates the
modular symbols to the central values of the twisted $L$-functions.

\begin{lemma}\label{lm-modmellin}
  For any primitive Dirichlet $\chi\mods{q}$, we have
$$
L\Bigl(\ftchi,\frac12\Bigr)=
\frac{\eps_\chi}{q^{1/2}} \underset{a\, (\text{{\rm mod }} q)}{\left.\sum\right.^{\ast}}\chi(-\ov a)\msym{\frac aq}_f.
$$
\end{lemma}

\begin{proof}
Observe that since $f$ has real Fourier coefficients, we have
$$\ov{f(x+iy)}=f(-x+iy),\ x,y\in\Rr,\ y>0$$
so that
\begin{equation}\label{modconj}
  \ov{\msym{\frac aq}_f}=
  2\pi\int_{0}^{\infty}f\Bigl(\,-\frac{a}q+iy\,\Bigr)\,dy={\msym{\frac{-a}q}_f}.
\end{equation}

Now denote
\begin{equation}\label{modspm}
  \msym{\frac aq}^\pm_f=\frac 12\,\Bigl(\, 
  \msym{\frac aq}_f\pm\msym{\frac{-a}q}_f\, \Bigr)	
\end{equation}
the even and odd parts of the modular symbols.  The Birch-Stevens
formula (see~\cite{Pollack}*{(2.2)} or~\cite{MTT}*{(8.6)}) states that
$$
L\Bigl(\ftchi,\frac12\Bigr)=
\frac{1}{\eps_{\ov\chi}\,q^{1/2}}\underset{a\, (\text{{\rm mod }} q)}{\left.\sum\right.^{\ast}}\ov\chi(a)\msym{\frac
  aq}^\pm_f=\frac{\chi(-1)\,\eps_\chi}{q^{1/2}}\underset{a\, (\text{{\rm mod }} q)}{\left.\sum\right.^{\ast}}\chi(\ov
a)\msym{\frac aq}^\pm_f
$$ 
where $a\ov a\equiv 1\mods q$, $\eps_\chi$ is the normalized Gau\ss\
sum of $\chi$ (cf.\ \eqref{gauss}), and the ``exponent'' $\pm$ is $\chi(-1)$.

Inserting \eqref{modspm} we obtain
$$
L\bigl(\ftchi,\frac12\bigr)=
\frac{\eps_\chi}{q^{1/2}}\underset{a\, (\text{{\rm mod }} q)}{\left.\sum\right.^{\ast}}\chi(-\ov a)\msym{\frac aq}_f,
$$
as claimed.
\end{proof}

By performing discrete Mellin inversion, we will be able to use our
results on moments of twisted central values to evaluate
asymptotically the first and second moments of the modular symbols,
and in fact also correlations between modular symbols for two cusp
forms. 

We define\label{pg-mfq}
$$
M_f(q)=\frac{1}{\vphi(q)}\underset{a\, (\text{{\rm mod }} q)}{\left.\sum\right.^{\ast}}\msym{\frac aq}_f.
$$
Moreover, we define
$$
f_q(z)=f(qz)=\sum_{n\geq 1}\lambda_f(n)n^{1/2}e(nqz).
$$
\par
If $g$ is a holomorphic primitive cusp form of weight $2$ and level
$r'$ coprime to $q$, and $u$, $v$ are integers coprime to $q$, we
define
$$
C_{f,g}(u,v;q)=\frac{1}{\varphi(q)}\underset{a\, (\text{{\rm mod }} q)}{\left.\sum\right.^{\ast}}
\Bigl(\msym{\frac{au}{q}}_f-M_f(q)\Bigr)
\overline{\Bigl(\msym{\frac{av}{q}}_g-M_g(q)\Bigr)}
$$
(where here and below, the sum is over invertible residue classes
modulo $q$).  In particular, note that the variance in
Theorem~\ref{th-MaRu} is $V_f(q)=C_{f,f}(1,1;q)$, so the second part
of the next result implies that theorem:

\begin{theorem}\label{th-MaRu2}
  Suppose that $q$ is prime. Write the levels $r$ and $r'$ of $f$ and
  $g$ as $r=\rho\delta$ and $r'=\rho'\delta$ where $\delta=(r,r')$ and
  $(\rho,\rho')=1$.
\par
\emph{(1)} We have
$$
M_f(q)=\Bigl(\,\frac{q^{1/2}}{q-1}\cdot
\frac{L_q(f_q,1/2)}{L_q(f,1/2)}-\frac{1}{q-1}\,\Bigr)L(f,1/2)=O(q^{-1/2}).
$$
\par
\emph{(2)} We have 
$$
  C_{f,g}(u,v;q)=\frac{q}{\vphi(q)^2}
    \sums_{\chi\mods{q}} L(\ftchi,1/2)\ov{L(\gtchi,1/2)}
    \chi(u)\ov{\chi(v)}.
$$
\par
\emph{(3)} In particular, if $r=r'$ and $\eps(f)\eps(g)=-1$, then
$C_{f,g}(1,1;q)=0$. Otherwise
\begin{gather*}
  C_{f,g}(1,1;q)=\gamma_{f,g}\frac{L^{\ast}(f\otimes
    g,1)}{\zeta(2)}+O(q^{-1/145}) \text{ if $f\not=g$},
  \\
  C_{f,f}(1,1;q)= 2 \prod_{p\mid r}(1+p^{-1})^{-1}
  \frac{L^{\ast}(\symf,1)} {\zeta(2)}\log q+\beta_f+O(q^{-1/145}),
\end{gather*}
where  $\beta_f$ is a constant, and
$$
\gamma_{f,g}=1+\eps(f)\eps(g)
\frac{\lambda_f(\rho)\lambda_{g}(\rho')}{\sqrt{\rho\rho'}}
$$
is a non-zero constant.
\end{theorem}

Part (3) with $f=g$ and $u=v=1$ confirms a conjecture of Mazur and
Rubin, as stated by Petridis and Risager~\cite{PeRi}*{Conj.\ 1.1}, in
the case of prime moduli $q$. Note that their statement of the
conjecture involves a quantity which they denote $L(\symf ,1)$ and
which should be interpreted as our $L^{\ast}(\symf ,1)$ (although they
do not state this formally, it is clear from their proof
of~\cite{PeRi}*{Th.\ 1.6} in Section 8 of loc.\ cit., especially Section
8.2.1). In fact, Theorem~\ref{th-MaRu2} and the computations in
Section~\ref{symsquare} show that the conjecture would not hold in
general if $L(\symf, 1)$ is interpreted as the special value of the
automorphic (or ``motivic'') symmetric square.

\begin{remark}
  If $f\not=g$ and $u=v=1$, and either $r\not=r'$ or
  $\eps(f)\eps(g)=1$, then we find by (3) a little bit of correlation
  between modular symbols as $q\to +\infty$: the modular symbols
  related to $f$ and $g$ do not become asymptotically
  independent.
\par
On the other hand, if $r=r'$, $\eps(f)\eps(g)=-1$ and $u=v=1$, the
correlation vanishes exactly, but this fact will not persist in
general if $u$ or $v$ is not $1$ modulo $q$, as shown by our
evaluation of the twisted second moment
(Theorem~\ref{th-second-final}). For instance, using
Lemma~\ref{SecondMomentEvaluationForProducts}, we get a formula with
non-zero leading term for suitable choices of $u$ and $v$, even if
$r=r'$ and $\eps(f)\eps(g)=-1$.
\end{remark}

\section{Proof of the theorem}

We observe first that $M_f(q)\in\Rr$ because of the relation
\eqref{modconj}.  We compute $M_f(q)$ exactly by analytic continuation
from a region of absolute convergence, using additive twists of
modular forms.

Let $a$ be coprime to $q$. We have
$$
yf\bigl(\,z+\frac aq\,\bigr)=y\sum_{n\geq
  1}\lf(n)n^{1/2}e\Bigl(n\Bigl(x+\frac{a}{q}\Bigr)\Bigr)
\exp(-2\pi ny).
$$
For any complex number $s$, we define
$$
  \msym{\frac aq}_{s,f}
  =2\pi \int_0^\infty yf\Bigl(\,\frac aq+iy\,\Bigr)y^{s}\frac{dy}{y}.
$$
As a function of $s$, this expression is holomorphic in the whole
complex plane. On the other hand, for $\Re(s)>1$, we have
\begin{align*}
  \msym{\frac{a}{q}}_{s,f}  &=(2\pi)^{1/2}\sum_{n\geq 1}\lf(n)e\Bigl(n\frac aq \Bigr)\int_0^\infty (2\pi ny)^{1/2}y^{1/2+s}\exp(-2\pi ny)\frac{dy}y\\
                            &=(2\pi)^{-s}\sum_{n\geq 1}\frac{\lf(n)e(n\frac aq)}{n^{1/2+s}}\int_0^\infty y^{1+s}e^{-y}\frac{dy}y\\
                            &=(2\pi)^{-s}\Gamma(1+s)L(f,a,\demi +s)
\end{align*}
where\label{pg-lfa}
$$
L(f,a,s)=\sum_{n\geq 1}\frac{\lf(n)e(n\frac aq)}{n^{s}}
$$
when the series converges absolutely.

Expressing the additive character in terms of multiplicative
characters, it follows that the series $L(f,a,s)$ has analytic
continuation to $\Cc$. Hence the identity above holds for all
$s\in\Cc$. In particular, we obtain
$$
\msym{\frac aq}_f=\msym{\frac aq}_{0,f}=L(f, a, 1/2).
$$
Since $q$ is prime, we deduce by direct computation that
\begin{multline*}
 \underset{a\, (\text{{\rm mod }} q)}{\left.\sum\right.^{\ast}}\msym{\frac aq}_{s, f} =(2\pi)^{-s} \Gamma(1+s) \Bigg(q\sum_{n\geq
    1}\frac{\lf(qn)}{(qn)^{s+1/2}}-L(f,s+1/2)\Bigg)\\
 =(2\pi)^{-s} \Gamma(1+s) \Bigl(\,q^{1/2}\frac{L_q(f_q,s+ 1/2)}{L_q(f, s+1/2)}-1\,\Bigr)L(f,s+ 1/2)
\end{multline*}
in $\Re s > 1/2$, where
\begin{displaymath}
\begin{split}
 L_q(f_q,1/2)=\sum_{\alpha\geq
  0}\frac{\lf(q^{\alpha+1})}{q^{\alpha/2}}
 & =\lf(q)+
    \frac{ L_q(f,1/2)}{q^{1/2}}\Bigl(\lambda_f(q^2) -\frac{\lf(q)}{q^{1/2}}\Bigr)\\
 & =\lf(q)+O(q^{-1/2}).
  \end{split}
  \end{displaymath}
Hence
$$
M_f(q)=\Bigl(\,\frac{q^{1/2}}{q-1}\cdot
\frac{L_q(f_q,1/2)}{L_q(f,1/2)}-\frac{1}{q-1}\,\Bigr)L(f,1/2)=O(q^{-1/2}).
$$
This proves the first part of Theorem~\ref{th-MaRu2}.
\par

 
Next, let $u$ and $v$ be integers coprime to $q$. From
Lemma~\ref{lm-modmellin}, we derive
\begin{multline*}
  \sums_{\chi\mods{q}} L(\ftchi,1/2)\ov{L(\gtchi,1/2)}
  \chi(u)\ov{\chi(v)} \\
  =\frac{1}{q}\sums_\chi\chi(u)\overline{\chi(v)}\underset{a, a'\, (\text{{\rm mod }} q)}{\left.\sum\right.^{\ast}}
  \chi(\ov a)\ov{\chi(\ov a')}\msym{\frac aq}_f\ov{\msym{\frac {a'}q}}_g\\
  =\frac{1}{q}\sum_\chi\chi(u)\overline{\chi(v)}
 \underset{a, a'\, (\text{{\rm mod }} q)}{\left.\sum\right.^{\ast}}\chi(\ov a) \ov{\chi(\ov a')}\msym{\frac
    aq}_f\ov{\msym{\frac {a'}q}}_g-
  \frac{\vphi(q)^2}qM_f(q)\ov{M_g(q)}\\
  =\frac{\vphi(q)}{q}\underset{\substack{a, a'\, (\text{{\rm mod }} q)\\u\bar{a}\equiv v\bar{a}' \, (\text{mod } q)}}{\left.\sum\right.^{\ast}}   
  \msym{\frac aq}_f\ov{\msym{\frac
      {a'}q}}_g-\frac{\vphi(q)^2}qM_f(q)\ov{M_g(q)},
\end{multline*}
hence, putting $b=a\bar{u}=a'\bar{v}$, we get
\begin{align*}
  C_{f,g}(u,v;q)
  &=\frac{1}{\vphi(q)}
   \underset{b\, (\text{{\rm mod }} q)}{\left.\sum\right.^{\ast}}\Bigl(\,\msym{\frac{ub}{q}}_f-M_f(q)\,\Bigr)
    \ov{\Bigl(\,\msym{\frac{vb}{q}}_g-M_g(q)\,\Bigr)}
  \\
  &=\frac{q}{\vphi(q)^2}
    \sums_{\chi\mods{q}} L(\ftchi,1/2)\ov{L(\gtchi,1/2)}
    \chi(u)\ov{\chi(v)},
\end{align*} 
which is the formula in Part (2) of the Theorem~\ref{th-MaRu2}.
\par
If $r=r'$, $u=v=1$ and $\eps(f)\eps(g)=-1$, then the second moment
vanishes exactly (see the last part of Theorem~\ref{th-second-final}),
which proves the first part of Part (3), and otherwise, we obtain the
last statement from Theorem~\ref{thsecondmoment} and
Proposition~\ref{pr-mt}.



\section{Modular symbols and trace functions}

As we have seen, the modular symbol $\langle a/q\rangle_f$, as a
function of $a$, depends only on the congruence class $a\mods q$ and
therefore defines a function on $\Zz/q\Zz$, where we put
$\langle 0/q\rangle_f=0$.

In the previous sections, we discussed how this function correlates
either with the constant function $1$, or with itself, or with the
modular symbol attached to another modular form. In this section, we
will see that we can also evaluate easily the correlations of modular
symbols and trace functions $t\colon \Fq\to\Cc$, as described in
Section~\ref{sec-trace}.

We consider here the correlation sums\label{pg-cft}
$$
C_{f}(t)=\frac{1}{\vphi(q)}\underset{a\, (\text{{\rm mod }} q)}{\left.\sum\right.^{\ast}}\msym{\frac aq}_f\ov{t(a)}.
$$
We will prove that these are small, except in very special cases. This
means that trace functions do not correlate with modular symbols.

\begin{proposition}\label{pr-tr-sym}
  Let $t$ be the trace function of a geometrically irreducible
  $\ell$-adic sheaf $\mcF$. We assume that $\mcF$ is not geometrically
  isomorphic to an Artin-Schreier sheaf or to the pull-back of such a
  sheaf by the map $x\mapsto x^{-1}$. Then we have
$$
C_f(t)\ll q^{-1/8+\eps},
$$
for any $\eps>0$, where the implied constant depends only on $\eps$,
$f$ and (polynomially) on the conductor of $\mcF$.
\end{proposition}

\begin{remark}
  The assumption on the sheaf holds for all the examples in
  Example~\ref{ex-trace}, except for $t(x)=e(f(x)/q)$ if the
  polynomial $f$ has degree $\leq 1$.
\end{remark}

\begin{proof}
  By Lemma \ref{lm-modmellin} and Theorem \ref{th-MaRu2}(1), we have
\begin{align*}
  C_f(t)&=\frac{1}{\vphi(q)}\underset{a\, (\text{{\rm mod }} q)}{\left.\sum\right.^{\ast}}\msym{\frac aq}_f\ov{t(a)}=
  \frac{1}{\vphi(q)^2}\sum_{\chi\mods{q}}\underset{a, a'\, (\text{{\rm mod }} q)}{\left.\sum\right.^{\ast}}\msym{\frac
    aq}_f\chi(\ov a)\chi(a')\ov{t(a')}
  \\
  &= \frac{q^{1/2}}{\vphi(q)^{3/2}}  \sums_{\chi\mods{q}} L(\ftchi,1/2) \chi(-1) \overline{\eps_\chi \widetilde{t}(\bar{\chi})}  + O\left(\frac{1}{q^{3/2}}\right)\\
  & = \frac{1}{\vphi(q)} \sums_{\chi\mods{q}} L(\ftchi,1/2) \chi(-1)
  \overline{\eps_\chi \widetilde{t}(\bar{\chi})} +
  O\left(\frac{1}{q}\right).
\end{align*}
\par
We compute that
\begin{align*}
  \eps_{\chi}\widetilde{t}(\ov{\chi})
  &=
    \frac{1}{\sqrt{q}}\sum_{x}\chi(x)e\Bigl(\frac{x}{p}\Bigr)
    \sum_y\ov{\chi(y)}\,t(y)\\
  &=\frac{1}{\sqrt{q}}\sum_{a}\ov{\chi(a)}
    \sum_{y/x=a}e\Bigl(\frac{x}{p}\Bigr)
    t(y),
\end{align*}
hence
$$
\ov{\eps_{\chi}\widetilde{t}(\ov{\chi})}= \frac{1}{\sqrt{q}}
\sum_{a\in\Fqt}\chi(a)\tau(a)
$$
where
$$
\tau(a)=\frac{1}{\sqrt{q}}\sum_{xy=a} e\Bigl(-\frac{\bar{x}}{p}\Bigr)
\ov{t(y)}
$$
is the convolution of $\ov{t}$ and $x\mapsto e(-x^{-1}/p)$. In other
words, $\chi\mapsto \ov{\eps_{\chi}\widetilde{t}(\ov{\chi})}$ is the
discrete Mellin transform of this convolution.
\par
We distinguish two cases. If $\mcF$ is not geometrically isomorphic to
a Kummer sheaf, then our assumptions on $\mcF$ imply that this
convolution is the trace function of a Mellin sheaf $\mcG$ with
conductor bounded polynomially in terms of $\cond(\mcF)$, and that
$\mcG$ is not geometrically isomorphic to $[x\mapsto a/x]^*\HYPK_2$
for any $a\in\Fqt$ (see Lemma~\ref{lm-except-symbols}). By
Theorem~\ref{th-m1-mellin}, we have therefore
$$
\frac{1}{\vphi(q)} \sums_{\chi\mods{q}} L(\ftchi,1/2) \chi(-1)
  \overline{\eps_\chi \widetilde{t}(\bar{\chi})}\ll q^{-1/8+\eps}
$$
for any $\eps>0$, and hence $C_f(t)\ll
q^{-1/8+\eps}$ for any $\eps>0$.
\par
In the case of a Kummer sheaf, we have
$t(x)=\alpha\chi_0(x)$ for some
$\alpha\in\Cc$ and some non-trivial multiplicative character
$\chi_0$ modulo $q$ and
$$
\wtilde{\chi}_0(\bar{\chi})=
\begin{cases}
  \alpha\varphi(q)^{1/2}&\text{ if } \chi=\chi_0\\
  0&\text{ otherwise,}
\end{cases}
$$
so that we get
$$
C_f(\chi_0)=\frac{1}{\vphi(q)^{1/2}} \chi_0(-1)\eps_{\chi_0}
L(f\otimes \chi_0,\demi)+O(q^{-1})\ll q^{-1/8+\eps}
$$
by the subconvexity estimate of Blomer and Harcos~\cite{BH}*{Th.\ 2}.
\end{proof}


\begin{remark}
For completeness, we consider the correlations in the exceptional
cases excluded in the previous proposition.  We assume that there
exists $l\in\Fqt$ such that either
\begin{equation}\label{eq-t1}
  t(x)=e\Bigl(-\frac{\ov lx}q\Bigr),\quad\quad x\in \Fq
\end{equation}
or
\begin{equation}\label{eq-t2}
  t(x)=e\Bigl(\frac{l\bar{x}}q\Bigr),\quad\quad x\in\Fqt,\quad\quad t(0)=0.
\end{equation}
(note that if $l=0$, the correlation sum is just the mean-value
$M_f(q)$ that we already investigated).

We follow the steps of the proof of Proposition~\ref{pr-tr-sym} for
these specific functions. In both cases, the Mellin transform
$\wtilde t$ is a multiple of a Gau\ss\ sum. More precisely, we obtain
$$
\ov{\eps_\chi\wtilde
  t(\ov\chi)}=\Bigl(\frac{q}{\vphi(q)}\Bigr)^{1/2}\chi(l),\quad\quad
\ov{\eps_\chi\wtilde
  t(\ov\chi)}=\Bigl(\frac{q}{\vphi(q)}\Bigr)^{1/2}\chi(l)\eps^{-2}_\chi
$$
in the case of~(\ref{eq-t1}) and of~(\ref{eq-t2}), respectively.
Using the notation of Chapter \ref{ch-first}, we therefore have
$$
C_f(t)=\Bigl(\frac{q}{\vphi(q)}\Bigr)^{1/2}\mcL(f;l,0),\quad\quad
C_f(t)=\Bigl(\frac{q}{\vphi(q)}\Bigr)^{1/2}\mcL(f;l,-2),
$$
respectively. By Corollary \ref{cor-k=-2}, we  conclude that
$$
C_f(t)=\frac{\lf(\ov l_q)}{\ov
  l_q^{1/2}}+O_{f,\eps}(q^{-1/8+\eps}),\quad\quad
C_f(t)=\eps(f)\frac{\lf((lr)_q)}{(lr)_q^{1/2}}
+O_{f,\eps}(q^{-1/8+\eps}),
$$
for any $\eps>0$, respectively.
\end{remark}

\chapter*{Notation index}

We list some of the notation used in this book that may not be
standard. Further notation and conventions are explained in
Section~\ref{21}. Some notation that is local to a single chapter are
omitted.

\newcommand{\pg}[1]{p. \pageref{#1}}
\newcommand{\pgb}[1]{\pageref{#1}}
\par
\bigskip
\par

{\small
  \begin{longtable}{lll}
    $\mathcal{F}_N$ & family of modular form &
                                               \pg{sec-general}
    \\
    $\expect_N(\cdot)$ & averaging over~$\mathcal{F}_N$ &
                                                          \pg{sec-general}
    \\
    $\proba_N(\cdot)$ & probability for~$\mathcal{F}_N$ &
                                                          \pg{sec-general}
    \\
    $M(f)$ & mollifier & \pg{pg-mollifier}
    \\
    $R(f)$ & resonator & \pg{pg-resonator}
    \\
    $\chi_r$ & trivial character modulo~$r$ & \pg{intro}
    \\
    $r$ & level of~$f$ & \pg{intro}
    \\
    $\mathcal{F}_q$ & family of Dirichlet characters modulo~$q$ &
                                                                  \pg{pg-fq}
    \\
    $\varphi^*(q)$ & number of primitive characters modulo~$q$ &
                                                                 \pg{pg-fq}
    \\
    $\theta(f\otimes \chi)$ & angle of root numbers in~$\Rr/2\pi\Zz$ & \pg{deftheta}
    \\
    $\widetilde{t}_e(\chi)$ & Evans sums & \pg{deftheta}
    \\
    $\mathrm{rk}_{an}(f\otimes\chi)$ & analytic rank & \pg{pg-an}
    \\
    $\displaystyle{\msym{\frac aq}_f}$ & modular symbol & \pg{pg-mr}
    \\
    $L^{\ast}(\symf,s)$ & imprimitive symmetric square & \pg{th-MaRu}
    \\
    $\mcL(f,s;\l,k)$, $\mcL(f;\l,k)$ & twisted first moments & \pg{moments}
    \\
    $\mcQ(f,s;\l,\ell')$, $\mcQ(f;\l,\ell')$
                    & twisted second moment & \pg{moments}
    \\
    $\eps_{\chi}$ & normalized Gau\ss\ sum & \pg{moments}
    \\
    $M(f\otimes\chi,s;\uple{x}_L)$ & mollifier & \pg{moments}
    \\
    $\bar{\ell}_q$ & integer in $[1,q]$ such that
                     $\ell\bar{\ell}_q\equiv 1\mods{q}$ & \pg{moments}
    \\
    $\mathrm{MT}(f,g,s;\ell,\ell')$,
    $\mathrm{MT}^{\pm}(f,g,s;\ell,\ell')$& main term for second moment
                                             & \pg{thsecondmoment},
                                               \pgb{th-second-final}
    \\
    $L_p(f,s)$ & local factor at~$p$ & \pg{hecke}
    \\
    $\alpha_{f,i}(p)$, $\alpha_{f\otimes g}(p)$ & Satake parameters
                                             & \pg{hecke}, \pgb{pg-explicit}
    \\
    $\lambda_f(n)$ & Hecke eigenvalues & \pg{hecke}
    \\
    $\kappa_f$ & spectral parameter & \pg{hecke}
    \\
    $\mu_{f,i}$ & archimedean parameters & \pg{hecke}
    \\
    $\mfa$ & parity parameter & \pg{lm-fneq}
    \\
    $\eps(f\otimes \chi)$, $\eps(f)$, $\eps(f\otimes g)$
                    & root numbers & \pg{lm-fneq}, \pgb{pg-fg}
    \\
    $\Lambda_f(n)$, $\Lambda_{f\otimes \chi}(n)$ & von Mangoldt
                                                   functions &\pg{pg-explicit}
    \\
    $L^*(f\otimes g,s)$ & imprimitive Rankin-Selberg convolution & \pg{lstarfg}
    \\
    $Q(\pi)$, $q(\pi)$, $Q(\pi\otimes\pi')$, $q(\pi\otimes\pi')$ & analytic conductor,
                                                                   conductor &
                                                                               \pg{pg-conductor}
    \\
    $V_{f,\pm, s}(y)$, $W_{f,g,\pm,s}$ & weight functions & \pg{pr-approx}
    \\
    $\eps(f,\pm, s)$, $\eps(f,g,\pm,s)$ & coefficients &
                                                         \pg{pr-approx}
    \\
    $\widetilde{W}_{\pm}$ & Voronoi transform of~$W$ & \pg{221}
    \\
    $\mathcal{J}_{\pm}$ & variants of Bessel functions & \pg{221}
    \\
    $\widehat{K}$ & normalized discrete Fourier transform of~$K$ &
                                                                   \pg{pg-what}
    \\
    $\bessel{K}$ & discrete Bessel transform of~$K$ & \pg{pg-what}
    \\
    $\mu_f(n)$ & Dirichlet convolution inverse of~$\lambda_f(n)$ &
                                                                   \pg{pg-factor}
    \\
    $\Kl_k(m;q)$ & hyper-Kloosterman sums & \pg{eq-hypk}
    \\
    $\cond(\sheaf{F})$ & conductor of~$\sheaf{F}$ & \pg{sec-trace}
    \\
    $\widetilde{t}(\chi)$ & discrete Mellin transform & \pg{pg-mellin}
    \\
    $t_1\star t_2$ & multiplicative convolution & \pg{pg-convol}
    \\
    $\theta_{\chi}$ & Frobenius conjugacy class & \pg{pg-thetachi}
    \\
    $\widetilde{t}_{rr}(\chi)$ & Rosenzweig-Rudnick sum & \pg{pg-rr}
    \\
    $\mcL^{\pm}(f;\ell,k)$ & even and odd twisted first moments
                                             & \pg{Lfcomp}
    \\
    $L(f\times g, 2s, u;\ell',\ell)$ & auxiliary Dirichlet series &
                                                                    \pg{Lfgdef}
    \\
    $\mathrm{ET}(f,g;\ell,\pm \ell')$ & error terms in second moments
                                             & \pg{pg-et}
    \\
    $\mathscr{L}(f;\bfx_L,\psi)$ & auxiliary first moment & \pg{8.2}
    \\
    $r(p)$ & auxiliary multiplicative function &
                                                 \pg{DefinitionResonator}
    \\
    $o^{\star}(1)$ & short-hand asymptotic notation &
                                                      \pg{SmallTailsLemma}
    \\
    $\lambda_f^{\ast}$, $\lambda_g^{\ast}$ & auxiliary arithmetic
                                             functions &
                                                         \pg{lambdaastdef}
    \\
    $R(\chi)$ & resonator & \pg{DefinitionResonatorPolynomial}
    \\
    $A_f(\chi)$ & amplifier & \pg{DefinitionAmplifier}
    \\
    $\widehat{\phi}(s)$ & Fourier-Laplace transform & \pg{pg-phi}
    \\
    $R(\ell_1,\ell_2,s)$ & auxiliary Dirichlet series & \pg{pg-r}
    \\
    $M_f(q)$ & average of modular symbols & \pg{pg-mfq}
    \\
    $f_q$ & $f(qz)$ & \pg{pg-mfq}
    \\
    $C_{f,g}(u,v;q)$ & correlation of modular symbols & \pg{pg-mfq}
    \\
    $L(f,a,s)$ & additive twist of $L(f,s)$ & \pg{pg-lfa}
    \\
    $C_f(t)$ & correlation of~$t$ with modular symbols & \pg{pg-cft}
  \end{longtable}
}


\end{document}